\documentclass[11pt, a4paper,reqno]{amsart}
\pdfoutput=1
\usepackage{tikz,rotating,latexsym,bm,stmaryrd,caption,}
\usetikzlibrary{positioning,intersections,decorations,shadings}
\usetikzlibrary{shapes}
\usetikzlibrary{arrows}
\usetikzlibrary{patterns}
\usepackage{tkz-euclide}
\usetikzlibrary{calc}
\usetikzlibrary{shapes.geometric}
\tikzset{snake it/.style={decorate, decoration=snake}}
\tikzset{zigzag/.style={decorate, decoration=zigzag}}

\definecolor{ao(english)}{rgb}{0.0, 0.5, 0.0}

\pgfdeclareverticalshading{rainbow}{100bp}
{color(0bp)=(orange);color(25bp)=(orange);
color(50bp)=(magenta); color(75bp)=(cyan!50);
 color(100bp)=(cyan!50)}

\usetikzlibrary{decorations.pathmorphing}

\definecolor{eng}{rgb}{0.0, 0.5, 0.0}
\definecolor{apple}{rgb}{0.55, 0.71, 0.0}
\definecolor{cadmium}{rgb}{0.0, 0.42, 0.24}
\definecolor{darkspringgreen}{rgb}{0.09, 0.45, 0.27}
\definecolor{amethyst}{rgb}{0.6, 0.4, 0.8}
\definecolor{ao}{rgb}{0.0, 0.0, 1.0}
\definecolor{atomictangerine}{rgb}{1.0, 0.6, 0.4}
\definecolor{carmine}{rgb}{0.59, 0.0, 0.09}

\definecolor{toggle}{rgb}{1.0, 0.94, 0.96}

\usepackage[normalem]{ulem}

 \newcommand{\Bee}{b}

 \newcommand{\csigma}{{{{\color{sigmacolor!90}\bm\sigma}}}}

\newcommand{\eps}{ \varepsilon}

\newcommand{\y}{{\underline{y}}}
\newcommand{\x}{{\underline{x}}}

\usepackage{standalone}

\usepackage{etex}
 
\tikzset{
  variable line width/.style={
    every variable line width/.append style={#1},
    to path={%
      \pgfextra{%
        \draw[every variable line width/.try,line width=\pgfkeysvalueof{/tikz/thickness}] (\tikztostart) -- (\tikztotarget);
      }%
      (\tikztotarget)
    },
  },
  thickness/.initial=0.6pt,
  every variable line width/.style={line cap=round, line join=round},
}

\usepackage{todonotes}

\usepackage{tikz}
\usetikzlibrary{matrix,  intersections, calc, decorations.pathreplacing} 

\usepackage{tikz-cd}

\newlength{\superthick}
\newlength{\cornerradius}
\tikzstyle{corner}=[rounded corners=\cornerradius]
\tikzstyle{dot}=[circle, inner sep=0pt, minimum size=4.8pt]
\tikzstyle{string}=[line width=\superthick]
\tikzstyle{std}=[string,dash pattern=on 0.9pt off 0.9pt]
\definecolor{realcyan!50}{rgb}{0,1,1}

\tikzstyle{unmarkedshading}=[copperred, pattern=north east lines, pattern color = copperred!40!white, rounded corners]
\tikzstyle{conditionalshading}=[plum, pattern=crosshatch dots, pattern color = plum, rounded corners]

 \usepackage{amsmath,amsthm,amsfonts,amssymb,mathrsfs,pb-diagram}
\usepackage[
bookmarks=true,colorlinks=true,linktoc=page,citecolor=green!80!black,linkcolor=green!80!black,urlcolor=green!80!black]{hyperref}
\usepackage{caption}
\usepackage{lipsum,wasysym}
\usepackage{mathtools}
\usepackage[a4paper,margin=2.1cm]{geometry}
\usepackage{cleveref}
 \usepackage{amsmath}
\mathchardef\mhyphen="2D
\usepackage{color}
\usepackage{xcolor}
\usepackage{ifthen}
\usepackage{sidecap}   
\definecolor{mediumblue}{rgb}{0.0, 0.0, 0.8}

\renewcommand{\geq}{\geqslant}
\renewcommand{\leq}{\leqslant}

\tikzset{wei/.style= 
{red,double=red,double
distance=0.5pt}}

\newcommand{\cF}{\stt}

\newcommand{\cR}{\mathcal{R}}

\newcommand{\Z}{\mathbb{Z}}

\tikzset{wei2/.style={red,double=red,double
distance=0.5pt}}

\allowdisplaybreaks
\numberwithin{equation}{section}
\usepackage{scalefnt}

\newtheorem{thm}{Theorem}[section]
\newtheorem{cor}[thm]{Corollary}
 
\newtheorem*{assump}{Standing assumptions}

\newtheorem{conj}[thm]{Conjecture}
\newtheorem{lem}[thm]{Lemma}
\newtheorem{prop}[thm]{Proposition}

\newtheorem*{prop*}{Proposition}
\newtheorem*{thmA}{Theorem A}

\newtheorem*{thmC*}{Theorem C}\newtheorem*{thm*}{Theorem D}
\newtheorem*{cor*}{Corollary}

\newtheorem*{conj*}{Conjecture A}

\newtheorem*{conj1*}{Conjecture B}
\newtheorem*{Acknowledgements*}{Acknowledgements}

\theoremstyle{rmk}

\theoremstyle{defn}
\newtheorem{rmk}[thm]{Remark}
\newtheorem{defn}[thm]{Definition}
\newtheorem{eg}[thm]{Example}

\newcommand{\rad}{\mathrm{rad}}
\newcommand{\res}{\mathrm{res}}

\newcommand{\Std}{{\rm Std}}

\newcommand{\Shape}{\operatorname{Shape}}

\newcommand{\la}{\lambda}

\newcommand{\K}{k}

\newcommand{\SSTS}{\mathsf{s}}

\newcommand{\SSTT}{\mathsf{t}}  
\newcommand{\SSTP}{\mathsf{p}}  
\newcommand{\SSTU}{\mathsf{u}}  
\newcommand{\SSTV}{\mathsf{v}}  
\newcommand{\SSTQ}{\mathsf{q}}  
\newcommand{\sts}{\mathsf{s}}  
\newcommand{\stt}{\mathsf{t}}  
\newcommand{\stu}{\mathsf{u}}  
\newcommand{\stv}{\mathsf{v}}  
\newcommand{\ZZ}{{\mathbb Z}}

\newcommand{\CC}{{\mathbb{C}}}

\let\<=\langle
\let\>=\rangle

\tikzset{
ultra thin/.style= {line width=0.05pt},
very thin/.style=  {line width=0.2pt},
thin/.style=       {line width=0.1pt},
semithick/.style=  {line width=0.6pt},
thick/.style=      {line width=0.8pt},
very thick/.style= {line width=1.2pt},
ultra thick/.style={line width=1.6pt}
}

\crefname{ques}{Question}{Questions}
\crefname{defn}{Definition}{Definitions}
\crefname{thm}{Theorem}{Theorems}
\crefname{prop}{Proposition}{Propositions}
\crefname{lem}{Lemma}{Lemmas}
\crefname{cor}{Corollary}{Corollaries}
\crefname{conj}{Conjecture}{Conjectures}
\crefname{section}{Section}{Sections}
\crefname{subsection}{Section}{Sections} 
\crefname{eg}{Example}{Examples}
\crefname{figure}{Figure}{Figures}
\crefname{rem}{Remark}{Remarks}
\crefname{rmk}{Remark}{Remarks}
\crefname{equation}{equation}{equation}

\Crefname{ques}{Question}{Questions}
\Crefname{defn}{Definition}{Definitions}
\Crefname{thm}{Theorem}{Theorems}
\Crefname{prop}{Proposition}{Propositions}
\Crefname{lem}{Lemma}{Lemmas}
\Crefname{cor}{Corollary}{Corollaries}
\Crefname{conj}{Conjecture}{Conjectures}
\Crefname{section}{Section}{Sections}
\Crefname{subsection}{Subsection}{Subsections}
\Crefname{eg}{Example}{Examples}
\Crefname{figure}{Figure}{Figures}
\Crefname{rem}{Remark}{Remarks}
\Crefname{rmk}{Remark}{Remarks}

\usepackage[hang,flushmargin]{footmisc}

\newcommand{\param}{q_0}
\newcommand{\paramt}{q_n}

\newcommand{\ti}{{\underline{i}}}

\newcommand{\idemp}{e_\ti}
\newcommand{\idmep}{e_\ti}

\usepackage{enumerate}

\def\CC{\mathbb{C}}

\def\<{\langle}	\def\>{\rangle}
\def\({[\![}	\def\){]\!]}
\def\Box{\mathrm{box}}

\def\odd{\mathrm{odd}}
\def\even{\mathrm{even}}
\newcommand{\sm}{\scalebox{.5}[1.0]{\hspace{.2pt}$-$}}

\def\TLC{\mathrm{2TL}}
\def\HC{H}
\def\BC{B}
\def\Cal{N}

\definecolor{zajj}{HTML}{008148}
\definecolor{plum}{HTML}{9448BC}
\definecolor{copperred}{HTML}{DA6244}
\definecolor{orangepeel}{HTML}{FFA62B}
\definecolor{mantis}{HTML}{73BD61}
\definecolor{teal}{HTML}{247BA0}

\tikzstyle over=[draw=white,double=black,line width=2pt, double distance=.4pt]
\tikzstyle{B}=[draw, fill=black, circle, inner sep=0pt, outer sep=0pt, minimum size=5pt]
\tikzstyle{V}=[draw, fill =black, circle, inner sep=0pt, minimum size=1.5pt]

\colorlet{marker}{black}
\tikzstyle{M}=[draw, black, fill =marker, circle, double, inner sep=0pt, minimum size=5pt]
\tikzstyle{M1}=[draw, black, fill =marker!70, circle, double, inner sep=0pt, minimum size=5pt]
\tikzstyle{M2}=[draw, black, fill =marker!70!black, circle, double, inner sep=0pt, minimum size=5pt]
\tikzstyle{bV}=[draw, fill =black, circle, inner sep=0pt, minimum size=3.5pt]
\tikzstyle{cV}=[draw, fill =white, circle, inner sep=0pt, minimum size=3.5pt]
\tikzstyle{BoxArr}=[xscale = .2, yscale=-.2]

\newcommand\TIKZ[2][]{\begin{tikzpicture}[baseline={([yshift=-.8ex]current bounding box.center)}, #1]#2\end{tikzpicture}}

\def\Over[#1,#2][#3,#4]{ 
	\draw[style=over]   (#1,#2) .. controls ++(0,#4*.5-#2*.5) and ++(0,-#4*.5+#2*.5) .. (#3,#4);}
\def\Cross[#1,#2][#3,#4]{
	\Over[#3,#2][#1,#4]\Over[#1,#2][#3,#4]}
\def\Tops[#1][#2][#3]{
	\foreach\x in {#1}{
		\draw (\x+.15,#2) -- (\x+.15,#2+.1) (\x-.15,#2) -- (\x-.15,#2+.1) ;
		\draw (\x+.15,#2+.1) arc (0:360:1.5mm and .75mm);}
	\foreach \x in {1,...,#3} {\draw (\x,#2)  to (\x,#2+.05); \node[V] at (\x,#2+.05){};}
	}
\def\Bottoms[#1][#2][#3]{
	\foreach\x in {#1}{
		\draw (\x+.15,#2) -- (\x+.15,#2-.1) (\x-.15,#2) -- (\x-.15,#2-.1) ;
		\draw (\x+.15,#2-.1) arc (0:-180:1.5mm and .75mm);}
	\foreach \x in {1,...,#3} {\draw (\x,#2)  to (\x,#2-.05); \node[V] at (\x,#2-.05){};}
	}
\def\Caps[#1][#2,#3][#4]{
	\Tops[#1][#3][#4]
	\Bottoms[#1][#2][#4]
	}
\def\Pole[#1][#2,#3]{
	\shade[left color=white,right color=white] (#1+.15,#2) rectangle (#1-.15,#3);
	\draw[over] (#1+.15,#2) to (#1+.15,#3) (#1-.15,#2) to (#1-.15,#3) ;}
\def\Label[#1,#2][#3][#4]{
	\node[above, inner sep=0pt] at (#3,#2+.1) {#4};
	\node[below, inner sep=0pt] at (#3,#1-.1) {#4};		}
\def\Ez[#1]{\draw [over, bend left=75] (1,#1+1) to (1-.4,#1+.7)  (1-.4,#1+.3)  to (1,#1) ;
		\draw[densely dotted]  (1-.4,#1)--(1-.4,#1+1) ; \node[V] at (1-.4,#1+.3){}; \node[V] at  (1-.4,#1+.7){};}
\def\Ek[#1][#2]{\draw [over, bend right=75] (#2,#1+1) to (#2+.4,#1+.7)  (#2+.4,#1+.3)  to (#2,#1) ;
		\draw[densely dotted]  (#2+.4,#1)--(#2+.4,#1+1) ; \node[V] at (#2+.4,#1+.3){}; \node[V] at  (#2+.4,#1+.7){};}

\newcommand{\FIVE}[3]{
\foreach \x/\y in {#2}{
	\ifnum \x=\y 
		\draw[very thick, copperred, bend right=90]
			(\x-.6,0) .. controls +(-.35,-.4) and +(.35,-.4) .. (\x-.4,0);
	\else
		\draw[very thick, copperred, bend right=90] (\x-.5,0) to (\y-.5,0); \fi}
\foreach \x/\y in {#3}{
	\ifnum \x=\y 
		\draw[very thick, teal, bend right=90]
			(\x-.6,1) .. controls +(-.35,.4) and +(.35,.4) .. (\x-.4,1);
	\else
	\draw[very thick, teal, bend left=90] (\x-.5,1) to (\y-.5,1);\fi}
\draw (0,0) rectangle (5,1); \foreach \x in {1,2,3,4}{\draw (\x,0) to (\x,1);}
\node[M] at (2,1){};
\foreach \x [count=\c from 1] in {#1}{\node at (\c - .5,.5) {\small \strut $\x$};}
}
\newcommand{\SIX}[3]{ 
\foreach \x/\y in {#2}{
	\ifnum \x=\y 
		\draw[very thick, copperred, bend right=90]
			(\x-.6,0) .. controls +(-.35,-.4) and +(.35,-.4) .. (\x-.4,0);
	\else
		\draw[very thick, copperred, bend right=90] (\x-.5,0) to (\y-.5,0); \fi}
\foreach \x/\y in {#3}{
	\ifnum \x=\y 
		\draw[very thick, teal, bend right=90]
			(\x-.6,1) .. controls +(-.35,.4) and +(.35,.4) .. (\x-.4,1);
	\else
	\draw[very thick, teal, bend left=90] (\x-.5,1) to (\y-.5,1);\fi}
\draw (0,0) rectangle (6,1); \foreach \x in {1,2,3,4,5}{\draw (\x,0) to (\x,1);}
\node[M] at (2,1){};
\foreach \x [count=\c from 1] in {#1}{\node at (\c - .5,.5) {\small \strut $\x$};}
}

\definecolor{sapphire}{HTML}{004FB6}
\definecolor{shamrock}{HTML}{0B9A59}
\definecolor{grape}{HTML}{6E11A0}
\definecolor{sunflower}{HTML}{F6AE2D}
\definecolor{maize}{HTML}{FDE34F}
\definecolor{rose}{HTML}{A40044}
\definecolor{caribbeancurrent}{HTML}{00737D}
\definecolor{cerulean}{HTML}{008DC1}
\definecolor{darkblue}{HTML}{1E2DBA}
\definecolor{palatinate}{HTML}{2B3DDB}
\definecolor{seagreen}{HTML}{10AFAC}
\definecolor{yellowgreen}{HTML}{B4C969}
\definecolor{icterine}{HTML}{ECFA56}
\definecolor{moss}{HTML}{7D9620}
\definecolor{apple}{HTML}{93B31E}
\definecolor{pupureus}{HTML}{943EB6}
\definecolor{burntorange}{HTML}{C96B27}
\definecolor{darkcyan}{HTML}{008B89}

\colorlet{taucolor}{darkblue}
\colorlet{taucolorlight}{palatinate}
\colorlet{sigmacolor}{darkcyan}
\colorlet{sigmacolorlight}{seagreen}

\tikzstyle{tauplane}=[line width=2, taucolorlight!40, line cap=round]
\tikzstyle{sigmaplane}=[line width=2, sigmacolorlight!40, line cap=round]

\colorlet{pathsregioncolor}{black!50}
\tikzstyle{pathsregion}=[pathsregioncolor, fill=pathsregioncolor!20, dotted]
\tikzstyle{gridstyle}=[thin, black!50]

\colorlet{minustwocolor}{apple}
\colorlet{plusonecolor}{icterine}
\tikzstyle{plusone}=[plusonecolor, opacity = .25]
\tikzstyle{minustwo}=[minustwocolor, opacity = .25]

\colorlet{path1color}{black!80}
\colorlet{path2color}{pupureus}
\colorlet{path3color}{burntorange}
\tikzstyle{path1}=[line width=2.5, path1color, line cap=round]
\tikzstyle{path2}=[line width=2, path2color, line cap=round]
\tikzstyle{path3}=[line width=1.75, path3color, line cap=round]

\def\TILE{to ++(1,1) to ++(-1,1) to ++(-1,-1) to ++(1,-1)}
\tikzstyle{tilemarkers}=[circle, inner sep=0pt, fill=white, fill opacity=.3, text opacity=1]

\newcommand{\extra}{\vartheta }

\title{}
\author{}

\author[C.\ Bowman]{Chris Bowman}
\address{Department of Mathematics, 
University of York, Heslington, York,  UK}
\email{chris.bowman-scargill@york.ac.uk}

\author[Z.\ Daugherty]{Zajj Daugherty}
\address{Department of Mathematics and Statistics,
Reed College, Portland OR, USA} 
\email{zdaugherty@reed.edu}

 \author[M.\ De Visscher]{Maud De Visscher}
\address{Department of Mathematics, City St George's, University of London,   London, UK}
\email{maud.devisscher.1@citystgeorges.ac.uk}

\author[R.\ Muth]{Rob Muth\qquad} 
	\address{ 
Department of Mathematics and Computer Science,
Duquesne University,
Pittsburgh PA, USA}
\email{muthr@duq.edu}

 		\author[L.\ Poulain D'Andecy]{Loic Poulain D'Andecy \ \ \ }

 	\address{Laboratoire de Math\'{e}matiques de Reims, CNRS UMR 9008
       Universit\'e de Reims Champagne-Ardenne,
  Reims,   France}
 	 \email{loic.poulain-dandecy@univ-reims.fr}

\begin{document}

  \title
 {The orientifold Temperley--Lieb algebra}

  \begin{abstract} We construct gradings on the simple modules of 2-boundary Temperley--Lieb algebras and symplectic blob algebras
 by realising 
the latter algebras  as quotients of  Varagnolo--Vasserot's orientifold quiver Hecke algebras. 
We prove that the symplectic blob algebras are graded cellular and provide a conjectural algorithm for calculating their graded 
decomposition matrices. 
 In doing so, we give the first explicit family of finite-dimensional graded quotients of the orientifold quiver Hecke algebras, providing a new entry point for the structure of  these algebras---in the spirit of Libedinsky--Plaza's ``blob algebra approach'' to modular representation theory. 
  \end{abstract}

  \maketitle

  \section{Introduction}\label{intro}

The Temperley--Lieb algebras first appeared as transfer matrix algebras of the Potts model of statistical mechanics \cite{MR498284}.
 These  Temperley--Lieb algebras and their modules were later categorified  
  \cite{MR1740682,MR2174270,MR2521250,MR2918294,MR2726291}   
and developed into powerful functorial knot invariants, whose impact was dramatically illustrated in   Piccirillo's solution of the   Conway knot problem---a landmark result in quantum topology \cite{MR4076631}.
It is even hoped that these   Temperley--Lieb categorifications could    provide 
 the desired 4-dimensional topological quantum field theories 
   of Crane--Frenkel's higher categorical approach to 
the smooth   Poincar\'{e} conjecture \cite{MR4353348,Manolescu2022FourDimensionalTopology,ICM1}.

Martin--Saleur, Green--Martin--Parker,
 and de Gier--Nichols    introduced new boundary conditions and greatly generalised the transfer matrix algebras of the classical Potts model \cite{MR1267001,MR2354870,MR2928127,GN}.
In the case of a single-boundary,    prophetic conjectures of Martin--Woodcock and Libedinsky--Plaza posited that these generalised transfer matrix algebras are governed by 
$p$-Kazhdan--Lusztig polynomials---an unexpected and powerful bridge between statistical mechanics and modular representation theory
\cite{MW00,MR4100120}. 
  Libedinsky--Plaza's conjecture seeks, as they put it, {\em “to raise intuitions from physics towards the (nowadays) obscure land of modular representation theory”}.
These conjectures have since been confirmed \cite{MR4611117}, opening the door to what might be called 
a categorical statistical mechanics, wherein    
quiver Hecke  algebras and diagrammatic 
Soergel bimodules offer new gradings and tools for diagonalising   transfer matrices, and 
      boundary phenomena  are encoded  within higher  quantum topological and categorical structures.



This paper seeks to push this frontier further by incorporating the transfer matrix algebras of  {\em two-boundary} Potts models into this emerging categorical framework.  
 For one-boundary Potts models (the generalised {\em blob algebras}) this categorical link was 
 provided by  a  realisation of these algebras as 
 graded quotients of the quiver Hecke algebra \cite{Pla13,PR13,MR4401509}.
For the two-boundary Potts models (the {\em symplectic blob algebras}, also known as the 
{\em double quotients of the  2-boundary Temperley--Lieb algebra}) 
we   construct the transfer matrix  algebras as 
 graded quotients of the {\em orientifold quiver Hecke algebras}  of Varagnolo--Vasserot 
 \cite{MR2827096} which first arose in their work verifying  the Enomoto--Kashiwara conjectures \cite{MR2279279}.


\begin{thmA}
The  
 symplectic  blob algebra 
 is the   graded  quotient of the orientifold   quiver Hecke  algebra 
$\mathscr{H}_n  (q_0q_n, -q_0 q_n^{-1}, \vartheta )$
by the relations 
$$ 	y_1 e _{\SSTT_{(0,\extra )}}   =0 
  \quad \text{ and }\quad
\idmep=0\text { if $\ti \neq \res(\SSTS)$ for some $\SSTS\in \Std_n$ }
$$where $ \Std_n$ is the set of ``standard orientifold tableaux" and ${\SSTT_{(0,\extra )}}$ is the minimal  amongst such   tableaux.  This algebra  has a  graded cellular basis 
and its module category is  graded highest weight.
\end{thmA}

  This new graded structure allows us to formulate an LLT-style  \cref{conjforus} for calculating the (graded) decomposition numbers  and  simple characters 
  of these   transfer matrix  algebras; this is   phrased in the language of graded {\em orientifold} standard tableaux and inspired by work of Kleshchev--Nash \cite{KN10}.

 \smallskip\noindent{\bf Representation theory in statistical mechanics. }
 For statistical mechanicists, the two-boundary Temperley–Lieb algebra governs lattice models with nontrivial boundaries—loop models, Potts models, and the XXZ spin chain—encoding the reflection equations that ensure edge integrability \cite{statmex1,statmex5,statmex31,statmex32}.   If the Yang–Baxter equation controls bulk solvability, these algebras capture the integrable physics of competing boundaries \cite{statmex1,statmex5}. %
Their simple modules map physical intuition to boundary spectra: each simple corresponds to an invariant sector, determining eigenvalues, degeneracies, and scaling exponents \cite{statmex5,statmex6,statmex32}. Classifying simples identifies generic versus indecomposable or logarithmic spectra, the latter marking critical or non-unitary behaviour, and predicts when boundary parameters yield new fixed points, spectral coincidences, or logarithmic conformal field theories  \cite{statmex6,statmex31,statmex32}.

Calculating  the simple modules of two-boundary Temperley--Lieb algebras/symplectic blob algebras at $q$ a root of unity has long been considered beyond hope---it was realised 20 years ago that constructing the simple modules in full generality was  beyond the remit of classical Kazhdan--Lusztig theory;  
  since then,     effort has  focussed on understanding  
the monomial bases and presentations, 
 block and quasi-hereditary structure, understanding 
 simple modules in special cases, and the construction of full tilting modules  
  \cite{statmexblob3,GN,MR2928127,MR2927180,statmexblob1,statmexblob2,statmexblob67,MR3818281,statmexblob66,DR25b}. 
 Our  \cref{conjforus}  proposes an explicit algorithm to compute all graded simple characters of two-boundary Temperley–Lieb and symplectic blob algebras across both generic and non-generic parameters.  
Reciprocally---and  in the spirit of Libedinsky–Plaza's   vision of importing {\em “physical intuition in one of the most difficult problems in representation theory”}  \cite{MR4100120}---we now discuss how the physicists’ construction of the  symplectic blob algebra  provides us with a structural bulkhead for tackling a major open problem in   categorical representation theory.

\smallskip\noindent{\bf Statistical mechanics  in  categorical representation theory. }
 The finite-dimensional algebra perspective has become so central to the study of quiver Hecke  algebras that it is almost taken for granted. 
Through their  cyclotomic quotients, quiver Hecke  algebras can be approached using the full machinery of finite-dimensional representation theory— 
making explicit computation and combinatorial analysis possible; 
for example this is how   graded  decomposition matrices are defined and how almost all results on  graded  simple characters are conjectured and proven \cite{LLT,ELpaper,MR4611117,bk09}. 
 These cyclotomic quotients are also of interest on a higher structural level, as they categorify the highest weight representations of  Drinfel'd--Jimbo quantum groups \cite{MR2822211,kk12}. 
 Consequently,  the cyclotomic viewpoint is now woven into nearly every   aspect of the subject.

 By way of contrast,  absolutely nothing is known about the cyclotomic quotients of orientifold quiver Hecke algebras---for example  there is no analogue of the Ariki--Koike construction, and  we cannot even determine when a cyclotomic quotient is   non-zero! 
 These  cyclotomic quotients  
   should be hoped to categorify simple modules of the Enomoto--Kashiwara algebras \cite{MR2827096,MR2279279} and to   have rich connections (via Schur--Weyl duality   \cite{MR4666131}) with the emerging theory of 
  $\imath$quantum groups. 
 With no clear algebraic path toward such a general construction, we turn instead to statistical mechanics:  
our Theorem~A realises   the symplectic blob algebra as  the first non-trivial finite-dimensional graded quotient of an orientifold quiver Hecke algebra.
Thus our  \cref{conjforus}  provides the first   approach to constructing a  family of graded simple modules of 
orientifold quiver Hecke algebras (for non-generic parameters), and serves as the first step towards a general theory of cyclotomic quotients.

  \smallskip\noindent{\bf Structure of the paper. } In Section 2 we recall the background on 2-boundary Hecke and  Temperley--Lieb algebras  and the  symplectic blob algebras.  In Section 3 we study the (calibrated) simple modules of these algebras at generic parameters---this provides us with an understanding of the kernel of the projection from
   the two boundary Temperley--Lieb algebra to the symplectic blob algebra
    in terms of the Jucys--Murphy elements of these algebras.  In Section 4 we define the orientifold Temperley--Lieb algebra as a quotient of the Varagnolo--Vasserot orientifold Hecke algebra  
    by relations that are   inspired by the  results of Section 3; 
       we then prove that the symplectic blob algebra factors through the orientifold Temperley--Lieb algebra.  
Section 5 contains the proof of Theorem A: we construct a graded cellular basis for the orientifold Temperley–Lieb algebra and simultaneously show that the   homomorphism of Theorem A is bijective.  
  In Section 6 we provide our conjectural algorithm for computing graded decomposition matrices.

\section{Two-boundary algebras}
\label{sec:two-boundary algebras}
\def\Ri{\cR}

\noindent 
Fix $n \in \ZZ_{>0}$, 
$\Bbbk$  a field of characteristic not equal to 2,  
and  set   $\Ri=\Bbbk(q, q_0, q_n)$.
%
%
In this section, we review the characterization of the two-boundary braid, Hecke, and Temperley--Lieb algebras as studied in \cite{DR25a, DR25b}. We recall that the Coxeter graph of type $ {C}_n$ is given by
\[\TIKZ[scale=1.1]{
	\foreach \x in {0,1, 2, 4,5,6}{
		\node(\x) at (\x,0){};
		}
	\foreach \x in {0,1, 2, 4,5}{		
				\draw (\x,0) circle (2.5pt);}
	\foreach \x in {1, 2}{
		\node[label=above:{$s_{\x}$}] at (\x,0){};
		\node[label=above:{$s_{n-\x}$}] at (6-\x,0){};}
	\node[label=above:{$s_0$}] at (0){};
	\node[label=above:{$\phantom{s_n}$}] at (6){};
	\draw[double distance = 2pt] (0)--(1)  ;
	\draw (1)--(2) (4)--(5);
	\draw[dashed] (2) to (4);
}
\]
and we let  $W( {C} _n)$ denote the corresponding  Weyl group of type  $C _n$;
we further recall that the Coxeter graph of   type $\widehat {C}_n$ is given by 
\[
\TIKZ[scale=1.1]{
	\foreach \x in {0,1, 2, 4,5,6}{
		\draw (\x,0) circle (2.5pt);
		\node(\x) at (\x,0){};
		}
	\foreach \x in {1, 2}{
		\node[label=above:{$s_{\x}$}] at (\x,0){};
		\node[label=above:{$s_{n-\x}$}] at (6-\x,0){};}
	\node[label=above:{$s_0$}] at (0){};
	\node[label=above:{$s_n$}] at (6){};
	\draw[double distance = 2pt] (0)--(1) (5)--(6);
	\draw (1)--(2) (4)--(5);
	\draw[dashed] (2) to (4);
}\]
and we let 
$ W(\widehat{C} _n)$ denote the corresponding affine Weyl group of type  $C _n$.
 We can  identify the Weyl group $W(C_n)$ with the group of signed permutations on $\{\pm1,\dots,\pm n\}$, with $s_0=(-1,1)$ and $s_i=(-i,-(i+1))(i,i+1)$ for $i=1,\dots,n-1$.  

\subsection{The two-boundary Hecke algebra $\HC_n$}
%
%
The    {\sf two-boundary Hecke algebra}, is generated over $\Ri$ by invertible elements $T_0, T_1, \dots, T_n$ with relations 
 \begin{align}
 \begin{split}\label{Hkdefn}
(  T_i-q_{ {i}})(  T_i+q_{ {i}}^{-1})			&=1 					\text{ for }\textcolor{purple}{0}\leq i \leq  n ,
\\
  T_i T_j			&= T_j T_i \text { for }|i-j|>1, 
\\
T_i T_{i+1}T_i 		&= T_i T_{i-1} T_i \text{ for }1<i<n-1, 
\\
 T_0T_1 T_0T_1				&=T_1T_0T_1 T_0, \\
T_nT_{n-1} T_nT_{n-1}				&=T_{n-1}T_nT_{n-1} T_n 
\end{split}
 \end{align} where $q_i = q$ for $i = 1, \dots, n-1$. 
Elements of $\HC_n$ can be represented as linear combinations of braids on a space with two rigid poles.
For $1\leq i \leq n-1$ the $T_i$ generators of $\HC_n$ are identified with the diagrams 
\begin{equation}\label{LRpics}
{\def\TOP{1.5} \def\K{6}
{T}_i=
\TIKZ[scale=.5]{
	\Pole[\K+.85][0,\TOP][\K]
	\Pole[.15][0,\TOP]
	\Over[3,0][4,\TOP]
	\Over[4,0][3,\TOP]
	 \foreach \x in {1,2,5,\K} {
		 \draw[] (\x,0) -- (\x,\TOP);
		 }
	\Caps[.15,\K+.85][0,\TOP][\K]
	\Label[0,\TOP][3][\footnotesize \strut $i$]
	\Label[0,\TOP][4][\footnotesize \strut $i$+1]
}
}
\end{equation}for $i=1, \dots, n-1$, and  the generators $T_0$ and $T_n$ are identified with the 
braid diagrams
$$
{\def\TOP{1.5}\def\K{6}
{T}_n=
\TIKZ[scale=.5]{
\Pole[.15][0,\TOP]
\Over[\K,0][\K+1.3,.5*\TOP]
\Pole[\K+.85][0,.5*\TOP][\K]
\Pole[\K+.85][.5*\TOP,\TOP][\K]
\Over[\K+1.3,.5*\TOP][\K,\TOP]
 \foreach \x in {1,...,5} {
	 \draw (\x,0) -- (\x,\TOP);
	 }
\Caps[.15,\K+.85][0,\TOP][\K]
} 
	\qquad \quad
{T}_0=
\TIKZ[scale=.5]{
	\Pole[\K+.85][0,\TOP][\K]
	\Pole[.15][0,.5*\TOP]
	\Over[1,0][-.3,.5*\TOP]
	\Over[-.3,.5*\TOP][1,\TOP]
	\Pole[.15][.5*\TOP,\TOP]
	 \foreach \x in {2,...,\K} {
	 \draw[] (\x,0) -- (\x,\TOP);
	  } 
\Caps[.15,\K+.85][0,\TOP][\K]
}
}$$where the multiplication of braid diagrams is given by vertical stacking of one diagram on top of another. 
We define 
\begin{equation}\label{eq:T0v}
T_{0^\vee} = T_1^{-1}T_2^{-1} \cdots T_{n-1}^{-1} T_n T_{n-1} \cdots T_1
\;=\; {\def\TOP{2}\def\K{6}
\TIKZ[scale=.5]{
	\Pole[.15][0,\TOP]
	\Over[1,0][\K+1.3,.5*\TOP]
	\Pole[\K+.85][0,\TOP][\K]
	\Over[1,\TOP][\K+1.3,.5*\TOP]
	 \foreach \x in {2,3,...,\K} {
		 \draw[over] (\x,0) -- (\x,\TOP);
		 }
	\Caps[.15,\K+.85][0,\TOP][\K]
}}\;.
\end{equation}
Then the {\sf Jucys--Murphy elements} of $\HC_n$ are defined  by 
\begin{equation}\label{eq:JMs}
X_1 = T_{0^\vee} T_0 \qquad \text{and} \qquad
	X_{i+1} = T_i X_i T_i     
\end{equation}for $i=1, \dots n-1$. 
The elements  $X_1, \dots, X_n$ form a maximal family of commutative elements of $\HC_n$. We  let   $W(C_n)$  act  on $\{X_1^{\pm 1}, \dots, X_n^{\pm 1}\}$  by $w \cdot X_i = X_{w(i)}$ for $w \in W$ and $i = \pm 1, \dots, \pm n$, where we set the convention $X_{-i} = X_i^{-1}$. Then the center of $\HC_n$ is 
\[Z(\HC_n) = \Ri[X_1^{\pm1}, \dots, X_n^{\pm1}]^{W(C_n)},\]
Laurent polynomials in $X_1, \dots, X_n$ that are symmetric under the action of ${W(C_n)}$. 
As we will see in \cref{sec:Calibrated}, we classify much of the representation theory $\HC_n$ (and important quotient algebras) in terms of the action of the Jucys--Murphy elements. 
Of particular interest 
in  \cref{sec:Calibrated} are the {\sf calibrated} representations---the     finite dimensional simple  representations 
 on which $X_1, \dots, X_n$ are simultaneously diagonalizable. 
 So it quickly becomes much more convenient to move to a diagrammatic presentation that highlights the commuting structure of these elements; this is  done by moving the right-hand pole to the left by conjugating by 
\begin{equation}\label{eq:sigma}
\sigma = 
\TIKZ[scale=.5]{
	\draw (-.7,1) .. controls (-.7,.15) .. (0,.15) -- (6,.15) .. controls (7,.15) .. (7,-1);
	\draw (-1,1) .. controls (-1,-.15) .. (0,-.15)-- (6,-.15) .. controls (7-.3,-.15) .. (7-.3,-1);
\Pole[.15][-1,1]
 \foreach \x in {1,...,6} {\draw[style=over] (\x,-1) -- (\x,1);}
\Tops[.15, -.85][1][6]
\Bottoms[.15, 6+.85][-1][6]
}\; .
\end{equation}
In particular, we have that 
\begin{equation}\label{eq:BraidT0vXi}
T_{0^\vee} = {\def\TOP{8}\def\K{5}\def\Left{-.85}\def\Right{.15}
\TIKZ[scale=.4]{
		\Pole[\Left][0,.5*\TOP] 
		\Over[\K,2.5][-1.3,4]
		\Over[-1.3,4][\K,5.5]
		\Pole[\Left][.5*\TOP,\TOP] 
		\Pole[\Right][0,\TOP] 
		\draw[, rounded corners] (1,\TOP) to (1,\TOP-.25) to [bend left=10] (5,.5*\TOP+2) to  (5,.5*\TOP+1.5)
								(1,0) to (1,.25) to [bend right=10] (5,.5*\TOP-2) to  (5,.5*\TOP-1.5);
		\draw[over, rounded corners] (2,\TOP) to (1.75,\TOP-.75) to (4,.5*\TOP+1.5) to  (4,.5*\TOP-1.5) to (2,.75)to  (2,0);
		\draw[over, rounded corners] (3,\TOP) to (3,\TOP-.75) to (1.5,.5*\TOP+2.5) to (3,.5*\TOP+1) to  (3,.5*\TOP-1) 
									 to (1.5,.5*\TOP-2.5)to (3,.75) to (3,0);
		\draw[over, rounded corners] (4,\TOP) to (4,\TOP-1) to (1.25,.5*\TOP+1.75) to (2,.5*\TOP+1)  
										to (2,.5*\TOP-1) to (1.25,.5*\TOP-1.75) to (4,1) to (4,0);
		\draw[over, rounded corners] (5,\TOP) to (5,\TOP-1.25) to (1,.5*\TOP+1) to (1,.5*\TOP-1) to  (5,1.25) to (5,0);
		\Caps[.15,-.85][0,\TOP][\K]
}}  \;\ =\; 
{\def\TOP{3} \def\K{5}
\TIKZ[scale=.4]{
	\Pole[-.85][0,.5*\TOP]
	\Over[1,0][-1.3,.5*\TOP]
	\Over[-1.3,.5*\TOP][1,\TOP]
	\Pole[-.85][.5*\TOP,\TOP]
	\Pole[.15][0,\TOP]
	 \foreach \x in {2,...,\K}  {
		 \draw[] (\x,0) -- (\x,\TOP);
		 }
	\Caps[.15,-.85][0,\TOP][\K]
	}}
	$$and so (by \cref{eq:JMs}) the $X_i $ can be pictured as follows
	$$
{
\def\TOP{2}\def\K{6}
X_i =
\TIKZ[scale=.5]{
		\Pole[-.85][0,1]
		\Pole[.15][0,1]
		 \foreach \x in {1,2} {\draw[] (\x,0) -- (\x,1);}
		\Over[3,0][-1.3,1]
		\Over[-1.3,1][3,2]
		\Pole[-.85][1,2]
		\Pole[.15][1,2]
		\foreach \x in {1,2} { \draw[over] (\x,1) -- (\x,2); }
		\foreach \x in {4,...,\K} {\draw[] (\x,0) -- (\x,\TOP);}
		\Caps[.15,-.85][0,\TOP][\K]
		\Label[0,\TOP][3][{\footnotesize \strut $i$}]
}\ }\end{equation}
for $i=2, \ldots, n$. As in \cite[Remark 2.3]{DR25a}, it is often convenient to replace the generator $T_n$ with the element $T_{0^\vee}$, noting that the latter also satisfies the relations
\begin{equation}
T_{0^\vee} T_i = T_i T_{0^\vee}  \qquad T_1 T_{0^\vee} T_1 T_{0^\vee} = T_{0^\vee} T_1 T_{0^\vee} T_1  \qquad  (T_{0^\vee} - q_n)(T_{0^\vee} + q_n^{-1}) = 0 
\end{equation}for all $i>1$.

\subsection{The two-boundary Temperley--Lieb algebra }
For $k \in \ZZ$ and $x \in \Ri^\times$, define 
\begin{equation}\label{eq:q-analogs}
[k] = \frac{q^k - q^{-k}}{q-q^{-1}} \quad \text{ and } \quad \(x\) = x + x^{-1},
\end{equation}
so that 
$\(q^{k}\) = [2k]/[k]$. 
Define $e_i = T_i - q_i$ (again, where $q_i = q$ for $i=1, \dots, n-1$).\footnote{We have set $a = a_0 = a_n = 1$ from \cite{DR25b}.} Note that the quadratic relation,
$ (T_i -q_i)(T_i + q_i^{-1}) = 0$,  from  \eqref{Hkdefn}
 is equivalent to
\begin{equation}
   e_i^2 = -\(q_i\)e_i.
\end{equation} 
We define the {\sf two-boundary Temperley--Lieb algebra}, $\TLC_n$, to be the   quotient of $\HC_n$ by  the relations 
\begin{equation}\label{eq:smash1}
e_1e_0e_1 = \(q_0q^{-1}\) e_1  \quad 
e_{n-1}e_n e_{n-1} = \(q_nq^{-1}\) e_{n-1} 
\quad 
e_ie_{i+1}e_i =e_i \quad  e_{i+1}e_ie_{i+1} = e_{i+1}   
  \end{equation}
  for  $ i = 1, \dots, n-2$.
We can then identify the elements of $\TLC_n$ with Temperley--Lieb diagrams with two side walls, writing 
\[
e_0 = \TIKZ{
\draw[densely dotted ] (.5*.5,0) rectangle (.5*5+.25,.75);
\foreach \x in {1,2,3,5}{\node[V] (b\x) at (.5*\x, 0) {}; \node[V] (t\x) at (.5*\x, .75){};}
\foreach \y in {1,2}{\node[V]  (l\y) at (.25, \y*.25){}; }
	\draw[bend right=30] (l2) to (t1) (b1) to (l1);
	\foreach \x in {2,3,5}{\draw (t\x)--(b\x);}
	\node at (.5*4,.5*.75) {$\dots$};	
}  \qquad  
e_n=\TIKZ{
\draw[densely dotted ] (.5*.5,0) rectangle (.5*5+.25,.75);
\foreach \x in {1,3,4,5}{\node[V] (b\x) at (.5*\x, 0) {}; \node[V] (t\x) at (.5*\x, .75){};}
\foreach \y in {1,2}{\node[V]  (r\y) at (.5*5+.25, \y*.25){};}
	\draw[bend right=30] (t5) to (r2) (r1) to (b5);
	\foreach \x in {1,3,4}{\draw (t\x)--(b\x);}
	\node at (.5*2,.5*.75) {$\dots$};	
}  \qquad  
e_i=\TIKZ{
\draw[densely dotted ] (.5*.5,0) rectangle (.5*8+.25,.75);
\foreach \x in {1,3,4,5,6,8}{\node[V] (b\x) at (.5*\x, 0) {}; \node[V] (t\x) at (.5*\x, .75){};}
	\draw[bend right=60] (t4) to (t5) (b5) to (b4);
	\foreach \x in {1,3,6,8}{\draw (t\x)--(b\x);}
	\node at (.5*2,.5*.75) {$\dots$};	\node at (.5*7,.5*.75) {$\dots$};
	\node[above] at (t4) {\tiny $i$};\node[below] at (b4) {\tiny $i$};
} 
\]
for $i = 1, \ldots, n-1$. In particular, we can translate between braid and Temperley--Lieb diagrams by local skein-type relations; on generators this can seen as follows
\begin{equation*}
\begin{array}{r@{\ }l@{\qquad}r@{\ }l@{\qquad}r@{\ }l}
T_0 &= e_0+q_0,  \qquad\qquad
	&T_i &= e_i + q, \qquad\qquad
	&T_n &= e_n + q_n,
\\
\TIKZ[scale=.35]{
		\Over[1,1.7][-.5,.5]
		\Pole[.15][-.7,1.7]
			\draw (.3,1.5+.2) arc (0:360:4pt and 3pt);
			\draw (.3,-.5-.2) arc (0:-180:4pt and 3pt);
		\Over[-.5,.5][1,-.7]
}
&= 
\hspace{0.05cm}	\TIKZ[scale=.6]{\Ez[0]}
	+  q_0\ 
	\TIKZ[scale=.35]{
		\Pole[.15][-.7,1.7]
			\draw (.3,1.5+.2) arc (0:360:4pt and 3pt);
			\draw (.3,-.5-.2) arc (0:-180:4pt and 3pt);
		\draw (1,1.7)--(1,-.7);
	}  
&
	\TIKZ[scale=.45]{
		\Cross[1,1.5][2,0]
	}
&= 
\TIKZ[scale=.45]{
		\draw[bend right=100] (1,1.5) to (2,1.5) (2,0) to (1,0);
	}
+ q \ 
	\TIKZ[scale=.45]{
		\draw (1,0)--(1,1.5) (2,0)--(2,1.5);
	} 
&
\TIKZ[xscale=-.35, yscale=.35]{
		\Over[-.5,.5][1,-.7]
		\Pole[.15][-.7,1.7]
			\draw (.3,1.5+.2) arc (0:360:4pt and 3pt);
			\draw (.3,-.5-.2) arc (0:-180:4pt and 3pt);
		\Over[1,1.7][-.5,.5]
	}
&= \hspace{0.015cm}	\TIKZ[scale=.6] {\Ek[0][0]}
+q_n\ 
	\TIKZ[xscale=-.35, yscale=.35]{
		\Pole[.15][-.7,1.7]
			\draw (.3,1.5+.2) arc (0:360:4pt and 3pt);
			\draw (.3,-.5-.2) arc (0:-180:4pt and 3pt);
		\draw (1,1.7)--(1,-.7);
	}
		\end{array}
\end{equation*}
We refer to  \cite[\S 3.3]{DR25b} for some of the key   diagrammatic relations that  follow from \eqref{eq:smash1}.

We now recall the construction of the diagrammatic basis of the two-boundary Temperley--Lieb algebra.  
	Take a rectangle with $n$ marked points on its upper and lower edges and an even number of marked points on both the left and right sides. 
	We draw non-intersecting arcs between pairs of marked points using each marked point once. Horizontal lines connecting the left and right side are permitted.
We say that such a diagram is {\sf  reduced} if no arc has both its end points on the lefthand-side of the rectangle or both its end points on the righthand-side of the rectangle.

\begin{prop}[{\cite[Proposition 3.5]{GN}}]
The set of all reduced diagrams forms a basis of the (infinite dimensional) 
2-boundary Temperley--Lieb algebra.
\end{prop}


  Similarly to \cite[\S 3.2]{GN}, we define the following elements of the two-boundary Temperley--Lieb algebra 
  $$I_0 = \prod_{i=0    }^{ \lfloor n/2 \rfloor}e_{2i}
  \qquad
  I_1 = \prod_{ i=1  }^{ \lceil n/2 \rceil } e_{2i-1}
  $$which can be visualised as follows 
\begin{equation}\label{eq:defI2}
    I_0 =   
	\begin{cases}
	\TIKZ[scale=.5, yscale=.9]{
		\draw[densely dotted ] (.5,0) rectangle (7.5,1.2);
		\foreach \x in {1,2,3,5,6,7}{\node[V] (b\x) at (\x, 0) {}; \node[V] (t\x) at (\x, 1.2){};}
		\foreach \x/\y in {2/3, 5/6}{\draw[bend right=60] (t\x) to (t\y);\draw[bend right=60] (b\y) to (b\x);}
		\foreach \y in {1,2}{\coordinate (l\y) at (.5, \y*.5-.15); \coordinate (r\y) at (7.5, \y*.5-.15); }
		\draw[bend right=30] (l2)node[V]{} to (t1) 	(b1)  to (l1) node[V]{} 	(t7) to (r2) node[V]{}	(r1) node[V]{} to (b7);
		\node at (4,.6){$\cdots$};
	}
	&\text{if $n\in 2\ZZ$,}\\[8pt] 
	\TIKZ[scale=.5, yscale=.9]{
		\draw[densely dotted ] (.5,0) rectangle (8.5,1.2);
		\foreach \x in {1,2,3,5,6,7,8}{\node[V] (b\x) at (\x, 0) {}; \node[V] (t\x) at (\x, 1.2){};}
		\foreach \y in {1,2}{\coordinate (l\y) at (.5, \y*.5-.15); \coordinate (r\y) at (8.5, \y*.5-.15); }
		\foreach \x/\y in {2/3, 5/6,7/8}{\draw[bend right=60] (t\x) to (t\y);\draw[bend right=60] (b\y) to (b\x);}
		\draw[bend right=30] (l2)node[V]{} to (t1) 	(b1)  to (l1) node[V]{} ;
		\node at (4,.6){$\cdots$};
	}
	&\text{if $n\in 1+2\ZZ$.}
\end{cases}
\end{equation}
and
\begin{equation}\label{eq:defI1}
   I_1     =
	\begin{cases}
	\TIKZ[scale=.5, yscale=.9]{
		\draw[densely dotted ] (.5,0) rectangle (7.5,1.2);
		\foreach \x in {1,2,3,4,6,7}{\node[V] (b\x) at (\x, 0) {}; \node[V] (t\x) at (\x, 1.2){};}
		\foreach \x/\y in {1/2, 3/4, 6/7}{\draw[bend right=60] (t\x) to (t\y);\draw[bend right=60] (b\y) to (b\x);}
		\node at (5,.6){$\cdots$};
	}
	&\text{if $n\in 2\ZZ$,}\\[8pt] 
	\TIKZ[scale=.5, yscale=.9]{
		\draw[densely dotted ] (.5,0) rectangle (8.5,1.2);
		\foreach \x in {1,2,3,4,6,7,8}{\node[V] (b\x) at (\x, 0) {}; \node[V] (t\x) at (\x, 1.2){};}
		\foreach \y in {1,2}{\coordinate (l\y) at (.5, \y*.5-.15); \coordinate (r\y) at (8.5, \y*.5-.15); }
		\foreach \x/\y in {1/2, 3/4, 6/7}{\draw[bend right=60] (t\x) to (t\y);\draw[bend right=60] (b\y) to (b\x);}
		\draw[bend right=30] (t8) to (r2) node[V]{}	(r1) node[V]{} to (b8);
		\node at (5,.6){$\cdots$};
	}
	&\text{if $n\in 1+2\ZZ$,}
	\end{cases}
\end{equation}
and we note that the elements $I_0$ and $I_1$ are quasi-idempotents.  Taking products of these elements we obtain diagrams with horizontal strands across the entire width of the diagram, as follows 
\begin{equation*}\label{eq:I121}
    I_1 I_0 I_1 = \begin{cases}
	\TIKZ[scale=.4, yscale=.9]{
		\draw[densely dotted ] (.5,0) rectangle (7.5,2);
		\foreach \x in {1,2,3,4,6,7}{\node[V] (b\x) at (\x, 0) {}; \node[V] (t\x) at (\x, 2){};}
		\foreach \y in {1,2,3,4}{\coordinate (l\y) at (.5, \y*.4){}; \coordinate(r\y) at (7.5, \y*.4){}; }
		\foreach \x/\y in {1/2, 3/4, 6/7}{\draw[bend right=60] (t\x) to (t\y);\draw[bend right=60] (b\y) to (b\x);}
		\draw (r3) node[V] {}--(l3) node[V] {} (r2) node[V] {} --(l2)node[V] {};
		\node at (5,.2){$\cdots$};	\node at (5,1.8){$\cdots$};
	}
	&\text{if $n\in 2\ZZ$,}
	\\[8pt] 
	\TIKZ[scale=.4, yscale=.9]{
		\draw[densely dotted ] (.5,0) rectangle (8.5,2);
		\foreach \x in {1,2,3,4,6,7,8}{\node[V] (b\x) at (\x, 0) {}; \node[V] (t\x) at (\x, 2){};}
		\foreach \y in {1,2,3,4}{\coordinate (l\y) at (.5, \y*.4){}; \coordinate(r\y) at (8.5, \y*.4){}; }
		\foreach \x/\y in {1/2, 3/4, 6/7}{\draw[bend right=60] (t\x) to (t\y);\draw[bend right=60] (b\y) to (b\x);}
		\draw[bend right=30] (t8) to (r4) node[V]{}	(r1) node[V]{} to (b8);
		\draw (r3) node[V] {}--(l3) node[V] {} (r2) node[V] {} --(l2)node[V] {};
		\node at (5,.2){$\cdots$};	\node at (5,1.8){$\cdots$};
	}&\text{if $n\in 1+2\ZZ$,}
	\end{cases}
\end{equation*}
and 
\begin{equation*} 
    I_0 I_1 I_0 = \begin{cases}
	\TIKZ[scale=.4, yscale=.9]{
		\draw[densely dotted ] (.5,0) rectangle (7.5,2);
		\foreach \x in {1,2,3,5,6,7}{\node[V] (b\x) at (\x, 0) {}; \node[V] (t\x) at (\x, 2){};}
		\foreach \x/\y in {2/3, 5/6}{\draw[bend right=60] (t\x) to (t\y);\draw[bend right=60] (b\y) to (b\x);}
		\foreach \y in {1,2,3,4}{\coordinate (l\y) at (.5, \y*.4); \coordinate (r\y) at (7.5, \y*.4); }
		\draw[bend right=30] (l4)node[V]{} to (t1) 	(b1)  to (l1) node[V]{} 	(t7) to (r4) node[V]{}	(r1) node[V]{} to (b7);
		\draw (r3) node[V] {}--(l3) node[V] {} (r2) node[V] {} --(l2)node[V] {};
		\node at (4,.2){$\cdots$};	\node at (4,1.8){$\cdots$};
	}
	&\text{if $n\in 2\ZZ$,}\\[8pt] 
	\TIKZ[scale=.4, yscale=.9]{
		\draw[densely dotted ] (.5,0) rectangle (8.5,2);
		\foreach \x in {1,2,3,5,6,7,8}{\node[V] (b\x) at (\x, 0) {}; \node[V] (t\x) at (\x, 2){};}
		\foreach \y in {1,2,3,4}{\coordinate (l\y) at (.5, \y*.4); \coordinate (r\y) at (8.5, \y*.4); }
		\foreach \x/\y in {2/3, 5/6,7/8}{\draw[bend right=60] (t\x) to (t\y);\draw[bend right=60] (b\y) to (b\x);}
		\draw[bend right=30] (l4)node[V]{} to (t1) 	(b1)  to (l1) node[V]{} ;
		\draw (r3) node[V] {}--(l3) node[V] {} (r2) node[V] {} --(l2)node[V] {};
		\node at (4,.2){$\cdots$};	\node at (4,1.8){$\cdots$};
	}
	&\text{if $n\in 1+2\ZZ$.}
	\end{cases}
\end{equation*}
%
%
%
 Analogous to working with our two poles positioned to the left of our braid diagrams, we can replace $e_n$ with 
\begin{equation}
e_{0^\vee} = T_{0^\vee} - q_n = 
(T_1^{-1} \cdots T_{n-1}^{-1}) e_n
(T_{n-1} \cdots T_1) , 
\end{equation}
(compare with  \eqref{eq:T0v}). 
Diagrammatically identify 
\[e_0^{\vee} = \TIKZ{
\draw[dotted] (.5*.5,0) rectangle (.5*5+.25,.75);
\foreach \x in {1,2,3,5}{\node[V] (b\x) at (.5*\x, 0) {}; \node[V] (t\x) at (.5*\x, .75){};}
\foreach \y in {1,2}{\node[cV]  (l\y) at (.25, \y*.25){}; }
	\draw[bend right=30] (l2) to (t1) (b1) to (l1);
	\foreach \x in {2,3,5}{\draw (t\x)--(b\x);}
	\node at (.5*4,.5*.75) {$\dots$};	
}\;.\]
Whilst  $e_0$ and $e_{0^\vee}$ do not commute, since $T_0(T_1 T_{0^{\vee}} T_1^{-1}) = (T_1 T_{0^{\vee}} T_1^{-1})T_0$, we do have that
$$e_0 (T_1 e_{0^{\vee}} T_1^{-1})  =  (T_1 e_{0^{\vee}} T_1^{-1})e_0 ,$$which can be pictured diagrammatically as follows 
\begin{equation}\label{eq:e0-e0check}
\begin{array}{ccc}
\TIKZ[scale=.5]{
\draw[dotted] (.5,0) rectangle (5+.5,4);
\foreach \x in {1,2,3,5}{\node[V] (b\x) at (\x, 0) {}; \node[V] (t\x) at (\x, 4){};}
\foreach \y in {1,2}{\node[V]  (c\y) at (.5, 4-\y*.33){}; }
\foreach \y in {1,2}{\node[cV]  (d\y) at (.5, 2-\y*.33){}; }
	\draw[bend right=50] (c1) to (1,4) (1,3) to (c2);
	\Cross[1,3][2,2] 
	\Cross[1,0][2,1]
	\draw (t2) to (2,3) (2,2) to (2,1);
	\draw[bend right=50] (d1) to (1,2) (1,1) to (d2);
	\foreach \x in {3,5}{\draw (t\x)--(b\x);}
	\node at (4,2) {$\dots$};	
}
&=&
\TIKZ[scale=.5]{
\draw[dotted] (.5,0) rectangle (5+.5,4);
\foreach \x in {1,2,3,5}{\node[V] (b\x) at (\x, 0) {}; \node[V] (t\x) at (\x, 4){};}
\foreach \y in {1,2}{\node[V]  (c\y) at (.5, 1-\y*.33){}; }
\foreach \y in {1,2}{\node[cV]  (d\y) at (.5, 3-\y*.33){}; }
	\draw[bend right=50] (c1) to (1,1) (1,0) to (c2);
	\Cross[2,3][1,4] 
	\Cross[2,2][1,1]
	\draw (2,3) to (2,2) (2,1) to (2,0);
	\draw[bend right=50] (d1) to (1,3) (1,2) to (d2);
	\foreach \x in {3,5}{\draw (t\x)--(b\x);}
	\node at (4,2) {$\dots$};	
}
\end{array}
\end{equation}

\begin{lem}\label{lem:Icheck}~
\begin{enumerate}[1.]
\item For $n$ even, we let 
$
\hat{I}_{0}  = e_0 e_2 \cdots e_{n-2} $ and 
$I_1^{\vee}  = T_1e_{0^{\vee}}T_1^{-1} e_3 \cdots e_{n-1} 
$. We have that 
\[I_0 =  \frac{-1}{q_0+q_0^{-1}}  \hat{I}_{0} I_{1}^{\vee} \hat{I}_0.\]

\item For $n$ odd, we let 
$
\hat{I}_{1} = e_1 e_3 \cdots e_{n-2}$
and $I_0^{\vee}  = e_{0^{\vee}}e_2 \cdots e_{n-1}$.
We have that 
\[I_1 = \hat{I}_{1} I_{0}^{\vee} \hat{I}_1.\]
\end{enumerate}
\end{lem}

\begin{proof}
When $n$ is even, we have that 
\begin{align}\label{ppff2}
I_0 &= e_0 e_2 \dots e_{n-2} e_n  
	 = e_0 e_2 \dots e_{n-2}  
(T_{n-1} \cdots T_1) e_{0^{\vee}} (T_1^{-1} \cdots T_{n-1}^{-1}) 	 
\end{align}
where the righthand-side of \eqref{ppff2} can be pictured as follows 
\begin{align}\label{eq:ppff1}
 	\TIKZ[scale=.5, yscale=.9]{
 		\draw[dotted] (.5,0) rectangle (7.5,8);
 		\foreach \x in {1,2,3,5,6,7}{\coordinate (b\x) at (\x, 0) {}; \coordinate (t\x) at (\x, 8){};}
 		\foreach \x/\y in {2/3, 5/6}{\draw[bend right=60] (\x,8) to (\y,8);}
 		\foreach \y in {1,2}{\coordinate (c\y) at (.5, 8-\y*.5); 
 				\coordinate (d\y) at (.5, 4-\y*.5); }
 		\draw[rounded corners] (d1)node[cV]{} 
 			to node[pos=.6, fill=white, sloped, inner sep=1pt]{\tiny$\cdots$}  (7,7) to (t7);
 		\draw[rounded corners] (d2)node[cV]{} 
 			to node[pos=.6, fill=white, sloped, inner sep=1pt]{\tiny$\cdots$}  (7,1) to (b7);
 		\draw[rounded corners, over]
 			(b5) to[bend right=20] (5,7) to[bend left=30] (6,7) 
 				to[bend left=20] (6,0);
 		\draw[rounded corners, over]
 			(b2) to[bend right=20]  (2,7) to[bend left=30] (3,7) 
 				to[bend left=20] (3,0);
 		\draw[bend right=30] (c1)node[V]{} to (1,8);
 		\draw[bend right=30, over]	(1,0)  to (c2);\node[V] at (c2){};
 		\node at (4,.2){\tiny$\cdots$};
 		\node at (4,8-.2){\tiny$\cdots$};
 		\foreach \x in {1,2,3,5,6,7}{\node[V] at (t\x){}; \node[V] at (b\x){};} 
 	}
 \;	=\;
 	\TIKZ[scale=.5, yscale=.9]{
 		\draw[dotted] (.5,0) rectangle (7.5,5);
 		\foreach \x in {1,2,3,5,6,7}{\coordinate (b\x) at (\x, 0) {}; \coordinate (t\x) at (\x, 5){};}
 		\foreach \x/\y in {2/3, 5/6}{
 			\draw[bend right=60] (\x,5) to (\y,5);
 			\draw[bend right=60] (\y,4) to (\x,4);
 			\draw[bend right=60] (\x+1,3) to (\y+1,3);
 			\draw[bend right=60] (\y+1,2) to (\x+1,2);
 			\draw[bend right=60] (\x,1) to (\y,1);
 			\draw[bend right=60] (\y,0) to (\x,0);}
 		\foreach \y in {1,2}{\coordinate (c\y) at (.5, 5-\y*.33); 
 				\coordinate (d\y) at (.5, 3-\y*.33); }
 		\draw[bend left] (d1)node[cV]{} to [bend right] (1,3);
 		\draw[bend left] (d2)node[cV]{} to [bend left] (1,2);
 		\draw[bend left] (c1)node[V]{} to [bend right] (1,5);
 		\draw[bend left] (c2)node[V]{} to [bend left] (1,4); 
 		\Cross[2,3][1,4] \Cross[1,1][2,2]
 		\draw (1,0) to (1,1) (2,3) to (2,2) (3,4) to (3,3) (3,2) to (3,1)
 			(6,1) to (6,2) (6,3) to (6,4) (7,3) to (7,5) (7,0) to (7,2);
 		\node at (4,.2){\tiny$\cdots$};
 		\node at (4,5-.2){\tiny$\cdots$};
 		\draw[white] (4,3) to node[sloped,black]{\tiny$\cdots$} (5,4);
 		\draw[white] (4,2) to node[sloped,black]{\tiny$\cdots$} (5,1);
 		\foreach \x in {1,2,3,5,6,7}{\node[V] at (t\x){}; \node[V] at (b\x){};} 
 	}
	\end{align}
	and the equality follows simply by isotopy.  Now the righthand-side of \eqref{eq:ppff1} is equal to  
\begin{align*}
&	 (e_0 e_2 \cdots e_{n-2}) (T_1 e_{0^{\vee}} T_1^{-1} e_3 e_5 \cdots e_{n-1})
		(e_2 \cdots e_{n-2})\\
	=&\tfrac{-1}{q_0 + q_0^{-1}} (e_0^2 e_2 \cdots e_{n-2}) (T_1 e_{0^{\vee}} T_1^{-1} e_3 e_5 \cdots e_{n-1})
		(e_2 \cdots e_{n-2})\\
	=&\tfrac{-1}{q_0 + q_0^{-1}} (e_0 e_2 \cdots e_{n-2}) (T_1 e_{0^{\vee}} T_1^{-1} e_3 e_5 \cdots e_{n-1})
		(e_0 e_2 \cdots e_{n-2})     \\ 
	=& \tfrac{-1}{q_0 + q_0^{-1}} \hat{I}_0 I_1^\vee \hat{I}_0,
\end{align*}where the penultimate equality follows by \eqref{eq:e0-e0check}.
Thus the result follows for $n$ even;
 the calculation for $n$ is odd follows similarly.
\end{proof}

\subsection{The symplectic blob algebra} 
We let  $Z$ denote the central element $Z = X_1 + \cdots + X_n + X_1^{-1} + \cdots + X_n^{-1}  \in Z(\HC_n)$. As   in \cite[Corollary \ 3.3]{DR25b}, we have that 
\begin{equation} 
I_1 I_0 I_1 = \varkappa I_1 \quad  \quad I_0 I_1 I_0 = \varkappa I_0, 
\end{equation}
where 
\begin{equation} \label{eq:I0I1I0=zI0} \varkappa = \begin{cases} 
\frac{1}{[n]} Z - \(q_0 q_n q^{-1}\) & \text{if $n\in 2\ZZ$, and}\\
\frac{1}{[n]} Z - \(q_0 q_n\) & \text{if $n\in 1+2\ZZ$.}
\end{cases}
\end{equation}
\begin{defn}[{\cite[Definition 3.6]{GN}}]\label{defblob}
Specialising the central element $\varkappa \mapsto  \kappa \in \Bbbk(q,q_0,q_n)$, we define the {\sf symplectic blob algebra}, $\BC_n(\kappa)$,  to be the quotient of $\TLC_n$ by the relations 
\begin{equation}
I_0 I_1 I_0 = \kappa I_0 \qquad  \qquad I_1 I_0 I_1 = \kappa I_1.
\end{equation}
\end{defn}
We emphasise that any element $\varkappa \in Z(\TLC_n)$
 acts as a scalar on every   finite-dimensional simple $\TLC_n$-module.  
 Therefore in order to understand all  finite-dimensional  simple $\TLC_n$-modules, it is enough to understand the  finite-dimensional  simple modules of  $\BC_n(\kappa)$ for all $\kappa$.

\begin{rmk}
In   \cite{MR2354870}, Green--Martin--Parker 
   define a diagrammatic algebra ``the symplectic blob algebra'', denoted 
  $B_n(\delta,\delta_L,\delta_R,\kappa_L,\kappa_R, \kappa_{LR})$. 
    For invertible parameters 
$$\begin{array}{cccc}
\delta=-(q+q^{-1})
&
\quad\delta_L=-(\param+\param^{-1})\quad
&
\quad\delta_R=-(\paramt+\paramt^{-1})		 \quad
\\
\quad \kappa_L = \param q^{-1} +  q \param ^{-1}\quad
&
\kappa_R=\paramt q^{-1} +  q \paramt^{-1}
&
\kappa_{LR}= \kappa . 
\end{array}
$$
It is proven that $B_n(\delta,\delta_L,\delta_R,\kappa_L,\kappa_R, \kappa_{LR})$ is isomorphic to 
  $\BC_n(\kappa)$ 
in  \cite[Theorem 3.4]{MR2928127}.
 
 \end{rmk}

\section{Calibrated representations   symplectic blob algebras }
\label{sec:Calibrated}

In this section we will consider the algebras from 
\cref{sec:two-boundary algebras} 
 solely over a generic field  $\Ri=
\Bbbk(q, q_0, q_n)$. 
 For all of these algebras, the   calibrated  modules can be defined to be 
  the finite-dimensional and simple modules on which the Jucys--Murphy elements 
  $X_1$, $X_2$, \dots, $X_n$ are simultaneously diagonalizable. From now on, we will use the parameters
\begin{equation}
\alpha_1=q_0q_n \qquad \text{and} \qquad 
\alpha_2=-q_0q_n^{-1}. 
\label{eq:alphas-defn}
\end{equation}

\subsection{Combinatorics of calibrated $\TLC_n$-modules. } All of the material of this subsection is a review of results from  
  \cite{DR25a} and  \cite[Theorem \ 4.3]{DR25b}. 
We define a {\sf residue} to be an $n$-tuple $\gamma=(\gamma_1,\gamma_2 ,\dots  \gamma_n) \in (\mathcal{R}^\times)^n$.  
The group $W(C_n)$   acts on residues by 
\[w \cdot (\gamma_1, \dots, \gamma_n) = (\gamma_{w^{-1}(1)}, \dots, \gamma_{w^{-1}(n)}),
\quad \text{writing } \gamma_{-i} = \gamma_i^{-1}.\] 
We let $\mathcal{R}_\gamma$     denote the 1-dimensional representation of 
$\mathcal{R}[X_1^{\pm1},X_2^{\pm1},\dots , X^{\pm1}_n]$ upon which $X_i$ acts as 
  $\gamma_i \in \mathcal{R}$.  
  The {\sf principal series module} corresponding to $\gamma$ is defined to be the module
  $$M(\gamma) = {\rm Ind}^{H_n}_{\mathcal{R}[X_1^{\pm1},X_2^{\pm1},\dots , X^{\pm1}_n]}(\mathcal{R}_\gamma)$$and this module has dimension $|W(C_n)|=2^nn!$.  We note that if $\gamma $ and $\gamma'$ are in the same 
  $W(C_n)$-orbit then $M(\gamma)\cong M(\gamma')$. 

For generic parameters, the classification of the calibrated $\TLC_n$-modules amounts to: first 
classifying the calibrated $H_n$-modules; 
  and then classifying which of these summands factors through the quotient onto $\TLC_n$. 
The first is a large part of the story in \cite{DR25a}, and is described both in terms of \emph{skew local regions} (elements of the reflection group $W$ that have certain behaviour on $\gamma$), and, in nice circumstances, by decorated box arrangements and corresponding generalizations of tableaux. The latter is described in  \cite[Theorem \ 4.3]{DR25b}. 
We now recall the required background material that  we need from  \cite{DR25a} and  \cite[Theorem \ 4.3]{DR25b}.

First, the principal series modules (up to isomorphism) 
for which a  calibrated submodule   survives  under the  quotient of \eqref{eq:smash1} are those of the form 
\begin{equation}\label{eq:gamma}
\gamma = (\gamma_1, \gamma_2, \gamma_3, \cdots, \gamma_n) 
=(\gamma_1, \gamma_1q^2, \gamma_1q^4,\dots , \gamma_1q^{2n-2}).
\end{equation}
Further, the $\HC_n$-summands  of $M(\gamma) $   that survive 
 can be encapsulated in terms of   specific box arrangement combinatorics, which we     now recall.  
%
%
To a residue $\gamma$ as in \eqref{eq:gamma}, 
 we associate a 2-row (Young) diagram consisting of  $2n$ boxes given by 
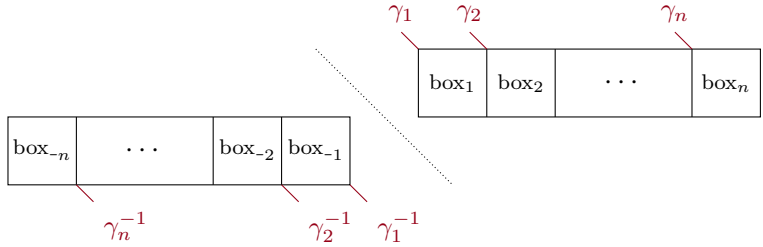
\[\TIKZ[scale=.9]{
\draw (1-.5,0) rectangle (6-.5,1);
\draw (-1+.5,0) rectangle (-6+.5,-1);
\foreach \x in {2,3,5}{\draw (\x-.5,0) to (\x-.5,1) (-\x+.5,0) to (-\x+.5,-1);}
\foreach \x/\y in {1/\gamma_1, 2/\gamma_2, 5/\gamma_n}{
	\draw[carmine] (\x-.5,1) to +(-.25, .25) node[above] {\small $\y$};}
\foreach \x/\y in {-1/\gamma_1^{-1}, -2/\gamma_2^{-1}, -5/\gamma_n^{-1}}{
	\draw[carmine] (\x+.5,-1) to +(.25, -.25) node[below right] {\small $\y$};}
\draw[densely dotted] (-1,1) to (1,-1);
\node at (4-.5,.5) {$\cdots$};
\node at (-4+.5,-.5) {$\cdots$};
\foreach \x/\y in {1/\Box_1, 2/\Box_2, 5/\Box_n}{
	\node at (\x,.5) {\scriptsize $\y$};}
\foreach \x/\y in {-1/\Box_{\sm 1}, -2/\Box_{\sm 2}, -5/\Box_{\sm n}}{
	\node at (\x,-.5) {\scriptsize $\y$};
	}
}
\]
letting $\Box_i$ be the box of residue   $\gamma_i$. 
We will  record  whether a 
 residue   is equal to one of our special points, that is 
$\gamma_i \in  \{\alpha_1^{\pm 1}, \alpha_2^{\pm 1}\}$,   
 by placing a bead to the northwest of $\Box_i$ and a bead to the southeast of $\Box_{-i}$. We call this picture the  {\sf (two row) shape} associated to $\gamma$.  For example, when $n=2$, there are nine such shapes, pictured in \cref{shapes-2}. 
 
 \begin{figure}[ht!]
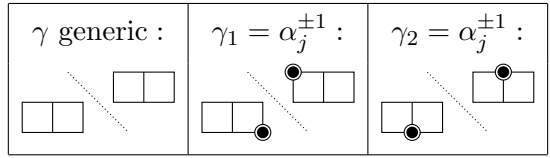

\[
{\def\arraystretch{1.5}
\begin{array}{|c|c|c|}\hline
\gamma\text{ generic}: & 
		\gamma_1 = \alpha_j^{\pm 1}: &
		\gamma_2 = \alpha_j^{\pm 1}: 
\\[3pt]
	\TIKZ[scale=.4]{ 
	\draw (1-.5,0) rectangle (3-.5,1);
	\draw (-1+.5,0) rectangle (-3+.5,-1);
	\draw (2-.5,0) to (2-.5,1) (-2+.5,0) to (-2+.5,-1);
	\draw[densely dotted] (-1,1) to (1,-1);
	} 
& 
	\TIKZ[scale=.4]{
	\draw (1-.5,0) rectangle (3-.5,1);
	\draw (-1+.5,0) rectangle (-3+.5,-1);
	\draw (2-.5,0) to (2-.5,1) (-2+.5,0) to (-2+.5,-1);
	\node[M] at (1-.5,1){};
	\node[M] at (-1+.5,-1){};
	\draw[densely dotted] (-1,1) to (1,-1);
	} 
&
	\TIKZ[scale=.4]{
	\draw (1-.5,0) rectangle (3-.5,1);
	\draw (-1+.5,0) rectangle (-3+.5,-1);
	\draw (2-.5,0) to (2-.5,1) (-2+.5,0) to (-2+.5,-1);
	\node[M] at (2-.5,1){};
	\node[M] at (-2+.5,-1){};
	\draw[densely dotted] (-1,1) to (1,-1);
	}
\\[10pt]\hline
\end{array}}
\]
\caption{The nine possible two row shapes for $n=2$ (where $9=1+2^2+2^2$ for the choices of $j=1,2$).}
\label{shapes-2}
\end{figure}

\begin{rmk}If one compares this combinatorial model directly to that of \cite{DR25a, DR25b}, some notable simplifications may become apparent. 
  In particular, there are additional arrangements of the marker \,\TIKZ{\node[M] at (0,0){};} needed in \cite[\S4]{DR25b} that are omitted here. We are able to do this due to the genericity conditions on parameters $q$, $\alpha_1$, and $\alpha_2$, driven by focusing on cases where $B_n(\kappa)$ is 
semisimple. Specifically, the conditions on these parameters imply (1) the rows of any relevant shape will not overlap, and (2) any given shape will not intersect with both the residues $\alpha_1$ and $\alpha_2$. Hence, for any marked shape, we can necessarily choose $\gamma_1$ so that the corresponding row has its marker placed to the NW of a box. 
\end{rmk}

\begin{defn} \label{TABLEAU-DEFN}
A {\sf tableau}   of a (two row) shape associated to $\gamma$ is a bijective function 
\[\cF: \{-n, \dots, - 1, 1, \dots n\} \to \{-n, \dots, - 1, 1, \dots n\},\] 
thought of as filling $\Box_i$ with $\cF(i)$. A tableau is {\sf standard} if it satisfies the following:
\begin{enumerate}[{\qquad Symmetry}]
\item[Symmetry:] $\cF(i) = - \cF({-i})$ for all $-n\leq i \leq n $.
\item[Adjacency:] $\cF(i) < \cF({i+1})$ for all $1\leq i \leq n $. 
\item [Beads:] If $\Box_i$ for $i\geq 1$ is to the right of  a bead, then $\cF(i)>0$.
\end{enumerate}
\end{defn}

\begin{eg}For example, let $\gamma=(	\alpha_1 q^{-2}	,	\alpha_1 )$ whose associated Young diagram is as follows 
$$\TIKZ[scale=.4]{
\draw (1-.5,0) rectangle (3-.5,1);
\draw (-1+.5,0) rectangle (-3+.5,-1);
\draw (2-.5,0) to (2-.5,1) (-2+.5,0) to (-2+.5,-1);
\node[M] at (2-.5,1){};
\node[M] at (-2+.5,-1){};
\draw[densely dotted] (-1,1) to (1,-1);
}$$where we note that the bead lies on the northwest corner of the box whose residue is the special point $\alpha_1$. 
The standard tableaux of shape $\gamma$ are as follows 
$$
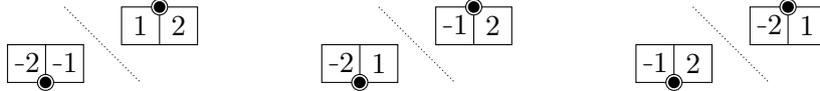
\TIKZ[scale=.5]{
\draw (1-.5,0) rectangle (3-.5,1);
\draw (-1+.5,0) rectangle (-3+.5,-1);
\draw (2-.5,0) to (2-.5,1) (-2+.5,0) to (-2+.5,-1);
\node[M] at (2-.5,1){};
\node[M] at (-2+.5,-1){};
\draw[densely dotted] (-1,1) to (1,-1);
\node at (1,.5) {$1$};	\node at (2,.5) {$2$};
\node at (-1,-.5) {$\sm 1$};	\node at (-2,-.5) {$\sm 2$};
} \qquad \qquad
\TIKZ[scale=.5]{
\draw (1-.5,0) rectangle (3-.5,1);
\draw (-1+.5,0) rectangle (-3+.5,-1);
\draw (2-.5,0) to (2-.5,1) (-2+.5,0) to (-2+.5,-1);
\node[M] at (2-.5,1){};
\node[M] at (-2+.5,-1){};
\draw[densely dotted] (-1,1) to (1,-1);
\node at (1,.5) {$\sm 1$};	\node at (2,.5) {$2$};
\node at (-1,-.5) {$1$};	\node at (-2,-.5) {$\sm 2$};
}  \qquad  \qquad
\TIKZ[scale=.5]{
\draw (1-.5,0) rectangle (3-.5,1);
\draw (-1+.5,0) rectangle (-3+.5,-1);
\draw (2-.5,0) to (2-.5,1) (-2+.5,0) to (-2+.5,-1);
\node[M] at (2-.5,1){};
\node[M] at (-2+.5,-1){};
\draw[densely dotted] (-1,1) to (1,-1);
\node at (1,.5) {$\sm 2$};	\node at (2,.5) {$1$};
\node at (-1,-.5) {$2$};	\node at (-2,-.5) {$\sm 1$};
}
$$
Because of the rotational symmetry of both the shapes and the standard  tableaux, and because $\{q^{2\ell} \gamma_1 ~|~ \ell = 0, \dots, n-1\}$ are distinct from $\{q^{-2\ell} \gamma_1^{-1} ~|~ \ell = -1, 0, \dots, n-1\}$, it is sufficient to just work with the half shape placed at $\gamma$. For example,
\begin{equation}\label{TABLEAU-DEFN2}
\TIKZ[scale=.5]{
\draw (1-.5,0) rectangle (3-.5,1);
\draw (-1+.5,0) rectangle (-3+.5,-1);
\draw (2-.5,0) to (2-.5,1) (-2+.5,0) to (-2+.5,-1);
\node[M] at (2-.5,1){};
\node[M] at (-2+.5,-1){};
\draw[densely dotted] (-1,1) to (1,-1);
\node at (1,.5) {$\sm 1$};	\node at (2,.5) {$2$};
\node at (-1,-.5) {$1$};	\node at (-2,-.5) {$\sm 2$};
} 
\quad \longleftrightarrow \quad 
\TIKZ[scale=.5]{
\draw (1-.5,0) rectangle (3-.5,1);
\draw (2-.5,0) to (2-.5,1) ;
\node[M] at (2-.5,1){};
\node at (1,.5) {$\sm 1$};	\node at (2,.5) {$2$};
}
\end{equation}
\end{eg}

For a given tableau $\cF$, the box filled with the entry $i$ has residue  $\gamma_{\cF^{-1}(i)}$.  
For brevity, we set 
\begin{equation}
\gamma^\cF_i := \gamma_{\cF^{-1}(i)} \qquad \gamma^\cF_{-i} := (\gamma^\cF_i)^{-1}.
\end{equation} 
The action of $W(C_n)$ on residues 
 can be recorded  on  tableaux in the natural fashion, that is  
 $w\cdot\cF$ for $w\in W(C_n)$  
 is the tableau determined by 
 $(w\cdot\cF)(i)=w\cdot(\cF(i))$   for $-n\leq i \leq n$.

\begin{thm}[{\cite[Theorem \ 3.3]{DR25a}, \cite[Theorem \ 4.1]{DR25b}}]\label{thm:calibratedTL}
Calibrated representations of $\TLC_n$ are indexed by the elements $\gamma$ as in \eqref{eq:gamma} such that $\gamma_i\neq\pm1$ for all $i=1,\dots,n$. The calibrated  module $\Cal(\gamma)$ has basis 
\[\{v_\cF ~|~ \cF\text{ is a standard  tableau of shape $\gamma$}\},\] 
with action given by 
\begin{align*}
X_i v_{\cF} &= \gamma^\cF_i v_\cF&&\text{for }i=1, \dots, n;\\
T_i v_{\cF} &= [T_i]_{\cF} v_{\cF} 
	+ \sqrt{-([T_i]_{\cF}-q)([T_i]_{\cF}+q^{-1})}\ v_{s_i \cF} 
			&&\text{for } i=1, \dots, n-1;  \\
T_0 v_{\cF} &= [T_0]_{\cF} v_{\cF} 
	+ \sqrt{-([T_0]_{\cF}-q_0)([T_0]_{\cF}+q_0^{-1})}\ v_{s_0 \cF};
\end{align*}
where
\begin{align*}
[T_i]_{\cF}  = \frac{q-q^{-1}}{1-\gamma^\cF_i \gamma^\cF_{-(i+1)}} 
\qquad\qquad [T_0]_{\cF}  = \frac{(q_0-q_0^{-1}) + (q_n - q_n^{-1})\gamma^\cF_{-1}}{1-(\gamma^\cF_{-1})^2} 
\end{align*}
for $1\leq i \leq n$.
\end{thm}

One might be concerned that $s_i \cdot  \cF$ is not necessarily standard, and hence $v_{s_i \stt}$ may not be an element of the basis of $\Cal(\gamma)$; however, the coefficient of $v_{s_i \stt}$ in each of the formulas above is 0 exactly when $s_i \cdot \stt$ is not standard.

Denoting $Q_0 = q_0 - q_0^{-1}$ and $Q_n = q_n - q_n^{-1}$, it is a straightforward set of calculations to find the useful identities
\begin{align*}
[T_0]_{\cF} - q_0 &= \gamma^\cF_{-1} \left(
	\frac{(q_0 \gamma^\cF_{-1} -q_0^{-1}\gamma^\cF_{1}) + Q_n}{1-(\gamma^\cF_{-1})^2}\right)\qquad 
[T_0]_{\cF} + q_0^{-1} &= \gamma^\cF_{-1} \left(
	\frac{(q_0 \gamma^\cF_{1} -q_0^{-1}\gamma^\cF_{-1}) + Q_n}{1-(\gamma^\cF_{-1})^2}\right).\\
\end{align*}
So (assuming, as we have, that $\gamma^\cF_{1} \ne 0$), we have that 
\begin{itemize} 
\item[$(i)$] $
[T_0]_{\cF} - q_0  = 0 $ if and only if 
$	\gamma^\cF_{1}  = \pm q_0 q_n^{\pm 1} = \alpha_1 \text{ or } \alpha_2$; 
\item[$(ii)$]
$
[T_0]_{\cF} + q_0^{-1}  = 0$  if and only if 
$	\gamma^\cF_{1}  = \pm q_0^{-1} q_n^{\mp 1}= \alpha_1^{-1} \text{ or } \alpha_2^{-1}$. 
\end{itemize}
 In particular, these are precisely the conditions 
  for $\stt^{-1}(1)$ to be a marked box. 

We can also take a moment now to get our hands on the action of $T_{0^\vee} = X_1 T_0^{-1} = X_1(T_0 - Q_0)$. To that end, it is straightforward to compute 
\begin{align}
&[T_0]_{\cF} - Q_0 = \gamma^\cF_{-1} \left(
	\frac{Q_n + Q_0\gamma^\cF_{-1} }{1-(\gamma^\cF_{-1})^2}\right) 
	\label{eq:T0-Q}\\
&([T_0]_{\cF} - q_0)([T_0]_{\cF} + q_0^{-1})
	= \frac{(\gamma^\cF_{-1})^2}{1-(\gamma^\cF_{-1})^2}
		\left(Q_0^2 + Q_k^2 - (\gamma^\cF_{1} - \gamma^\cF_{-1})^2 + Q_0Q_k(\gamma^\cF_{1} + \gamma^\cF_{-1})\right).  \label{eq:(T0-q)(T_0+q)}
\end{align}
This implies that 
\begin{align*}
T_0^{-1} v_{\cF} &= ([T_0]_{\cF} - Q_0) v_{\cF} 
	+ \sqrt{-([T_0]_{\cF}-q_0)([T_0]_{\cF}+q_0^{-1})}\ v_{s_0 \cF},
	\end{align*}
and hence 
\begin{align}
T_{0^\vee}v_{\cF} &= X_1 T_0^{-1} v_{\cF} \nonumber\\
&= X_1([T_0]_{\cF} - Q_0) v_{\cF} 
	+ X_1\sqrt{-([T_0]_{\cF}-q_0)([T_0]_{\cF}+q_0^{-1})}\ v_{s_0 \cF}\nonumber\\
&= [T_{0^\vee}]_{\cF} v_{\cF} 
	+ {\gamma^\cF_{-1}\sqrt{-([T_{0^\vee}]_{\cF}-q_n)([T_{0^\vee}]_{\cF}+q_n^{-1})}}\ v_{s_0 \cF},
	\end{align}
where we can explicitly calculate the coefficient 
\begin{equation}
 [T_{0^\vee}]_{\cF} = 
 	\frac{Q_n + Q_0\gamma^\cF_{-1}}{1-(\gamma^\cF_{-1})^2}.
\end{equation}
The last step is a combination of \eqref{eq:T0-Q}, together with 
\[([T_0]_{\cF}-q_0)([T_0]_{\cF}+q_0^{-1})=([T_{0^\vee}]_{\cF}-q_n)([T_{0^\vee}]_{\cF}+q_n^{-1}),\] 
due to the symmetry under $q_0 \leftrightarrow q_n$ in the righthand side of \eqref{eq:(T0-q)(T_0+q)} (or one can simply expand both and compare). As before, we have 
\begin{itemize}
\item[$(i)$] $[T_{0^\vee}]_{\cF} - q_n  = 0 $  if and only if 
$	\gamma^\cF_{1} = \pm q_0^{\pm 1} q_n = \alpha_1 \text{ or } \alpha_2^{-1}$;

\item[$(ii)$]
$ [T_{0^\vee}]_{\cF} + q_n^{-1}  = 0$  if and only if 
$	\gamma^\cF_{1}  = \pm q_0^{\mp 1} q_n^{- 1} = \alpha_1^{-1} \text{ or } \alpha_2$.  
\end{itemize}

\subsection{Calibrated  $\BC_n(\kappa)$-modules}\label{ytieuowrytoiueytiouewr}
Our goal in this section is to classify which calibrated $\TLC_n$-modules further factor through the quotient onto the blob algebra $\BC_n(\kappa) $, see Definition \ref{defblob}. 
The modules that survive the quotient
 $\BC_n(\kappa)$ are those where
  either
 $\varkappa$    acts as  $\kappa \in \mathcal{R}$, 
or $I_0$ and $I_1$ act by $0$. In particular, we want to classify those $\Cal(\gamma)$ for which 
\begin{enumerate}
\item[$(i)$] the element $\varkappa$ from \eqref{eq:I0I1I0=zI0} acts by the fixed $\kappa$; or 
\item[$(ii)$] at least one of $e_0, e_2, \dots$ acts by 0 and at least one of $e_1, e_3, \dots$ acts by 0.
\end{enumerate} 
By \cite[Proposition \ 4.4]{DR25b}, the first happens exactly when 
\begin{equation}\label{eq:good gamma}
\kappa = \(\gamma_1 q^{n-1}\)- \begin{cases} 
 \(\alpha_1 q^{-1}\) & \text{when $n$ is even, and}\\
\(\alpha_1\) & \text{when $n$ is odd.}
\end{cases}
\end{equation}
We recall that we work in a generic setting, that is over the field $\Bbbk(q, q_0, q_n)$ and we assume moreover that $\kappa$ is as in \eqref{eq:good gamma} for some $\gamma_1$ such that $\gamma_1\notin \pm q^{\ZZ}$ and $\gamma_1\notin \alpha_i^{\pm1}q^{2\ZZ}$. For the remainder of this section, we explore the second case: classifying those calibrated modules on which $I_0$ and $I_1$ act by $0$, and prove the following proposition.

%
%

\begin{prop}\label{prop:calibrated B modules}
Let $\gamma = (\gamma_1, \dots, \gamma_n)$ with $\gamma_i = q^{2(i-1)}\gamma_1$, and consider the $\TLC_n$-module $\Cal(\gamma)$. Then $\Cal(\gamma)$ is also a calibrated $B_n(\kappa)$-module precisely when one of the following circumstances occurs.
\begin{enumerate}[(I)]

\item    $\gamma_1$ is related to $\kappa$ by \eqref{eq:good gamma},  in which case $N(\gamma)$ is $2^n$-dimensional; or

\item There is  some $i \le (n+1)/2$ for which $\gamma_i \in \{\alpha_1^{\pm 1}, \alpha_2^{\pm 1}\}$, with some additional restrictions for small $i$ as follows:
\begin{itemize}
\item If $i =  (n+1)/2$, then $\gamma_i = \alpha_1$.
\item If $i =  {n}/{2}$, then $\gamma_i \ne \alpha_1^{-1}$.
\end{itemize}
\end{enumerate}
Moreover this provides a complete set of non-isomorphic simple $B_n(\kappa)$-modules.
\end{prop}

Let us assume for the remainder of this section that we are not in case (I), namely 
 that 
$\varkappa$ does not act on $N(\gamma)$ by $\kappa\in \Bbbk$.
  As we will see in the following lemma, this necessarily narrows our search to those shapes marked by a bead; in particular, they will necessarily be marked on one of the first $n/2$ boxes. To prove this, recall that  $e_0 = T_0 - q_0$, $e_{0^\vee} = T_{0^\vee} - q_n$,  and $e_i = T_i - q$ for $i = 1, \dots, n-1$. So we can see from \cref{thm:calibratedTL} (and the calculations following that theorem) that 
\begin{align}
e_0 v_\cF &= 0 &\text{ exactly when } &&& 
	\gamma^\cF_{1}\in \{\alpha_1, \alpha_2\}, \nonumber\\
e_{0^\vee} v_\cF &= 0 &\text{ exactly when } &&& 
	\gamma^\cF_{1}\in \{\alpha_1, \alpha_2^{-1}\},
	\label{eq:e-zeros}\\
e_i v_\cF &= 0 &\text{ exactly when } &&& 
	\gamma^\cF_{i}(\gamma^\cF_{i+1})^{-1} = q^2.\nonumber
\end{align}

\begin{lem}\label{lem:bad shapes}
If both $I_0$ and $I_1$ act on $N(\gamma)$ by $0$, then $\gamma_i\in  \{\alpha_1^{\pm 1}, \alpha_2^{\pm 1}\}$ for some $i\le (n+1)/2$.
\end{lem}
\begin{proof}
Consider the alternating  tableaux 
\begin{align*}
\TIKZ[scale=.5]{
\node[left] at (-5, .5){\strut $\cF  = $};
\draw (-5,0) rectangle (5.2,1);
\foreach \x in {-4, -2, -1, 0, 1, 2, 4}{\draw (\x,0) to (\x,1);}
\foreach \c [count=\x from -2] in {\sm 3,  \sm 1 ,2, 4}
	{\node at (\x+.5, .5) {\footnotesize $\c$};}
\node at (4.6, .5) {\footnotesize $n \sm 1$};
\node at (-4.5, .5) {\footnotesize $\sm n $};
\node at (-3,.5) {$\cdots$};
\node at (3,.5) {$\cdots$};
\draw[carmine] (0,1) to +(-.5, .5) node[above] {\footnotesize $\gamma_{\frac{n+1}{2}+1}$};
\draw[black!70, thin] (-.5,0) to +(0, -.5) node[below, inner sep=2pt] {\tiny (center)};
}
\TIKZ[scale=.5]{
\node[left] at (-5, .5){\strut \qquad or \qquad $\cF  = $};
\draw (-5,0) rectangle (5.2,1);
\foreach \x in {-4, -2, -1, 0, 1, 2, 4}{\draw (\x,0) to (\x,1);}
\foreach \c [count=\x from -2] in {\sm 4, \sm 2 ,1, 3}
	{\node at (\x+.5, .5) {\footnotesize $\c$};}
\node at (-4.5, .5) {\footnotesize $ \sm n$};
\node at (4.6, .5) {\footnotesize $  n\sm1  $};
\node at (-3,.5) {$\cdots$};
\node at (3,.5) {$\cdots$};
\draw[carmine] (0,1) to +(-.5, .5) node[above] {\footnotesize $\gamma_{\frac{n}{2}+1}$};
\draw[black!70, thin] (0,0) to +(0, -.5) node[below, inner sep=2pt] {\tiny (center)};
}
\end{align*}for $n$ odd or even, respectively. 
If $\gamma_i \notin \{\alpha_1^{\pm 1}, \alpha_2^{\pm 1}\}$ for all $i\le (n+1)/2$, then this  tableau is standard; and $e_j v_{\cF} \ne 0$ for all $j = 1, \dots n-1$. So, of $n$ is even, then 
\[I_1 = e_1 e_3 \dots e_{n-1} \quad \text{does not act by $0$ on $v_{\cF}$}.\]
If $n$ is odd, then since 
\[\gamma_1^\cF = \gamma_{\frac{n+1}{2}}^{-1} \notin \{\alpha_1^{\pm 1}, \alpha_2^{\pm 1}\},\]
we have $e_0v_{\cF} \ne 0$. So $I_0 = e_0 e_2 \cdots e_{n-1}$  does not act by $0$.
\end{proof}

Given the exceptional behaviour for $n=1$ and $n=2$, we now take a moment to consider the small rank cases. 
\begin{eg}
For $n=1$, we have $I_0 = e_0$ and $I_1 = e_1 =  e_{0^\vee}$. Hence, by \eqref{eq:e-zeros}, we are limited to the one-dimensional modules coming from marked shapes:  
\[\TIKZ[scale=.5]{
\node[left] at (0, .5) {$\Cal(\alpha) = \CC v_\cF$ { \quad where \quad } $\cF=$\strut};
\draw (0,0) rectangle (1,1); 
\draw (0,1) node[M]{} to[densely dotted] +(-.5,.5) node[above left, inner sep = 1pt]{\small$\alpha$};
\node at (.5, .5) {$1$};
} \]
for $\alpha \in \{ \alpha_1^{\pm 1}, \alpha_2^{\pm 1}\}$. Moreover   
 $I_0 v_\cF = 0$ if and only if  $\alpha = \alpha_1$ or $\alpha_2$; and 
 $I_1v_\cF = 0$ if and only if   $\alpha = \alpha_1$ or $\alpha_2^{-1}$. 
Hence, there is exactly one $\TLC_1$-module annihilated by both $I_0$ and $I_1$, namely: $\Delta(\alpha_1)$.
Similarly for $n=2$, $I_0 = e_0e_{2}$ and $I_1 = e_1$. But the only $\Cal(\gamma)$ annihilated by $e_1$ is 
\[\TIKZ[scale=.5]{
\node[left, inner sep = 2pt] at (0, .5) 
	{$\Cal((\alpha, q^2\alpha)) = \CC v_\cF$, \quad {where} \quad $\cF = $};
\draw (0,0) rectangle (2,1) (1,0) to (1,1); 
\node[M] at (0,1){};
\node at (.5, .5) {$1$}; \node at (1.5, .5) {$2$}; 
}. \]
For example, in $\Cal((q^{-2} \alpha, \alpha))$, the  tableau 
\[
\TIKZ[scale=.5]{
\node[left, inner sep = 2pt] at (0, .5) 
	{$\cF = $};
\draw (0,0) rectangle (2,1) (1,0) to (1,1); 
\node[M] at (1,1){};
\node at (.5, .5) {$\sm 1$}; \node at (1.5, .5) {$2$};
}\]
is standard, and $I_1 v_{\cF} \ne 0$. For the same reason, $\Cal(\gamma)$ is not anihilated by $I_1$ if $\gamma$ is unmarked.

\end{eg}

We are now ready to prove \cref{prop:calibrated B modules}. 
\begin{proof}[Proof of \cref{prop:calibrated B modules}]
Part (I) follows from \cite[Proposition  4.4]{DR25b}, so consider part (II). As we saw in \cref{lem:bad shapes}, we can restrict to those $\gamma$ where $\gamma_i \in \{\alpha_1^{\pm 1}, \alpha_2^{\pm 1}\}$ for some $i \leq (n+1)/2$.

\smallskip\noindent
{\bf Case: $i< n/2$}.    Here, for any standard  tableau $\cF$ of shape $\gamma$, enough of boxes $1, \dots, n$ are filled with a positive values to ensure that $\cF$ satisfies \emph{both} of the following:
\begin{enumerate}[\quad {$\bullet$}]
\item There is at least one pair $(2j, 2j+1)$ for which $\cF^{-1}(2j+1) = \cF^{-1}(2j) +1$ (the boxes filled with $2j$ and $2j+1$ are on adjacent diagonals), so that $e_{2j} v_\cF = 0$. Hence $\Cal(\gamma)$ is  annihilated by $I_0$.
\item There is at least one pair $(2j-1, 2j)$ for which $\cF^{-1}(2j) = \cF^{-1}(2j-1) +1$ (the boxes filled with $2j-1$ and $2j$ are on adjacent diagonals), in which case $e_{2j-1} v_\cF = 0$.  So $\Cal(\gamma)$ is also annihilated by $I_1$.
\end{enumerate}

\smallskip\noindent
{\bf  Case: $n$ odd and $i=(n+1)/2$}.   
  Again, a pigeonhole argument can be applied to show $I_0 v_\cF = I_1 v_{\cF} = 0$ for almost standard  tableaux $\cF$ as above. The exceptions for $I_0$ are those  standard  tableaux for which 
\[e_j v_{\cF} \ne 0 \quad  \text{ for all $j = 2, 4, \dots, n-1$,}\]
all of which satisfy $\cF^{-1}(1) = i$ (the box on diagonal $\gamma_i$ is filled with 1). 
(See \cref{fig:special-shape-tabuleaux-5}
 for the four such examples when $n=5$.)\ \  On each of these,  it follows from \eqref{eq:e-zeros} that $e_0 v_\cF = 0$ (and hence $I_0 v_\cF = 0$) exactly when $\gamma_i = \alpha_1$ or $\alpha_2$.

To show that we also have $I_1 \Cal(\gamma) = 0$ if and only if $\gamma_i = \alpha_1$, we will have to work a bit harder: the action of $e_n$ is not so straightforward on the basis $\{v_{\cF} ~|~ \cF \text{ standard }\}$. First, observe that there is a unique standard  tableau $\cF^*$ for which $e_j v_{\cF^*} \ne 0$ for \emph{all}  $j = 1, 2, \dots, n-1$: 
\[\begin{array}{c}
\cF_\star^{-1}: 2i - 1 \mapsto i+(n-1)/2 \quad \text{and} \quad 2i \mapsto i - (n+1)/2.\\[8pt]
\end{array}\]
pictured as follows.
\[
\TIKZ[scale=.55, font=\small]{
\draw (-5.6,0) rectangle (5,1);
\foreach \x in {-4, -2, -1, 0, 1, 2, 4}{\draw (\x,0) to (\x,1);}
\foreach \c [count=\x from -2] in {\sm 4, \sm 2 ,1, 3}
	{\node at (\x+.5, .5) {$\c$};}
\node at (4.5, .5) {$n$};
\node at (-4.8, .5) {\scriptsize $\sm n\! +\! 1$};
\node at (-3,.5) {$\cdots$};
\node at (3,.5) {$\cdots$};
\draw[carmine] (0,1) to +(-.5, .5) node[above] {\footnotesize $\gamma_{(n+1)/2}$};
\node[M] at (0,1){};
\draw[black!70, thin] (.5,0) to +(0, -.5) node[below, inner sep=2pt] {\tiny (center)};
}
\]
Using \cref{lem:Icheck}, we will build the action of $I_1 = \hat{I}_{1} I_{0}^{\vee} \hat{I}_1$ on $v_{\cF^*}$ one term at a time. Again, one can follow along with the example $n=5$ in \cref{fig:special-shape-tabuleaux-5}.
For $n$ odd, let $S_\odd$ and $S_\even$ be the subgroups of $W$ given by 
\[S_\odd = \< s_1, s_3, \dots, s_{n-2}\>
	\quad \text{ and } \quad 
	S_\even  = 
		\< s_2, s_4, \dots, s_{n-1}\>.\]
The set of  tableaux that are not annihilated by $\hat{I}_1  = e_1 e_3 \cdots e_{n-2}$ is  as follows 
\begin{align*}
F_\odd &= \{ \text{standard tableaux $\cF$ of shape $\gamma$}~|~ \cF^{-1}(2i) \ne \cF^{-1}(2i-1) +1 \text{ for all }i =1, \dots, (n-1)/2\}\\
	&= S_\odd \cF_\star.
\end{align*}
In fact, $S_\odd$ acts freely on this orbit: the condition $\cF^{-1}(2i) \ne \cF^{-1}(2i-1) +1$ is exactly what ensures $s_{2i-1}\cF$ is standard. Hence, for any $\cF \in F_\odd$, 
\[\hat{I}_1 v_{\cF } 
=\textstyle  \sum_{w \in S_\odd} c_w v_{w  \cF_\star},   \]
for some  $ c_w \ne 0 \text{ for all }w \in  S_\odd$.
But again, there is only one standard  tableau $\cF_\star$ for which there are \emph{no} pairs $(i, i+1)$ on adjacent diagonals. 
Hence, for all $w \in S_\odd - \{ {\rm id}\}$,  we have 
$e_{j} v_{w\cF_\star} = 0$ for all $j = 2, 4, \dots, n-1$. So for any $\cF \in F_\odd$, letting 
$c = c_{\rm id}$ be the   coefficient of $v_{\cF_\star}$ in $\hat{I}_1 
v_{ {\cF }}$, we have 
\begin{align*}
e_2 e_4 \cdots e_{n-1} \hat{I}_1 v_{\cF} 
 = e_2 e_4 \cdots e_{n-1} c v_{\cF } 
\textstyle 	 =  \sum_{w \in S_\even} c d_w v_{w{\cF_\star}},  
	\end{align*}
for some  $ d_w \ne 0 \text{ for all }w \in  S_\even$. 
The second equality above follows from similar reasoning as above for the odd permutations: every  tableau in  $S_\even \cF_\star$ is standard, and $S_\even$ acts freely on this orbit.

Now, since $w(1) = 1$ for all $w \in S_\even$, it follows from \eqref{eq:e-zeros} that $e_{0^{\vee}} v_{w\cF_\star} = 0$ exactly when $\gamma_{i} = \alpha_1$ (we already assumed $\gamma_i = \alpha_1$ or $\alpha_2$ when working with $I_0$ above). 
If so, then for all $\cF \in F_\odd$,
\[I_1 v_{\cF} = \hat{I}_1 
 {I_{0}^\vee \hat{I}_1 v_{\cF}} =  \hat{I}_1 ( {I_{0}^\vee \hat{I}_1 v_{\cF}} ) =\hat{I}_1 \times 0 
 =0,\]and hence $I_1 \Cal(\gamma) = 0$. Otherwise, if $\gamma_i \ne \alpha_1$, then by considering the coefficient of $v_{\cF_\star}$ in the final expansion of the action by $\hat{I}_1$, we can see that $I_1 v_{\cF_\star} \ne 0$ (the only element of the set $S_\odd S_\even$ that stabilizes $\cF_\star$  is $1$). Hence, $I_0 \Cal(\gamma) = I_1 \Cal(\gamma) = 0$ if and only if $\gamma = \alpha_1$.

\medskip

\paragraph{\bf Case: $n$ even and $i=\frac{n}{2}$.}\ This case is very similar to the last case: we handle the exceptional standard  tableaux by considering the orbit of a particularly special  tableau under certain actions. But now that $e_0$ and $e_n$ are both factors in $I_0$, the conditions on $\gamma_i$ become inclusive rather than exclusive. 

The same pigeonhole arguments as above can now be used to show $I_1 v_\cF = 0$ for all standard $\cF$; and that $I_0  v_\cF = 0$ for most standard $\cF$.  And again, those exceptions where $e_j v_\cF \ne 0$ for all $j = 2, 4, \dots, n-2$ all satisfy $\cF^{-1}(1) = 1$. If $\gamma_i = \alpha_1$ or $\alpha_2$, then we're done: $e_0  v_\cF = 0$ on all such  tableaux, and hence $I_0 v_\cF = 0$. Otherwise, we will proceed by studying the action of $I_0 = \frac{1}{\(q_0\)}\hat{I}_{0} I_{1}^{\vee} \hat{I}_0$ one term at a time. See \cref{fig:special-shape-tabuleaux-6} for the example where $n=6$.

\begin{figure}[ht!]
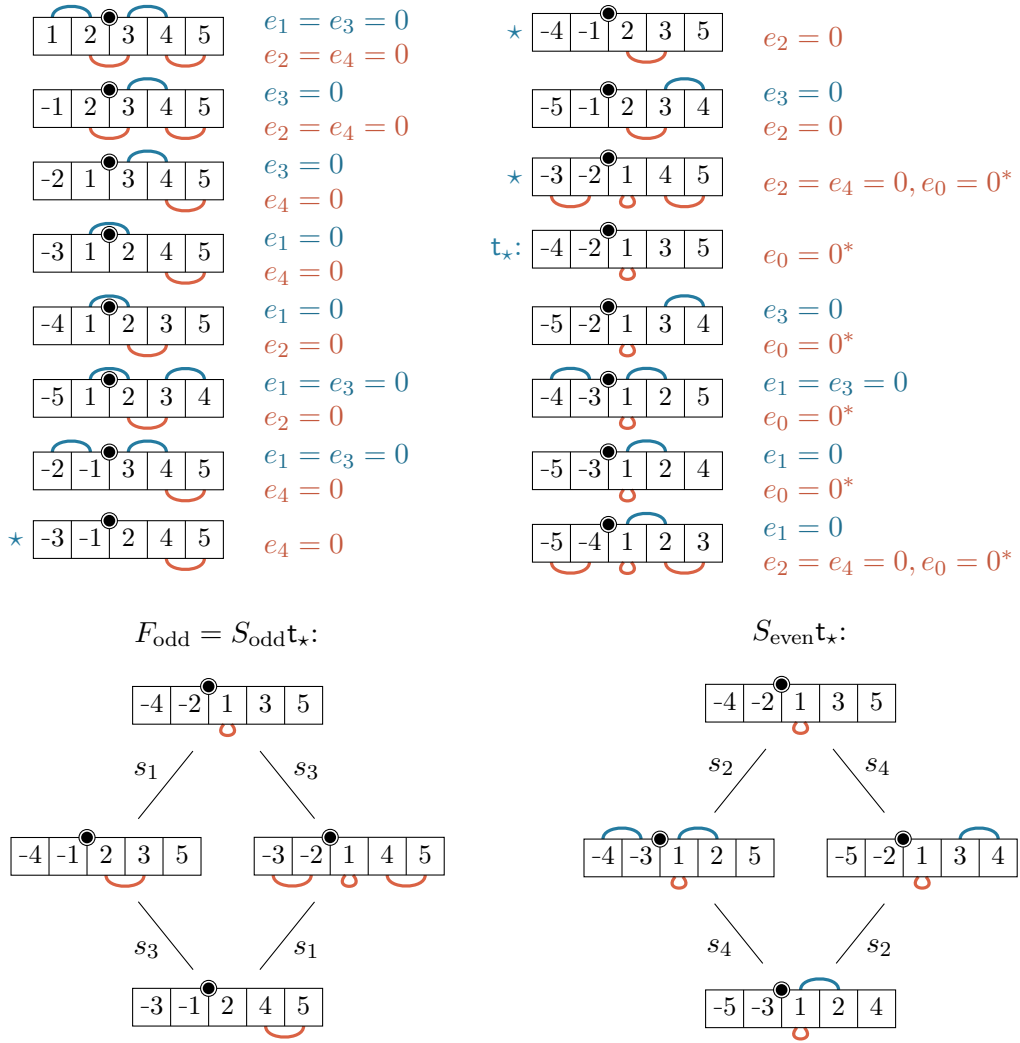

  \[
 \begin{array}{rl@{\qquad}rl}
 \TIKZ[scale=.5]{\FIVE{1,2,3,4,5}{2/3,4/5}{1/2,3/4}}
 &
 \begin{array}{l} \color{teal!90!black}
 	e_1=e_3 = 0\\ \color{copperred!90!black}
 	e_2 = e_4 = 0\end{array}
 & \TIKZ[scale=.5]{\FIVE{\sm 4,\sm 1,2,3,5}{3/4}{} \node[teal, left] at (0,.5){$\star$};}
 &
 \begin{array}{l}  \color{copperred!90!black}
 	 e_2 = 0 \end{array}\\
 \TIKZ[scale=.5]{\FIVE{\sm 1,2,3,4,5}{2/3,4/5}{3/4}}
 &
 \begin{array}{l} \color{teal!90!black}
 	 e_3 = 0\\ \color{copperred!90!black}
 	 e_2 = e_4 = 0\end{array}
 &\TIKZ[scale=.5]{\FIVE{\sm 5,\sm 1,2,3,4}{3/4}{4/5}}
 &
 \begin{array}{l} \color{teal!90!black}
 	 e_3 = 0\\ \color{copperred!90!black}
 	 e_2 = 0\end{array}\\
 \TIKZ[scale=.5]{\FIVE{\sm 2,1,3,4,5}{4/5}{3/4}}
 &
 \begin{array}{l} \color{teal!90!black}
 	 e_3= 0\\ \color{copperred!90!black}
 	 e_4 = 0\end{array}
 &\TIKZ[scale=.5]{\FIVE{\sm 3,\sm 2,1,4,5}{1/2,3/3,4/5}{} \node[teal, left] at (0,.5){$\star$};}
 &
 \begin{array}{l}  \color{copperred!90!black}
 	e_2 = e_4 = 0, e_0 = 0^* \end{array}\\
 \TIKZ[scale=.5]{\FIVE{\sm 3,1,2,4,5}{4/5}{2/3}}
 &
 \begin{array}{l} \color{teal!90!black}
 	 e_1= 0\\ \color{copperred!90!black}
 	 e_4= 0\end{array}
 &\TIKZ[scale=.5]{\FIVE{\sm 4,\sm 2,1,3,5}{3/3}{} \node[teal, left] at (0,.5){$\cF_\star$:};}
 &
 \begin{array}{l}  \color{copperred!90!black}
 	 e_0= 0^*\end{array}\\
 \TIKZ[scale=.5]{\FIVE{\sm 4,1,2,3,5}{3/4}{2/3}}
 &
 \begin{array}{l} \color{teal!90!black}
 	 e_1 = 0\\ \color{copperred!90!black}
 	  e_2 = 0\end{array}
 &\TIKZ[scale=.5]{\FIVE{\sm 5,\sm 2,1,3,4}{3/3}{4/5}}
 &
 \begin{array}{l}  \color{teal!90!black}
 	 e_3 = 0\\ \color{copperred!90!black}
 	 e_0 = 0^*\end{array}\\
 \TIKZ[scale=.5]{\FIVE{\sm 5,1,2,3,4}{3/4}{2/3,4/5}}
 &
 \begin{array}{l} \color{teal!90!black}
 	 e_1 = e_3 = 0\\ \color{copperred!90!black}
 	 e_2 = 0\end{array}
 &\TIKZ[scale=.5]{\FIVE{\sm 4,\sm 3,1,2,5}{3/3}{1/2,3/4}}
 &
 \begin{array}{l} \color{teal!90!black}
 	e_1 = e_3 = 0\\ \color{copperred!90!black}
 	e_0 = 0^*\end{array}\\
 \TIKZ[scale=.5]{\FIVE{\sm 2,\sm 1,3,4,5}{4/5}{1/2,3/4}}
 &
 \begin{array}{l} \color{teal!90!black}
 	e_1 = e_3  = 0\\ \color{copperred!90!black}
 	e_4 = 0\end{array}
 &\TIKZ[scale=.5]{\FIVE{\sm 5,\sm 3,1,2,4}{3/3}{3/4}}
 &
 \begin{array}{l} \color{teal!90!black}
 	e_1 = 0\\ \color{copperred!90!black}
 	 e_0 = 0^*\end{array}\\
 \TIKZ[scale=.5]{\FIVE{\sm 3,\sm 1,2,4,5}{4/5}{} \node[teal, left] at (0,.5){$\star$};}
 &
 \begin{array}{l}  \color{copperred!90!black}
 	e_4 = 0\end{array}
 &\TIKZ[scale=.5]{\FIVE{\sm 5,\sm 4,1,2,3}{1/2, 3/3, 4/5}{3/4}}
 &
 \begin{array}{l} \color{teal!90!black}
 	e_1 = 0 \\  \color{copperred!90!black}
 	e_2 = e_4 = 0, e_0 = 0^*\end{array}
 \end{array}
 \]

 \[\TIKZ[xscale=.8]{
 \node at (0, 5) {$F_\odd = S_\odd \cF_\star$:};
 \node (0) at (0, 4) {$\TIKZ[scale=.5]{\FIVE{\sm 4,\sm 2,1,3,5}{3/3}{}
 }$};
 \node (1) at (-2, 2) {$\TIKZ[scale=.5]{\FIVE{\sm 4,\sm 1,2,3,5}{3/4}{}}$};
 \node (3) at (2, 2) {$\TIKZ[scale=.5]{\FIVE{\sm 3,\sm 2,1,4,5}{1/2,3/3,4/5}{}}$};
 \node (13) at (0, 0) {$\TIKZ[scale=.5]{\FIVE{\sm 3,\sm 1,2,4,5}{4/5}{}}$};
 \draw (0) to node[above left, inner sep=2pt]{$s_1$} (1);
 \draw (0) to node[above right, inner sep=2pt]{$s_3$} (3);
 \draw (1) to node[below left, inner sep=2pt]{$s_3$} (13);
 \draw (3) to node[below right, inner sep=2pt]{$s_1$} (13);
 }\qquad\qquad
 \TIKZ[xscale=.8]{
 \node at (0, 5) {$S_\even \cF_\star$:};
 \node (0) at (0, 4) {$\TIKZ[scale=.5]{\FIVE{\sm 4,\sm 2,1,3,5}{3/3}{} 
 	}$};
 \node (2) at (-2, 2) {$\TIKZ[scale=.5]{\FIVE{\sm 4,\sm 3,1,2,5}{3/3}{3/4,1/2}}$};
 \node (4) at (2, 2) {$\TIKZ[scale=.5]{\FIVE{\sm 5,\sm 2,1,3,4}{3/3}{4/5}}$};
 \node (24) at (0, 0) {$\TIKZ[scale=.5]{\FIVE{\sm 5,\sm 3,1,2,4}{3/3}{3/4}}$};
 \draw (0) to node[above left, inner sep=2pt]{$s_2$} (2);
 \draw (0) to node[above right, inner sep=2pt]{$s_4$} (4);
 \draw (2) to node[below left, inner sep=2pt]{$s_4$} (24);
 \draw (4) to node[below right, inner sep=2pt]{$s_2$} (24);
 }\]
\caption{Example for the proof of \cref{prop:calibrated B modules} when $n=5$ and $i= \frac{5+1}{2} = 3$. Pictured below are 
the standard  tableaux $\cF$ of shape $\gamma = (\beta q^{-4}, \beta q^{-2}, \beta , \beta q^{2},  \beta q^{4})$. They are marked by which of $e_0, e_1, \dots, e_4$ act by $0$ on the corresponding weight vector $v_\cF$ in $\Cal(\gamma)$ (where $e_0$ acts by $0$ if and only if $\beta \in\{ \alpha_1,\alpha_2\}$). The  tableau marked $\cF_\star$ is the unique  tableau for which $e_j v_{\cF_\star} \ne 0$ for all $j = 1,\dots, 4$. The other  tableaux marked by $\star$ are the remaining of the elements of $F_\odd$, those whose corresponding weight vectors are not annihilated by any odd $e_{2j-1}$. Below that list, find the orbits of $\cF_\star$ under the action of  $S_\odd=  \<s_1,s_3\>$ and of  $S_\even = \<s_2,s_4\>$.}\label{fig:special-shape-tabuleaux-5}

\end{figure}

Let $\cF_\star$ be the unique standard  tableau of shape $\gamma$ satisfying $e_{j}v_{\cF_\star} \ne 0$ for all $j = 2, 3, \dots, n-1$:
\[\begin{matrix}
\begin{array}{rc@{\ }c@{\ }cl}
	\cF_\star^{-1}\colon 
		& 1 &\mapsto& \frac{n}{2}, &\\[3pt]
		& 2j &\mapsto& j + \frac{n}{2} &\quad \text{ for } j=1, \dots, \frac{n}{2}, \text{ and }\\[3pt]
		& 2j + 1 &\mapsto& j - \frac{n}{2} &\quad \text{ for } j=1, \dots, \frac{n}{2}-1.
		\end{array}\\[10pt]
 \end{matrix}\]
(there is no standard  tableau where $e_1, \dots, e_{n-1}$ all act by 0 in this case) which we picture as follows.
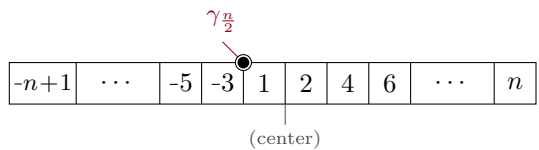
\[\TIKZ[scale=.55, font = \small]{
\draw (-6.6,0) rectangle (6,1);
\foreach \x in {-5, -3, -2, -1, 0, 1, 2, 3,5}{\draw (\x,0) to (\x,1);}
\foreach \c [count=\x from -3] in {\sm 5,\sm 3, 1,2,4,6}
	{\node at (\x+.5, .5) {$\c$};}
\node at (5.5, .5) {$n$};
\node at (-5.8, .5) {\footnotesize $\sm n \! + \! 1$};
\node at (-4,.5) {$\cdots$};
\node at (4,.5) {$\cdots$};
\draw[carmine] (-1,1) to +(-.5, .5) node[above] {\footnotesize $\gamma_{\frac{n}{2}}$};
\node[M] at (-1,1){};
\draw[black!70, thin] (0,0) to +(0, -.5) node[below, inner sep=2pt] {\tiny (center)};
}
\]

\begin{figure}[ht!]
  \[
 \begin{array}{rl@{\qquad}rl}
 \TIKZ[scale=.5]{\SIX{1,2,3,4,5,6}{2/3,4/5}{1/2,3/4,5/6}}
 &
 \begin{array}{l} \color{teal!90!black}
 	 e_1 = e_3 = e_5 = 0\\ \color{copperred!90!black}
 	 e_2 = e_4 = 0\end{array}
 &\TIKZ[scale=.5]{\SIX{\sm 6,\sm 1,2,3,4,5}{5/6,3/4}{4/5}}
 &
 \begin{array}{l} \color{teal!90!black}
 	e_3  = 0\\ \color{copperred!90!black}
 	e_2 = e_4  = 0\end{array}\\
 \TIKZ[scale=.5]{\SIX{\sm 1,2,3,4,5,6}{2/3,4/5}{5/6,3/4}}
 &
 \begin{array}{l} \color{teal!90!black}
 	e_3 = e_5 = 0\\ \color{copperred!90!black}
 	e_2 = e_4 = 0\end{array}
 &\TIKZ[scale=.5]{\SIX{\sm 3,\sm 2,1,4,5,6}{3/3,4/5,1/2}{5/6}}
 &
 \begin{array}{l} \color{teal!90!black}
 	e_5  = 0\\ \color{copperred!90!black}
 	e_2 = e_4 = 0\end{array}\\
 \TIKZ[scale=.5]{\SIX{\sm 2,1,3,4,5,6}{4/5}{5/6,3/4}}
 &
 \begin{array}{l} \color{teal!90!black}
 	e_3 = e_5 = 0\\ \color{copperred!90!black}
 	e_4 = 0\end{array}
 &\TIKZ[scale=.5]{\SIX{\sm 4,\sm 2,1,2,5,6}{3/3}{5/6} \node[copperred, left] at (0,.5){$\star$};}
 &
 \begin{array}{l} \color{teal!90!black}
 	e_5  = 0 \end{array}\\
 \TIKZ[scale=.5]{\SIX{\sm 3,1,2,4,5,6}{4/5}{5/6, 2/3}}
 &
 \begin{array}{l} \color{teal!90!black}
 	 e_1 = e_5 = 0\\ \color{copperred!90!black}
 	 e_4 = 0\end{array}
 &\TIKZ[scale=.5]{\SIX{\sm 5,\sm 2,1,3,4,6}{3/3}{4/5} \node[copperred, left] at (0,.5){$\star$};}
 &
 \begin{array}{l} \color{teal!90!black}
 	e_3 = 0 \end{array}\\
 \TIKZ[scale=.5]{\SIX{\sm 4,1,2,3,5,6}{3/4}{5/6,2/3}}
 &
 \begin{array}{l} \color{teal!90!black}
 	e_1 = e_5 = 0\\ \color{copperred!90!black}
 	e_2 = 0\end{array}
 &\TIKZ[scale=.5]{\SIX{\sm 6,\sm 2,1,3,4,5}{3/3,5/6}{4/5}}
 &
 \begin{array}{l} \color{teal!90!black}
 	e_3 = 0\\ \color{copperred!90!black}
 	e_4 = 0\end{array}\\
 \TIKZ[scale=.5]{\SIX{\sm 5,1,2,3,4,6}{3/4}{4/5,2/3}}
 &
 \begin{array}{l} \color{teal!90!black}
 	e_1 = e_3 = 0\\ \color{copperred!90!black}
 	e_2 = 0\end{array}
 &\TIKZ[scale=.5]{\SIX{\sm 4,\sm 3,1,2,5,6}{3/3}{1/2,3/4,5/6} \node[copperred, left] at (0,.5){$\star$};}
 &
 \begin{array}{l} \color{teal!90!black}
 	e_1 = e_3 = e_5 = 0 \end{array}\\
 \TIKZ[scale=.5]{\SIX{\sm 6,1,2,3,4,5}{3/4,5/6}{4/5,2/3}}
 &
 \begin{array}{l} \color{teal!90!black}
 	e_1 = e_3 = 0\\ \color{copperred!90!black}
 	e_2 = e_4 = 0\end{array}
 &\TIKZ[scale=.5]{\SIX{\sm 5,\sm 3,1,2,4,6}{3/3}{3/4} \node[copperred, left] at (0,.5){$\cF_\star$:};}
 &
 \begin{array}{l} \color{teal!90!black}
 	e_1 = 0 \end{array}\\
 \TIKZ[scale=.5]{\SIX{\sm 2,\sm 1,3,4,5,6}{4/5}{5/6,3/4,1/2}}
 &
 \begin{array}{l} \color{teal!90!black}
 	e_1 = e_3 = e_5 = 0\\ \color{copperred!90!black}
 	e_4 = 0\end{array}
 &\TIKZ[scale=.5]{\SIX{\sm 6,\sm 3,1,2,4,5}{3/3, 5/6}{3/4}}
 &
 \begin{array}{l} \color{teal!90!black}
 	e_1  = 0\\ \color{copperred!90!black}
 	e_4  = 0\end{array}\\
 \TIKZ[scale=.5]{\SIX{\sm 3,\sm 1,2,4,5,6}{4/5}{5/6}}
 &
 \begin{array}{l} \color{teal!90!black}
 	e_5 = 0\\ \color{copperred!90!black}
 	e_4 = 0\end{array}
 &\TIKZ[scale=.5]{\SIX{\sm 5,\sm 4,1,2,3,6}{3/3, 4/5,1/2}{3/4}}
 &
 \begin{array}{l} \color{teal!90!black}
 	e_1 = 0\\ \color{copperred!90!black}
 	e_2 = e_4 = 0\end{array}\\
 \TIKZ[scale=.5]{\SIX{\sm 4,\sm 1,2,3,5,6}{3/4}{5/6}}
 &
 \begin{array}{l} \color{teal!90!black}
 	e_5  = 0\\ \color{copperred!90!black}
 	e_2 = 0\end{array}
 &\TIKZ[scale=.5]{\SIX{\sm 6,\sm 4,1,2,3,5}{3/3, 4/5}{3/4}}
 &
 \begin{array}{l} \color{teal!90!black}
 	e_1 = 0\\ \color{copperred!90!black}
 	e_2 = 0\end{array}\\
 \TIKZ[scale=.5]{\SIX{\sm 5,\sm 1,2,3,4,6}{3/4}{4/5}}
 &
 \begin{array}{l} \color{teal!90!black}
 	e_3 = 0\\ \color{copperred!90!black}
 	e_2 = 0\end{array}
 &\TIKZ[scale=.5]{\SIX{\sm 6,\sm 5,1,2,3,4}{3/3, 4/5}{3/4,1/2,5/6}}
 &
 \begin{array}{l} \color{teal!90!black}
 	e_1 = e_3 = e_5 = 0\\ \color{copperred!90!black}
 	e_2 = 0\end{array}
 \end{array}\]

 \[\TIKZ{
 \node at (0, 5) {$S_\odd \cF_\star$:};
 \node (0) at (0, 4) {$\TIKZ[scale=.5]{\SIX{\sm 5, \sm 3, 1,2,4,6}{3/3}{3/4}
 	  }$};
 \node (3) at (-2, 2) {$\TIKZ[scale=.5]{\SIX{\sm 5,\sm 4,1,2,3,6}{1/2,3/3,4/5}{3/4}}$};
 \node (5) at (2, 2) {$\TIKZ[scale=.5]{\SIX{\sm 6, \sm 3,1,2,4,5}{3/3, 5/6}{3/4}}$};
 \node (35) at (0, 0) {$\TIKZ[scale=.5]{\SIX{\sm 6,\sm 4,1,2,3,5}{3/3, 4/5}{3/4}}$};
 \draw (0) to node[above left, inner sep=2pt]{$s_3$} (3);
 \draw (0) to node[above right, inner sep=2pt]{$s_5$} (5);
 \draw (3) to node[below left, inner sep=2pt]{$s_5$} (35);
 \draw (5) to node[below right, inner sep=2pt]{$s_3$} (35);
 }\qquad\qquad
 \TIKZ{
 \node at (0, 5) {$F_\even = S_\even \cF_\star$:};
 \node (0) at (0, 4) {$\TIKZ[scale=.5]{\SIX{\sm 5,\sm 3,1,2,4,6}{3/3}{3/4} 
  }$};
 \node (2) at (-2, 2) {$\TIKZ[scale=.5]{\SIX{\sm 5,\sm 2,1,3,4,6}{3/3}{4/5}}$};
 \node (4) at (2, 2) {$\TIKZ[scale=.5]{\SIX{\sm 4,\sm 3,1,2,5,6}{3/3}{3/4,5/6,1/2}}$};
 \node (24) at (0, 0) {$\TIKZ[scale=.5]{\SIX{\sm 4,\sm 2,1,3,5,6}{3/3}{5/6}}$};
 \draw (0) to node[above left, inner sep=2pt]{$s_2$} (2);
 \draw (0) to node[above right, inner sep=2pt]{$s_4$} (4);
 \draw (2) to node[below left, inner sep=2pt]{$s_4$} (24);
 \draw (4) to node[below right, inner sep=2pt]{$s_2$} (24);
 }\]
\caption{Example for the proof of \cref{prop:calibrated B modules} when $n=6$ and $i= \frac{6}{2} = 3$. Pictured below are 
the standard  tableaux $\cF$ of shape $\gamma = (\beta q^{-4}, \beta q^{-2}, \beta , \beta q^{2},  \beta q^{4}, \beta q ^6)$. They are marked by which of $e_0, e_1, \dots, e_5$ act by $0$ on the corresponding weight vector $v_\cF$ in $\Cal(\gamma)$ (where $e_0$ acts by $0$ if and only if $\beta \in\{ \alpha_1,\alpha_2\}$). The  tableau marked $\cF_\star$ is the unique tableau for which $e_j v_{\cF_\star} \ne 0$ for all $j  = 2,3,\dots, 5$. The other tableaux  marked by $\star$ are the remaining of the elements of $F_\even$, those whose corresponding weight vectors are not annihilated by any even $e_{2j}$. Below that list, find the orbits of $\cF_\star$ under the action of  $S_\odd=  \<s_3,s_5\>$ and of  $S_\even = \<s_2,s_4\>$.
}\label{fig:special-shape-tabuleaux-6}

\end{figure}

For $n$ even, define 
\[S_\odd =\<s_3, s_5, \dots, s_{n-1}\>
	\quad \text{ and } \quad 
	S_\even  = 
		\<s_2, s_4, \dots, s_{n-2}\> \]
(notice $s_1 \notin S_\odd$).    We again have that  the set of  tableaux that are not annihilated by $\hat{I}_0 = e_0 e_2 \cdots e_{n-2}$ is exactly the orbit
\begin{align*}
F_\even&= \left\{ \text{ standard  tableaux $\cF$ of shape $\gamma$}~\left|~ \cF^{-1}(2j+1) \ne \cF^{-1}(2j) +1 \text{ for all }j =1, \dots, \tfrac{n}{2}-1\right.\right\}\\
	&= S_\even \cF_\star;
\end{align*}
and $S_\even$ acts freely on this orbit. Further, if $\gamma_i \ne \alpha_1$ or $\alpha_2$, then $e_0$ acts by a (non-zero) scalar on each $v_\cF$ for $\cF \in F_\even$. 
Hence, for any $\cF \in F_\even$, 
\[\textstyle \hat{I}_0 v_\cF = \sum_{w \in S_\even} c_w v_{w \cdot \cF},   \]
for some $c_w \ne 0 \text{ for all }w \in  S_\even$.
And for all $\cF \in F_\even - \{\cF_\star\}$,  we have $e_{j} v_{\cF} = 0$ for all $j = 3, 5, \dots, n-1$. So for $\cF \in F_\even$, letting $c = c_{\rm id}$ be the (non-zero) coefficient of $v_{\cF_\star}$ in $\hat{I}_0 v_\cF$, we have 
\begin{align*}
e_3 e_5 \cdots e_{n-1} \hat{I}_0 v_{\cF} 
 = e_3 e_5 \cdots e_{n-1} c v_{\cF_\star}  
	\textstyle  =  \sum_{w \in S_\odd} c d_w v_{w{\cF_\star}},   
	\end{align*}
for some   $d_w \ne 0 \text{ for all }w \in S_\odd.$

We have $w(1) = 1$ and $w(2) = 2$ for all $w \in  S_\odd$, and hence $\gamma^{w{\cF_\star}}_1 = \gamma_i$ and $\gamma^{w{\cF_\star}}_2 = q^2\gamma_i$. So neither  $s_0w\cF_\star$ nor $s_1w\cF_\star$ is standard. By \cref{thm:calibratedTL}, we can then conclude that 
\begin{align*}
T_1 e_{0^{\vee}} T_1^{-1} v_{w\cF_\star} 
 = [T_1]_{w\cF_\star} \left( [T_{0^\vee}]_{w\cF_\star}  -q_n\right)[T_1]_{w\cF_\star}^{-1} v_{w\cF_\star} 
 = \left( [T_{0^\vee}]_{\cF_\star}  -q_n\right)v_{w\cF_\star},
\end{align*}
which is $0$ exactly when $\gamma_i = \alpha_1$ or $\alpha_2^{-1}$ (by \eqref{eq:e-zeros}). If so, then  
\[I_0 v_{\cF} = 
\tfrac{1}{\(q_0\)}\hat{I}_0 
{I_{1}^\vee \hat{I}_0 v_{\cF}} 
=\tfrac{1}{\(q_0\)}\hat{I}_0 
({I_{1}^\vee \hat{I}_0 v_{\cF}} )
=\tfrac{1}{\(q_0\)}\hat{I}_0 
\times 0 
 =0,\]
and hence $I_0 \Cal(\lambda_2) = 0$. Otherwise, the coefficient of $v_{\cF_\star}$ is nonzero in the final expansion of the action by $\hat{I}_0$ (the only element of the set $S_\even S_\odd$ that stabilizes $\cF_\star$  is $1$). 
 Thus $I_0$ annihilates $N(\gamma)$ if and only if $\gamma_i = \alpha_2^{\pm 1}$ or $\alpha_1$.

 \smallskip
 \noindent{\bf Complete list of non-isomorphic simples.} Finally, our indexing set of calibrated simple modules can easily be seen to be in bijection with 
 \cite[Theorem 8.13]{MR2354870}. They are also easily seen to be pairwise non-isomorphic by looking at the eigenvalues of the Jucys--Murphy elements $X_i$, that is by comparing sequences of residues.
\end{proof}

It will be convenient to shift $\gamma_1$ satisfying \eqref{eq:good gamma} by a power of $q$ as follows: 
we set $\vartheta= \gamma_1 q^n$ if $n$ is even and 
$\vartheta= \gamma_1 q^{n-1}$ if $n$ is odd.  
Thus from now on we will use in addition to $q$ the three parameters $\alpha_1$, $\alpha_2$, and $\extra$ related to $q_0$,$q_n$, and $\kappa$ by
\begin{equation}\label{parametersfinal}
\alpha_1=q_0q_n, \quad 
\alpha_2=-q_0q_n^{-1}, \quad 
\kappa = \begin{cases}
    \(\extra q^{-1}\)-\(\alpha_1 q^{-1}\) 
        & n\in 2\ZZ\\
    \(\extra\)-\(\alpha_1\) 
    & n\in 1+2\ZZ
\end{cases}
\end{equation}
(see \eqref{eq:alphas-defn} and \eqref{eq:good gamma}).

\subsection{Shapes and standard tableaux}
For the remainder of the paper, we will fix the following notation
$$ 
\Lambda_0 =   \{(0,\extra) \}
\qquad
\Lambda_1 =   \{(0,\extra),(1,\alpha_1) \}
\qquad
\Lambda_2 =   \{(0,\extra),(2,\alpha_1),(2,\alpha_2),(2,\alpha_2^{-1}) \}
$$and $\Lambda_n = \Lambda_{n-2}\cup \{(n,\beta) \mid \beta \in \{\alpha_1^{\pm1},\alpha_2^{\pm1}\}\}$ for $n>2$. We call elements of $\Lambda_n$ \emph{shapes of size $n$}.
We place a partial ordering on $\Lambda_n$ as follows: $(k_1,\beta_1) < (k_2,\beta_2)$ for $(k_i,\beta_i)\in \Lambda_n$ if and only if $k_1<k_2$.  
\Cref{prop:calibrated B modules} states  that $\Lambda_n$ provides  an indexing set of the simple modules 
of the symplectic blob algebra with generic parameters, and moreover that a basis of this algebra is indexed by  pairs of standard tableaux of shape $(k,\beta)\in \Lambda_n$  (by \cref{thm:calibratedTL} and  Artin--Wedderburn theory).

Note that $\Lambda_n$ consists of shapes $( k,\beta)$ for $0< k \leq n$ with the same parity of $n$ and with some restrictions on $\beta$ when
 $k$ is small, together with the special shape $(0, \extra )$. 
We will identify a shape $\lambda=(k,\beta)$ with a row of $n$ boxes with a bead on the top-left corner of the $(\lfloor \tfrac{n-k}{2}\rfloor +1)$th box (from left to right) indexed by $\beta\in\{\alpha_1^{\pm1},\alpha_2^{\pm1}, \extra\}$.

\begin{eg}
 For $n=4$ the shapes of 
 $(4,\alpha_1)$,  $(4,\alpha_2)$,  $(4,\alpha_1^{-1})$,  $(4,\alpha_2^{-1})$        are as follows 
$$
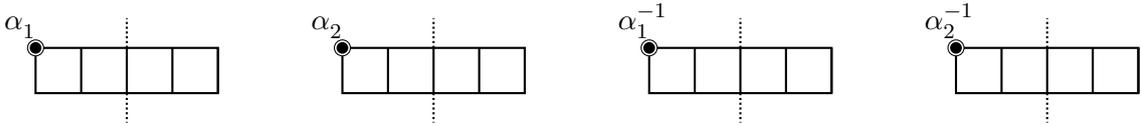
\begin{minipage}{3.25cm}\begin{tikzpicture}[scale=0.6]
 \draw[thick,](0,0) --++(0:1) coordinate (x1)
  --++(0:1) coordinate (x2)
   --++(0:1) coordinate (x3)
    --++(0:1) coordinate (x4)
   
                                                                            --++(-90:1)--++(180:4)--++(90:1);
 \foreach \i in {1,2,3,4}
 {
 \draw[thick ](x\i)--++(-90:1);
 \path(x\i)--++(-90:0.5)--++(180:0.5) coordinate (y\i);
 }
%

 \path(x1)--++(180:1) coordinate (M);
 \path(M)--++(90:0.6)--++(180:0.3) node {$\phantom{_x}\alpha_1^{\phantom{-1}}$};
  \node[M] at (M){};

\path(0,0)--++(0:4/2) coordinate (x35); 
 
 \draw[   thick,densely dotted](x35) --++(90:0.65);
\path(x35)--++(-90:1) coordinate (x35);
\draw[   thick,,densely dotted](x35) --++(-90:.65);

\end{tikzpicture}\end{minipage}
\qquad
\begin{minipage}{3.25cm}\begin{tikzpicture}[scale=0.6]
 \draw[thick,](0,0) --++(0:1) coordinate (x1)
  --++(0:1) coordinate (x2)
   --++(0:1) coordinate (x3)
    --++(0:1) coordinate (x4)
   
                                                                            --++(-90:1)--++(180:4)--++(90:1);
 \foreach \i in {1,2,3,4}
 {
 \draw[thick ](x\i)--++(-90:1);
 \path(x\i)--++(-90:0.5)--++(180:0.5) coordinate (y\i);
 }
%

 \path(x1)--++(180:1) coordinate (M);
 \path(M)--++(90:0.6)--++(180:0.3) node {$\phantom{_x}\alpha_2^{\phantom{-1}}$};
  \node[M] at (M){};

\path(0,0)--++(0:4/2) coordinate (x35); 
 
 \draw[   thick,densely dotted](x35) --++(90:0.65);
\path(x35)--++(-90:1) coordinate (x35);
\draw[   thick,,densely dotted](x35) --++(-90:.65);

\end{tikzpicture}\end{minipage}
\qquad
\begin{minipage}{3.25cm}\begin{tikzpicture}[scale=0.6]
 \draw[thick,](0,0) --++(0:1) coordinate (x1)
  --++(0:1) coordinate (x2)
   --++(0:1) coordinate (x3)
    --++(0:1) coordinate (x4)
   
                                                                            --++(-90:1)--++(180:4)--++(90:1);
 \foreach \i in {1,2,3,4}
 {
 \draw[thick ](x\i)--++(-90:1);
 \path(x\i)--++(-90:0.5)--++(180:0.5) coordinate (y\i);
 }
%

 \path(x1)--++(180:1) coordinate (M);
 \path(M)--++(90:0.6)--++(180:0.3) node {$\phantom{_x}\alpha_1^{-1}$};
  \node[M] at (M){};

\path(0,0)--++(0:4/2) coordinate (x35); 
 
 \draw[   thick,densely dotted](x35) --++(90:0.65);
\path(x35)--++(-90:1) coordinate (x35);
\draw[   thick,,densely dotted](x35) --++(-90:.65);

\end{tikzpicture}\end{minipage}\qquad
\begin{minipage}{3.25cm}\begin{tikzpicture}[scale=0.6]
 \draw[thick,](0,0) --++(0:1) coordinate (x1)
  --++(0:1) coordinate (x2)
   --++(0:1) coordinate (x3)
    --++(0:1) coordinate (x4)
   
                                                                            --++(-90:1)--++(180:4)--++(90:1);
 \foreach \i in {1,2,3,4}
 {
 \draw[thick ](x\i)--++(-90:1);
 \path(x\i)--++(-90:0.5)--++(180:0.5) coordinate (y\i);
 }
%

 \path(x1)--++(180:1) coordinate (M);
 \path(M)--++(90:0.6)--++(180:0.3) node {$\phantom{_x}\alpha_2^{-1}$};
  \node[M] at (M){};

\path(0,0)--++(0:4/2) coordinate (x35); 
 
 \draw[   thick,densely dotted](x35) --++(90:0.65);
\path(x35)--++(-90:1) coordinate (x35);
\draw[   thick,,densely dotted](x35) --++(-90:.65);

\end{tikzpicture}\end{minipage}
$$
and the shapes coming from $\Lambda_{2}$ are 
 $(2,\alpha_1),(2,\alpha_2),(2,\alpha_2^{-1})$ and $(0,\extra)$ pictured as follows 
$$
 \begin{minipage}{3.25cm}
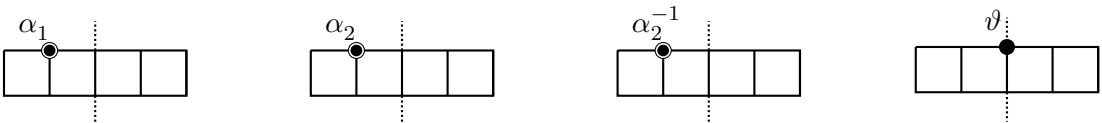
\begin{tikzpicture}[scale=0.6]
 \draw[thick,](0,0) --++(0:1) coordinate (x1)
  --++(0:1) coordinate (x2)
   --++(0:1) coordinate (x3)
    --++(0:1) coordinate (x4)
   
                                                                            --++(-90:1)--++(180:4)--++(90:1);
 \foreach \i in {1,2,3,4}
 {
 \draw[thick ](x\i)--++(-90:1);
 \path(x\i)--++(-90:0.5)--++(180:0.5) coordinate (y\i);
 }
%

 \path(x2)--++(180:1) coordinate (M);
 \path(M)--++(90:0.6)--++(180:0.3) node {$\phantom{_x}\alpha_1^{\phantom{-1}}$};
  \node[M] at (M){};

\path(0,0)--++(0:4/2) coordinate (x35); 
 
 \draw[   thick,densely dotted](x35) --++(90:0.65);
\path(x35)--++(-90:1) coordinate (x35);
\draw[   thick,,densely dotted](x35) --++(-90:.65);

\end{tikzpicture}\end{minipage}
\qquad
\begin{minipage}{3.25cm}\begin{tikzpicture}[scale=0.6]
 \draw[thick,](0,0) --++(0:1) coordinate (x1)
  --++(0:1) coordinate (x2)
   --++(0:1) coordinate (x3)
    --++(0:1) coordinate (x4)
   
                                                                            --++(-90:1)--++(180:4)--++(90:1);
 \foreach \i in {1,2,3,4}
 {
 \draw[thick ](x\i)--++(-90:1);
 \path(x\i)--++(-90:0.5)--++(180:0.5) coordinate (y\i);
 }
%

 \path(x2)--++(180:1) coordinate (M);
 \path(M)--++(90:0.6)--++(180:0.3) node {$\phantom{_x}\alpha_2^{\phantom{-1}}$};
  \node[M] at (M){};

\path(0,0)--++(0:4/2) coordinate (x35); 
 
 \draw[   thick,densely dotted](x35) --++(90:0.65);
\path(x35)--++(-90:1) coordinate (x35);
\draw[   thick,,densely dotted](x35) --++(-90:.65);

\end{tikzpicture}\end{minipage}
\qquad
\begin{minipage}{3.25cm}\begin{tikzpicture}[scale=0.6]
 \draw[thick,](0,0) --++(0:1) coordinate (x1)
  --++(0:1) coordinate (x2)
   --++(0:1) coordinate (x3)
    --++(0:1) coordinate (x4)
   
                                                                            --++(-90:1)--++(180:4)--++(90:1);
 \foreach \i in {1,2,3,4}
 {
 \draw[thick ](x\i)--++(-90:1);
 \path(x\i)--++(-90:0.5)--++(180:0.5) coordinate (y\i);
 }
%

 \path(x2)--++(180:1) coordinate (M);
 \path(M)--++(90:0.6)--++(180:0.3) node {$\phantom{_x}\alpha_2^{{-1}}$};
  \node[M] at (M){};

\path(0,0)--++(0:4/2) coordinate (x35); 
 
 \draw[   thick,densely dotted](x35) --++(90:0.65);
\path(x35)--++(-90:1) coordinate (x35);
\draw[   thick,,densely dotted](x35) --++(-90:.65);

\end{tikzpicture}\end{minipage}
\qquad
\begin{minipage}{3.25cm}\begin{tikzpicture}[scale=0.6]
 \draw[thick,](0,0) --++(0:1) coordinate (x1)
  --++(0:1) coordinate (x2)
   --++(0:1) coordinate (x3)
    --++(0:1) coordinate (x4)
   
                                                                            --++(-90:1)--++(180:4)--++(90:1);
 \foreach \i in {1,2,3,4}
 {
 \draw[thick ](x\i)--++(-90:1);
 \path(x\i)--++(-90:0.5)--++(180:0.5) coordinate (y\i);
 }
%

 \path(x3)--++(180:1) coordinate (M);
 \path(M)--++(90:0.6)--++(180:0.3) node {$ \vartheta  $};
 \fill(M) circle (5pt);

\path(0,0)--++(0:4/2) coordinate (x35); 
 
 \draw[   thick,densely dotted](x35) --++(90:0.65);
\path(x35)--++(-90:1) coordinate (x35);
\draw[   thick,,densely dotted](x35) --++(-90:.65);

\end{tikzpicture}\end{minipage}$$

\end{eg}

We recall that a {\sf tableau} $\SSTT$ of shape $\lambda$ is a filling of the boxes by integers from $\{\pm1,\dots,\pm n\}$ which contains exactly one entry equal to $i$ or $-i$ for each $i\in \{1,\ldots ,n\}$.  The Weyl group $W(C_n)$ acts on the set of all tableaux of a given shape via its action on the entries in the boxes. We say that the tableau  $\SSTT$ is {\sf standard} if the entries are strictly increasing from left to right and, moreover, if $\lambda \neq (0,\extra)$, only the boxes to the left of the bead can (but do not have to) contain negative entries. We write $\Shape(\SSTT)=\lambda\in \Lambda_n$.  We refer to \cref{hereisatab} for an example.
For $n\in \ZZ_{\geq 0}$ we denote:
\[\Std_n(\lambda)=\{\text{standard tableaux $\SSTT$ with $\Shape(\SSTT)=\lambda$}\}
\]
and we set $\Std_n=\cup_{\lambda\in\Lambda_n}\Std_n(\lambda)$.

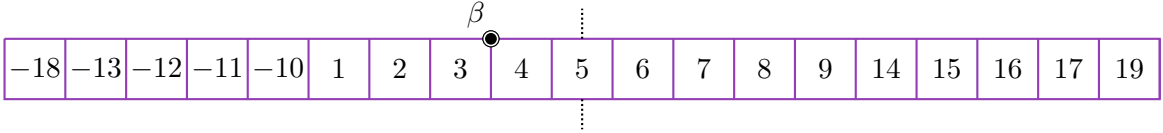
\begin{figure}[ht!]
 $$
 \begin{tikzpicture}[scale=0.8]
 \draw[thick,path2color](0,0) --++(0:1) coordinate (x1)
  --++(0:1) coordinate (x2)
   --++(0:1) coordinate (x3)
    --++(0:1) coordinate (x4)
     --++(0:1) coordinate (x5)
      --++(0:1) coordinate (x6)
       --++(0:1) coordinate (x7)
        --++(0:1) coordinate (x8)
               --++(0:1) coordinate (x9)
        --++(0:1) coordinate (x10)
               --++(0:1) coordinate (x11)
                      --++(0:1) coordinate (x12)
                             --++(0:1) coordinate (x13)
                                    --++(0:1) coordinate (x14)
                                           --++(0:1) coordinate (x15)
                                                  --++(0:1) coordinate (x16)
                                                         --++(0:1) coordinate (x17)
                                                                --++(0:1) coordinate (x18)
                                                                       --++(0:1) coordinate (x19)              
                                                                           --++(-90:1)--++(180:19)--++(90:1);
 \foreach \i in {1,2,3,4,5,6,7,8,...,19}
 {
 \draw[thick,path2color](x\i)--++(-90:1);
 \path(x\i)--++(-90:0.5)--++(180:0.5) coordinate (y\i);
 }

 \path(y8) node {$3$};
 \path(y7) node {$2$};
 \path(y6) node {$1$};
 \path(y5) node {$-10$};
 \path(y4) node {$-11$};
 \path(y3) node {$-12$};
 \path(y2) node {$-13$};
 \path(y1) node {$-18$};

 \path(y9) node {$4$};
 \path(y10) node {$5$};
 \path(y11) node {$6$};
 \path(y12) node {$7$};
 \path(y13) node {$8$};
 \path(y14) node {$9$};
 \path(y15) node {$14$};

 \path(y16) node {$15$};
 \path(y17) node {$16$};
 \path(y18) node {$17$};
 \path(y19) node {$19$};

 \fill(x8) circle (4pt); 
 \path(x8)--++(90:0.4)--++(180:0.25) node {$\beta$};
 \node[M] at (x8){};

\path(0,0)--++(0:19/2) coordinate (x35); 
 
 \draw[   thick,densely dotted](x35) --++(90:0.5);
\path(x35)--++(-90:1) coordinate (x35);
\draw[   thick,,densely dotted](x35) --++(-90:.5);

  \end{tikzpicture}
 $$ 
\caption{A standard tableau of shape $(k,\beta)$ for $k=3$
and $n=19$.}
\label{hereisatab}
\end{figure}

%
%
%
%

\color{black}

\section{The  orientifold  quiver Temperley--Lieb algebra}

In this section, we define our main algebraic object, the orientifold Temperley--Lieb algebra. We define it from the orientifold quiver Hecke algebra and make the connections with the symplectic blob algebra using the results on calibrated representations from the preceding section.
%
From now on, we specialize our parameters in a field $\Bbbk$  of \textbf{characteristic not equal to 2}   
 subject to the following assumptions.

  \begin{assump}
 We let $q,\alpha_1 ,\alpha_2,  \extra \in \Bbbk^\times$  
 and $q^{2e}=1$ for some minimal $e>2$ (we allow $e=\infty$).
 We further require that  $\alpha_i^{\pm1},\alpha_i^{\pm1}q^{\pm 2}$
 are   12 distinct points not equal to $\pm1$ and that 
   $\extra \not \in \{\pm1,\alpha_i^{\pm1},  \pm q, \pm q^{2}\}$. 
 
    \end{assump}

%

\subsection{The orientifold quiver Hecke algebra}

The type $C$ Weyl group $W_n:=W(C_n)=\langle s_k \mid 0\leq k < n\rangle $ acts as  signed permutations on
 $(\Bbbk^\times)^n$ as follows. The generator $s_0$ inverts the first component and the generator $s_k$ for  $k>0$, acts by transposing the $k$th and $(k+1)$th components. 
 We consider the following subset of  $\Bbbk^\times$:
\[I=\{\alpha_1^{\pm 1}q^{2l}\}_{l\in\mathbb{Z}}\cup\{\alpha_2^{\pm 1}q^{2l}\}_{l\in\mathbb{Z}}\cup \{\extra^{\pm 1}q^{2l}\}_{l\in\mathbb{Z}}.\]
 By Proposition \ref{prop:calibrated B modules}  the eigenvalues of all the Jucys--Murphy elements $X_1,\dots,X_n$ of the symplectic blob algebra are elements of  $I$.
 We now recall  the definition of the graded algebras of interest in this paper.

\begin{defn}[\cite{MR4085039} and \cite{MR4666131}]   \label{definition:vv algebra}
 {We define $\mathscr{H}_n  (\alpha_1,\alpha_2, \vartheta )= \oplus _\Lambda \mathscr{H}_n  ^{\Lambda}(\alpha_1,\alpha_2, \vartheta )$, where for $\Lambda$ a $W(C_n)$-orbit in $I^n$, the algebra $\mathscr{H}_n  ^{\Lambda}(\alpha_1,\alpha_2, \vartheta )$ is the associative $\Bbbk$-algebra generated by elements}
\begin{equation*}
\{ \psi_a \}_{0\leq a\leq n-1} \cup \{ y_j \}_{1\leq j\leq n} \cup \{ \idemp \}_{\underline{i} \in \Lambda}
\end{equation*}
subject to the following defining relations. We have the commutation relations 
\begin{align}
\textstyle  \sum_{\underline{i} \in \Lambda} \idmep & = 1,   \quad\! \idmep e_{\underline{j}} = \delta_{\ti,\underline{j}}\idmep  
 , \quad\!
y_ry_s=y_sy_r ,   \quad\! y_{r}\idmep=\idmep y_{r} , \quad\!
\psi_a\psi_b =\psi_b\psi_a 
 \quad\! \psi_0y_t = y_t \psi_0
  \label{Rel:V2}
\end{align}
providing 
$0\leq a \leq b-2$ and $t>1$. For $0\leq a<n$ and $1\leq b < n$  {and $\ti=(i_1,\dots,i_n)\in\Lambda$} we have 
 \begin{align}
\psi_{a}\idmep &= e_{s_a(\ti)}\psi_a
 \label{Rel:V3}
\\ 
 \label{Rel:V4}
(\psi_b y_j - y_{s_b(j)}\psi_b)\idmep&=
\begin{cases}
			-\idmep &     j=b,\ i_b=i_{b+1} \\
			\idmep &    j=b+1,\ i_b=i_{b+1} \\
			0 \qquad\qquad\qquad\quad &  \textrm{otherwise}
\end{cases} 
 \\
 \psi_b^2\idmep &= \begin{cases}
0 \qquad\qquad\qquad\quad					&  i_b=i_{b+1} \\
\idmep 			& 	 \text{$i_{b+1}\notin\{q^2i_{b},i_b, q^{-2}i_{b}\}$} \\
(y_{b+1} - y_b)\idmep &   i_{b+1}=q^2i_{b}   \\
(y_b - y_{b+1})\idmep &  
\textrm{$i_{b+1}=q^{-2}i_{b} $\,,} 
 \\
 \end{cases}
 \label{Rel:V6}
 \\
 \label{Rel:V5}
(\psi_{b}\psi_{b+1}\psi_{b} - \psi_{b+1}\psi_{b}\psi_{b+1})\idemp &
=  
\begin{cases}
\idmep								 &  i_b=i_{b+2}=\textcolor{black}{q^{-2}}i_{b+1} 				\\
-\idmep 								&  i_b=i_{b+2}=\textcolor{black}{q^{2}}i_{b+1} 			  \\
0 \qquad\qquad\qquad\quad									&   \text{otherwise} 				
\end{cases}
%
%
%
%
%
%
%
%
\\
(\psi_0 y_1 + y_1\psi_0)\idmep&=
\begin{cases}
			0 \qquad\qquad\qquad\quad & i_1^{-1}\neq i_1 	\\
			2\idmep &   i_1^{-1}= i_1 
\end{cases} 
 \label{Rel:V8}
 \\
 \psi_0^2\idemp&=\left\lbrace
\begin{array}{ll}
0 \qquad\qquad\qquad\quad &   i_1^{-1}=i_1\,, \\
\idmep &  
   i_1\neq i_1^{-1} \text{ and } 
   {i_1^{\pm1} \not \in\{\alpha_1,\alpha_2\}}
\\
y_1\idmep  & 
{i_1 \in\{\alpha_1,\alpha_2\}}
 \\
-y_1\idmep  
&  
{i_1^{-1}  \in\{\alpha_1,\alpha_2\}}
 \end{array}
\right. \label{Rel:V10} 
\\
(\psi_0 \psi_1\psi_0 \psi_1 -  \psi_1\psi_0 \psi_1\psi_0 )\idmep&= 
\begin{cases}
2\psi_0\idmep  
& 
 q^2i_1  =  i_2 \in \{\pm1\}
   \\
-2\psi_0\idmep & 
\textrm{$   q^{-2}i_1 =i_2 \in \{\pm1\} 		$}  \\ 
 -\psi_1\idmep & 
  		 {  i_2^{-1}= i_1 \in\{\alpha_1,\alpha_2\}}\\ 
\psi_1\idmep &   
 {i_2= i_1^{-1} \in\{\alpha_1,\alpha_2\}}\\
0 \qquad\qquad\qquad\quad& \textrm{otherwise.}
\end{cases}
\label{Rel:V11}
\end{align}
 
 \end{defn}
 
 We denote ${}^\star$ the involutive anti-isomorphism sending each generator to itself (the defining relations are easily checked to be preserved by ${}^\star$).

\begin{thm}[\cite{MR4085039} and \cite{MR4666131}]  
\label{Loic-iso}  
Let $m\,:\, I\to \mathbb{Z}_{\geq0}$ with finite support. The cyclotomic quotient of the algebra  $\mathscr{H}_n  (\alpha_1,\alpha_2 , \vartheta)$ by the relations $y_1^{m(i_1)}e_\ti=0$ for all $\ti\in I^n$ is isomorphic to the cyclotomic quotient of the  2-boundary Hecke algebra $H_n$ by the relation $\prod_{i\in I}(X_1-i)^{m(i)}=0$. 
\end{thm}

%
%
%
%

\begin{thm}[\cite{MR4085039,MR4250039}]
The  algebra  $\mathscr{H}_n  ^{\Lambda}$ has a $\Z$-grading   given as follows,
$$\begin{array}{ccccc}
{\rm  deg}(\idmep)=0 \qquad  
& {\rm  deg}(\psi_b \idmep)	 =\left\{
\begin{array}{l l l}	
  1
 & \quad \textrm{ if  } i_{b+1}= q^{\pm2} i_b \\  
-2 & \quad \textrm{ if  } 	i_b = i_{b+1}\,,\\
 {0}  &\quad \textcolor{black}{\textrm{\ otherwise.}}	
\end{array}\right .
\\ 
 {\rm  deg}(y_j \idmep)=2 \qquad 
& {\rm  deg}(\psi_0\idmep)  = 
\begin{cases}	
-2 &  i_1^{-1} = i_1				\\	
 {\delta_{i_1, \alpha_1 }+\delta_{i_1, 	 \alpha_1 ^{-1}	}+\delta_{i_1 ,  \alpha_2 }+\delta_{i_1, \alpha_2^{-1} }}& \textrm{otherwise}
 \end{cases}
 \\
\end{array}.$$ 
  \end{thm}

\subsection{More tableaux combinatorics}
We denote by $\SSTT_{\lambda}$ the following specific  element of $\Std_n(\lambda)$: $1$ is in the box with the bead, then we put $-2,-4,-6,...$ to its left until the leftmost box, and the remaining integers, with positive signs, in increasing order in the remaining boxes.  Examples are depicted in \cref{hereisatab2,ttableau}.

\begin{figure}[ht!]
 $$
 \begin{tikzpicture}[scale=0.8]
 \draw[thick,path1color](0,0) --++(0:1) coordinate (x1)
  --++(0:1) coordinate (x2)
   --++(0:1) coordinate (x3)
    --++(0:1) coordinate (x4)
     --++(0:1) coordinate (x5)
      --++(0:1) coordinate (x6)
       --++(0:1) coordinate (x7)
        --++(0:1) coordinate (x8)
               --++(0:1) coordinate (x9)
        --++(0:1) coordinate (x10)
               --++(0:1) coordinate (x11)
                      --++(0:1) coordinate (x12)
                             --++(0:1) coordinate (x13)
                                    --++(0:1) coordinate (x14)
                                           --++(0:1) coordinate (x15)
                                                  --++(0:1) coordinate (x16)
                                                         --++(0:1) coordinate (x17)
                                                                --++(0:1) coordinate (x18)
                                                                       --++(0:1) coordinate (x19)              
                                                                            --++(-90:1)--++(180:19)--++(90:1);
 \foreach \i in {1,2,3,4,5,6,7,8,...,19}
 {
 \draw[thick,path1color](x\i)--++(-90:1);
 \path(x\i)--++(-90:0.5)--++(180:0.5) coordinate (y\i);
 }

 \path(y8) node {$-2$};
 \path(y7) node {$-4$};
 \path(y6) node {$-6$};
 \path(y5) node {$-8$};
 \path(y4) node {$-10$};
 \path(y3) node {$-12$};
 \path(y2) node {$-14$};
 \path(y1) node {$-16$};

 \path(y9) node {$1$};
 \path(y10) node {$3$};
 \path(y11) node {$5$};
 \path(y12) node {$7$};
 \path(y13) node {$9$};
 \path(y14) node {$11$};
 \path(y15) node {$13$};
 \path(y16) node {$15$};

 \path(y17) node {$17$};
 \path(y18) node {$18$};
 \path(y19) node {$19$};

 \fill(x8) circle (4pt); 
 \path(x8)--++(90:0.4)--++(180:0.25) node {$\beta$};
  \node[M] at (x8){};

\path(0,0)--++(0:19/2) coordinate (x35); 
 
 \draw[   thick,densely dotted](x35) --++(90:0.5);
\path(x35)--++(-90:1) coordinate (x35);
\draw[   thick,,densely dotted](x35) --++(-90:.5);

\end{tikzpicture}
 $$
\caption{The tableau $\stt_{(k,\beta)}$ for $k=3$ and $n=19$. 
We have that $\res(\stt_{(k,\beta)})=(\beta,\,\beta^{-1}q^{2},\,\beta q^2,\, \beta^{-1}q^{4} ,\, \beta q^{4},\,\beta^{-1}q^{6},\,\beta q ^6,\,\dots)$.}
\label{hereisatab2}

\end{figure}
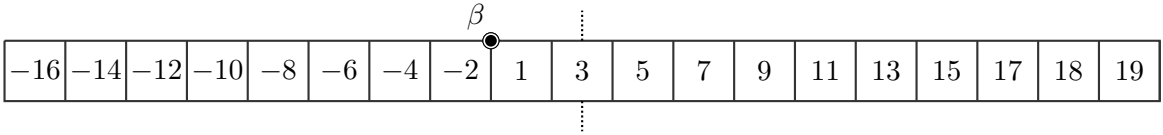

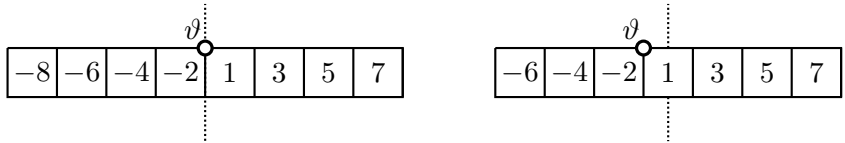
\begin{figure}[ht!]

 $$
\begin{tikzpicture}[scale=0.65]
\draw[thick](0,0) --++(0:1) coordinate (x1)
 --++(0:1) coordinate (x2)
  --++(0:1) coordinate (x3)
   --++(0:1) coordinate (x4)
    --++(0:1) coordinate (x5)
     --++(0:1) coordinate (x6)
      --++(0:1) coordinate (x7)
       --++(0:1) coordinate (x8)
--++(-90:1)--++(180:8)--++(90:1);
\foreach \i in {1,2,3,4,5,6,7,8}
{
\draw[thick](x\i)--++(-90:1);
\path(x\i)--++(-90:0.5)--++(180:0.5) coordinate (y\i);
}

\path(y1) node {$-8$};
\path(y2) node {$-6$};
\path(y3) node {$-4$};
\path(y4) node {$-2$};
\path(y5) node {$1$};
\path(y6) node {$3$};
\path(y7) node {$5$};
\path(y8) node {$7$};

\draw[   thick,densely dotted](x4) --++(90:0.9);
\draw[   thick,,densely dotted](x4) --++(-90:1.9);
\draw[ very thick,fill=white](x4) circle (4pt); 
\path(x4)--++(90:0.4)--++(180:0.25) node {$\extra$};
 \end{tikzpicture}
\qquad
\quad
\begin{tikzpicture}[scale=0.65]
\draw[thick](0,0) --++(0:1) coordinate (x1)
 --++(0:1) coordinate (x2)
  --++(0:1) coordinate (x3)
   --++(0:1) coordinate (x4)
    --++(0:1) coordinate (x5)
     --++(0:1) coordinate (x6)
      --++(0:1) coordinate (x7)
--++(-90:1)--++(180:7)--++(90:1);
\foreach \i in {1,2,3,4,5,6,7}
{
\draw[thick](x\i)--++(-90:1);
\path(x\i)--++(-90:0.5)--++(180:0.5) coordinate (y\i);
}

\path	(x3)--++(0:0.5) coordinate 	(x35);

\path(y1) node {$-6$};
\path(y2) node {$-4$};
\path(y3) node {$-2$};
\path(y4) node {$1$};
\path(y5) node {$3$};
\path(y6) node {$5$};
\path(y7) node {$7$};

\draw[   thick,densely dotted](x35) --++(90:0.9);
\path(x35)--++(-90:1) coordinate (x35);
\draw[   thick,,densely dotted](x35) --++(-90:.9);
\draw[ very thick,fill=white](x3) circle (4pt); 
\path(x3)--++(90:0.4)--++(180:0.25) node {$\extra$};
 \end{tikzpicture}
$$

\caption{The tableau $\stt_{(0, \extra )}$ for $n=8$ and $n=7$ respectively.}
\label{ttableau}
\end{figure}

For any standard tableau $\SSTT\in \Std_n(\la)$ we let $w_\SSTT\in W(C_n)$ be defined by $\SSTT = w_\SSTT (\SSTT_{\la})$.

We define the residue sequence of a tableau as follows. First assign a content to each of the boxes in $\lambda$ by setting that the content of the box with the bead on it is equal to $\beta$, and then the contents are multiplied by $q^2$ for each step to the right and by $q^{-2}$ for each step to the left.
Now,  let $\SSTT\in\Std_n$ and $i\in\{1,\dots,n\}$. If $+i$ appears in $\SSTT$,   we define $\res_i(\SSTT)$ to be the content of the box containing $+i$. Otherwise, we define $\res_i(\SSTT)$ to be the inverse of the content of the box containing $-i$. Finally, we set:
\[\res(\SSTT)=(\res_1(\SSTT),\dots,\res_n(\SSTT))\in I^n \]
see \cref{hereisatab2} for an example. 
The action of the Weyl group $W(C_n)$ on tableaux translates via residue sequences as an action on $I^n$, where the generator $s_0$ inverts the first coordinate and the generators $s_j$ ($1\leq j\leq n-1$) swap the $j$th and $(j+1)$th coordinates. 

The results of the preceding section on the generic representation theory show that the common spectrum of the Jucys--Murphy elements of the symplectic blob algebra is given by the residue sequences of standard tableaux, namely,
\[{\rm Spec}(X_1,\dots,X_n)=\{\res(\SSTT)\,\mid \  \SSTT\in\Std_n\}.\]
For the shape $(0,\extra)$, this requires a short verification using (\ref{eq:good gamma}) to check that (\ref{parametersfinal}) is indeed the correct relation between $\extra$ and $\kappa$.
The definition of the tableaux $\SSTT_{\lambda}$ is justified by the following property they satisfy. Note that it implies in particular that $\res(\SSTT_{\lambda})=\res(\SSTT_{\lambda'})$ if and only if $\lambda=\lambda'$.
\begin{prop}\label{propdominant1}
For all $\lambda\in\Lambda_n$, if $\res(\SSTT_{\lambda})=\res(\stu)$ for $\SSTT_{\lambda}\neq\stu\in\Std_n$ then $\Shape(\stu)<\lambda$.    
\end{prop}
\begin{proof}
It is slightly more convenient to prove the statement for an extended set of shapes $\tilde{\Lambda}_n$ obtained by allowing back the forbidden shapes $(1,\beta)$ with $\beta\neq\alpha_1$ if $n$ is odd and the forbidden shape $(2,\alpha_1^{-1})$ if $n$ is even.

As a preliminary, we claim that for $\lambda=(1,\beta)$ and for $\stu\in\Std_n(\lambda)$ we have $\res(\stu)=\res(\SSTT_{\lambda})$ only if $\stu=\SSTT_{\lambda}$. Indeed, for $\stu$, there are two options for the position for $n$: with a $+$ sign at the right end or with a $-$ sign at the left end. The two residues are $\beta q^{n-1}$ and $\beta^{-1}q^{n-1}$. They are different since $\alpha_i^{\pm1}$ are all different, and from $\res(\stu)=\res(\SSTT_{\lambda})$, we get that $n$ is the entry of the rightmost box, as in $\SSTT_{\lambda}$. Similarly, for $n-1$, there are two options leading to residues $\beta q^{n-3}$ and $\beta^{-1}q^{n-1}$. These are different since $\alpha_i^{\pm1}$ and $q^2\alpha_i^{\pm1}$ are all different. Therefore $n-1$ in $\stu$ has to be with a minus sign at the left end as in $\SSTT_{\lambda}$. 
Reproducing this reasoning for all $i< n-1$, we get that $\stu=\SSTT_{\lambda}$.

Now let us check the proposition for $n=1$ and $n=2$. For $n=1$, this is an immediate verification using the fact  that $\alpha_i^{\pm1}$ and $\extra$ are all different. For $n=2$ and $\lambda=(2,\beta)$, again this follows from $\alpha_i^{\pm1}$ being all distinct. For $\lambda=(0,\extra)$, we have $\res(\SSTT_{\lambda})=(\extra,\extra^{-1}q^2)$ and it must be different from:
\[(\beta,\beta q^2)\ \ \text{for all }\beta\in\{\alpha_i^{\pm1}\}\,,\ \ \ (\extra q^{-2},\extra)\,,\ \ (\extra^{-1}q^2,\extra)\,,\ \ (\extra^{-1},\extra^{-1}q^2).\]
The result then follows from $\extra^2\not \in\{ 1,q^2\}$ and $q^2\neq1$.

Then let $n\geq 3$ and $\lambda\in\Lambda_n$ and assume that $\res(\SSTT_{\lambda})=\res(\stu)$ with $\Shape(\stu)\nless\lambda$.    

Let $\lambda=(k,\beta)\in \Lambda_n$ with $k\neq 0$ so that $\Shape(\stu)=(c,\beta')$ with $c\geq k$. First assume that $k\neq1$. The letter $n$ is at the right end of $\SSTT_{\lambda}$ and removing it we have that $\SSTT_{\lambda}{\downarrow}_{\leq n-1}$ is $\SSTT_{\lambda'}$ with $\lambda'=(k-1,\beta)$. Now we note that removing $n$ in $\stu$, we have that $\stu{\downarrow}_{\leq n-1}$ is a standard tableau of shape either $(c-1,\beta')$ or $(c+1,\beta')$. In any case its shape is $\nless \lambda'$ and this is impossible by induction hypothesis.

If $k=1$, we remove both $n$ and $n-1$ from $\SSTT_{\lambda}$ and we have that $\SSTT_{\lambda}{\downarrow}_{\leq n-2}=\SSTT_{\lambda'}$ with $\lambda'=(k,\beta)\in\tilde{\Lambda}_{n-2}$. Now removing both $n$ and $n-1$ from $\stu$, we always get a standard tableau of size $n-2$ unless $\stu$ was of shape $(1,\beta')$ and both $n-1$ and $n$ were at the right end of $\stu$. Outside of this case, we conclude as before that this is impossible by induction hypothesis. In the remaining case, the condition $\res_n(\SSTT_{\lambda})=\res_n(\stu)$ implies that $\beta'=\beta$, and we conclude using the preliminary result from the beginning of the proof.

Finally let $\lambda=(0,\extra)$. Again we remove $n$ and $n-1$ and we have $\SSTT_{\lambda}{\downarrow}_{\leq n-2}=\SSTT_{\lambda'}$ with $\lambda'=(0,\extra)\in\Lambda_{n-2}$. We also remove $n$ and $n-1$ in $\stu$ and as before the only case where the induction hypothesis is not immediately applicable is when $\stu$ is of shape $(1,\beta')$ and both $n-1$ and $n$ are at the right end of $\stu$. In this case ($n$ is odd) we would have from  $\res_n(\SSTT_{\lambda})=\res_n(\stu)$ that $\extra q^{n-1}=\beta'q^{n-1}$ which is impossible.
\end{proof}

It turns out that the standard tableaux $\SSTT_{\lambda}$ satisfy another uniqueness property. We recall that the generator $s_0$ of the Weyl group replaces the first entry of a sequence by its inverse.
\begin{prop}\label{propdominant2}
For all $\lambda\in\Lambda_n$ with $\lambda\neq(0,\extra)$, if $s_0(\res(\SSTT_{\lambda}))=\res(\stu)$ for $\stu\in\Std_n$ then $\Shape(\stu)<\lambda$.    
\end{prop}
\begin{proof}
We use the same approach by induction as used  in the proof of \cref{propdominant1} with a very small modification. This time we extend the set of shapes in $\tilde{\Lambda}_n$ as before, but only for $n>1$.

First if $\lambda=(1,\beta)$, the same proof repeats to show that if $\res(\stu)=s_0\cdot\res(\SSTT_{\lambda})$ with $\stu$ of shape $\lambda$, then it must be that $\stu$ and $\SSTT_{\lambda}$ coincide for all entries from $2$ to $n$. 
Since replacing  $1$  with a  $-1$ in $\SSTT_{\lambda}$ we  {\em do not} obtain a standard tableau, we deduce that such a $\stu$ cannot exist.

The induction step works exactly the same, but we have to be a bit more careful about the small values of $n$ since there is one more possibility to fall outside of the sets $\tilde{\Lambda}_{n-2}$ when reducing from $n$ to $n-2$ in the induction step. This is taken care of below.

For $n=1$, the only shape is $\lambda=(1,\alpha_1)$ and we have $\res(\SSTT_{\lambda})=(\alpha_1)$ and applying $s_0$ results into $(\alpha_1^{-1})$ which is different, using our standing assumption on $\alpha_1$. Thus $s_0(\res(\SSTT_{\lambda}))$ cannot be the residue sequence of a standard tableau.

For $n=2$, the only shapes are $(2,\beta)$ and for them the verification is immediate.

Finally, we also need to deal with $n=3$ and $\stu$ of shape $(3,\beta)$ since this is where an induction step of size $2$ would possibly lead us outside $\Lambda_1$. The residue sequence of $\stu$ is $(\beta,\beta q^2,\beta q^4)$ and we need to be sure that it is different from all $s_0(\res(\SSTT_{\lambda}))$ with $\lambda\in\tilde{\Lambda}_3$. This amounts to checking that it is different from
\[(\beta'^{-1},\beta'q^2,\beta'q^4)\ \ \ \text{and}\ \ \ (\beta'^{-1},\beta'^{-1}q^2,\beta'q^2).\]
This is obviously true for the first one since $\beta\neq\beta^{-1}$. An equality with the second one leads to $\beta'=\beta^{-1}$ and $\beta=q^2\beta^{-1}$ which is excluded from our standing assumptions on $\alpha_1,\alpha_2$.
\end{proof}

\subsection{The orientifold quiver Temperley--Lieb algebra} With our combinatorics in place, we are now ready to define the orientifold quiver Temperley--Lieb algebra.  
This can be seen as a natural generalisation of the ideas of \cite{PR13}.

\begin{defn}\label{whatisTL}
We define the orientifold quiver Temperley--Lieb algebra, ${\rm TL}_n(\alpha_1,\alpha_2, \extra )$, to be the quotient of $\mathscr{H}_n  (\alpha_1,\alpha_2 , \vartheta )$ by the relations
\begin{align}\label{fkdjhgdlsjhgdjlskhgsdfjklhgjkdflh}
\idemp= 0   \text{ for $\ti \neq \res(\SSTT)$ 				for some   $\SSTT\in \Std_n$		}
\qquad\quad
y_1 e_{\res(\stt_{(0,\extra)})}= 0 
\end{align}
where $\stt_{(0,\extra)}$ is the distinguished standard tableau of shape $(0,\extra$) (see \cref{ttableau}).
\end{defn}
  
We are ready to prove the first half of the isomorphism theorem relating ${\rm TL}_n(\alpha_1,\alpha_2, \extra )$ to the symplectic blob algebra $B_n(\kappa)$.

\begin{prop} \label{quotientthm}
 There is a surjective homomorphism 
 from $\mathscr{H}_n(\alpha_1,\alpha_2, \vartheta)$ to the symplectic blob algebra $B_n(\kappa)$
 and this homomorphism factors through the orientifold quiver Temperley--Lieb algebra ${\rm TL}_n(\alpha_1,\alpha_2, \extra )$. 

\end{prop}
\begin{proof}
Since $B_n(\kappa)$ is a finite-dimensional quotient of the two-boundary Hecke algebra $H_n$, the Jucys--Murphy element $X_1$ automatically satisfies in $B_n(\kappa)$ a characteristic equation of finite order. Moreover, the study of calibrated representations of $B_n(\kappa)$ in the generic semisimple case (Section \ref{sec:Calibrated}) shows that its eigenvalues all lie in the set 
\[I=\{\alpha_1^{\pm 1}q^{2l}\}_{l\in\mathbb{Z}}\cup\{\alpha_2^{\pm 1}q^{2l}\}_{l\in\mathbb{Z}}\cup \{\extra^{\pm 1}q^{2l}\}_{l\in\mathbb{Z}}.\]
Recall that in the isomorphism from \cref{Loic-iso}, the idempotents $e_\ti$ correspond to projectors on the common generalised eigenspaces for the Jucys--Murphy elements $X_1,\dots,X_n$. The sequence $\ti=(i_1,\dots,i_n)$ corresponds to the sequence of eigenvalues of $X_1,\dots,X_n$. Therefore, \cref{Loic-iso} applies for some map $m\ :\ I\to \mathbb{Z}_{\geq 0}$ with finite support. Ignoring the precise form of $m$, we still get a surjective morphism
\[\Theta\ :\ \mathscr{H}_n  (\alpha_1,\alpha_2 , \vartheta)\ \to\ B_n(\kappa),\]
and we must argue that the relations of \eqref{fkdjhgdlsjhgdjlskhgsdfjklhgjkdflh} are in the kernel of this morphism. The relation $\idemp= 0$ for $\ti \neq \res(\SSTT)$ for some   $\SSTT\in \Std_n$	is clearly satisfied in $B_n(\kappa)$ from the study of calibrated representations, since the set $\{\res(\SSTT) \mid  \SSTT\in\Std_n\}$ gives all possible sequences of common eigenvalues of $X_1,\dots,X_n$. 

We now consider the relation $y_1 e_{\res(\stt_{(0,\extra)})}=0$. We first recall that in the isomorphism from \cref{Loic-iso}, the element $y_je_\ti$ is the nilpotent part of the Jucys--Murphy element $X_j$ in the common eigenspace corresponding to $\ti \in I^n$. In particular it has to be $0$ if the eigenspace is of dimension 1. Now Proposition \ref{propdominant1} implies that the residue sequence of the standard tableau $\stt_{(0,\extra)}$ never appears as a residue sequence of another standard tableau in $\Std_n$, since the shape $(0,\extra)$ is at the bottom of the order. This means that the common eigenspace corresponding to the residue sequence $\res(\stt_{(0,\extra)})$ is of dimension 1, and therefore the relation $y_j e_{\res(\stt_{(0,\extra)})}=0$ for all $j=1,\dots,n$ is satisfied in $B_n(\kappa)$.
 \end{proof}

  \section{Cellular basis of the  orientifold  quiver Temperley--Lieb algebra }\label{sec-cellbasis}

 
The main result of \cref{ytieuowrytoiueytiouewr} was that the set of shapes $\Lambda_n$ provides  an indexing set of the simple modules of the symplectic blob algebra with generic parameters (and a construction of these simple $B_n(\kappa)$-modules).  In this section, we will show   {\em for arbitrary parameters} satisfying the standing assumption (see the preceding section or the introduction) that $\Lambda_n$ provides the poset
of an integral graded  cellular structure on the orientifold quiver Temperley-Lieb algebra ${\rm TL}_n(\alpha_1,\alpha_2,\extra)$. This will also complete the second half of the proof of the isomorphism theorem between the symplectic blob algebra $B_n(\kappa)$ and  ${\rm TL}_n(\alpha_1,\alpha_2,\extra)$.

\subsection{Orientifold paths}
 
We will use many times the following basic property of standard tableaux: after having put the first $k$ entries in a shape in order to make a standard tableau, we have at most two choices for the entry $k+1$ (one to the right with a $+$ sign and one to the left with a $-$ sign). This will allow to recast the orientifold tableaux as paths in a 2-dimensional Euclidean space. As we have already seen, the residue sequence of standard tableaux plays an essential role when studying the orientifold quiver Temperley-Lieb algebra (or the symplectic blob algebra). We will therefore choose our embedding of tableaux in the Euclidian space in a way which allows us to easily identify the residue classes of standard tableaux.

Define $e$ to be the minimal positive integer, if it exists, such that $q^{2e} = 1$. If no such integer exists, we set $e=\infty$. By our standing assumption, we have $e>2$.

 For $\beta \in \{\alpha_1^{\pm1},\alpha_2^{\pm1},\extra^{\pm1}\}$ 
    we  write $\beta = q^{\Bee}$, where $\Bee$ is just a formal symbol if $\beta\notin q^{\ZZ}$. We also assume that $0<\Bee<2e$ if $e<\infty$ and $\beta\in q^\ZZ$. We define a lattice $
 \mathcal{L}_\beta$ by  
 $$
 \mathcal{L}_\beta=(\Bee,0)+\ZZ\eps_1+\ZZ\eps_2=\{(\Bee + m-n,m+n) \mid m,n \in \ZZ \}, 
 $$
where $\eps_1:=(1,1)$ and $\eps_2:=(-1,1)$. We will draw this lattice in $\mathbb{R}^2$ by placing a vertex $(\Bee,0)$ labelled by $\beta$, and by drawing from this vertex the edges corresponding to $\eps_1$-steps and $\eps_2$-steps. Note that our $y$-axis is oriented downwards, so that $\eps_1$ is a SE step and $\eps_2$ is a SW step. We will only need and only show the vertices with non-negative $y$-coordinate.

Note that we have points with $x$-coordinate $0$ in the lattice $\mathcal{L}_{\beta}$ if and only if $b\in \ZZ$ (that is, $\beta\in q^{\ZZ}$). In this case, we draw a  hyperplane through the points with $x=0$. Also, in this case 
   $\mathcal{L}_{\beta} =\mathcal{L}_{\beta^{-1}}$ and we merge these two labelled lattices by placing both marked points $\beta$ and $\beta^{-1}$ on the same lattice. Namely, the vertex $(-b,0)$ of $\mathcal{L}_{\beta}$ is labelled by $\beta^{-1}$.

 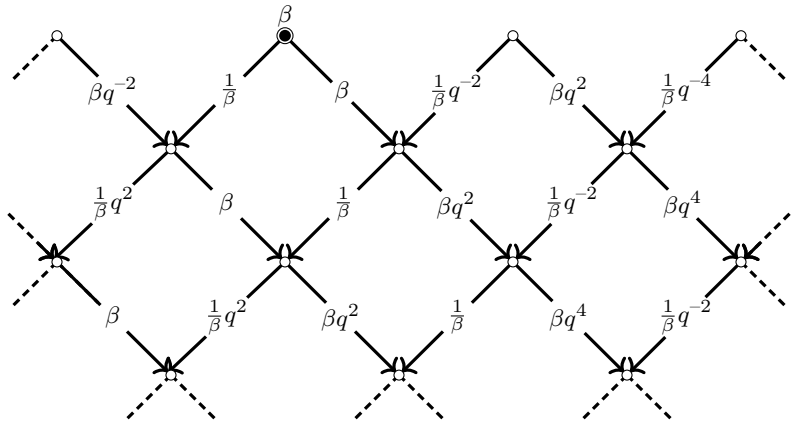
\begin{figure}[ht!] 

$$\scalefont{0.85}
\begin{tikzpicture}[scale=1.5]
 \node (0,0.2) [above] {$\beta$};
\path(0,0)--(1,-1) coordinate (Y11) coordinate  [midway] (X);
\draw[very thick, ->](0,0)--(0.975,-0.975) ;
\fill[white](X) circle (0.125cm);
\draw(X) node {$\beta$};

\path(Y11)--++(1,-1) coordinate (Y22)  coordinate [midway] (X);
\draw[very thick, ->](Y11)--++(0.975,-0.975);
\fill[white](X) circle (0.18cm);
\draw(X) node {$\beta q^2$};

\path(Y11)--++(1,1) coordinate (Y20)  coordinate [midway] (X);
\draw[very thick, ->](Y20)--++(-1+.025,-1+0.025); 
\fill[white](X) circle (0.26cm);
\draw(X) node {$\tfrac{1}{\beta}q^{\text{--}2}  $};

\path(Y20)--++(-1,-1) coordinate (Y02)  ;
\draw[very thick, ->](Y02)--++(-1+.025,-1+0.025) coordinate [midway] (X); 
\fill[white](X) circle (0.18cm);
\draw(X) node {$\tfrac{1}{\beta}   $};

\draw[very thick, ->](2,0)-- (3-0.025,-1+0.025) coordinate [midway] (X); 
\fill[white](X) circle (0.18cm);
\draw(X) node {$ {\beta} q^2  $};

\draw[very thick, ->](3,-1)-- (4-0.025,-2+0.025) coordinate [midway] (X); 
\fill[white](X) circle (0.18cm);
\draw(X) node {$ {\beta} q^4  $};

\draw[very thick, ->](4,0)-- (3+0.025,-1+0.025) coordinate [midway] (X); 
\fill[white](X) circle (0.25cm);
\draw(X) node {$ \tfrac{1}{\beta} q^{\text{--}4}  $};

\draw[very thick, ->](3,-1)-- (2+0.025,-2+0.025) coordinate [midway] (X); 
\fill[white](X) circle (0.25cm);
\draw(X) node {$ \tfrac{1}{\beta} q^{\text{--}2}  $};

\draw[very thick, ->](2,-2)-- (1+0.025,-3+0.025) coordinate [midway] (X); 
\fill[white](X) circle (0.18cm);
\draw(X) node {$ \tfrac{1}{\beta}     $};

\draw[very thick, ->](0,0)-- (-1+0.025,-1+0.025) coordinate [midway] (X); 
\fill[white](X) circle (0.18cm);
\draw(X) node {$ \tfrac{1}{\beta}     $};

\draw[very thick, ->](2,-2)-- (3-0.025,-3+0.025) coordinate [midway] (X); 
\fill[white](X) circle (0.2cm);
\draw(X) node {$ {\beta} q^{4}  $};

\draw[very thick, ->](4,-2)-- (3+0.025,-3+0.025) coordinate [midway] (X); 
\fill[white](X) circle (0.25cm);
\draw(X) node {$ \tfrac{1}{\beta} q^{\text{--}2}  $};

\draw[very thick, ->](-2,0)-- (-1-0.025,-1+0.025) coordinate [midway] (X); 
\fill[white](X) circle (0.18cm);
\draw(X) node {$ {\beta}q^{\text{--}2}     $};

\draw[very thick, ->](-2+1,0-1)-- (1-1-0.025,-1-1+0.025) coordinate [midway] (X); 
\fill[white](X) circle (0.18cm);
\draw(X) node {$ {\beta}    $};

\draw[very thick, ->](-2+1-1,0-1-1)-- (1-1-0.025-1,-1-1+0.025-1) coordinate [midway] (X); 
\fill[white](X) circle (0.18cm);
\draw(X) node {$ {\beta}     $};

\draw[very thick, ->](-2+1+1,0-1-1)-- (1+1-1-0.025,-2-1+0.025) coordinate [midway] (X); 
\fill[white](X) circle (0.18cm);
\draw(X) node {$ {\beta}q^{2}     $};

\draw[very thick, ->](-2+1,0-1)-- (-1-1-0.025,-1-1+0.025) coordinate [midway] (X); 
\fill[white](X) circle (0.25cm);
\draw(X) node {$\tfrac{1} {\beta}q  ^2    $};

\draw[very thick, ->](-2+1+1,0-1-1)-- (-1-1-0.025+1,-1-1-1+0.025) coordinate [midway] (X); 
\fill[white](X) circle (0.25cm);
\draw(X) node {$\tfrac{1} {\beta}q  ^2    $};

\draw[very thick,densely dashed,<-] 
 (-2-0.025,-2+0.025)--++(135:0.57);
 \draw[very thick,densely dashed,<-] 
(4+0.025,-2+0.025)--++(45:0.57);

\draw[very thick, ,<-] 
 (-2-0.025,-2+0.025)--++(135:0.2);
 \draw[very thick, ,<-] 
(4+0.025,-2+0.025)--++(45:0.2);

\draw[very thick,densely dashed] (-2,0)--++(-135:0.57)		(-2,-2)--++(-135:0.57)
(-2,-2)--++(-45:0.57)	
(-1,-3)--++(-135:0.57)
(-1,-3)--++(-45:0.57)
( 1,-3)--++(-135:0.57)
(1,-3)--++(-45:0.57)
(3,-3)--++(-135:0.57)
(3,-3)--++(-45:0.57)
(4,-2)--++(-45:0.57)
(4,-0)--++(-45:0.57);
;

\draw[fill=white](3,-3) circle (1.25pt);
\draw[fill=white](3,-1) circle (1.25pt);
\draw[fill=white](4,0) circle (1.25pt);
\draw[fill=white](4,-2) circle (1.25pt);
\draw[fill=white](-2,0) circle (1.25pt);
\draw[fill=white](-2,-2) circle (1.25pt);
\draw[fill=white](-1,-3) circle (1.25pt);
\draw[fill=white](-1,-1) circle (1.25pt);
\draw[fill=white](0,-2) circle (1.25pt);
\draw[fill=white](1,-3) circle (1.25pt);
\draw[fill=white](Y20) circle (1.25pt);
\draw[fill=white](Y11) circle (1.25pt);
\draw[fill=white](Y22) circle (1.25pt);
\node[M] at (0,0){};

\end{tikzpicture}
$$

\caption{ The residues of the steps in the lattice $\mathcal{L}_\beta$. }
\label{figresiduelattice}\end{figure}

Now, independently of whether $\beta^{-1}$ is or is not on the same lattice than $\beta$, if we have another $\beta'\in\{\alpha_1^{\pm1},\alpha_2^{\pm1},\extra^{\pm1}\}$ such that $\beta' \in \beta q^{2\ZZ}$ then 
we have that $\mathcal{L}_{\beta}=\mathcal{L}_{\beta'}$ 
and we merge these two labelled lattices by placing both marked points $\beta$ and $\beta'$ on the same lattice. Namely if $\beta'=q^{2c}\beta$ then $\beta'$ labels the vertex $(b+2c,0)$ in $\mathcal{L}_\beta$.

At the end, our lattice $\mathcal{L}_{\beta}$ may have as many as 6 different types of markings depending on how many special points among $\{\alpha_1^{\pm1},\alpha_2^{\pm1},\extra^{\pm1}\}$ are in the same $q^{2\ZZ}$-orbits. It may or may not have a hyperplane depending if the special point $\beta$ is in $q^{\ZZ}$ or not.

Suppose that $q^{2e}=1$ for some $e>0$. Then in every lattice $\mathcal{L}_{\beta}$, we also label by $\beta$ the horizontal translates of the marked point $(b,0)$ by multiples of $2e$. Therefore all points $(b+2re,0)$, $r\in\ZZ$, are marked by $\beta$. We do that for every type of marked points.

Suppose furthermore that we have a first hyperplane at $x=0$ in our lattice $\mathcal{L}_{\beta}$. Then we draw parallel hyperplanes at $x = re$ for all $r\in \mathbb{Z}$.


To each step $+\eps_1$ or $+\eps_2$ in the lattice $\mathcal{L}_\beta$, we associate a residue $i\in I$. This is defined by first associating the residue $\beta$ to the step $(b-j,j)\xrightarrow{+\eps_1} (b-j+1,j+1)$ (for all $j\geq 0$) and the residue $\beta^{-1}$ to the step $(b+j,j)\xrightarrow{+\eps_2} (b+j-1,j+1)$ (for all $j\geq 0$). We then extend to every step in the lattice by setting 
$$\res ((x+1,y+1)\xrightarrow{+\eps_1} (x+2,y+2)) = q^2 \res((x,y)\xrightarrow{+\eps_1}(x+1,y+1))$$
and 
$$\res ((x-1,y+1)\xrightarrow{+\eps_2} (x-2,y+2)) = q^2 \res((x,y)\xrightarrow{+\eps_2}(x-1,y+1)).$$
for every $(x,y)\in \mathcal{L}_\beta$. This is illustrated in \cref{figresiduelattice,fighyperplaneresidues}.


\begin{figure}[ht!]
$$ \begin{minipage}{13.3cm}
  \scalefont{0.85}
\begin{tikzpicture}[scale=0.45]

\draw[line width = 2.7, densely dotted] (-4-7,1)--(-4-7,-15);
\draw[line width = 2.7, densely dotted] (-4,1)--(-4,-15);
\draw[line width = 2.7, densely dotted] (-4+7,1)--(-4+7,-15);
\draw[line width = 2.7, densely dotted] (-4+7+7,1)--(-4+7+7,-15);

 \foreach \i  in {-16,-14,...,12}
 {
 \draw[very thick,fill=white](\i,0) circle (4pt);
 \draw[very thick,fill=white](\i,-14) circle (4pt);
 }

 \foreach \i  in {0,-2,...,-14}
 {
 \draw[fill=white](-16,\i) circle (4pt);
 \draw[fill=white](12,\i) circle (4pt);
 }

\draw(-8,0) node [above] {$\beta^{-1}$};
\draw(0,0)node [above] {$\beta$};
\draw(-14,0)node [above] {$\beta$};

 \node[M] at (0,0){};
  \node[M] at (-8,0){};
  \node[M] at (-8-6,0){};
  \node[M] at (6,0){};

\clip(-16,0) rectangle (12,-14);

 \node (0,0.2) [above] {$\beta$};

\foreach \j in {-12,...,-1,0,1,2,...,12}
{
 \foreach \i  in {-2,-1,0,1,2}
 {
\draw[line width=3, red, ](0+14*\i-\j ,-\j+0)--(1 +14*\i-\j ,-\j+-1 ); 
\draw[line width=3, violet, ](1+14*\i-\j,-\j+-1)--(2 +14*\i-\j,-\j+-2 ); 
\draw[line width=3, orange, ](2+14*\i-\j,-\j+-2)--(3 +14*\i-\j,-\j+-3); 
\draw[line width=3, green, ](3+14*\i-\j,-\j+-3)--(4 +14*\i-\j,-\j+-4); 
\draw[line width=3, pink, ](4+14*\i-\j,-\j+-4)--(5 +14*\i-\j,-\j+-5); 
\draw[line width=3, gray ](5+14*\i-\j,-\j+-5)--(6 +14*\i-\j,-\j+-6); 
\draw[line width=3, cyan ](6+14*\i -\j,-\j+-6)--(7 +14*\i-\j,-\j+-7); 
\draw[line width=3, red, ](7+14*\i -\j,-\j+-7)--(8 +14*\i -\j,-\j+-8 ); 
}
}

\foreach \j in {-12,...,-1,0,1,2,...,12}
{
 \foreach \i  in {-7,...,-2,-1,0,1,2,...,12}
 {
\draw[line width=3, green, ](0+14*\i  +\j , 0-\j )--++(-135:1.4142135 ); 
\draw[line width=3, pink, ](-1+14*\i  +\j , -1-\j )--++(-135:1.4142135 ); 
 \draw[line width=3, gray, ](-2+14*\i  +\j , -2-\j )--++(-135:1.4142135 ); 
\draw[line width=3, cyan, ](-3+14*\i  +\j , -3-\j )--++(-135:1.4142135 ); 
\draw[line width=3, red, ](-4+14*\i  +\j , -4-\j )--++(-135:1.4142135 ); 
\draw[line width=3, violet ](-5+14*\i  +\j , -5-\j )--++(-135:1.4142135 ); 
\draw[line width=3, orange, ](-6+14*\i  +\j , -6-\j )--++(-135:1.4142135 ); 
\draw[line width=3, green, ](-7+14*\i  +\j , -7-\j )--++(-135:1.4142135 ); 
}
}

%

\foreach \j in {16,14,...,-2,-4,-6,-8,-10,-12,-14}
{
 \foreach \i  in {-26,-24,...,28,30}
 {
 \draw[very thick,fill=white](\i,\j) circle (4pt);
 }
}

\foreach \j in {16,14,...,-2,-4,-6,-8,-10,-12,-14}
{
 \foreach \i  in {-26,-24,...,28,30}
 {
 \draw[very thick,fill=white](\i+1,\j+1) circle (4pt);
 }
}

 \foreach \i  in {-16,-14,...,12}
 {
 \draw[very thick,fill=white](\i,0) circle (4pt);
 \draw[very thick,fill=white](\i,-14) circle (4pt);
 }

 \foreach \i  in {0,-2,...,-14}
 {
 \draw[fill=white](-16,\i) circle (4pt);
 \draw[fill=white](12,\i) circle (4pt);
 }

 \node[M] at (0,0){};
 \node[M] at (0,0){};
  \node[M] at (-8,0){};
  \node[M] at (-8-6,0){};
  \node[M] at (6,0){};

\end{tikzpicture}\end{minipage}
 \begin{minipage}{2cm}
 \begin{tikzpicture} 
 \draw[very thick] (0,0) circle (1cm);
 \draw[very thick, <->](0,-1.4)--(0,1.4);
  \draw[very thick, <->]( -1.35,0)--(1.35,0);
  \draw[thick, fill= cyan](51.428:1) circle (3pt);
    \draw[thick, fill= red](51.428*2:1) circle (3pt);
        \draw[thick, fill= violet](51.428*3:1) circle (3pt);
    \draw[thick, fill= orange](51.428*4:1) circle (3pt);     
    \draw[thick, fill= green](51.428*5:1) circle (3pt);      
    \draw[thick, fill= pink](51.428*6:1) circle (3pt);     
        \draw[thick, fill= gray](0:1) circle (3pt);
 \end{tikzpicture}\end{minipage}$$

\caption{ On the left we picture the residues of the steps in the lattice $\mathcal{L}_\beta$ for $e=7$ and $\beta=q^4$. The colours of the 7 distinct powers of $q^2$ are given on the right.}
\label{fighyperplaneresidues}
\end{figure}
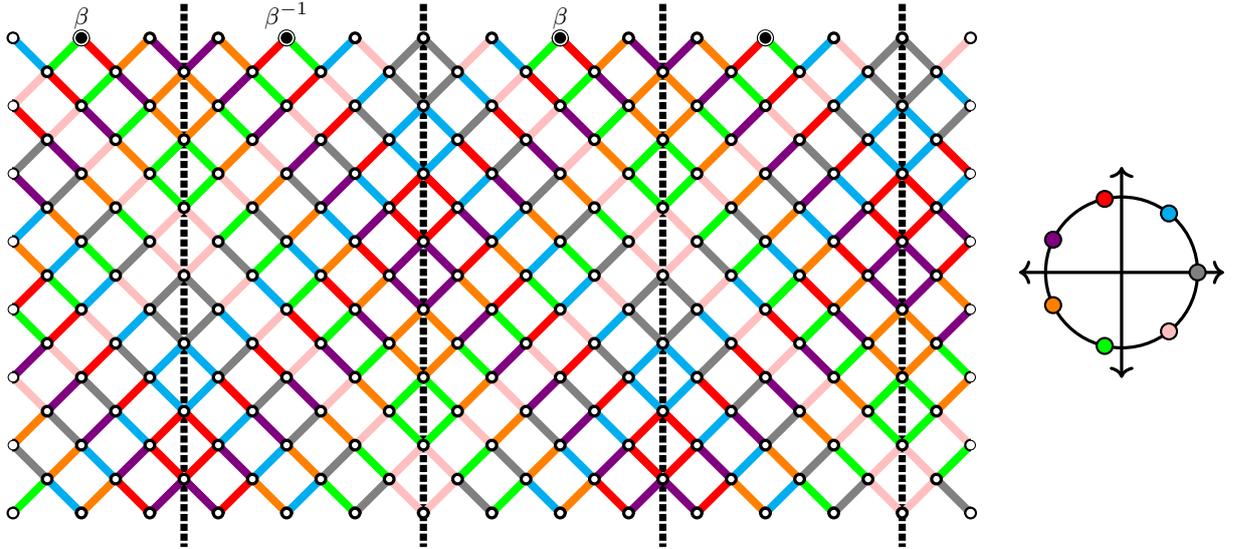

Note that if $\mathcal{L}_\beta = \mathcal{L}_{\beta'}$ then the residues are consistently defined. Note further that reflections through the hyperplanes $x=0$ and $x=re$ for $r\in \mathbb{Z}$, if present, preserve the residues. Moreover, translation by $(2re,0)$ for any $r\in \ZZ$ (when $e<\infty$) also preserves the residues. This is illustrated in \cref{fighyperplaneresidues}

We are now ready to embed our standard tableaux into the lattices $\mathcal{L}_\beta$ for $\beta \in \{ \alpha_1^{\pm 1}, \alpha_2^{\pm}, \extra^{\pm}\}$.
Let $0\leq k \leq n$ with $k \equiv n$ modulo 2 when $k\neq 0$ and let $\stt \in \Std_n(k,\beta)$.  We define $m(\stt) = \lfloor \frac{n-k}{2}\rfloor - \# \text{negative entries in $\stt$}$. (Note that $m(\stt)$ can be negative when $\stt\in \Std_n(0,\extra)$). We depict $\stt$ as a path in $\mathcal{L}_\beta$ starting at the vertex $(b-1-2m(\stt),1)$ as follows: 
we read the entries of $\stt$ in increasing modulus, and we  take a step $+\eps_1:=(1,-1)$ if the entry if positive and a step $+\eps_2:=(-1,1)$ if the entry is negative. 
We  also use the following notation  $\eps_{\overline{1}}=\eps_2 $ and $\eps_{\overline{2}}=\eps_1 $.
We note that all paths of shape $(k,\beta)$ will end at the same point. The end point of these paths is given by $(\Bee-1+k, n+1)$ in all cases except when $k= 0$ when $n$ odd when the paths end at $(\Bee, n+1)$. With this defintion, the residue sequence of the tableau is precisely the residue sequence obtained by taking the residues of each step in the path on $\mathcal{L}_\beta$. 
Examples are depicted in \cref{thisistheexampleblow5hghghghgh} and \cref{figspecialmodules}

\begin{figure}[ht!]
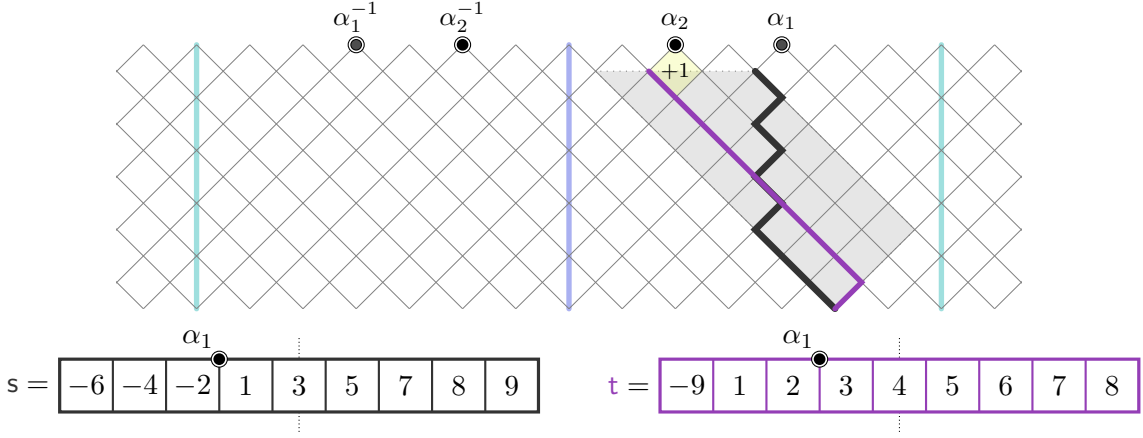

\[
\def\pE{14}
\def\pN{9}
\def\pA{8}
\def\pB{3}
\TIKZ[yscale=-1, scale=.35, font=\scriptsize]{
\pgfmathsetmacro{\pEParity}{Mod(\pE,2)}
\pgfmathsetmacro{\pNParity}{Mod(\pN,2)}
\coordinate (max-start) at (\pA-1, 0);
\path (max-start) to ++(\pN-2*\pB, \pN) coordinate (end);
\path (end) to ++(-\pN, -\pN) coordinate (min-start);
\path (end) to ++(\pB, -\pB) coordinate (p-corner);
\draw[pathsregion] (max-start) to (min-start) to (end) to (p-corner) to (max-start); 
\coordinate (beta2) at (4, -1); \coordinate (beta2inv) at (-4, -1);
\coordinate (beta1) at (8, -1); \coordinate (beta1inv) at (-8, -1);
\filldraw[plusone] (beta2) \TILE;
	\path (beta2) to ++(0,1) node {$+1$};
\draw[tauplane] (0,-1) to +(0,\pN+1);
\draw[sigmaplane] (-\pE,-1) to +(0,\pN+1);
\draw[sigmaplane] (\pE,-1) to +(0,\pN+1);
\begin{scope}
\clip (-17, -1) rectangle (17, \pN);
	\foreach \x in {-20, ..., 20}{
		\draw[gridstyle] (2*\x,-1) to +(\pN+1,\pN+1);
		\draw[gridstyle] (2*\x,-1) to +(-\pN-1,\pN+1);
	}
\end{scope}
\foreach \x in {beta1, beta1inv}
	{\node[M1] at (\x){};}
\foreach \x in {beta2, beta2inv}
	{\node[M2] at (\x){};}
\node[above] at (beta1) {\small \strut $\alpha_1$};
	\node[above] at (beta2) {\small \strut $\alpha_2$};
\node[above] at (beta1inv) {\small \strut $\alpha_1^{-1}$};
	\node[above] at (beta2inv) {\small \strut $\alpha_2^{-1}$};
\draw[path1] (max-start) 
	\foreach \x in {1, ..., 3}{ to ++(1,1) to ++(-1,1)} to (end);
\draw[path2] (3,0) to ++(8,8) to ++(-1,1);
}
\]
\[
\TIKZ[scale=0.7]{
\node[left] at (0, 0){${\color{path1color}\sts} =$};
\draw[densely dotted](.5*9, .9) --++(0,-2*.9);
\draw[very thick, path1color, fill=white] (0,-.5) rectangle (9,.5);
\foreach \x [count = \c from 1] in {-6, -4, -2, 1, 3, 5, 7, 8, 9}{
\node at (\c-.5, 0) {$\x$};
\ifnum \c = 9
	\breakforeach
	\fi
\draw[path1color, thick] (\c, .5) to +(0,-1);
} 
\node[M, label = {100, inner sep=1.5pt}:{$\alpha_1$}] at (3, .5){};
}
\qquad 
\TIKZ[scale=0.7]{
\node[left] at (0, 0){${\color{path2color}\stt} =$};
\draw[densely dotted](.5*9, .9) --++(0,-2*.9);
\draw[very thick, path2color, fill=white] (0,-.5) rectangle (9,.5);
\foreach \x [count = \c from 1] in {-9, 1, 2, ..., 8}{
\node at (\c-.5, 0) {$\x$};
\ifnum \c = 9
	\breakforeach
	\fi
\draw[path2color, thick] (\c, .5) to +(0,-1);
} 
\node[M, label = {100, inner sep=1.5pt}:{$\alpha_1$}] at (3, .5){};
}
\]
\caption{Here  $e=14$ and $\alpha_1=q^8$ and $\alpha_2=q^4$.
We depict two tableaux $\color{path1color}\sts$ and $\color{path2color}\stt$ in $\Std_9(3,\alpha_1)$ and their corresponding paths.
The first ({\color{path1color}black}) path  has degree 0 and the second ({\color{path2color}purple}) path has degree 1.
We note that the grey shading around these paths allows us to uniquely associate them to tableau by telling us where the relevant special point is. The shape of these tableaux is $(3,\alpha_1)$.
We have that $x_{\color{path1color}\sts} (0)=x_{\color{path1color}\sts} (2)=
x_{\color{path1color}\sts} (4)=x_{\color{path1color}\sts} (6)=7$	 and 
$x_{\color{path1color}\sts} (8)	=9$.
}
 \label{thisistheexampleblow5hghghghgh}
 \end{figure}

\begin{rmk}\label{shape?}
 
When we  draw  any   path $\stt$  in $\mathcal{L}_\beta$ representing a tableau in $\Std_n(k, \beta)$ we depict this in a greyed-region $\mathbb{T}_{(k,\beta)}$ that records the shape $(k, \beta)$ of the tableau.  
We do this by shading the  convex hull of  the set of  all  paths 
$\sts \in \Std_n(k,\beta)$. Except when $\beta=\extra$ (where it is a triangle), this shaded area always has the shape of a trapezoid. The width of the trapezoid is directly related to $k$, it grows as $k$ becomes smaller (in fact the width is equal to $\frac{n-k}{2}$). When $k\neq 0$, the top right corner of the region $\mathbb{T}_{(k,\beta)}$ for $\beta = q^b$ is the vertex $(b-1,1)$. See \cref{thisistheexampleblow5hghghghgh} and \cref{figspecialmodules} for examples.
 \end{rmk}

 It will be convenient for us to consider other paths on the lattices $\mathcal{L}_\beta$, not only those corresponding to standard tableaux. All our paths will start at some vertex $(c,1)$ in our lattices, have length $n$ and only take steps of the form $+\eps_1$ or $+\eps_2$. For a path $\SSTP$ on $\mathcal{L}_\beta$, we let $x_\SSTQ(i)$ denote the $x$-coordinate of the path after $0\leq i \leq n$ steps.  

Once we fix its starting point, a path $\SSTP$ is completely determined by its sequence of steps $\SSTP=(\varepsilon_{p(1)},\ldots,\varepsilon_{p(n)})$ with $p(i)=1$ or $2$. Given such a path with starting point $(c,1)$ in the lattice $\mathcal{L}_{\beta}$, we denote by $-\SSTP$  
the path $(\varepsilon_{ \overline{p(1)}},\ldots,\varepsilon_{\overline{ p(n)}})$ with starting point $(-c,1)$ in the lattice $\mathcal{L}_{\beta^{-1}}$. Note that in the case where $\mathcal{L}_{\beta}=\mathcal{L}_{\beta^{-1}}$ and therefore contains the hyperplane $x=0$, the path $ - \SSTP $ is the full reflection of $\SSTP$ through this hyperplane. Otherwise, $- \SSTP $ lives in a different lattice than $\SSTP$.
Assume now that $q^{2e}=1$ for some minimal $e>0$. Then for a path $\SSTP$ and for $r\in\mathbb{Z}$, we define $\rho_{2re}(\SSTP)$ to be the horizontal translates of $\SSTP$ by $2re$.

 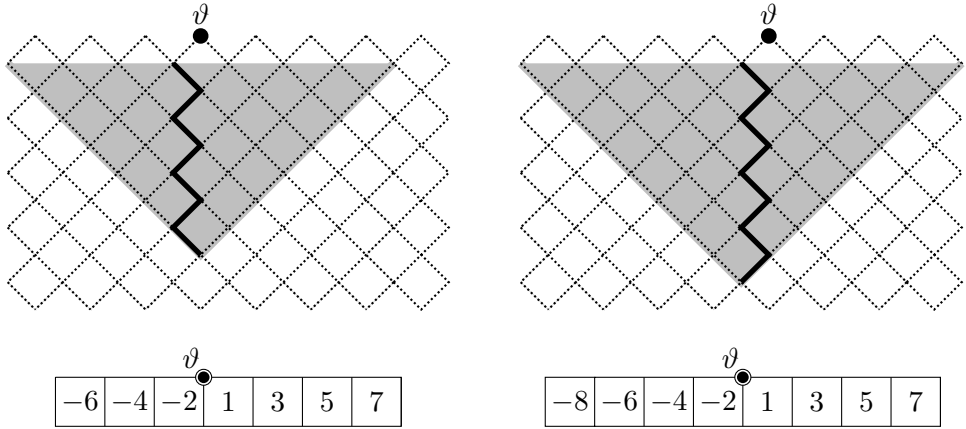
\begin{figure}[ht!]
$$ \begin{minipage}{6cm}
 $$
  \begin{tikzpicture} [scale=1.28251]

 \begin{scope}

 \draw[gray!30](0,0) coordinate (top00)
 --++(-45:0.4)  --++(45:0.4) coordinate (top1)
  --++(-45:0.4)  --++(45:0.4) coordinate (top2)
   --++(-45:0.4)  --++(45:0.4) coordinate (top3)
    --++(-45:0.4)  --++(45:0.4) coordinate (top4)
     --++(-45:0.4)  --++(45:0.4)coordinate  (top5)
  --++(-45:0.4)   coordinate(right1) --++(-135:0.4) 
  --++(-45:0.4)  coordinate (right2)--++(-135:0.4)  
  --++(-45:0.4)   coordinate (right3)--++(-135:0.4) 
  --++(-45:0.4) coordinate (right4) --++(-135:0.4)  
  --++(-45:0.4)  coordinate (right5)
  --++(-135:0.4)  
coordinate (bottom5)    --++(135:0.4)    
  --++(-135:0.4)  
coordinate (bottom4)    --++(135:0.4)    
  --++(-135:0.4)  
coordinate (bottom3)    --++(135:0.4)    
  --++(-135:0.4)  
coordinate (bottom2)    --++(135:0.4)    
  --++(-135:0.4)  
coordinate (bottom1)    --++(135:0.4)    
  --++(-135:0.4)  
coordinate (bottom0)    --++(135:0.4)    
  --++(-135:0.4)  
coordinate (bottomm1)    
--++(135:0.4)   
coordinate (left5)    --++(45:0.4)     
--++(135:0.4)   
coordinate (left4)    --++(45:0.4)     
--++(135:0.4)   
coordinate (left3)    --++(45:0.4)     
--++(135:0.4)   
coordinate (left2)    --++(45:0.4)     
 --++(135:0.4)   
coordinate (left1)    --++(45:0.4)     
coordinate (topm1)
--++(-45:0.4)    --++(45:0.4)      ;
 
  \path(top00)--++(-135:0.8)--++(135:0.4) 
  --++(-45:0.8*2)--++(45:0.4*4)   --++(-45:0.8*2)--++(45:0.4*4) 
   --++(-135:0.4)--++(135:0.4) 
  coordinate (PP);
 \draw[gray!50,fill=gray!50,line width =3] (PP)--++(-135:0.4*7) --++(135:0.4*7);

 \path(top3)    --++(-45:0.2)  --++(45:0.2) coordinate (top45);
  \path(bottom3)    --++(-45:0.2)  --++(45:0.2) coordinate (bottom45);

 	\draw[ thick,,densely dotted](left5)--++(-45:0.4);
 	\draw[ thick,,densely dotted](left4)--(bottom0); 
 	\draw[ thick,,densely dotted](left3)--(bottom1); 
 	\draw[ thick,,densely dotted](left2)--(bottom2);  
 	\draw[ thick,,densely dotted](left1)--(bottom3); 
	\draw[ thick,,densely dotted](topm1)--(bottom4); 
	\draw[ thick,,densely dotted](top00)--(bottom5); 
	\draw[ thick,,densely dotted](top1)--(right5); 
	\draw[ thick,,densely dotted](top2)--(right4); 
		\draw[ thick,,densely dotted](top3)--(right3); 
	\draw[ thick,,densely dotted](top4)--(right2); 
		\draw[ thick,,densely dotted](top5)--(right1);

		\draw[ thick,,densely dotted](topm1)--++(-135:0.4); 
		\draw[ thick,,densely dotted](top00)--++(-135:0.4*3); 
		\draw[ thick,,densely dotted](top1)--++(-135:0.4*5);
		\draw[ thick,,densely dotted](top2)--++(-135:0.4*7);
		\draw[ thick,,densely dotted](top3)--++(-135:0.4*9);
\draw[ thick,,densely dotted](top4)--++(-135:0.4*10);
\draw[ thick,,densely dotted](top5)--++(-135:0.4*10);
\draw[ thick,,densely dotted](right1)--++(-135:0.4*9);
\draw[ thick,,densely dotted](right2)--++(-135:0.4*7);
\draw[ thick,,densely dotted](right3)--++(-135:0.4*5);
\draw[ thick,,densely dotted](right4)--++(-135:0.4*3);
\draw[ thick,,densely dotted](right5)--++(-135:0.4*1);

\draw[ thick,,densely dotted](right5)--++(-45:0.4*1)
--++(45:0.4*1)
--++(135:0.4*2)   --++(45:0.4*2)   --++(135:0.4*2)  --++(45:0.4*2)   --++(135:0.4*1)--++(-135:0.4*1)
(right5)--++(45:0.4*2)--++(135:0.4*2)    --++(45:0.4*2)   --++(135:0.4*2)
;

\path(top00)--++(-45:0.4)--++(-45:0.4)--++(45:0.4)  coordinate (here);
\path(top00)--++(-45:0.8)--++(45:0.8) coordinate (top00);

\draw[line width =2,fill] (top00) circle (1.5pt) node [above] {$\vartheta$} ;

\draw[line width =2] (here) 
--++(-45:0.4*1)
--++(-135:0.4*1)
--++(-45:0.4*1)
--++(-135:0.4*1)
--++(-45:0.4*1)
--++(-135:0.4*1)
   --++(-45:0.4*1)
 coordinate (here2);

\end{scope}

 \end{tikzpicture}
 $$
  $$
\begin{tikzpicture}[scale=0.65]
\draw(0,0) --++(0:1) coordinate (x1)
 --++(0:1) coordinate (x2)
  --++(0:1) coordinate (x3)
   --++(0:1) coordinate (x4)
    --++(0:1) coordinate (x5)
     --++(0:1) coordinate (x6)
      --++(0:1) coordinate (x7)
--++(-90:1)--++(180:7)--++(90:1);
\foreach \i in {1,2,3,4,5,6,7}
{
\draw(x\i)--++(-90:1);
\path(x\i)--++(-90:0.5)--++(180:0.5) coordinate (y\i);
}

\path(y1) node {$-6$};
\path(y2) node {$-4$};
\path(y3) node {$-2$};
\path(y4) node {$1$};
\path(y5) node {$3$};
\path(y6) node {$5$};
\path(y7) node {$7$};

\fill(x3) circle (4pt); 
\path(x3)--++(90:0.4)--++(180:0.25) node {$\vartheta$};
\node[M] at (x3){};
 \end{tikzpicture}
$$
\end{minipage}
%
%
%
%
%
%
\qquad
 \begin{minipage}{6cm}
 $$
  \begin{tikzpicture} [scale=1.28251]

 \begin{scope}

 \draw[gray!30](0,0) coordinate (top00)
 --++(-45:0.4)  --++(45:0.4) coordinate (top1)
  --++(-45:0.4)  --++(45:0.4) coordinate (top2)
   --++(-45:0.4)  --++(45:0.4) coordinate (top3)
    --++(-45:0.4)  --++(45:0.4) coordinate (top4)
     --++(-45:0.4)  --++(45:0.4)coordinate  (top5)
  --++(-45:0.4)   coordinate(right1) --++(-135:0.4) 
  --++(-45:0.4)  coordinate (right2)--++(-135:0.4)  
  --++(-45:0.4)   coordinate (right3)--++(-135:0.4) 
  --++(-45:0.4) coordinate (right4) --++(-135:0.4)  
  --++(-45:0.4)  coordinate (right5)
  --++(-135:0.4)  
coordinate (bottom5)    --++(135:0.4)    
  --++(-135:0.4)  
coordinate (bottom4)    --++(135:0.4)    
  --++(-135:0.4)  
coordinate (bottom3)    --++(135:0.4)    
  --++(-135:0.4)  
coordinate (bottom2)    --++(135:0.4)    
  --++(-135:0.4)  
coordinate (bottom1)    --++(135:0.4)    
  --++(-135:0.4)  
coordinate (bottom0)    --++(135:0.4)    
  --++(-135:0.4)  
coordinate (bottomm1)    
--++(135:0.4)   
coordinate (left5)    --++(45:0.4)     
--++(135:0.4)   
coordinate (left4)    --++(45:0.4)     
--++(135:0.4)   
coordinate (left3)    --++(45:0.4)     
--++(135:0.4)   
coordinate (left2)    --++(45:0.4)     
 --++(135:0.4)   
coordinate (left1)    --++(45:0.4)     
coordinate (topm1)
--++(-45:0.4)    --++(45:0.4)      ;
 
  \path(top00)--++(-135:0.8)--++(135:0.4) 
  --++(-45:0.8*2)--++(45:0.4*4)   --++(-45:0.8*2)--++(45:0.4*4) 
  coordinate (PP);
 \draw[gray!50,fill=gray!50,line width =3] (PP)--++(-135:0.4*8) --++(135:0.4*8);

 \path(top3)    --++(-45:0.2)  --++(45:0.2) coordinate (top45);
  \path(bottom3)    --++(-45:0.2)  --++(45:0.2) coordinate (bottom45);

 	\draw[ thick,,densely dotted](left5)--++(-45:0.4);
 	\draw[ thick,,densely dotted](left4)--(bottom0); 
 	\draw[ thick,,densely dotted](left3)--(bottom1); 
 	\draw[ thick,,densely dotted](left2)--(bottom2);  
 	\draw[ thick,,densely dotted](left1)--(bottom3); 
	\draw[ thick,,densely dotted](topm1)--(bottom4); 
	\draw[ thick,,densely dotted](top00)--(bottom5); 
	\draw[ thick,,densely dotted](top1)--(right5); 
	\draw[ thick,,densely dotted](top2)--(right4); 
		\draw[ thick,,densely dotted](top3)--(right3); 
	\draw[ thick,,densely dotted](top4)--(right2); 
		\draw[ thick,,densely dotted](top5)--(right1);

		\draw[ thick,,densely dotted](topm1)--++(-135:0.4); 
		\draw[ thick,,densely dotted](top00)--++(-135:0.4*3); 
		\draw[ thick,,densely dotted](top1)--++(-135:0.4*5);
		\draw[ thick,,densely dotted](top2)--++(-135:0.4*7);
		\draw[ thick,,densely dotted](top3)--++(-135:0.4*9);
\draw[ thick,,densely dotted](top4)--++(-135:0.4*10);
\draw[ thick,,densely dotted](top5)--++(-135:0.4*10);
\draw[ thick,,densely dotted](right1)--++(-135:0.4*9);
\draw[ thick,,densely dotted](right2)--++(-135:0.4*7);
\draw[ thick,,densely dotted](right3)--++(-135:0.4*5);
\draw[ thick,,densely dotted](right4)--++(-135:0.4*3);
\draw[ thick,,densely dotted](right5)--++(-135:0.4*1);

\draw[ thick,,densely dotted](right5)--++(-45:0.4*1)
--++(45:0.4*1)
--++(135:0.4*2)   --++(45:0.4*2)   --++(135:0.4*2)  --++(45:0.4*2)   --++(135:0.4*1)--++(-135:0.4*1)
(right5)--++(45:0.4*2)--++(135:0.4*2)    --++(45:0.4*2)   --++(135:0.4*2)
;

\path(top00)--++(-45:0.4)--++(-45:0.4)--++(45:0.4)  coordinate (here);
\path(top00)--++(-45:1.2)--++(45:1.2) coordinate (top00);

\draw[line width =2,fill] (top00) circle (1.5pt) node [above] {$\vartheta$} ;
 
 \path(here)--++(-45:0.4)--++(45:0.4) coordinate (here);
\draw[line width =2] (here) 
--++(-45:0.4*1)
--++(-135:0.4*1)
--++(-45:0.4*1)
--++(-135:0.4*1)
--++(-45:0.4*1)
--++(-135:0.4*1)
   --++(-45:0.4*1)--++(-135:0.4*1)
 coordinate (here2);

\end{scope}

 \end{tikzpicture}
 $$
  $$
\begin{tikzpicture}[scale=0.65]
\draw(0,0) --++(0:1) coordinate (x1)
 --++(0:1) coordinate (x2)
  --++(0:1) coordinate (x3)
   --++(0:1) coordinate (x4)
    --++(0:1) coordinate (x5)
     --++(0:1) coordinate (x6)
      --++(0:1) coordinate (x7)
--++(-90:1)--++(180:8)--++(90:1)
--++(0:1)
--++(-90:1);
\foreach \i in {1,2,3,4,5,6,7}
{
\draw(x\i)--++(-90:1);
\path(x\i)--++(-90:0.5)--++(180:0.5) coordinate (y\i);
}
\path(y1)--++(180:1) node {$-8$};

\path(y1) node {$-6$};
\path(y2) node {$-4$};
\path(y3) node {$-2$};
\path(y4) node {$1$};
\path(y5) node {$3$};
\path(y6) node {$5$};
\path(y7) node {$7$};

\fill(x3) circle (4pt); 
\path(x3)--++(90:0.4)--++(180:0.25) node {$\vartheta$};
\node[M] at (x3){};
 \end{tikzpicture}
$$
\end{minipage}$$

 \caption{The shaded region given by all paths corresponding to standard tableaux of shape $(0,\extra)$ for $n$ even and $n$ odd.}\label{figspecialmodules} \end{figure}

Moreover, if hyperplanes are present, then for 
 $1\leq i \leq n$ 
we   write 
$\csigma \cdot _i \SSTP=\SSTQ$ if  $x_i(\SSTP)=re$
 for some $r\in \ZZ$ and 
$$
 \varepsilon_{q(t)}=\left\{\begin{array}{ll}
\varepsilon_{p(t)}& \ \text{for}\ 1\leq t\leq i\\
 \varepsilon_{ \overline{p(t)}}& \ \text{for}\ i+1\leq t\leq n .\end{array}\right.$$ 
In other words the paths $\SSTP$ and $\SSTQ$ agree up to some point $x_i(\SSTP)=x_i(\SSTQ)$ for $i\geq 1$ 
which lies on 
some hyperplane after which the rest of the path $\SSTQ$ is obtained from the rest of the path $\SSTP$ by reflection through this hyperplane. 

We write $\SSTP\sim\SSTQ$ if the two paths are obtained one from another by a finite sequence of operations of $\sigma\cdot_i$, of translations $\rho_{2re}$ and of the operation $\SSTP\mapsto-\SSTP$.

The following lemma follows directly from the definitions.
 
\begin{lem}\label{res=sim}
Let $\SSTS,\SSTT \in \Std_n$ denote two standard tableaux.  
Then we have $\SSTS \sim \SSTT$ if and only if $\res(\SSTS)=\res(\SSTT)$ 
\end{lem}

\begin{rmk} It will be useful to visualise  a standard tableau of shape $(0,\extra)$ both in the lattice $\mathcal{L}_{\extra}$ as described above but also in the lattice $\mathcal{L}_{\extra^{-1}}$ simply by mapping each $\SSTT\in \mathcal{L}_{\extra}$ to $-\SSTT\in \mathcal{L}_{\extra^{-1}}$
\end{rmk}

It is important to realise that if we draw a  path $\sts$ in $\mathcal{L}_\beta$   without 
contextualising it within a grey shaded region (as in  \cref{shape?}), 
 this path could correspond to several different tableaux; 
 for example both of the following tableaux
\[
\TIKZ[scale=0.7]{
\node[left] at (0, 0){${\color{path2color}\stt} =$};
\draw[densely dotted](.5*9, .9) --++(0,-2*.9);
\draw[very thick, path2color, fill=white] (0,-.5) rectangle (9,.5);
\foreach \x [count = \c from 1] in {-9, 1, 2, ..., 8}{
\node at (\c-.5, 0) {$\x$};
\ifnum \c = 9
	\breakforeach
	\fi
\draw[path2color, thick] (\c, .5) to +(0,-1);
} 
\node[M, label = {100, inner sep=1.5pt}:{$\alpha_1$}] at (3, .5){};
}
\qquad 
\TIKZ[scale=0.7]{
\node[left] at (0, 0){${\color{path2color}\stt'} =$};
\draw[densely dotted](.5*9, .9) --++(0,-2*.9);
\draw[very thick, path2color, fill=white] (0,-.5) rectangle (9,.5);
\foreach \x [count = \c from 1] in {-9, 1, 2, ..., 8}{
\node at (\c-.5, 0) {$\x$};
\ifnum \c = 9
	\breakforeach
	\fi
\draw[path2color, thick] (\c, .5) to +(0,-1);
} 
\node[M, label = {100, inner sep=1.5pt}:{$\alpha_2$}] at (1, .5){};
}
\]
correspond to the purple path in
  \cref{thisistheexampleblow5hghghghgh} but with different grey regions (the grey region corresponding to $(3,\alpha_1)$ is depicted in   \cref{thisistheexampleblow5hghghghgh},
  we leave it as an exercise for the reader to draw this path within 
   the grey region corresponding to $(7,\alpha_2)$).
We would like to know the maximal   $(k,\beta)\in \Lambda_n$ 
such that a given path can be depicted in the corresponding grey region (in the example above, the maximal such element is $(7,\alpha_2)$).  
Assume that $\sts$ is a path    in $\mathcal{L}_\beta$
 with $x_{\sts}(0)\leq x_\sts(n)$.
(If this is not the case, consider the path $-\sts$ in $\mathcal{L}_{\beta^{-1}}$
 with $x_{-\sts}(0)\leq x_{-\sts}(n)$.) 
We start by finding the first marked point to the right of the starting point of $\sts$, that is we consider the minimal integer $m_\sts\geq 0 $ such that 
$q^{2m_\sts+x_{\sts}(0)+1}\in \{\alpha_i^{\pm1}\}$, {\em if it exists}. 
If $$\la=(	
x_\sts(n) - x_{\sts}(0) - 2m_\sts	,q^{2m_\sts+x_{\sts}(0)+1})\in \Lambda_n$$then we define the {\sf maximal shape} of $\sts$ to be ${\sf max}(\sts)=\la \in \Lambda_n$.
If no such $\la\in \Lambda_n$ exists, then we define 
${\sf max}(\sts)=(0,\vartheta)\in \Lambda_n$ if 
$q^{x_\sts(n)}=\vartheta$ when $n$ is odd and $q^{x_\sts(n)+1}=\vartheta$ when $n$ is even. Finally if it is not the case, then the path $\sts$ does not correspond to any standard tableau and we leave ${\sf max}(\sts)$ undefined in this case.

\subsection{Permutations from orientifold paths} In this subsection, we will define an element $\psi_\stt$ for each $\stt\in \Std_n(k,\beta)$ using the lattices $\mathcal{L}_\beta$. 
The region,  $\mathbb T_{(k,\beta)}$, containing all paths $\SSTS \in \Std_n(k,\beta)$ is tiled by $(1\times 1)$-unit tiles (we include the tiles in the top row, which are only half shaded); 
  we say that two of these tiles are {\sf neighbouring} if they meet at an edge. 
The path $\SSTT_{(k,\beta)}$ lies within the region $\mathbb T_{(k,\beta)}$  and divides this region into a lefthand-side 
$\mathbb L_{(k,\beta)}$   
 and a righthand-side $ \mathbb R_{(k,\beta)}$,  by definition.  
 We define an {\sf admissible tiling} $\mathbb T$ to be any collection of tiles  in $\mathbb{T}_{(k,\beta)}$ which does not contain a pair of neighbouring tiles $T, T'$ such that 
 $T \in \mathbb L_{(k,\beta)}$ and  $T' \in \mathbb R_{(k,\beta)}$. 
 We define the length of an admissible tiling $ \mathbb T$ to be the number of tiles it contains.  
 We place an ordering on the tiles within  	$\mathbb{L}_{(k,\beta)}$	(resp. $\mathbb{R}_{(k,\beta)}$	) as follows: 
\begin{itemize} 
\item  Given two neighbouring  tiles $T, T'\in \mathbb L_{(k,\beta)}$,  we write $T	\leq_L T'$  
if the $x$-coordinate of the centroid of  $T'$  is strictly less than that of $T$.  
\item  Given two  neighbouring tiles $T, T'\in \mathbb R_{(k,\beta)}$,  we write $T	\leq_R T'$ if the $x$-coordinate of the centroid of  $T'$  is strictly greater than  that of $T$.  
\end{itemize}
Note that we say that an $x$-coordinate of the form $b+m$ is strictly greater than an $x$-coordinate of the form $b+m'$ when $m>m'$ (even if $b$ is just a formal symbol).

We extend $\leq _L$ to a partial order on $  \mathbb L_{(k,\beta)}$ 
(and $\leq _R$ to a partial order on $  \mathbb R_{(k,\beta)}$)   by transitivity.  
Let $\mathbb T$ be an admissible tiling and $T$ be a tile in $\mathbb T$. 
We say that   $T \in  \mathbb L_{(k,\beta)}\cap \mathbb T$  (respectively 
$T \in  \mathbb R_{(k,\beta)}\cap \mathbb T$) is {\sf supported} in $\mathbb{T}$ if every tile less than $T$ in the ordering $\leq _L$ (respectively $\leq _R$) belongs to~$\mathbb T$.
 
%
%

 \begin{defn}
   We say that  an admissible tiling  $ \mathbb T$  
  is {\sf reduced} if every tile in  $ \mathbb T$ is   supported.  We define the length $\ell (\mathbb{T})$ of a reduced admissible tiling $\mathbb{T}$ to be the number of tiles in $\mathbb{T}$.
 \end{defn}

\begin{prop}\label{inductive1}
We have a bijection between paths $\stt \in \Std_n(k,\beta)$ and admissible reduced tilings 
$\mathbb T\subseteq \mathbb{T}_{(k,\beta)}$. 
\end{prop}

\begin{proof}
Given a path $\stt \in \Std_n(k,\beta)$ we can tile the region lying strictly between 
$\stt \in \Std_n(k,\beta)$ and $\stt_{(k,\beta)}  \in \Std_n(k,\beta)$ to obtain an admissible reduced tiling. 
Conversely, given an admissible reduced tiling $\mathbb T$, we can associate a path $\stt$ by  
 drawing the south-east/south-westerly path that traces 
out the edge of $\mathbb T$.  
We invite the reader to verify that these maps are mutual inverses of one another.
\end{proof}

\begin{defn}\label{Soergeldegreee}
For each $\stt\in \Std_n(k,\beta)$, we denote by $\mathbb{T}_\stt$ the corresponding admissible reduced tiling. 
 We define the degree of a  tile $T$ in $\mathbb{T}_\stt$ to be 
 \begin{itemize}
\item $+1$ if $T$ is in the top row and has  a marked point labelled by $\alpha_i^{\pm1}$;
\item $+1$ $T$ is not in the top row and a single vertex of $T$ lies on a hyperplane;
\item $-2$ if a hyperplane bisects $T$;
\item $0$ otherwise.
 \end{itemize}
    We define the degree of the path $\SSTT\in\Std_n(k,\beta)$ to be the sum over the degrees of all tiles in $\mathbb{T}_\stt$.  See 
    \cref{thisistheexampleblow5hghghghgh,thisistheexampleblow5,newlabel1} for  examples.
    \end{defn}
Note that the tiles having a marked point labelled by $\extra$ are of degree $0$.

\begin{figure}[ht!]
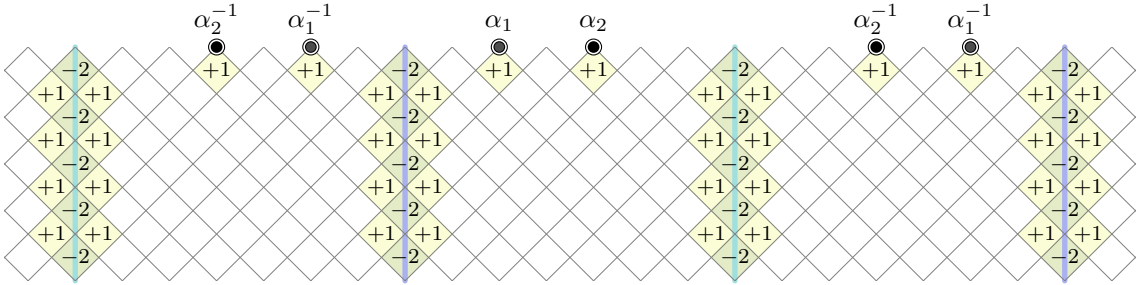

\[
\def\pE{14}
\def\pN{9}
\TIKZ[yscale=-1, scale=.31, font=\scriptsize]{
\pgfmathsetmacro{\pEParity}{Mod(\pE,2)}
\coordinate (beta1) at (4, -1);
	\coordinate (beta1inv1) at (-4, -1);
	\coordinate (beta1inv2) at (-4+2*\pE, -1);
\coordinate (beta2) at (8, -1);
	\coordinate (beta2inv1) at (-8, -1);
	\coordinate (beta2inv2) at (-8+2*\pE, -1);
\foreach \x in {beta1, beta1inv1, beta1inv2, beta2, beta2inv1, beta2inv2}
	{\filldraw[plusone] (\x) \TILE;
	\path (\x) to ++(0,1) node {$+1$};}
\foreach \x in {-1, ..., 2}{
	\foreach \y in {-1, 1,...,5}{
		\filldraw[minustwo] (\pE*\x, \y) \TILE;
		\foreach \side in {-1,1}{
		\filldraw[plusone] (\pE*\x+\side, \y+1) \TILE;}
	}
	\filldraw[minustwo] (\pE*\x, 7) \TILE;
	}
\draw[tauplane] (0,-1) to +(0,\pN+1);
\draw[tauplane] (2*\pE,-1) to +(0,\pN+1);
\draw[sigmaplane] (-\pE,-1) to +(0,\pN+1);
\draw[sigmaplane] (\pE,-1) to +(0,\pN+1);
\foreach \x in {-1, ..., 2}{
	\foreach \y in {-1, 1,...,5}{
		\node at (\pE*\x, \y+1) {$-2$};
		\foreach \side in {1,-1}{\node at (\pE*\x+\side, \y+2)  { $+1$};}
	}
	\node at (\pE*\x, 8) {$-2$};}
\begin{scope}
\clip (-17, -1) rectangle (31, 9);
	\foreach \x in {-27, ..., 41}{
		\draw[gridstyle] (2*\x,-1) to +(11,11);
		\draw[gridstyle] (2*\x,-1) to +(-11,11);
	}
\end{scope}
\foreach \x in {beta1, beta1inv1, beta1inv2}
	{\node[M1] at (\x){};}
\foreach \x in {beta2, beta2inv1, beta2inv2}
	{\node[M2] at (\x){};}
\node[above] at (beta1) {\small \strut $\alpha_1$};
	\node[above] at (beta2) {\small \strut $\alpha_2$};
\node[above] at (beta1inv1) {\small \strut $\alpha_1^{-1}$};
	\node[above] at (beta2inv1) {\small \strut $\alpha_2^{-1}$};
\node[above] at (beta1inv2) {\small \strut $\alpha_1^{-1}$};
	\node[above] at (beta2inv2) {\small \strut $\alpha_2^{-1}$};
}
\]
 
 \caption{The degrees of tiles for $e=14$ and $\alpha_1=q^4$ and $\alpha_2=q^8$.
 }
 \label{thisistheexampleblow5}
 \end{figure}

\begin{figure}[ht!]
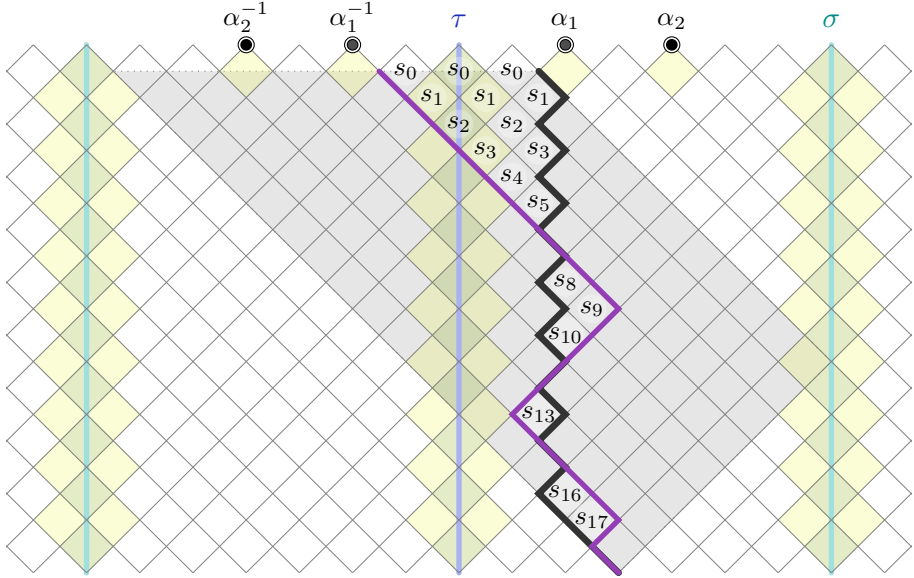

 \[
\def\pE{14}
\def\pN{19}
\def\pA{4}
\def\pB{17}
\TIKZ[yscale=-1, scale=.35, font=\scriptsize]{
\pgfmathsetmacro{\pEParity}{Mod(\pE,2)}
\pgfmathsetmacro{\pNParity}{Mod(\pN,2)}
\coordinate (max-start) at (\pA-1, 0);
\path (max-start) to ++(\pN-\pB+\pNParity, \pN) coordinate (end);
\path (end) to ++(-\pN, -\pN) coordinate (min-start);
\path (end) to ++(.5*\pB-.5*\pNParity, -.5*\pB+.5*\pNParity) coordinate (p-corner);
\draw[pathsregion] (max-start) to (min-start) to (end) to (p-corner) to (max-start); 
\coordinate (beta1) at (4, -1); \coordinate (beta1inv) at (-4, -1);
\coordinate (beta2) at (8, -1); \coordinate (beta2inv) at (-8, -1);
\foreach \x in {beta1, beta1inv, beta2, beta2inv}
	{\filldraw[plusone] (\x) \TILE;}
\foreach \x in {-1,0,1}{
	\foreach \y in {-1, 1,...,15}{
		\filldraw[minustwo] (\pE*\x, \y) \TILE;
		\foreach \side in {-1,1}{
		\filldraw[plusone] (\pE*\x+\side, \y+1) \TILE;}
	}
	\filldraw[minustwo] (\pE*\x, 17) \TILE;
	}
\draw[tauplane] (0,-1) to +(0,\pN+1);
\draw[sigmaplane] (-\pE,-1) to +(0,\pN+1);
\draw[sigmaplane] (\pE,-1) to +(0,\pN+1);
\begin{scope}
\clip (-17, -1) rectangle (17, \pN);
	\foreach \x in {-20, ..., 20}{
		\draw[gridstyle] (2*\x,-1) to +(\pN+1,\pN+1);
		\draw[gridstyle] (2*\x,-1) to +(-\pN-1,\pN+1);
	}
\end{scope}
\foreach \x in {beta1, beta1inv}
	{\node[M1] at (\x){};}
\foreach \x in {beta2, beta2inv}
	{\node[M2] at (\x){};}
\node[above] at (beta1) {\small \strut $\alpha_1$};
	\node[above] at (beta2) {\small \strut $\alpha_2$};
\node[above] at (beta1inv) {\small \strut $\alpha_1^{-1}$};
	\node[above] at (beta2inv) {\small \strut $\alpha_2^{-1}$};
\node[above, taucolor] at (0,-1) {\strut \normalsize $\tau$};
\node[above, sigmacolor] at (\pE, -1) {\strut \normalsize $\sigma$};
\begin{scope}[font= \small]
\foreach \y/\n in {0/3, 1/3, 2/2, 3/2, 4/1, 5/1}
	{\pgfmathsetmacro{\yParity}{Mod(\y,2)}
	\foreach \x in {1, ..., \n}{\node[tilemarkers] at (\pA-2*\x+\yParity, \y) {$s_{\y}$}; }}
\foreach \y/\x in {8/0, 9/1, 10/0, 13/-1, 16/0, 17/1}
	{\node[tilemarkers] at (\pA+\x, \y) {$s_{\y}$};}
\end{scope}
\draw[path1] (max-start) 
	\foreach \x in {1, ..., 8}{ to ++(1,1) to ++(-1,1)} to (end);
\draw[path2] (-3,0) to ++(9,9) to ++(-4,4) to ++(4,4) to ++(-1,1) to ++(1,1);
}
\]
 
 \caption{We depict two paths ${\color{path1color}\stt_{(3,\beta)}}$  and ${\color{path2color}\sts} \in \Std_{19}(3,\beta)$ with 
 $\beta= \alpha_1=q^4$ and $\alpha_2=q^8$ and $e=14$.  
 The  permutation $w_\stt$  can be read off of the diagram (see a particular word for this permutations below). The path  $\sts$ has degree 1; to see this note that there are 2 degree $-2$ tiles and 3 degree $+1$ tiles lying between $\sts$ and  $\stt_{(3,\beta)}$.  
 }
 \label{newlabel1}
 \end{figure}

\begin{defn}
Let $\mathbb T$ be an admissible reduced  tiling containing $\ell$ tiles , 
we define a {\sf tiling tableau} to be a map $\tau :\mathbb T \to \{1,2,\dots,\ell\}$ such that 
$\tau (T) < \tau (T')$ for any pair of tiles $T,T'$ satisfying $T<_LT'$ or $T<_RT'$. 
\end{defn}
 
 The $y$-coordinate of the top vertex of a  tile $T \in \mathbb T_{(k,\beta)}$  is  equal to some $y\in \ZZ_{\geq 0}$ by construction which we call the height of the tile $T$; 
 we define the content of the tile $T$ to be the corresponding reflection $s_y \in W(C_n)$.  
Let $\mathbb T$ be an admissible tiling containing $\ell$ tiles.
Given a tiling tableau $\tau: \mathbb T \to \{1,\dots,\ell\}$  we define an associated word 
$w_\tau$ to be the ordered product of the $s_y$ for $y \in \ZZ_{\geq0}$ given by reading the contents of the tiles of 
$ \mathbb T$ in the order specified by $\tau$.

\begin{prop}\label{useful!!!}
 Given $\SSTT \in \Std_n(k,\beta) $
  we have that the reduced expressions of $w_\SSTT$ are given by the  
words $w_\tau$ coming from the set of all tiling tableaux of $\mathbb{T}_\stt$. 
\end{prop}

\begin{proof}
This follows by inductive application of \cref{inductive1}.
 \end{proof}  
 
 We define the length of the tableau $\stt\in \Std_n(k,\beta)$ by $\ell(\stt) =\ell(w_\stt) = \ell(\mathbb{T}_\stt)$. 
   It will be useful for us to choose one specific tiling tableau $\tau (\stt)$ for the admissible reduced tiling $\mathbb{T}_\stt$ corresponding to each $\stt\in \Std_n(k,\beta)$. This is equivalent to choosing one specific reduced expression for $w_\stt$.

\begin{defn} Let $\stt\in \Std_n(k,\beta)$. 
We define $\tau(\stt)$ denote the  tiling tableau in which we first fill $\mathbb L_{(k,\beta)}\cap \mathbb T_\stt$ by successively  adding the tile of minimal  content at each step
and we then fill $\mathbb R_{(k,\beta)}\cap \mathbb T_\stt$ by 
successively  adding the tile of maximal  content at each step. See \cref{newlabel1} for an example. 
\end{defn}

\begin{figure}[ht!]
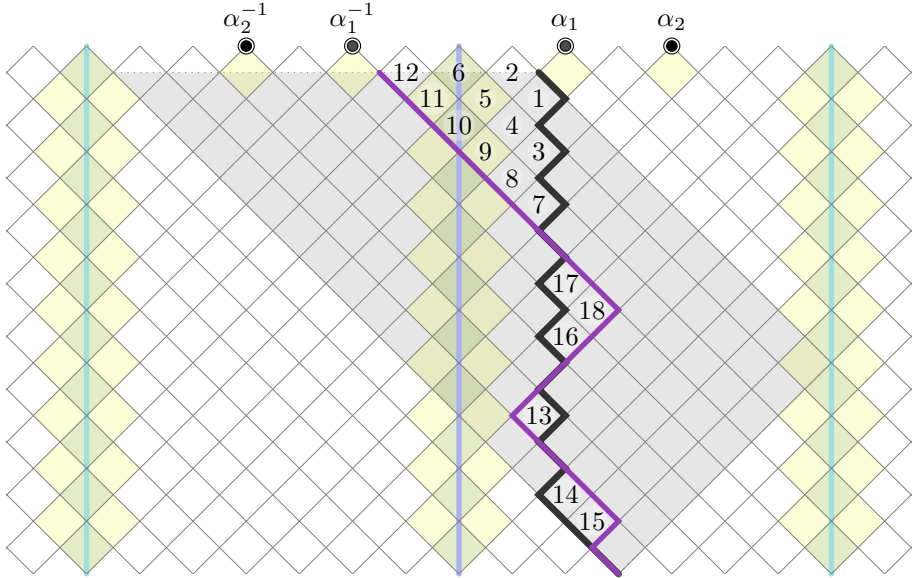


\[
\def\pE{14}
\def\pN{19}
\def\pA{4}
\def\pB{17}
\TIKZ[yscale=-1, scale=.35, font=\scriptsize]{
\pgfmathsetmacro{\pEParity}{Mod(\pE,2)}
\pgfmathsetmacro{\pNParity}{Mod(\pN,2)}
\coordinate (max-start) at (\pA-1, 0);
\path (max-start) to ++(\pN-\pB+\pNParity, \pN) coordinate (end);
\path (end) to ++(-\pN, -\pN) coordinate (min-start);
\path (end) to ++(.5*\pB-.5*\pNParity, -.5*\pB+.5*\pNParity) coordinate (p-corner);
\draw[pathsregion] (max-start) to (min-start) to (end) to (p-corner) to (max-start); 
\coordinate (beta1) at (4, -1); \coordinate (beta1inv) at (-4, -1);
\coordinate (beta2) at (8, -1); \coordinate (beta2inv) at (-8, -1);
\foreach \x in {beta1, beta1inv, beta2, beta2inv}
	{\filldraw[plusone] (\x) \TILE;}
\foreach \x in {-1,0,1}{
	\foreach \y in {-1, 1,...,15}{
		\filldraw[minustwo] (\pE*\x, \y) \TILE;
		\foreach \side in {-1,1}{
		\filldraw[plusone] (\pE*\x+\side, \y+1) \TILE;}
	}
	\filldraw[minustwo] (\pE*\x, 17) \TILE;
	}
\draw[tauplane] (0,-1) to +(0,\pN+1);
\draw[sigmaplane] (-\pE,-1) to +(0,\pN+1);
\draw[sigmaplane] (\pE,-1) to +(0,\pN+1);
\begin{scope}
\clip (-17, -1) rectangle (17, \pN);
	\foreach \x in {-20, ..., 20}{
		\draw[gridstyle] (2*\x,-1) to +(\pN+1,\pN+1);
		\draw[gridstyle] (2*\x,-1) to +(-\pN-1,\pN+1);
	}
\end{scope}
\foreach \x in {beta1, beta1inv}
	{\node[M1] at (\x){};}
\foreach \x in {beta2, beta2inv}
	{\node[M2] at (\x){};}
\node[above] at (beta1) {\small \strut $\alpha_1$};
	\node[above] at (beta2) {\small \strut $\alpha_2$};
\node[above] at (beta1inv) {\small \strut $\alpha_1^{-1}$};
	\node[above] at (beta2inv) {\small \strut $\alpha_2^{-1}$};
\begin{scope}[font= \small]
\foreach \x/\y [count=\c from 1] in {-1/1, -2/0, -1/3, -2/2, -3/1, -4/0, -1/5, -2/4, -3/3, -4/2, -5/1, -6/0, -1/13, 0/16, 1/17, 0/10, 0/8, 1/9}
	{\node[tilemarkers] at (\pA+\x, \y) {$\c$}; }
\end{scope}
\draw[path1] (max-start) 
	\foreach \x in {1, ..., 8}{ to ++(1,1) to ++(-1,1)} to (end);
\draw[path2] (-3,0) to ++(9,9) to ++(-4,4) to ++(4,4) to ++(-1,1) to ++(1,1);
}
\]
 
 \caption{ We depict the tiling tableau  $\tau(\sts)$ for 
   $\sts \in \Std_{19}{(8,\beta)}$   with 
 $\beta= \alpha_1=q^4$ and $\alpha_2=q^8$.  
 }
 \end{figure}

\begin{eg}
 For    ${\color{violet} \sts }\in \Std_{19}(8,\beta)$ as in \cref{newlabel1} the reduced word for 
 $w_{\color{violet}\sts}$ corresponding to the   tiling tableau  $\tau(\sts)$  is as follows:
$$w_{\color{violet}\sts}=  (s_9s_8) (s_{10})(s_{17}s_{16})s_{13}(s_0s_1s_2s_3s_4 s_5)(s_0s_1s_2s_3)(s_0 s_1).
 $$
  
\end{eg}

 Given $\SSTT \in \Std_n(b,\beta) $ and some reduced expression
  $w_\stt = s_{i_1}s_{i_2}\dots s_{i_\ell}$ we define 
  $\psi_\stt:= \psi_{i_1}\psi_{i_2}\dots \psi_{i_\ell}\in {\rm TL}_n(\alpha_1, \alpha_2, \extra)$.  We note that $\psi_\stt$ is well-defined (independently of this choice of reduced expression) by the commuting relations and \cref{useful!!!}.

\begin{prop}
 Given $\stt \in \Std_n(k,\beta) $ 
 we have that 
 $\deg(\psi_\stt e_{\res(\stt_{(k,\beta)})})=\deg(\stt)$. 
\end{prop}

\begin{proof}
Recall that the residue sequence of a tableau $\stt\in \Std_n(k,\beta)$ can be read from the steps of the associated path on $\mathcal{L}_\beta$. Each tile  in $\mathbb T_\stt$ of height $j \in \ZZ_{\geq0}$ corresponds to a generator $\psi_j$ in the 
element $\psi_\stt e_{\res(\stt_{(k,\beta)})}$.   
The tiles of height $j\in \ZZ_{>0}$ bisected by hyperplanes correspond to 
$\psi_je(\ti)$ with $i_j=i_{j+1}$ and have degree $-2$ as required.  
The tiles of height $0$ bisected by hyperplanes correspond to 
$\psi_0e(\ti)$, with $i_1\in\{\pm1\}$  and have degree $-2$ as required.  
The tiles of height $k\in \ZZ_{>0}$  touching a  hyperplane  correspond to $\psi_ke(\ti)$ with $i_{k+1}=i_kq^{\pm2}$ and have degree $+1$ as required.  
The tiles of height $0$  touching a marked point different from $\extra$ correspond to $\psi_0e(\ti)$ with $i_1=\alpha_i^{\pm1}$ and have degree $+1$ as required. All other tiles correspond to generators of degree 0, and have degree $0$ as required.
\end{proof}

 \subsection{Ladder tableaux and idempotent ideals} From now on, we will almost always write $e_\stt$ instead of $e_{\res(\stt)}$ for standard tableaux $\stt$.
 For $0\leq m \leq n$ we set 
 $$\textstyle 
 e_{\leq m}= 
\sum_{0\leq k \leq m}
\sum _{(k,\beta)\in \Lambda_n} e_{\stt_{(k,\beta)}} 
\qquad 
e_{< m}= 
\sum_{0\leq k < m}
\sum _{(k,\beta)\in \Lambda_n} e_{\stt_{(k,\beta)}} 
$$and we define 
\begin{align*}\textstyle
{\rm TL}_n(\alpha_1,\alpha_2, \extra ) ^{\leq m}
&=
{\rm TL}_n(\alpha_1,\alpha_2, \extra )  
e_{\leq m}
{\rm TL}_n(\alpha_1,\alpha_2, \extra )
\\
{\rm TL}_n(\alpha_1,\alpha_2, \extra ) ^{< m}
&=
{\rm TL}_n(\alpha_1,\alpha_2, \extra )  
e_{< m}
{\rm TL}_n(\alpha_1,\alpha_2, \extra ). 
\end{align*}

\begin{defn}

We say that a tableau $\sts \in \Std_n(k,\beta) $  is a {\sf ladder tableau}  if ${\sf max}(\sts) = (k,\beta)$ and 
$$x_\sts(n)-x_\sts(0) = \max\{ x_\stt (n) - x_\stt(0) \, : \, \stt\in \Std_n \, \,\text{with} \,\, \res(\stt) = \res(\sts)\}.$$ 
 
\end{defn}

\begin{figure}[ht!]
\[
\def\pN{33} 
\def\pB{15} 
\def\pA{3} 
\def\pE{7} 
\def\pWidth{25} 
\TIKZ[yscale=-1, scale=.25]{
\pgfmathsetmacro{\pC}{ceil(.5*\pN - .5*\pB)}
\pgfmathsetmacro{\pK}{floor(\pWidth/\pE)}
\pgfmathsetmacro{\pShift}{ceil(\pK/2)}
\pgfmathsetmacro{\pEParity}{Mod(\pE,2)}
\pgfmathsetmacro{\pNParity}{Mod(\pN,2)}
\pgfmathsetmacro{\pLeft}{ceil((\pA - \pB-1)/\pE)-1}
\coordinate (max-start) at (\pA, 0);
\path (max-start) to ++(\pN-\pB+\pNParity, \pN) coordinate (end);
\path (end) to ++(-\pN, -\pN) coordinate (min-start);
\path (end) to ++(.5*\pB-.5*\pNParity, -.5*\pB+.5*\pNParity) coordinate (p-corner);
\filldraw [pathsregion] (max-start) to (min-start) to (end) to (p-corner); 
\begin{scope}
\clip (\pA - \pB,-1.5) rectangle (\pA + \pN -.5*\pB + .5*\pNParity+1,\pN);
\foreach \x in {0,..., \pK}{
\coordinate (beta-marker\x) at (-2*\pShift*\pE + \pA+1  + \x*2*\pE, -1);
\coordinate (betainv-marker\x) at (2*\pShift*\pE- \pA -1 - \x*2*\pE, -1);
	\filldraw[plusone] (beta-marker\x) \TILE;
	\filldraw[plusone] (betainv-marker\x) \TILE;
}
\foreach \x in {-\pK, ..., \pK}{
	\foreach \y in {-1, 1,...,\pN}{
		\filldraw[minustwo] (2*\pE*\x, \y) \TILE;
		\filldraw[plusone] 
			(2*\pE*\x+1, \y+1) \TILE;
		\filldraw[plusone] 
			(2*\pE*\x-1, \y+1) \TILE;
		\filldraw[minustwo] (2*\pE*\x+\pE, \y+\pEParity) \TILE;
		\filldraw[plusone] 
			(2*\pE*\x+\pE+1, \y+\pEParity+1) \TILE;
		\filldraw[plusone] 
			(2*\pE*\x+\pE-1, \y+\pEParity+1) \TILE;
	}}
\foreach \x in {-\pK, ..., \pK}{
		\draw[tauplane] (\x*2*\pE,-1) to +(0,\pN+1);
		\draw[sigmaplane] (\x*2*\pE+\pE,-1) to +(0,\pN+1);
		}
\pgfmathsetmacro{\pGrids}{max(\pN, \pWidth)}
\foreach \x in {-\pGrids, ..., \pGrids}{
	\draw[gridstyle] (2*\x,-1) to +(-\pN-2,\pN+2);
	\draw[gridstyle] (2*\x,-1) to +(\pN+2,\pN+2);
}
\foreach \x in {0,..., \pK}{
	\node[M] at (beta-marker\x){};
	\node[M] at (betainv-marker\x){};
}
\end{scope}
\node[above] at (\pA+1, -1) {\small \strut $\alpha_1$};
\node[above] at (-\pA-1, -1) {\small \strut $\alpha_1^{-1}$};
\draw[path1] (max-start) 
\foreach \x in {1, ..., 7}{ to ++(1,1) to ++(-1,1)} to (end);
\draw[path2] (-\pA, 0) 
	to ++(-1,1) 
	to ++(4,4)
	to ++(-1,1) 
	to ++(13,13) 
	to ++(-1,1) 
	to ++(11,11)
	to ++(-1,1)
	to ++(1,1);
 }
\]
\caption{Here $e=7$ and $\alpha_1= q^4$ with $\alpha_2 \not \in \alpha_1q^{2\ZZ}$. 
We depict two    ladder tableaux   $ {\color{path1color}\stt_{(15,\alpha_1)}}, {\color{path2color}\sts} \in \Std_{33}(15,\alpha_1)\in \Lambda_{33}$. 
} 
\label{heresaladder}
\end{figure}

\begin{eg}
Let $\beta= \alpha_1=q^4$ and $\alpha_2=q^8$. 
Both  $\stt_{(k,\beta)}$  and $\sts$ depicted in 
  \cref{newlabel1} are examples of ladder tableaux of  shape $(3,\beta) \in \Lambda_{19}$. In fact it is easy to see that $\stt_{(k,\beta)}$ is a ladder tableau for all $(k,\beta)\in \Lambda_n$. Two more examples are depicted in \cref{heresaladder}. The tableaux given in \cref{fignotladder} are not ladder tableaux. This can be seen by observing that ${\sf max}(s) \neq (15,\alpha_1)$ and $x_\stt(33)-x_\stt(0)$ is not maximal in its residue class. 
  Note that ladder tableaux are not unique in their residue classes. See \cref{figladdernotunique} for eight different ladder tableaux in the same residue class.  
\end{eg}

\begin{figure}[ht!]
\[
\def\pN{33} 
\def\pB{15} 
\def\pA{3} 
\def\pE{7} 
\def\pWidth{25} 
\TIKZ[yscale=-1, scale=.25]{
\pgfmathsetmacro{\pC}{ceil(.5*\pN - .5*\pB)}
\pgfmathsetmacro{\pK}{floor(\pWidth/\pE)}
\pgfmathsetmacro{\pShift}{ceil(\pK/2)}
\pgfmathsetmacro{\pEParity}{Mod(\pE,2)}
\pgfmathsetmacro{\pNParity}{Mod(\pN,2)}
\pgfmathsetmacro{\pLeft}{ceil((\pA - \pB-1)/\pE)-1}
\coordinate (max-start) at (\pA, 0);
\path (max-start) to ++(\pN-\pB+\pNParity, \pN) coordinate (end);
\path (end) to ++(-\pN, -\pN) coordinate (min-start);
\path (end) to ++(.5*\pB-.5*\pNParity, -.5*\pB+.5*\pNParity) coordinate (p-corner);
\filldraw [pathsregion] (max-start) to (min-start) to (end) to (p-corner); 
\begin{scope}
\clip (\pA - \pB,-1.5) rectangle (\pA + \pN -.5*\pB + .5*\pNParity+1,\pN);
\foreach \x in {0,..., \pK}{
\coordinate (beta-marker\x) at (-2*\pShift*\pE + \pA+1  + \x*2*\pE, -1);
\coordinate (betainv-marker\x) at (2*\pShift*\pE- \pA -1 - \x*2*\pE, -1);
	\filldraw[plusone] (beta-marker\x) \TILE;
	\filldraw[plusone] (betainv-marker\x) \TILE;
}
\foreach \x in {-\pK, ..., \pK}{
	\foreach \y in {-1, 1,...,\pN}{
		\filldraw[minustwo] (2*\pE*\x, \y) \TILE;
		\filldraw[plusone] 
			(2*\pE*\x+1, \y+1) \TILE;
		\filldraw[plusone] 
			(2*\pE*\x-1, \y+1) \TILE;
		\filldraw[minustwo] (2*\pE*\x+\pE, \y+\pEParity) \TILE;
		\filldraw[plusone] 
			(2*\pE*\x+\pE+1, \y+\pEParity+1) \TILE;
		\filldraw[plusone] 
			(2*\pE*\x+\pE-1, \y+\pEParity+1) \TILE;
	}}
\foreach \x in {-\pK, ..., \pK}{
		\draw[tauplane] (\x*2*\pE,-1) to +(0,\pN+1);
		\draw[sigmaplane] (\x*2*\pE+\pE,-1) to +(0,\pN+1);
		}
\pgfmathsetmacro{\pGrids}{max(\pN, \pWidth)}
\foreach \x in {-\pGrids, ..., \pGrids}{
	\draw[gridstyle] (2*\x,-1) to +(-\pN-2,\pN+2);
	\draw[gridstyle] (2*\x,-1) to +(\pN+2,\pN+2);
}
\foreach \x in {0,..., \pK}{
	\node[M] at (beta-marker\x){};
	\node[M] at (betainv-marker\x){};
}
\end{scope}
\node[above] at (\pA+1, -1) {\small \strut $\alpha_1$};
\node[above] at (-\pA-1, -1) {\small \strut $\alpha_1^{-1}$};
%
\draw[path1,cyan, line width =3] (3, 0) --(7,4)--(0,11)--(end); 
\draw[path2] (3, 0) to (0,3)--(8,11)--(7,12)--(25,30)--(end); 
\draw[path1,line width=1.5] (-5, 0) to (-6,1)--(0,7)--(-1,8)--(15+7,24+7)--(14+7,25+7)--(end);
 }
\]
\caption{Here $e=7$ and $\alpha_1= q^4$ with $\alpha_2 \not \in \alpha_1q^{2\ZZ}$. 
Here we depict examples of tableaux $\color{black}\sts$, $\color{violet}\stt$,  
$\color{cyan}\stu$    which are {\em not} ladder tableaux. 
} \label{fignotladder}
\end{figure}

%

\begin{figure}[ht!]
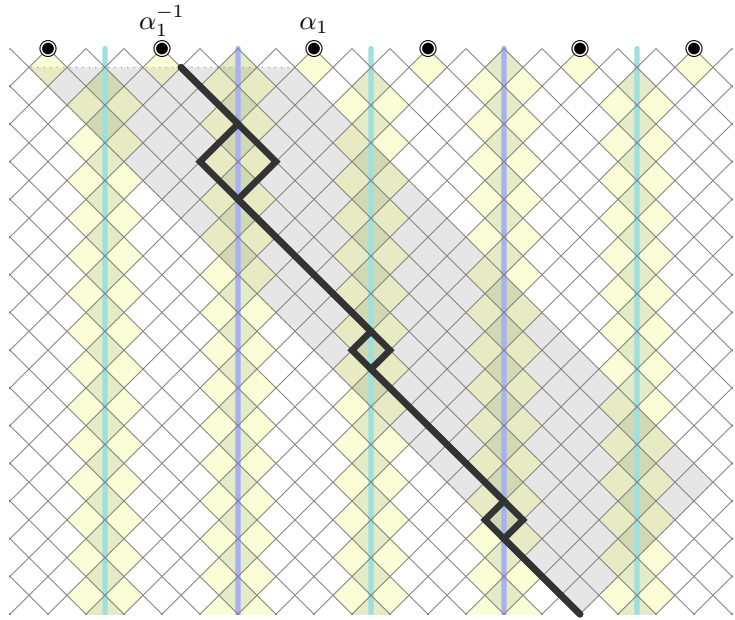

\[
\def\pN{29} 
\def\pB{15} 
\def\pA{3} 
\def\pE{7} 
\def\pWidth{25} 
\TIKZ[yscale=-1, scale=.25]{
\pgfmathsetmacro{\pC}{ceil(.5*\pN - .5*\pB)}
\pgfmathsetmacro{\pK}{floor(\pWidth/\pE)}
\pgfmathsetmacro{\pShift}{ceil(\pK/2)}
\pgfmathsetmacro{\pEParity}{Mod(\pE,2)}
\pgfmathsetmacro{\pNParity}{Mod(\pN,2)}
\pgfmathsetmacro{\pLeft}{ceil((\pA - \pB-1)/\pE)-1}
\coordinate (max-start) at (\pA, 0);
\path (max-start) to ++(\pN-\pB+\pNParity, \pN) coordinate (end);
\path (end) to ++(-\pN, -\pN) coordinate (min-start);
\path (end) to ++(.5*\pB-.5*\pNParity, -.5*\pB+.5*\pNParity) coordinate (p-corner);
\filldraw [pathsregion] (max-start) to (min-start) to (end) to (p-corner); 
\begin{scope}
\clip (\pA - \pB,-1.5) rectangle (\pA + \pN -.5*\pB + .5*\pNParity+1,\pN);
\foreach \x in {0,..., \pK}{
\coordinate (beta-marker\x) at (-2*\pShift*\pE + \pA+1  + \x*2*\pE, -1);
\coordinate (betainv-marker\x) at (2*\pShift*\pE- \pA -1 - \x*2*\pE, -1);
	\filldraw[plusone] (beta-marker\x) \TILE;
	\filldraw[plusone] (betainv-marker\x) \TILE;
}
\foreach \x in {-\pK, ..., \pK}{
	\foreach \y in {-1, 1,...,\pN}{
		\filldraw[minustwo] (2*\pE*\x, \y) \TILE;
		\filldraw[plusone] 
			(2*\pE*\x+1, \y+1) \TILE;
		\filldraw[plusone] 
			(2*\pE*\x-1, \y+1) \TILE;
		\filldraw[minustwo] (2*\pE*\x+\pE, \y+\pEParity) \TILE;
		\filldraw[plusone] 
			(2*\pE*\x+\pE+1, \y+\pEParity+1) \TILE;
		\filldraw[plusone] 
			(2*\pE*\x+\pE-1, \y+\pEParity+1) \TILE;
	}}
\foreach \x in {-\pK, ..., \pK}{
		\draw[tauplane] (\x*2*\pE,-1) to +(0,\pN+1);
		\draw[sigmaplane] (\x*2*\pE+\pE,-1) to +(0,\pN+1);
		}
\pgfmathsetmacro{\pGrids}{max(\pN, \pWidth)}
\foreach \x in {-\pGrids, ..., \pGrids}{
	\draw[gridstyle] (2*\x,-1) to +(-\pN-2,\pN+2);
	\draw[gridstyle] (2*\x,-1) to +(\pN+2,\pN+2);
}
\foreach \x in {0,..., \pK}{
	\node[M] at (beta-marker\x){};
	\node[M] at (betainv-marker\x){};
}
\end{scope}
\node[above] at (\pA+1, -1) {\small \strut $\alpha_1$};
\node[above] at (-\pA-1, -1) {\small \strut $\alpha_1^{-1}$};
%
\draw[path1,] (-3, 0) -- (0,3)
--(2,5)--(0,7)
--(7,14)
--(8,15)--(7,16)
--(14,23)
--(15,24)
--(14,25)
--(end)
; 
\draw[path1,] (-3, 0) -- (0,3)
--(-2,5)--(0,7)
--(7,14)
--(6,15)--(7,16)
--(14,23)
--(13,24)
--(14,25)
--(end)
; 
 }
\]
\caption{Here $e=7$ and $\alpha_1= q^4$ with $\alpha_2 \not \in \alpha_1q^{2\ZZ}$. 
We depict eight different ladder tableaux in the same residue class.}\label{figladdernotunique}
\end{figure}

\begin{prop}\label{ladderssss}
For a ladder tableau $\stt \in \Std_n (k,\beta)$ we have that
 $e_\stt \in {\rm TL}_n(\alpha_1,\alpha_2, \extra )^{\leq k}$.

\end{prop}

\begin{proof}
We will choose the unique  ladder tableau $\sts\in \Std_n(k,\beta)$ in the residue class  $\stt$ for which $\mathbb T_\sts$ is maximal. In \cref{figladdernotunique}, this is the tableau which takes the left side before crossing $\stt_{(k,\beta)}$ and the right sides afterwards.
We proceed by induction on $\ell(\mathbb T_\sts)$ with the $\ell(\mathbb T_\sts)=0$ case being trivial 
(since $\sts=\stt_{(k,\beta)}$ in this case).   
By definition of our ladder tableaux  we have that $\square \in \mathbb T_\sts$
implies that $\square$ does not touch a marked point labelled by $\alpha_i^{\pm1}$, since it would contradict the 
maximality of ${\sf max}(\sts)\in \Lambda_n$. 
We consider the tile $\square $  containing the maximal entry, $\ell$, of the tiling tableau $\tau(\sts)$.
The tile $\square $ has height $h\in \ZZ_{\geq0}$ say, and   has degree 
$-2$, $0$, or $+1$ and these form the cases of the proof.

\smallskip\noindent\textbf{Case 1:} If $\square $ is of degree zero, 
then $  \mathbb T_\sts =   \mathbb T_\sts' \cup  \square$ with $\sts' \in \Std_n(k,\beta)$ a ladder tableau.
By construction $\psi_h^2e_\sts=e_\sts$  since the tile is of degree 0 and therefore
 $$
 e_\sts  =  \psi_h^2e_\sts=\psi_h e_{\sts'}  \psi_h \in {\rm TL}_n(\alpha_1,\alpha_2, \extra )^{\leq k}
 $$by induction.

 \smallskip\noindent
 \textbf{Case 2:} If $\square $ is a tile of degree $-2$, then 
  $\square  $ lies on a hyperplane. We first assume that $h>0$. In this case the $h$th and $(h+1)$th entries of $e_\sts$ are equal.  
In which case we let $\square'$ be such that $\tau_{\sts}(\square')=(\ell-1)$ 
and we let $\stu $ be such  that 
$\mathbb T_{\sts} =\mathbb T_{\stu} \cup \square'\cup \square$. 
If  $\square \in \mathbb L_{(k,\beta)}\cap \mathbb T_\sts$ (respectively  
$\square \in \mathbb R_{(k,\beta)}\cap \mathbb T_\sts$)
then $\square'$ has height $h+1\in \ZZ_{\geq0}$ (respectively  $h-1\in \ZZ_{\geq0}$).
We consider the former case as the latter is identical. An example is the tile $\square$ of height $6$ in \cref{heresaladder} on the hyperplane in an inner corner of the path ${\color{violet}\sts}$. The tile $\square'$ is SE of the tile $\square$.
  We have that 
  $\res_a(\sts) = \res_a(\stu)$  for $a\neq h,h+1,h+2 $ 
  and 
  $$(\res_h(\sts), \res_{h+1}(\sts), \res_{h+2}(\sts)) = (i,i,iq^2), \quad (\res_h(\stu), \res_{h+1}(\stu), \res_{h+2}(\stu)) = (i,iq^2, i).$$ This is illustrated in \cref{figcase2}.

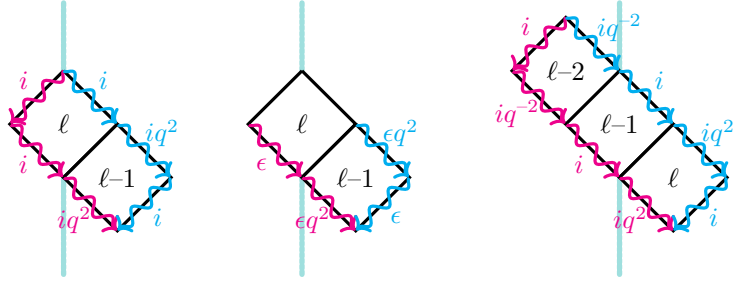
\begin{figure}[ht!]
$$
\begin{tikzpicture}[scale=0.5]\scalefont{0.9}


\draw[very thick,densely dotted,sigmaplane] (0,1.8)--(0,-2.7*2);

\draw[very thick,fill=white](0,0)--++(-45:2) coordinate (X)--++(-135:2)--++(135:2)--++(45:2); 

\path(0,0) --++(-45:1) --++(-135:1) node {$\ell$};
\draw[very thick,fill=white](X)--++(-45:2)  --++(-135:2)--++(135:2)--++(45:2); 
\path(X) --++(-45:1) --++(-135:1) node {$\ell\text{--}1$};

\draw[very thick,magenta,snake it,->](0,0)--++(-135:1.975) coordinate [midway] (here);
\path[magenta](here) --++(135:0.5)--++(90:0.1) node {$i$};

\path(-135:2)--++(-45:0.025) coordinate (X);

\draw[very thick,magenta,snake it,->](X)--++(-45:1.95) coordinate [midway] (here);
\path[magenta](here) --++(-135:0.5)  node {$i$};

\path(-135:2)--++(-45:2.0125) coordinate (X);

\draw[very thick,magenta,snake it,->](X)--++(-45:1.975) coordinate [midway] (here);
\path[magenta](here) --++(-135:0.6)  node {$iq^{2}$};

\draw[very thick,cyan,snake it,->](0,0)--++(-45:1.975) coordinate [midway] (here);
\path[cyan](here) --++(45:0.5)--++(90:0.1) node {$i$};

\path(-45:2)--++(-45:0.025) coordinate (X);

\draw[very thick,cyan,snake it,->](X)--++(-45:1.95) coordinate [midway] (here);
\path[cyan](here) --++(45:0.67)  node {$iq^{2}$};

\path(-45:4)--++(-135:.0125) coordinate (X);

\draw[very thick,cyan,snake it,->](X)--++(-135:1.975) coordinate [midway] (here);
\path[cyan](here) --++(-45:0.5)  node {$ i$};

\end{tikzpicture}
\qquad
\begin{tikzpicture}[scale=0.5]\scalefont{0.9}


\draw[very thick,densely dotted,sigmaplane] (0,1.8)--(0,-2.7*2);

\draw[very thick,fill=white](0,0)--++(-45:2) coordinate (X)--++(-135:2)--++(135:2)--++(45:2); 

\path(0,0) --++(-45:1) --++(-135:1) node {$\ell$};
\draw[very thick,fill=white](X)--++(-45:2)  --++(-135:2)--++(135:2)--++(45:2); 
\path(X) --++(-45:1) --++(-135:1) node {$\ell\text{--}1$};

%

\path(-135:2)--++(-45:0.025) coordinate (X);

\draw[very thick,magenta,snake it,->](X)--++(-45:1.95) coordinate [midway] (here);
\path[magenta](here) --++(-135:0.5)  node {$\epsilon$};

\path(-135:2)--++(-45:2.0125) coordinate (X);

\draw[very thick,magenta,snake it,->](X)--++(-45:1.975) coordinate [midway] (here);
\path[magenta](here) --++(-135:0.6)  node {$\epsilon q^{2}$};

%

\path(-45:2)--++(-45:0.025) coordinate (X);

\draw[very thick,cyan,snake it,->](X)--++(-45:1.95) coordinate [midway] (here);
\path[cyan](here) --++(45:0.67)  node {$\epsilon q^{2}$};

\path(-45:4)--++(-135:.0125) coordinate (X);

\draw[very thick,cyan,snake it,->](X)--++(-135:1.975) coordinate [midway] (here);
\path[cyan](here) --++(-45:0.5)  node {$ \epsilon$};

\end{tikzpicture}
\qquad
\begin{tikzpicture}[scale=0.5]\scalefont{0.9}


\draw[very thick,densely dotted,sigmaplane] (0,1.8)--(0,-2.7*2);

\draw[very thick,fill=white](0,0)--++(135:2)  --++(-135:2)--++(-45:2);

\path(0,0) --++(-135:1) --++(135:1) node {$\ell\text{--}2$};

\draw[very thick,fill=white](0,0)--++(-45:2) coordinate (X)--++(-135:2)--++(135:2)--++(45:2); 

\path(0,0) --++(-45:1) --++(-135:1) node {$\ell\text{--}1$};
\draw[very thick,fill=white](X)--++(-45:2)  --++(-135:2)--++(135:2)--++(45:2); 
\path(X) --++(-45:1) --++(-135:1) node {$\ell$};

\draw[very thick,magenta,snake it,->](135:2)--++(-135:1.975) coordinate [midway] (here);
\path[magenta](here) --++(135:0.5)--++(90:0.1) node {$i$};

\path(-135:2)--++(-45:0.025)--++(135:2) coordinate (X);

\draw[very thick,magenta,snake it,->](X)--++(-45:1.95) coordinate [midway] (here);
\path[magenta](here) --++(-135:0.6)--++(180:0.15)  node {$iq^{\text{--}2}$};

\path(-135:2)--++(-45:0.025) coordinate (X);

\draw[very thick,magenta,snake it,->](X)--++(-45:1.95) coordinate [midway] (here);
\path[magenta](here) --++(-135:0.5)  node {$i$};

\path(-135:2)--++(-45:2.0125) coordinate (X);

\draw[very thick,magenta,snake it,->](X)--++(-45:1.975) coordinate [midway] (here);
\path[magenta](here) --++(-135:0.6)  node {$iq^{2}$};

\draw[very thick,cyan,snake it,->](135:2)--++(-45:1.975) coordinate [midway] (here);
\path[cyan](here) --++(45:0.6)--++(0:0.29) node {$iq^{\text{--}2}$};

\draw[very thick,cyan,snake it,->](0,0)--++(-45:1.975) coordinate [midway] (here);
\path[cyan](here) --++(45:0.5)--++(90:0.1) node {$i$};

\path(-45:2)--++(-45:0.025) coordinate (X);

\draw[very thick,cyan,snake it,->](X)--++(-45:1.95) coordinate [midway] (here);
\path[cyan](here) --++(45:0.67)  node {$iq^{2}$};

\path(-45:4)--++(-135:.0125) coordinate (X);

\draw[very thick,cyan,snake it,->](X)--++(-135:1.975) coordinate [midway] (here);
\path[cyan](here) --++(-45:0.5)  node {$ i$};

\end{tikzpicture}$$
\caption{The cases from the proof of \cref{ladderssss}.} 
\label{figcase2}
\end{figure}

%

 We will show that 
$$e_\sts \in  {\rm TL}_n(\alpha_1,\alpha_2, \extra ) e_\stu  {\rm TL}_n(\alpha_1,\alpha_2, \extra )$$
then use induction to deduce the result. 
First observe that
 \begin{align*}
  (y_{h+1}\psi_{h}y_{h+1}\psi_h - \psi_{h}y_{h+1}\psi_h y_h)e_\sts
  &= (y_{h+1}(\psi_{h}y_{h+1})\psi_h - \psi_{h}(y_{h+1}\psi_h) y_h)e_\sts \\
  &= (y_{h+1}(y_h\psi_h +1)\psi_h - \psi_{h}(\psi_hy_h+1) y_h)e_\sts \\
  &= (y_{h+1}\psi_h - \psi_{h} y_h)e_\sts \\
  &= e_\sts
 \end{align*}
 where the second and the fourth equality follows from (4.4) , and the third equality follows from (4.5), noting that $\res_h(\sts) = \res_{h+1}(\sts)$.
 Thus it will suffice to show that 
  $$\psi_hy_{h+1}\psi_h e_\sts
  \in  {\rm TL}_n(\alpha_1,\alpha_2, \extra )^{\leq k}.
  $$
Using (4.5) and noting that $\res_{h+1}(\sts)=i$ and $\res_{h+2}(\sts) = iq^2$ we have
$$y_{h+1}e_\sts = y_{h+2}e_\sts -\psi^2_{h+1}e_\sts.$$
So we get
  \begin{align*}
  \psi_hy_{h+1}\psi_h e_\sts &= (\psi_h y_{h+2}\psi_h - \psi_h \psi_{h+1}^2 \psi_h)e_\sts \\
  &= (\psi_h^2y_{h+2} - \psi_h \psi_{h+1}^2 \psi_{h})e_\sts \\
  &= - \psi_h \psi_{h+1}^2 \psi_h e_\sts\\
  &= - \psi_h \psi_{h+1}e_\stu \psi_{h+1}\psi_h
   \in {\rm TL}_n(\alpha_1,\alpha_2, \extra ) e_\stu  {\rm TL}_n(\alpha_1,\alpha_2, \extra )\subseteq 
 {\rm TL}_n(\alpha_1,\alpha_2, \extra )^{\leq k}
  \end{align*}
    where the second equality follows from the commutation relations (4.2), and the third from (4.5) (noting that $\res_h(\sts) = \res_{h+1}(\sts)$) and the final line follows from (4.3) and induction. 

 We now consider the subcase with  $h=0$. We continue with our assumption 
that  $\square $ is an $s_0$-tile of degree $-2$ lying 
  on a hyperplane. It means that the first entry of  $e_\sts$  is equal to $\pm 1$ in which case we let $\square'$ be such that $\tau_{\sts}(\square')=(\ell-1)$ 
and we let $\stu $ be such  that 
$\mathbb T_{\sts} =\mathbb T_{\stu} \cup \square'\cup \square$. 
In this case we have that    
   $\res_a(\sts) = \res_a(\stu)$  for $a\neq 1,2$ 
  and 
  $$(\res_1(\sts), \res_{2}(\sts)) = (\epsilon, \epsilon q^2), \qquad (\res_1(\stu), \res_{1}(\stu)) = (\epsilon q^2, \epsilon)$$
  for $\epsilon \in \{1, -1\}$.  This is illustrated in \cref{figcase2}.
We have that
  $$e_\sts =\tfrac{1}{4}(
  y_1 
  \psi_0
  y_1
  \psi_0+   
  \psi_0
  y_1
  \psi_0y_1 ) e_\sts=
  \tfrac{1}{4} y_1 (
  \psi_0
  y_1
  \psi_0) e_\sts +   
 \tfrac{1}{4}( \psi_0
  y_1
  \psi_0 ) y_1    e_\sts
  $$ by two applications of \eqref{Rel:V8} (and selective bracketing).  Thus, it is enough to show that 
  $$\psi_0y_1\psi_0e_\sts \in {\rm TL}_n(\alpha_1,\alpha_2, \extra ) e_\stu  {\rm TL}_n(\alpha_1,\alpha_2, \extra ).$$%
 Now, we have
 \begin{align*}
 \psi_0 y_1 \psi_0 e_\sts &=    
 \psi_0 (y_2-\psi_1^2)\psi_0  e_\sts \\
 &= (\psi_0^2y_2 - \psi_0\psi_1^2\psi_0)e_\sts \\
 &= -\psi_0 \psi_1^2  \psi_0 e_\sts \\
 &= -\psi_0\psi_1 e_\stu \psi_1 \psi_0 e_\sts \in  {\rm TL}_n(\alpha_1,\alpha_2, \extra )^{\leq k}
 \end{align*} 
where the first equality follows from (4.5), the second from (4.2), the third from (4.8) and the last by (4.3).

\smallskip\noindent\textbf{Case 3:} If $\square $ is a tile of degree $1$, then, as $\sts$ is a ladder tableau with $\mathbb{T}_\sts$ maximal, we have 
  either $\square \in \mathbb L_{(k,\beta)}\cap \mathbb T_\sts $  and 
  the rightmost vertex of $\square $ lies on a hyperplane;
  or  $\square \in \mathbb R_{(k,\beta)}\cap \mathbb T_\sts $  and 
  the leftmost vertex of $\square $ lies on a hyperplane. Note that $\square$ is not in the top row since we have no tile in the top row of $\mathbb{T}_\sts$ of degree 1 (no marked point), again since $\sts$ is a ladder tableau.  
  We consider the latter case as the former is identical. 
  In this case we let $\square'$ be such that $\tau_{\sts}(\square')=\ell-1$ 
  and $\square''$ be such that $\tau_{\sts}(\square'')=\ell-2 $.  We    let $\stu$ be such  that 
$\mathbb T_{\sts} =\mathbb T_{\stu} \cup \square''\cup \square'\cup \square$. 
We have that 
 $\res_a(\sts)=\res_a(\stu) $  for all $a\neq h-2,h-1,h,h+1$ 
and 
\begin{align*}
(\res_{h-2}(\sts), \res_{h-1}(\sts),\res_{h}(\sts),\res_{h+1}(\sts))				&= (q^{-2}i, i, q^2i, i),
\\
 (\res_{h-2}(\stu), \res_{h-1}(\stu),\res_{h}(\stu),\res_{h+1}(\stu))				&= (i, q^{-2}i, i, q^2i).
 \end{align*}
This is illustrated in \cref{figcase2}.
%
%
%
%
%
%
We claim that 
\begin{align}\label{subinme}
e_\sts=  \psi_{h} \psi_{h-1} \psi_{h-2} 
 e_\stu \psi_{h-2} \psi_{h-1} \psi_{h} 
 \end{align}
 and we note that the proof will follow by induction once we verify the claim since $\stu$ is a ladder tableau with $\mathbb{T}_\stu$ maximal in its residue class. First observe that
 \begin{align}\nonumber 
  \psi_3 \psi_2 \psi_1 
 e(i, q^{-2}i,	i,	iq^2		) \psi_1 \psi_2 \psi_3
 &= \psi_3 \psi_2 \psi_1^2 e(q^{-2}i, i, i, q^2i)\psi_2 \psi_3 \\
 \nonumber 
&=  \psi_3 \psi_2   
(y_2-y_1) e(  q^{-2}i, i,	i,	iq^2 		)  \psi_2 \psi_3
\\
\nonumber 
&
 =
  \psi_3 \psi_2   
 y_2  e(  q^{-2}i, i,	i,	iq^2 		)  \psi_2 \psi_3
\\
\nonumber 
&
 = -\psi_3 e(q^{-2}i, i, i, q^2i) \psi_2 \psi_3 \\
 \nonumber 
 &= \psi_3 \psi_2 \psi_3 e(q^{-2}i , i , q^2 i , i) \\
 \label{skjghskljdhgkjldhkjgdshgkjdshfgjkdfhgkjdsfghkdshgfjkdshfgjksfhdgkjhgdksjlghsd}
 &=
(- \psi_2 \psi_3    \psi_2 +1)
  e(  q^{-2}i,i,  iq^2 ,i		)   
 \end{align}
 where the first and the fifth follows from (4.3), the second follows from (4.5), the third and the fourth follow from (4.4) and (4.5) and the last one from (4.6).
 Thus we have 
 $$\psi_h \psi_{h-1} \psi_{h-2} e_\stu \psi_{h-2}\psi_{h-1}\psi_{h} = (-\psi_{h-1}\psi_h \psi_{h-1} + 1) e_\sts.$$
Now, note that  $s_{h-1}s_ns_{h-1}(\res(\sts))$ is not the residue sequence of a standard tableau, so using (4.3) and (4.10) we have that $\psi_{h-1}\psi_h \psi_{h-1}e_\sts = 0$ and so the claim follows.
 \end{proof}

\begin{cor}\label{cor_residues}
Let $\sts\in\Std_n$. We have $e_{\sts}\in{\rm TL}_n(\alpha_1,\alpha_2, \extra )^{\leq k}$ where $k$ is maximal such that there is a standard tableau in the residue class of $\sts$ of shape $(k,\beta)$.
\end{cor}
\begin{proof}
Take $\stt$ a ladder tableau in the residue class of $\sts$ and let $(k',\beta')$ be its shape. From the preceding proposition, we have $e_{\sts}=e_{\stt}\in{\rm TL}_n(\alpha_1,\alpha_2, \extra )^{\leq k'}$ which is included in ${\rm TL}_n(\alpha_1,\alpha_2, \extra )^{\leq k}$ with $k$ as in the statement, by maximality of $k$.
\end{proof}

\subsection{The graded cellular basis  of the orientifold Temperley--Lieb algebra}
We are now ready to construct graded cellular bases of the orientifold Temperley--Lieb algebras.
This requires a few preparatory lemmas which will help us when proving the cellular ideal structure.

\begin{lem}\label{zajjlemma}
Let $\SSTT \in \Std_n(k, \beta)$ for $0 < k \leq n$ and $\beta \in \{\alpha_1^{\pm1},\alpha_2^{\pm1}\}$.
If  $s_0(\SSTT) \not  \in \Std_n(k, \beta)$ then 
$$\psi_0 \psi_\SSTT e_{\SSTT_{(k, \beta)}}\in {\rm TL}_n(\alpha_1,\alpha_2, \extra )^{<k}.$$
\end{lem}

\begin{proof}
Note that $\sts = s_0(\stt)$ precisely when the path $\sts$ is obtained from $\stt$ by swapping the first $+\epsilon_1$ (respectively $+\epsilon_2$) step by a $+\epsilon_2$ (respectively $+\epsilon_1$) step. Equivalently, $\sts = s_0(\stt)$ if and only if $x_{\stt}(0) = c$, $x_{\sts}(0) = c\pm 2$ and $x_{\stt}(i) = x_{\sts}(i)$ for all $1\leq i\leq n$. 
Now it's easy to see that $\sts = s_0(\stt)\notin \Std_n(k,\beta)$ precisely when $x_{\stt}(0) = b-1$,  in which case $\SSTT   = w \SSTT_{(k, \beta)}$
 for some $w \in \langle s_2,\dots s_{n-1}\rangle$. Therefore 
$$\psi_0 ( \psi_\SSTT e_{\SSTT_{(k, \beta)}}) =  \psi_\SSTT (\psi_0e_{\SSTT_{(k, \beta)}})$$and so it will suffice  to show that 
$$ \psi_0 e_{\SSTT_{(k, \beta)}}
=  e(s_0\res(\SSTT_{(k, \beta)})) \psi_0
 \in  {\rm TL}_n(\alpha_1,\alpha_2, \extra )^{<k} $$ which 
follows immediately  from \cref{propdominant2} and \cref{cor_residues}.
\end{proof}

\begin{lem} \label{zajjlemma2} 
Let    $0 < k \leq n$ and $\beta \in \{\alpha_1^{\pm1},\alpha_2^{\pm1}\}$.
We have that 
$$\psi_j  e_{\SSTT_{(k, \beta)}}\in {\rm TL}_n(\alpha_1,\alpha_2, \extra )^{<k} $$for all
$n-k< j< n$.
\end{lem}

\begin{proof}
Our assumption that $n-k< j < n$ is equivalent to saying that
 the $j$ and $j+1$ steps in $\stt_{(k,\beta)}$ are both $+\epsilon_1$. 
 If $\SSTT_{(k, \beta)}(j)$ does not lie on a hyperplane, then 
we have  that there is no tableau with residue sequence 
 $s_j(\res(  \SSTT_{(k, \beta)}) )$  and 
   hence $\psi_j e_{\SSTT_{(k, \beta)}}=0 \in {\rm TL}_n(\alpha_1,\alpha_2, \extra )$ by \cref{whatisTL}.   
  If $\SSTT_{(k, \beta)}(j)$ does lie on a hyperplane, then   we let $\SSTU:= s_j(\csigma \cdot_j  \SSTT_{(k, \beta)})$ and we let $\SSTT \sim \SSTU$ denote   
  the corresponding  ladder tableau in this equivalence class, if it exists (see \cref{heresaladder2} for examples).   If no such $\stt$ exists then $e_\stu = 0$ and we're done.
Now, it is easy to see that if $\stt$ exists then  $\Shape(\SSTT)<(k,\beta)$ since the width of the path must have increased. We have using \cref{cor_residues}
\begin{equation}\label{asdkfjghsdrfljghfdjs}
\psi_j  e_{\SSTT_{(k, \beta)}}= e_\SSTT \psi_j \in  {\rm TL}_n(\alpha_1,\alpha_2, \extra )^{<k} 
\end{equation}as required.
\end{proof}

\begin{figure}[ht!]
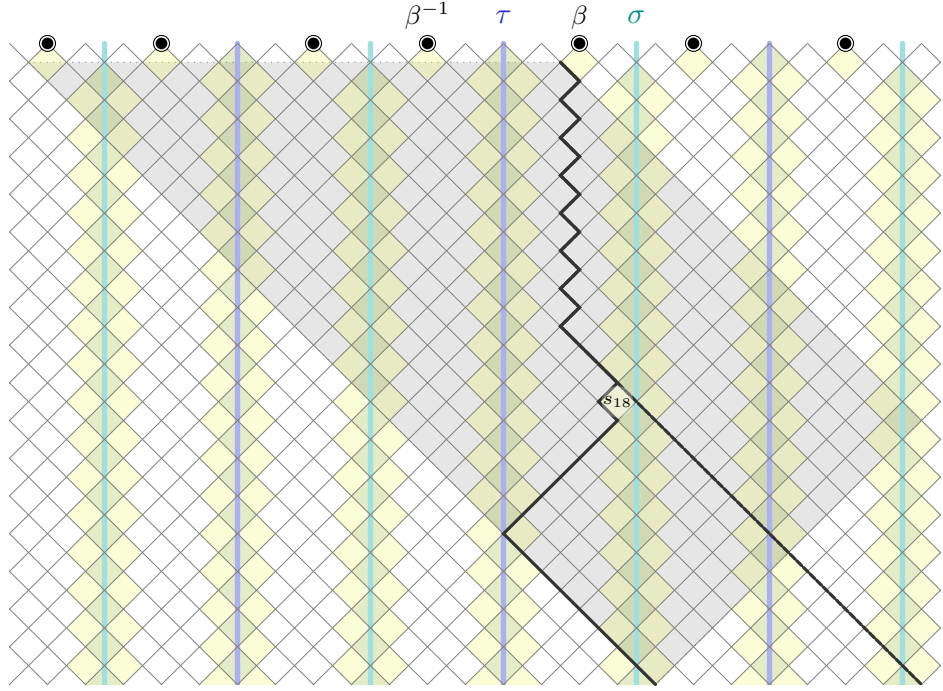

\[
\def\pN{33} 
\def\pB{29} 
\def\pA{3} 
\def\pE{7} 
\def\pWidth{25} 
\TIKZ[yscale=-1, scale=.25]{
\pgfmathsetmacro{\pC}{ceil(.5*\pN - .5*\pB)}
\pgfmathsetmacro{\pK}{floor(\pWidth/\pE)}
\pgfmathsetmacro{\pShift}{ceil(\pK/2)}
\pgfmathsetmacro{\pEParity}{Mod(\pE,2)}
\pgfmathsetmacro{\pNParity}{Mod(\pN,2)}
\pgfmathsetmacro{\pLeft}{ceil((\pA - \pB-1)/\pE)-1}
\coordinate (max-start) at (\pA, 0);
\path (max-start) to ++(\pN-\pB+\pNParity, \pN) coordinate (end);
\path (end) to ++(-\pN, -\pN) coordinate (min-start);
\path (end) to ++(.5*\pB-.5*\pNParity, -.5*\pB+.5*\pNParity) coordinate (p-corner);
\filldraw [pathsregion] (max-start) to (min-start) to (end) to (p-corner) to (max-start); 
\begin{scope}
\clip (\pA - \pB,-1.5) rectangle (\pA + \pN -.5*\pB + .5*\pNParity+1,\pN);
\foreach \x in {0,..., \pK}{
\coordinate (beta-marker\x) at (-2*\pShift*\pE + \pA+1  + \x*2*\pE, -1);
\coordinate (betainv-marker\x) at (2*\pShift*\pE- \pA -1 - \x*2*\pE, -1);
	\filldraw[plusone] (beta-marker\x) \TILE;
	\filldraw[plusone] (betainv-marker\x) \TILE;
}
\foreach \x in {-\pK, ..., \pK}{
	\foreach \y in {-1, 1,...,\pN}{
		\filldraw[minustwo] (2*\pE*\x, \y) \TILE;
		\filldraw[plusone] 
			(2*\pE*\x+1, \y+1) \TILE;
		\filldraw[plusone] 
			(2*\pE*\x-1, \y+1) \TILE;
		\filldraw[minustwo] (2*\pE*\x+\pE, \y+\pEParity) \TILE;
		\filldraw[plusone] 
			(2*\pE*\x+\pE+1, \y+\pEParity+1) \TILE;
		\filldraw[plusone] 
			(2*\pE*\x+\pE-1, \y+\pEParity+1) \TILE;
	}}
\foreach \x in {-\pK, ..., \pK}{
		\draw[tauplane] (\x*2*\pE,-1) to +(0,\pN+1);
		\draw[sigmaplane] (\x*2*\pE+\pE,-1) to +(0,\pN+1);
		}
\pgfmathsetmacro{\pGrids}{max(\pN, \pWidth)}
\foreach \x in {-\pGrids, ..., \pGrids}{
	\draw[gridstyle] (2*\x,-1) to +(-\pN-2,\pN+2);
	\draw[gridstyle] (2*\x,-1) to +(\pN+2,\pN+2);
}
\foreach \x in {0,..., \pK}{
	\node[M] at (beta-marker\x){};
	\node[M] at (betainv-marker\x){};
}
\end{scope}
\node[above] at (\pA+1, -1) {\small \strut $\beta$};
\node[above] at (-\pA-1, -1) {\small \strut $\beta^{-1}$};
\node[above, taucolor] at (0,-1) {\strut \normalsize $\tau$};
\node[above, sigmacolor] at (\pE, -1) {\strut \normalsize $\sigma$};
\draw[path1, densely dotted, very thick ] (max-start) 
\foreach \x in {1, ..., 7}{ to ++(1,1) to ++(-1,1)}
	to ++(19,19);
\draw[path1, very thick ] (max-start) 
\foreach \x in {1, ..., 7}{ to ++(1,1) to ++(-1,1)}
	 to ++(3, 3) 
	 to ++(-1,1) to ++(1,1)
	 to ++(-6,6) to (end);
\node[tilemarkers] at (6, 18) {\tiny $s_{18}$};
%
}
\]
\caption{Here $e=14$ and $\alpha_1= q^4$ with $\alpha_2 \not \in \alpha_1q^{2\ZZ}$. 
We depict two ladder tableaux  which 
 illustrate the second case of the proof of \cref{zajjlemma2}. 
We have that 
 $s_{18}  $ to the residue sequence of one path, we obtain the residue sequence of the other path.
 }
\label{heresaladder2}
\end{figure}


\begin{prop}\label{ydies}
Let $\SSTT \in \Std_n(k, \beta)$ for $(k, \beta)\in \Lambda_n$.
We have that 
$$y_j  e_{\SSTT_{(k, \beta)}}\in {\rm TL}_n(\alpha_1,\alpha_2, \extra )^{<k} $$for all $1\leq j \leq n$
(where we set ${\rm TL}_n(\alpha_1,\alpha_2, \extra )^{<0} = \{0\}$).

\end{prop}

\begin{proof}
We first suppose that  $j=1$.  If    $k\neq 0$ then   $s_0(\SSTT_{(k,\beta)}) \not  \in \Std_n(k, \beta)$ and so the result follows from \cref{zajjlemma} together with the fact that $y_1e_{\SSTT_{(k, \beta)}}=\pm \psi_0^2e_{\SSTT_{(k, \beta)}}$ using (4.8). If $k=0$ then $y_1  e_{\SSTT_{(0, \extra)}}=0$ by \cref{whatisTL}.
We now suppose that  $j=2$ and $k\neq n$.  
For $\beta \in \{\alpha_1^{\pm1},\alpha_2^{\pm1},\extra\}$ we have $e_{\SSTT_{(k,\beta)}}=e(\beta,\beta^{-1}q^2,...)$ and 
\begin{align*}
\psi_1\psi_0\psi_1e(\beta,\beta q^{-2},...)\psi_1\psi_0\psi_1
& = \psi_1\psi_0(y_2-y_1)e(\beta q^{-2},\beta,...)\psi_0\psi_1\\
& = \psi_1(y_2+y_1)e(\beta^{-1}q^2,\beta,...)\psi_1\\
& = (y_1+y_2)e(\beta,\beta^{-1}q^2,...)\\
& =(y_1+y_2)e_{\SSTT_{(k,\beta)}}\ ,
\end{align*}
where the first equality follows from (4.3) and (4.5), the second equality uses (4.3) and (4.8) with the fact that $\beta q^{-2}$ is not $\pm1$ and not in $\{\alpha_i^{\pm1}\}$ (see our standing assumptions), and the third equality uses (4.3), (4.4) and (4.5) with the fact that $\beta^{-1}q^2\notin\{\beta,\beta q^{\pm2}\}$ (again from our standing assumptions). We have that $e(\beta,\beta q^{-2},...)=0$ since the only possible residues after $\beta$ in a standard tableau are $\beta q^2$ and $\beta^{-1}q^2$ and these are different from $\beta q^{-2}$. Thus we have $y_2e_{\stt_{(k,\beta)} }= -y_1e_{\stt_{(k,\beta)}}=0$ as seen above.

We now suppose that $j=2$ and $k=n$. In this case, we have $e_{\SSTT_{(k,\beta)}}=e(\beta,\beta q^2,...)$ and
\[y_2e_{\SSTT_{(k,\beta)}}=y_2e(\beta,\beta q^2,...)=(\psi_1^2+y_1)e(\beta,\beta q^2,...)=(\psi_1^2+y_1)e_{\SSTT_{(k,\beta)}}\ \]
using (4.5). 
Now the result follows from  the case $j=1$ and \cref{zajjlemma2}.

Now assume, by induction, that the result holds for $j\geq 2$.  Suppose first that within
 $\SSTT_{(k,\beta)}$  the $j$th  and $(j+1)$th steps are $+\epsilon_1$. This implies that $\res_{j+1}(\stt_{(k,\beta)})= q^2\res_{j}(\stt_{(k,\beta)})$, so using (4.5), we have
\[y_{j+1}e_{\SSTT_{(k,\beta)}}=(\psi_j^2+y_j)e_{\SSTT_{(k,\beta)}}\ ,\]
and we can apply \cref{zajjlemma2} and the induction hypothesis.

Finally, suppose that within
 $\SSTT_{(k,\beta)}$ the $j$th  and $(j+1)$ steps are $+\epsilon_1$ and $+\epsilon_2$, respectively (that is, $j$ is odd). This means that $e_{\SSTT_{(k,\beta)}}=e(...,\beta^{-1}q^{j-1},\beta q^{j-1},\beta^{-1}q^{j+1},...)$ where we are showing only the residues in positions $j-1,j,j+1$. We have:
 \begin{align*}
& \psi_{j-1}\psi_je(...,\beta q^{j-1},\beta^{-1} q^{j+1},\beta^{-1}q^{j-1},...)\psi_j\psi_{j-1}\\
& = \psi_{j-1}(y_{j+1}-y_j)e(...,\beta q^{j-1},\beta^{-1}q^{j-1},\beta^{-1}q^{j+1},...)\psi_{j-1}\\
& = (y_{j+1}-y_j)e(...,\beta^{-1}q^{j-1},\beta q^{j-1},\beta^{-1} q^{j+1},...)\\
& = (y_{j+1}-y_j)e_{\SSTT_{(k,\beta)}}
\end{align*}
where the first inequality follows from (4.3) and (4.5) noting that the $j$th and $j+1$th residues differ by $q^2$. The second inequality follows again  from (4.3) and (4.5) using that $\beta^{-1}\notin\{\beta,\beta q^{\pm2}\}$. The result now follows using the induction hypothesis and the fact that $e(...,\beta q^{j-1},\beta^{-1} q^{j+1},\beta^{-1}q^{j-1},...)=0$ since it cannot be the residues sequence of a standard tableau. This latter fact is easily seen in the path model, where we can see that $s_{j}s_{j-1}\SSTT_{(k,\beta)}$ is not a valid path.

The other case where $j$ is even and the $j$ and $j+1$ steps are $+\epsilon_2$ and $+\epsilon_1$ respectively is similar and we leave it to the reader (one reverses the order of $j$ and $j-1$). 
\end{proof}

\begin{prop}\label{BKW}

Let $\ti  \in I^n$, and $w\in W(C_n)$ be an element written as a product of simple transpositions: $w=s_{t_1}s_{t_2}\dots s_{t_m}$ for some $0\leq t_1, \dots, t_m <n$. 
\begin{itemize}
\item [$(i)$]	If $s_{t_1}s_{t_2}\dots s_{t_m}$ is not a reduced word, then 
$\psi_{t_1}\psi_{t_2}\dots \psi_{t_m}e(i)\in \mathscr{H}_n  (\alpha_1,\alpha_2 , \vartheta)$ can be written as a linear combination of 
$$\psi_{t_{a_1}}\psi_{t_{a_2}}\dots \psi_{t_{a_b}}f(y)e(\ti)$$for 
$1\leq a_1 \leq \dots \leq a_b \leq m$, $b<m$,  
and 
$s_{t_{a_1}}s_{t_{a_2}}\dots s_{t_{a_b}}$ is a reduced word, and 
$f(y)$ is a polynomial in $y_1, \dots ,y_n$.

\item [$(ii)$]	
 We have that 
$$y_k\psi_{t_1}\psi_{t_2}\dots \psi_{t_m}e(\ti)=
 \psi_{t_1}\psi_{t_2}\dots \psi_{t_m} 	y_{w^{-1}(k)}	e(\ti)+\dots 
 $$where the $\dots $ denotes a   linear combination of terms of the form 
$$\psi_{t_{a_1}}\psi_{t_{a_2}}\dots \psi_{t_{a_b}}f(y)e(\ti)$$for 
$1\leq a_1 \leq \dots \leq a_b \leq m$, $b< m$,  
and 
$f(y)$ is a polynomial in $y_1, \dots ,y_n$.

\end{itemize}
\end{prop}

\begin{proof}
The proof proceeds exactly as in \cite[Lemma 2.5 and Proposition 2.5]{bkw11} with the   minor caveat that 
we appeal to   Matsumoto's theorem for $W(C_n)$, rather than Matsumoto's theorem for the finite symmetric group.
\end{proof}

\begin{prop}\label{almostcellular}
The algebra  ${\rm TL}_n(\alpha_1,\alpha_2, \extra )$ has a filtration by two-sided ideals 
\begin{align*}
0\leq{\rm TL}_n(\alpha_1,\alpha_2, \extra ) ^{\leq 0} \leq  {\rm TL}_n(\alpha_1,\alpha_2, \extra ) ^{\leq 1}\leq 
 {\rm TL}_n(\alpha_1,\alpha_2, \extra ) ^{\leq 3}\leq \dots 
\leq   {\rm TL}_n(\alpha_1,\alpha_2, \extra ) ^{\leq n} 
\\0\leq {\rm TL}_n(\alpha_1,\alpha_2, \extra ) ^{\leq 0}\leq 
 {\rm TL}_n(\alpha_1,\alpha_2, \extra ) ^{\leq 2}\leq \dots 
\leq   {\rm TL}_n(\alpha_1,\alpha_2, \extra ) ^{\leq n}  
\end{align*}for $n$ odd or even respectively, where
$ {\rm TL}_n(\alpha_1,\alpha_2, \extra ) ^{\leq n}  = {\rm TL}_n(\alpha_1,\alpha_2, \extra ) ^{\leq n} $.
Here  each 
 subquotient 
 $$
  {\rm TL}_n(\alpha_1,\alpha_2, \extra ) ^{\leq k}/  {\rm TL}_n(\alpha_1,\alpha_2, \extra ) ^{<k}
 $$%
 decomposes as 
\begin{equation}\label{bigsum}
\bigoplus _{\{\beta \mid (k,\beta) \in \Lambda_n\}} 
\Bbbk\text{--}{\rm span}
\{\psi_\SSTS e_{\SSTT_{(k,\beta)}} \psi_\SSTT^* +{\rm TL}_n(\alpha_1,\alpha_2, \extra ) ^{<k}
\mid 
\SSTS,\SSTT \in \Std_n(k,\beta)
\}
\end{equation}
and the decomposition of  \eqref{bigsum} is as a 
direct sum of   ${\rm TL}_n(\alpha_1,\alpha_2, \extra )$-bimodules.

\end{prop}

\begin{proof}
We proceed by induction on $0\leq k \leq n$.    For a fixed $0\leq k \leq n$ and any  $\ell \in \ZZ_{\geq0}$ we define 
$$	\Delta	_{\leq\ell} 	(k,\beta)  :=  	\Bbbk\text{--}{\rm span} \{	\psi_\SSTU e_{\SSTT_{(k,\beta)}} \psi_\SSTV^* + {\rm TL}_n(\alpha_1,\alpha_2, \extra ) ^{<k}
\mid 
\SSTU,\SSTV \in \Std_n(k,\beta)
,
  \ell(\SSTU)+\ell(\SSTV)\leq \ell 
\} .
$$We define $\Delta	_{<\ell} 	(k,\beta) $ similarly.
For a given  $\ell= \ell(\sts)+\ell(\stt)$ we will prove the following two claims by induction 
$$
 y_j	\psi_\SSTS e_{\SSTT_{(k,\beta)}} \psi_\SSTT^* 
  \in 
		\Delta	_{ <\ell} 	(k,\beta)   
	\qquad
 \psi_j	\psi_\SSTS e_{\SSTT_{(k,\beta)}} \psi_\SSTT^* 
\begin{cases}	
= \psi_\SSTP e_{\SSTT_{(k,\beta)}} \psi_\SSTT^* 		&\text{if }\sts< s_j(\SSTS)=\SSTP \in \Std_n(k,\beta)	
\\
\in	\Delta	_{<\ell} 	(k,\beta) 					&\text{otherwise.}
\end{cases}
$$
and hence deduce the result. 

\smallskip
\noindent{\bf Claim 1. } By \cref{BKW}$(ii)$ we have that 
$$
 y_j	\psi_\SSTS e_{\SSTT_{(k,\beta)}} \psi_\SSTT^* 
 =
\sum _{\SSTP \leq \SSTS} 	\psi_\SSTP f(y) e_{\SSTT_{(k,\beta)}} \psi_\SSTT^*  .
$$
If $\SSTP=\SSTS$ then we note that $f(y)$ is a polynomial of degree 1 (in fact, it's simply equal to $y_{w_\SSTS(j)}$) and hence this term is zero by   \cref{ydies}.  All other terms belong to $	\Delta	_{<\ell} 	(k,\beta)   $ by definition.

\smallskip
\noindent{\bf Claim 2. } We will break this claim down according to how (if at all) the tableau $s_j(\SSTS)$ fails to be a standard. 
  If $\SSTP:=s_j(\SSTS)\in \Std_n(k,\beta)$ is standard and $\SSTP>\SSTS$  then there is nothing to prove.  
 If $s_j(\SSTS)$ is standard with $ s_j(\SSTS)< \SSTS$ then  $  s_j w_\SSTS$ is non-reduced and we can rewrite 
 $\psi_j 	\psi_\SSTS e_{\SSTT_{(k,\beta)}} \psi_\SSTT^* $ in the required form using 
 \cref{BKW}$(i)$ to simplify the non-reduced word followed by \cref{ydies} (as any positive degree term in the polynomial $f(y)$ is zero).

Finally it remains to consider the case where  $ s_j(\SSTS)$ is not a standard tableau.  In this case, we let 
$\mathbb T_\sts $ denote the (admissible, reduced) tiling of  $ \sts \in \Std_n(k,\beta)$. 
The permutation $s_j w_\sts$ has admissible (non-reduced) tiling 
  $\mathbb T_\sts \cup \square$ for some unsupported $\square$ (with $y$-coordinate $y \in \ZZ_{\geq0}$).
There exist two tiles $\square'$ and $\square''$ adjacent to $\square$ which are strictly lower in the tile ordering; and 
without loss of generality we have that $\square'  \in \mathbb T_\sts $ and $\square'' \not \in \mathbb T_\sts $.  
(Note that  $\square',\square''  \in \mathbb T_\sts $ implies that $s_j(\SSTS)>\SSTS$ is standard 
and $\square',\square''  \not \in \mathbb T_\sts $ implies that $s_j(\SSTS)<\SSTS$ is standard.) 

{\em Base cases of induction. }
We first consider the base cases of the induction. 
These base cases are as follows: 
$(0)$ $\sts=\stt_{(k,\beta)}$ and  $\SSTP = s_j(\sts)\not \in \Std_n(k,\beta)$   
$(1)$ $\sts = s_j  (\stt_{(k,\beta)})\in \Std_n(k,\beta)$ and $ s_{j\pm1}(\sts)\not\in \Std_n(k,\beta)$ 
or 
$(2)$ $\sts = s_0 s_1  ( \stt_{(k,\beta)})\in \Std_n(k,\beta)$ and $s_1(\sts)\not \in \Std_n(k,\beta)$.
Case $(0)$ was already taken care of in \cref{zajjlemma2,zajjlemma}. 
 In case $(1)$ we have that $1\leq j \leq n-k$ and we 
 can assume that 
 $j$ is odd (the $j$ is even case is identical) and that $ s_{j+1}(\sts)\not\in \Std_n(k,\beta)$ (the case $ s_{j-1}(\sts)\not\in \Std_n(k,\beta)$ isidentical).
We have that 
 $$e_\sts= e(\ldots , 	\beta^{-1}q^{j+1} , \beta q^{j-1}	,  \beta q^{j+1} , \ldots  )
 $$
 where we have only written the residues in position $j, j+1$ and $j+2$. Note that 
 $$s_{j+1}(\res (\sts)) = e(\ldots , 	\beta^{-1}q^{j+1} , \beta q^{j+1}, \beta q^{j-1} , \ldots  )
$$ is not the residue sequence of a standard tableau. Indeed after the residue $\beta^{-1}q^{j+1}$, the two options are $\beta q^{j-1}$ and $\beta^{-1}q^{j+3}$ and none of them are equal to $\beta q^{j+1}$ by our standing assumptions on $\beta$.
Therefore $e_{s_{j+1}(\res (\sts))} = 0$ by (4.10). It follows from (4.3) that
$$\psi_{j+1}\psi_je_{\stt_{(k,\beta)}}\psi_{\stt}^* =  \psi_{j+1}e_\sts \psi_je_{\stt_{(k,\beta)}}\psi_{\stt}^* = e_{s_{j+1}(\res (\sts))} \psi_{j+1}\psi_je_{\stt_{(k,\beta)}}\psi_{\stt}^* = 0.$$

In case $(2)$ we have that 
 $$e_\sts= e( \beta q^{-2}, \beta 
 ,\ldots)
 \qquad
 e_{s_1(\res(\sts))} = = e(    \beta , \beta q^{-2}  , \ldots ).
 $$ 
 Again, this is not the residue sequence of a standard tableau as   $\beta q^{-2} \notin \{ \beta^{-1}q^2, \beta q^2\}$ by our standing assumptions. Thus the result follows as in case $(1)$.

%
%
%

{\em  Inductive step. }
Now assume that $\ell(\sts)\geq 2$ and $w_\sts \neq s_0s_1$.  Then we are in one of the following three cases:
$(i)$ $\sts =  s_0 s_1s_0  \stu$ for $\stu \in \Std_n(k,\beta)$, and $j=1$,
$(ii)$ $\sts = s_{j-1} s_{j} \stu$ with $\stu \in \Std_n(k,\beta)$
   or 
 $(iii)$ $\sts = s_{j+1} s_{j} \stu$ with $\stu \in \Std_n(k,\beta)$.
  In each of these three cases we have that $s_j(\sts)$ is not standard. 
Cases $(i)$ and $(ii)$ are depicted in \cref{OH-zajj}.

\begin{figure}[ht!]
\[
 \begin{tikzpicture}[scale=1.38251]
 \begin{scope}
\path (0,0)--++(135:0.4)--++(-135:0.4)--++(135:0.4)--++(-135:0.4) coordinate (sart1);
\path (0,0)--++(45:0.4)--++(-45:0.4)
	--++(45:0.4)--++(-45:0.4)
	--++(45:0.4)--++(-45:0.4)
	--++(45:0.4)--++(-45:0.4) coordinate(pp)
	;
\path(pp)--++(-135:0.4*5)--++(-45:0.4*5) coordinate(pp2);
\path (pp2)--++(135:0.4)--++(-135:0.4)
	--++(135:0.4)--++(-135:0.4)
	--++(135:0.4)--++(-135:0.4)
	--++(135:0.4)--++(-135:0.4)--++(135:0.4)--++(-135:0.4)--++(135:0.4)--++(-135:0.4)  coordinate(pp3)
	;
\clip (sart1)--(pp)--(pp2)--(pp3);
\path (0,0) coordinate (top00)
	--++(-45:0.4)  --++(45:0.4) coordinate (top1)
	--++(-45:0.4)  --++(45:0.4) coordinate (top2)
	--++(-45:0.4)  --++(45:0.4) coordinate (top3)
	--++(-45:0.4)  --++(45:0.4) coordinate (top4)
	--++(-45:0.4)  --++(45:0.4)coordinate  (top5)
	--++(-45:0.4)   coordinate(right1) --++(-135:0.4) 
	--++(-45:0.4)  coordinate (right2)--++(-135:0.4)  
	--++(-45:0.4)   coordinate (right3)--++(-135:0.4) 
	--++(-45:0.4) coordinate (right4) --++(-135:0.4)  
	--++(-45:0.4)  coordinate (right5)
	--++(-135:0.4)  
	coordinate (bottom5)    --++(135:0.4)    --++(-135:0.4)  
	coordinate (bottom4)    --++(135:0.4)    --++(-135:0.4)  
	coordinate (bottom3)    --++(135:0.4)    --++(-135:0.4)  
	coordinate (bottom2)    --++(135:0.4)    --++(-135:0.4)  
	coordinate (bottom1)    --++(135:0.4)    --++(-135:0.4)  
	coordinate (bottom0)    --++(135:0.4)    --++(-135:0.4)  
	coordinate (bottomm1)    --++(135:0.4)   
	coordinate (left5)    --++(45:0.4)     --++(135:0.4)   
	coordinate (left4)    --++(45:0.4)     --++(135:0.4)   
	coordinate (left3)    --++(45:0.4)     --++(135:0.4)   
	coordinate (left2)    --++(45:0.4)     --++(135:0.4)   
	coordinate (left1)    --++(45:0.4)     
	coordinate (topm1) --++(-45:0.4)    --++(45:0.4);
\path(top00)--++(-135:0.4) coordinate (PP);
 \path(top3)    --++(-45:0.2)  --++(45:0.2) coordinate (top45);
 \path(bottom3)    --++(-45:0.2)  --++(45:0.2) coordinate (bottom45);
\begin{scope}[gridstyle]
\draw
	(left5) to ++(-45:0.4)
	(left4) to (bottom0)
 	(left3) to (bottom1)
 	(left2) to (bottom2)  
	(left1) to (bottom3) 
	(topm1) to (bottom4)
	(top00) to (bottom5)
	(top1) to (right5)
	(top2) to (right4) 
	(top3) to (right3)
	(top4) to (right2)
	(top5) to (right1); 
\foreach \x [count=\c from 1] in {m1, 00, 1, 2, 3}{
	\draw(top\x) to ++(-135:0.4*\c*2 - .4);}
\draw
	(top4)--++(-135:0.4*10)
	(top5)--++(-135:0.4*10);
\foreach \x [count=\c from 1] in {5,4,...,1}{
	\draw(right\x) to ++(-135:0.4*\c*2 - .4);}
\end{scope}
\path(0,0)--++(45:0.8)--++(-45:1.2) coordinate (X);
\fill[opacity=0.3,cyan] 
	(X)--++(-135:0.4*3)  --++(-45:0.4*3)
	 --++(45:0.4) --++(135:0.4)  --++(45:0.4) 
	 --++(135:0.4) --++(45:0.4) --++(135:0.4);
\path (X)--++(135:0.4*1) --++(-135:0.4*1) --++(135:0.4*1) --++(-135:0.4*1) 
--++(135:0.4*1) --++(-135:0.4*1) coordinate (JJ)
--++(-45:0.4*3) ;
\fill[opacity=0.3,cyan](X)--++(135:0.4*1) --++(-135:0.4*3)  coordinate (PP)--++(-45:0.4*1);
\fill[opacity=0.3,](PP)--++(135:0.4*2) --++(45:0.4*1) --++(-45:0.4*1) 
 --++(45:0.4*1) --++(-45:0.4*1) ;
\fill[magenta,opacity=0.4](JJ)--++(-135:0.4)--++(-45:0.4)--++(45:0.4)--++(135:0.4);;
 \path(JJ)--++(-45:0.2)--++(-135:0.2)--++(-90:0.025) node {\scalefont{0.9}$s_1$};
\path(JJ)
 --++(-45:0.4)--++(45:0.4)
 --++(-45:0.2)--++(-135:0.2)--++(-90:0.025) node {\scalefont{0.9}$s_1$};
\path(JJ)
 --++(45:0.4)
 --++(-45:0.2)--++(-135:0.2)--++(-90:0.025) node {\scalefont{0.9}$s_0$};
\path(JJ)
 --++(45:0.4) --++(-45:0.4)--++(45:0.4)
 --++(-45:0.2)--++(-135:0.2)--++(-90:0.025) node {\scalefont{0.9}$s_0$};
 \path(JJ)--++(-45:0.6)--++(-135:0.2) node {\scalefont{0.9}$\times$};
\draw[path1](X)--++(-45:0.4)
--++(-135:0.4)--++(-45:0.4)
--++(-135:0.4)--++(-45:0.4)
--++(-135:0.4)--++(-45:0.4)
--++(-135:0.4)--++(-45:0.15)
coordinate (Y);
\draw[line width = 2.5, path1color,densely dotted] (Y) --++(-45:0.4);
\node[M] at (X) {};
\end{scope}
\node[above] at (X) {$\beta$};
\end{tikzpicture}
 \qquad\quad
\begin{tikzpicture}[scale=1.38251]
\begin{scope}
\path (0,0)--++(135:0.4)--++(-135:0.4)--++(135:0.4)--++(-135:0.4) coordinate (sart1);
\path (0,0)--++(45:0.4)--++(-45:0.4)
	--++(45:0.4)--++(-45:0.4)
	--++(45:0.4)--++(-45:0.4)
	--++(45:0.4)--++(-45:0.4) coordinate(pp)
	;
\path(pp)--++(-135:0.4*5)--++(-45:0.4*5) coordinate(pp2);
\path (pp2)--++(135:0.4)--++(-135:0.4)
	--++(135:0.4)--++(-135:0.4)
	--++(135:0.4)--++(-135:0.4)
	--++(135:0.4)--++(-135:0.4)--++(135:0.4)--++(-135:0.4)--++(135:0.4)--++(-135:0.4)  coordinate(pp3)
	;
\clip (sart1)--(pp)--(pp2)--(pp3);
\path (0,0) coordinate (top00)
	--++(-45:0.4)  --++(45:0.4) coordinate (top1)
	--++(-45:0.4)  --++(45:0.4) coordinate (top2)
	--++(-45:0.4)  --++(45:0.4) coordinate (top3)
	--++(-45:0.4)  --++(45:0.4) coordinate (top4)
	--++(-45:0.4)  --++(45:0.4)coordinate  (top5)
	--++(-45:0.4)   coordinate(right1) --++(-135:0.4) 
	--++(-45:0.4)  coordinate (right2)--++(-135:0.4)  
	--++(-45:0.4)   coordinate (right3)--++(-135:0.4) 
	--++(-45:0.4) coordinate (right4) --++(-135:0.4)  
	--++(-45:0.4)  coordinate (right5)
	--++(-135:0.4)  
	coordinate (bottom5)    --++(135:0.4)    --++(-135:0.4)  
	coordinate (bottom4)    --++(135:0.4)    --++(-135:0.4)  
	coordinate (bottom3)    --++(135:0.4)    --++(-135:0.4)  
	coordinate (bottom2)    --++(135:0.4)    --++(-135:0.4)  
	coordinate (bottom1)    --++(135:0.4)    --++(-135:0.4)  
	coordinate (bottom0)    --++(135:0.4)    --++(-135:0.4)  
	coordinate (bottomm1)    --++(135:0.4)   
	coordinate (left5)    --++(45:0.4)     --++(135:0.4)   
	coordinate (left4)    --++(45:0.4)     --++(135:0.4)   
	coordinate (left3)    --++(45:0.4)     --++(135:0.4)   
	coordinate (left2)    --++(45:0.4)     --++(135:0.4)   
	coordinate (left1)    --++(45:0.4)     
	coordinate (topm1) --++(-45:0.4)    --++(45:0.4);
\path(top00)--++(-135:0.4) coordinate (PP);
 \path(top3)    --++(-45:0.2)  --++(45:0.2) coordinate (top45);
 \path(bottom3)    --++(-45:0.2)  --++(45:0.2) coordinate (bottom45);
\begin{scope}[gridstyle]
\draw
	(left5) to ++(-45:0.4)
	(left4) to (bottom0)
 	(left3) to (bottom1)
 	(left2) to (bottom2)  
	(left1) to (bottom3) 
	(topm1) to (bottom4)
	(top00) to (bottom5)
	(top1) to (right5)
	(top2) to (right4) 
	(top3) to (right3)
	(top4) to (right2)
	(top5) to (right1); 
\foreach \x [count=\c from 1] in {m1, 00, 1, 2, 3}{
	\draw(top\x) to ++(-135:0.4*\c*2 - .4);}
\draw
	(top4)--++(-135:0.4*10)
	(top5)--++(-135:0.4*10);
\foreach \x [count=\c from 1] in {5,4,...,1}{
	\draw(right\x) to ++(-135:0.4*\c*2 - .4);}
\end{scope}
\path(0,0)--++(45:0.8)--++(-45:1.2) coordinate (X);
\path (X)--++(45:0.15) coordinate (Y);
\fill[opacity=0.3,cyan] (X)--++(-135:0.4*3) coordinate (JJ)--++(-45:0.4*2)
--++(-135:0.4*1)  --++(-45:0.4*2)
 --++(45:0.4) --++(135:0.4) --++(45:0.4) --++(135:0.4) --++(45:0.4) --++(135:0.4) --++(45:0.4) --++(135:0.4);
\path (X)--++(-135:0.4*4) coordinate (JJ);
  \fill[  opacity=0.25](JJ)--++(45:0.4)--++(-45:0.4*2)--++(-135:0.4)--++(135:0.4*2);
\fill[ magenta,opacity=0.4](JJ)--++(-135:0.4)--++(-45:0.4)--++(45:0.4)--++(135:0.4);
\path(JJ)--++(-45:0.2)--++(-135:0.2)--++(-90:0.025) node {\scalefont{0.8}$s_{j}$};
\path(JJ)
--++(45:0.4) 
--++(-45:0.2)--++(-135:0.2)--++(-90:0.025) node {\scalefont{0.8}$s_{j-1}$};
\path(JJ)
--++(-45:0.4)--++(45:0.4)
--++(-45:0.2)--++(-135:0.2)--++(-90:0.025) node {\scalefont{0.8}$s_{j}$};
%
\path(JJ)--++(-45:0.6)--++(-135:0.2) node {\scalefont{0.9}$\times$};

\draw[line width = 2.5, path1color, densely dotted](Y)--++(45:0.4);
\draw[path1](Y)--(X)--++(-45:0.4)
--++(-135:0.4)--++(-45:0.4)
--++(-135:0.4)--++(-45:0.4)
--++(-135:0.4)--++(-45:0.4)
--++(-135:0.4)--++(-45:0.15)
coordinate (Y);
\draw[line width = 2.5, path1color,densely dotted](Y)--++(-45:0.4);
\end{scope}
\end{tikzpicture}
\] 
 
 \caption{
 Cases $(i)$ and $(ii)$ of the inductive step.  
The union of the grey and blue tiles is equal to $\mathbb T_\sts$ for $\sts$ a standard tableau; 
the pink tile breaks the standardness condition.  The $\times$ marks the missing tile $\square''\not \in \mathbb T_\sts$.  We emphasise that we have drawn $\sts$ of minimal length such that the inclusion of the pink tile violates the standardness condition; the non-minimal cases follow by the commutativity relations.
 }
 \label{OH-zajj}
 \end{figure}
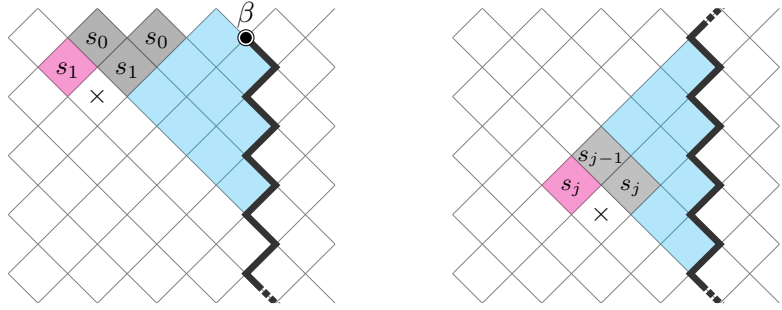

We consider case  $(ii)$ first (and we remark that  case $(iii)$ can be dealt with in an  identical manner). 
In case $(ii)$, using (4.6), we have that 
$$
 \psi_{j}	\psi_\SSTS e_{\SSTT_{(k,\beta)}} \psi_\SSTT^* = 
 \psi_{j} \psi_{j-1}  \psi_{j}		\psi_\SSTU	 e_{\SSTT_{(k,\beta)}} \psi_\SSTT^* =  
 \begin{cases}
 \psi_{j-1} \psi_{j}  \psi_{j-1}		\psi_\SSTU	 e_{\SSTT_{(k,\beta)}} \psi_\SSTT^* 	  \\
(  \psi_{j-1} \psi_{j}  \psi_{j-1}	 \pm1)	\psi_\SSTU	 e_{\SSTT_{(k,\beta)}} \psi_\SSTT^* 	  \end{cases}
$$
depending on whether $\sts(j-1)$ lies on a hyperplane or not. 
Note that $s_{j-1}(\stu)$ is not standard and $\ell(\stu)=\ell(\sts)-2$, so by induction we have $\psi_{j-1}\psi_\stu e_{\stt_{(k,\beta)}}\psi_\stt^* \in \Delta_{<\ell -2}(k,\beta)$ and $\psi_{j-1}\psi_j\psi_{j-1}\psi_\stu e_{\stt_{(k,\beta)}}\psi_\stt^*\in \Delta_{<\ell}(k,\beta)$. Moreover, by definition we have $\psi_\stu e_{\stt_{(k,\beta)}}\psi_\stt^*\in \Delta_{\leq \ell - 2}(k,\beta) \subseteq \Delta_{< \ell }(k,\beta)$, so in both cases we have $ \psi_{j}	\psi_\SSTS e_{\SSTT_{(k,\beta)}} \psi_\SSTT^* \in \Delta_{< \ell }(k,\beta)$.

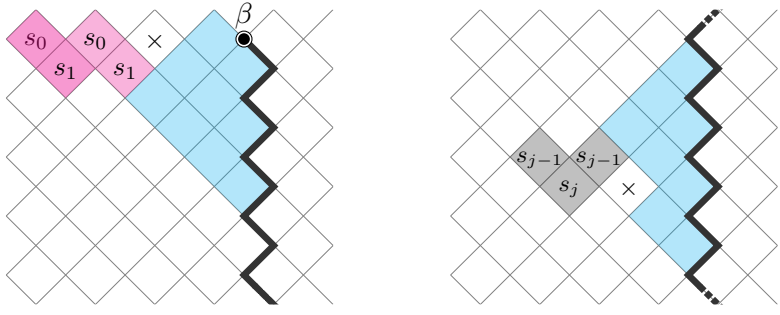
\begin{figure}[ht!]
\[\begin{tikzpicture}[scale=1.38251]
\begin{scope}
\path (0,0)--++(135:0.4)--++(-135:0.4)--++(135:0.4)--++(-135:0.4) coordinate (sart1);

\path (0,0)--++(45:0.4)--++(-45:0.4)
--++(45:0.4)--++(-45:0.4)
--++(45:0.4)--++(-45:0.4)
--++(45:0.4)--++(-45:0.4) coordinate(pp)
;
\path(pp)--++(-135:0.4*5)--++(-45:0.4*5) coordinate(pp2);
\path (pp2)--++(135:0.4)--++(-135:0.4)
--++(135:0.4)--++(-135:0.4)
--++(135:0.4)--++(-135:0.4)
--++(135:0.4)--++(-135:0.4)--++(135:0.4)--++(-135:0.4)--++(135:0.4)--++(-135:0.4)  coordinate(pp3)
;
\clip (sart1)--(pp)--(pp2)--(pp3);

\path (0,0) coordinate (top00)
 --++(-45:0.4)  --++(45:0.4) coordinate (top1)
  --++(-45:0.4)  --++(45:0.4) coordinate (top2)
   --++(-45:0.4)  --++(45:0.4) coordinate (top3)
    --++(-45:0.4)  --++(45:0.4) coordinate (top4)
     --++(-45:0.4)  --++(45:0.4)coordinate  (top5)
  --++(-45:0.4)   coordinate(right1) --++(-135:0.4) 
  --++(-45:0.4)  coordinate (right2)--++(-135:0.4)  
  --++(-45:0.4)   coordinate (right3)--++(-135:0.4) 
  --++(-45:0.4) coordinate (right4) --++(-135:0.4)  
  --++(-45:0.4)  coordinate (right5)
  --++(-135:0.4)  
coordinate (bottom5)    --++(135:0.4)    
  --++(-135:0.4)  
coordinate (bottom4)    --++(135:0.4)    
  --++(-135:0.4)  
coordinate (bottom3)    --++(135:0.4)    
  --++(-135:0.4)  
coordinate (bottom2)    --++(135:0.4)    
  --++(-135:0.4)  
coordinate (bottom1)    --++(135:0.4)    
  --++(-135:0.4)  
coordinate (bottom0)    --++(135:0.4)    
  --++(-135:0.4)  
coordinate (bottomm1)    
--++(135:0.4)   
coordinate (left5)    --++(45:0.4)     
--++(135:0.4)   
coordinate (left4)    --++(45:0.4)     
--++(135:0.4)   
coordinate (left3)    --++(45:0.4)     
--++(135:0.4)   
coordinate (left2)    --++(45:0.4)     
 --++(135:0.4)   
coordinate (left1)    --++(45:0.4)     
coordinate (topm1)
--++(-45:0.4)    --++(45:0.4);

 \path(top00)--++(-135:0.4) coordinate (PP);
 
 \path(top3)    --++(-45:0.2)  --++(45:0.2) coordinate (top45);
  \path(bottom3)    --++(-45:0.2)  --++(45:0.2) coordinate (bottom45);
 
\draw[gridstyle]
	(left5) to ++(-45:0.4)
	(left4) to (bottom0)
 	(left3) to (bottom1)
 	(left2) to (bottom2)  
	(left1) to (bottom3) 
	(topm1) to (bottom4)
	(top00) to (bottom5)
	(top1) to (right5)
	(top2) to (right4) 
	(top3) to (right3)
	(top4) to (right2)
	(top5) to (right1);

		\draw[gridstyle](topm1)--++(-135:0.4); 
		\draw[gridstyle](top00)--++(-135:0.4*3); 
		\draw[gridstyle](top1)--++(-135:0.4*5);
		\draw[gridstyle](top2)--++(-135:0.4*7);
		\draw[gridstyle](top3)--++(-135:0.4*9);
\draw[gridstyle](top4)--++(-135:0.4*10);
\draw[gridstyle](top5)--++(-135:0.4*10);
\draw[gridstyle](right1)--++(-135:0.4*9);
\draw[gridstyle](right2)--++(-135:0.4*7);
\draw[gridstyle](right3)--++(-135:0.4*5);
\draw[gridstyle](right4)--++(-135:0.4*3);
\draw[gridstyle](right5)--++(-135:0.4*1);

\path(0,0)--++(45:0.8)--++(-45:1.2) coordinate (X);

\fill[opacity=0.3,cyan] (X)--++(-135:0.4*3)  --++(-45:0.4*3)
 --++(45:0.4) --++(135:0.4)  --++(45:0.4) --++(135:0.4) --++(45:0.4) --++(135:0.4);

\path (X)--++(135:0.4*1) --++(-135:0.4*1) --++(135:0.4*1) --++(-135:0.4*1) 
--++(135:0.4*1) --++(-135:0.4*1) coordinate (JJ)
--++(-45:0.4*3) ;

 \fill[opacity=0.3,cyan](X)--++(135:0.4*1) --++(-135:0.4*3)  coordinate (PP)--++(-45:0.4*1) 
 ;
 
 \fill[opacity=0.3,magenta](PP)--++(135:0.4*2) --++(45:0.4*1) --++(-45:0.4*2) 
 --++(-135:0.4*1) ;

 \fill[ magenta,opacity=0.4](JJ)--++(-135:0.4)--++(-45:0.4)--++(45:0.4)--++(135:0.4);;

 \path(JJ)--++(-45:0.2)--++(-135:0.2)--++(-90:0.025) node {\scalefont{0.9}$s_1$};

 \path(JJ)--++(-45:0.2)--++(-135:0.2)--++(-90:0.025) --++(135:0.4)node {\scalefont{0.9}$s_0$};

 \fill[ magenta,opacity=0.4](JJ)
 --++(135:0.4)
 --++(-135:0.4)--++(-45:0.4)--++(45:0.4)--++(135:0.4);;

 \path(JJ)
 --++(-45:0.4)--++(45:0.4)
 --++(-45:0.2)--++(-135:0.2)--++(-90:0.025) node {\scalefont{0.9}$s_1$};

 \path(JJ)
 --++(45:0.4)
 --++(-45:0.2)--++(-135:0.2)--++(-90:0.025) node {\scalefont{0.9}$s_0$};
 \path(JJ)
 --++(45:0.4) --++(-45:0.4)--++(45:0.4)
 --++(-45:0.2)--++(-135:0.2)--++(-90:0.025) node {\scalefont{0.9}$\times$};


\draw[path1](X)--++(-45:0.4)
--++(-135:0.4)--++(-45:0.4)
--++(-135:0.4)--++(-45:0.4)
--++(-135:0.4)--++(-45:0.4)
--++(-135:0.4)--++(-45:0.15)
coordinate (Y);
\draw[path1,densely dotted](Y)--++(-45:0.4);
\end{scope}
\node[M] at (X) {};
\node[above] at (X) {$\beta$};
\end{tikzpicture}
 \qquad\quad
\begin{tikzpicture} [scale=1.38251]
\begin{scope}
\path (0,0)--++(135:0.4)--++(-135:0.4) --++(135:0.4)--++(-135:0.4)coordinate (sart1);

\path (0,0)--++(45:0.4)--++(-45:0.4)
--++(45:0.4)--++(-45:0.4)
--++(45:0.4)--++(-45:0.4)
--++(45:0.4)--++(-45:0.4) coordinate(pp)
;
\path(pp)--++(-135:0.4*5)--++(-45:0.4*5) coordinate(pp2);
\path (pp2)--++(135:0.4)--++(-135:0.4)
--++(135:0.4)--++(-135:0.4)
--++(135:0.4)--++(-135:0.4)
--++(135:0.4)--++(-135:0.4)--++(135:0.4)--++(-135:0.4)  --++(135:0.4)--++(-135:0.4)coordinate(pp3)
;

\clip (sart1)--(pp)--(pp2)--(pp3);

\path (0,0) coordinate (top00)
 --++(-45:0.4)  --++(45:0.4) coordinate (top1)
  --++(-45:0.4)  --++(45:0.4) coordinate (top2)
   --++(-45:0.4)  --++(45:0.4) coordinate (top3)
    --++(-45:0.4)  --++(45:0.4) coordinate (top4)
     --++(-45:0.4)  --++(45:0.4)coordinate  (top5)
  --++(-45:0.4)   coordinate(right1) --++(-135:0.4) 
  --++(-45:0.4)  coordinate (right2)--++(-135:0.4)  
  --++(-45:0.4)   coordinate (right3)--++(-135:0.4) 
  --++(-45:0.4) coordinate (right4) --++(-135:0.4)  
  --++(-45:0.4)  coordinate (right5)
  --++(-135:0.4)  
coordinate (bottom5)    --++(135:0.4)    
  --++(-135:0.4)  
coordinate (bottom4)    --++(135:0.4)    
  --++(-135:0.4)  
coordinate (bottom3)    --++(135:0.4)    
  --++(-135:0.4)  
coordinate (bottom2)    --++(135:0.4)    
  --++(-135:0.4)  
coordinate (bottom1)    --++(135:0.4)    
  --++(-135:0.4)  
coordinate (bottom0)    --++(135:0.4)    
  --++(-135:0.4)  
coordinate (bottomm1)    
--++(135:0.4)   
coordinate (left5)    --++(45:0.4)     
--++(135:0.4)   
coordinate (left4)    --++(45:0.4)     
--++(135:0.4)   
coordinate (left3)    --++(45:0.4)     
--++(135:0.4)   
coordinate (left2)    --++(45:0.4)     
 --++(135:0.4)   
coordinate (left1)    --++(45:0.4)     
coordinate (topm1)
--++(-45:0.4)    --++(45:0.4)      ;

 \path(top00)--++(-135:0.4) coordinate (PP);
 
 \path(top3)    --++(-45:0.2)  --++(45:0.2) coordinate (top45);
  \path(bottom3)    --++(-45:0.2)  --++(45:0.2) coordinate (bottom45);

 	\draw[gridstyle](left5)--++(-45:0.4);
 	\draw[gridstyle](left4)--(bottom0); 
 	\draw[gridstyle](left3)--(bottom1); 
 	\draw[gridstyle](left2)--(bottom2);  
 	\draw[gridstyle](left1)--(bottom3); 
	\draw[gridstyle](topm1)--(bottom4); 
	\draw[gridstyle](top00)--(bottom5); 
	\draw[gridstyle](top1)--(right5); 
	\draw[gridstyle](top2)--(right4); 
		\draw[gridstyle](top3)--(right3); 
	\draw[gridstyle](top4)--(right2); 
		\draw[gridstyle](top5)--(right1);

		\draw[gridstyle](topm1)--++(-135:0.4); 
		\draw[gridstyle](top00)--++(-135:0.4*3); 
		\draw[gridstyle](top1)--++(-135:0.4*5);
		\draw[gridstyle](top2)--++(-135:0.4*7);
		\draw[gridstyle](top3)--++(-135:0.4*9);
\draw[gridstyle](top4)--++(-135:0.4*10);
\draw[gridstyle](top5)--++(-135:0.4*10);
\draw[gridstyle](right1)--++(-135:0.4*9);
\draw[gridstyle](right2)--++(-135:0.4*7);
\draw[gridstyle](right3)--++(-135:0.4*5);
\draw[gridstyle](right4)--++(-135:0.4*3);
\draw[gridstyle](right5)--++(-135:0.4*1);

\path(0,0)--++(45:0.8)--++(-45:1.2) coordinate (X);
\path (X)--++(45:0.15) coordinate (Y);

\fill[opacity=0.3,cyan] (X)--++(-135:0.4*3) coordinate (JJ)--++(-45:0.4*2)
--++(-135:0.4*1)  --++(-45:0.4*2)
 --++(45:0.4) --++(135:0.4) --++(45:0.4) --++(135:0.4) --++(45:0.4) --++(135:0.4) --++(45:0.4) --++(135:0.4);

\path (X)--++(-135:0.4*4) coordinate (JJ);
  \fill[  opacity=0.25](JJ)
 --++(135:0.4)
  --++(-135:0.4)
   --++(-45:2*0.4)
     --++(45:2*0.4)
       --++(135:1*0.4)
   ;

\path(JJ)
--++(45:0.4) 
--++(-45:0.2)--++(-135:0.2)--++(-90:0.025) node {\scalefont{0.8}$s_{j-1}$};
\path(JJ)
--++(-45:0.4)--++(45:0.4)
--++(-45:0.2)--++(-135:0.2)--++(-90:0.025) node {\scalefont{0.8}$\times$};

\path(JJ)--++(-135:0.4) 
--++(45:0.4) 
--++(-45:0.2)--++(-135:0.2)--++(-90:0.025) node {\scalefont{0.8}$s_{j}$};

\path(JJ)--++(-135:0.4) --++(135:0.4) 
--++(45:0.4) 
--++(-45:0.2)--++(-135:0.2)--++(-90:0.025) node {\scalefont{0.8}$s_{j-1}$};


\draw[path1](Y) to (X)--++(-45:0.4)
--++(-135:0.4)--++(-45:0.4)
--++(-135:0.4)--++(-45:0.4)
--++(-135:0.4)--++(-45:0.4)
--++(-135:0.4)--++(-45:0.15)
coordinate (Z);
\draw[line width = 2.5, path1color, densely dotted](Y)--++(45:0.4);
\draw[line width = 2.5, path1color, densely dotted](Z)--++(-45:0.4);
\end{scope}
\end{tikzpicture}
\]
 
 \caption{
Rewriting cases $(i)$ and $(ii)$ of the inductive step.  
Compare the ``missing'' tile $\times$ here with that in 
\cref{OH-zajj} and notice that it is smaller in the ordering on tiles (and that this is the inductive step).
  }
 \label{OH-zajj2}
 \end{figure}

We now consider case $(i)$. 
Here, using (4.9), we have that 
$$
 \psi_1	\psi_\SSTS e_{\SSTT_{(j,\beta)}} \psi_\SSTT^* = 
 \psi_1 \psi_0 \psi_1 \psi_0	\psi_\SSTU e_{\SSTT_{(j,\beta)}} \psi_\SSTT^* = 
\begin{cases}
(\psi_0 \psi_1 \psi_0 \psi_1\pm 2\psi_0)
  	\psi_\SSTU e_{\SSTT_{(j,\beta)}} \psi_\SSTT^* 

 \\
(\psi_0 \psi_1 \psi_0 \psi_1  \pm \psi_1)
  	\psi_\SSTU e_{\SSTT_{(j,\beta)}} \psi_\SSTT^* 
 \\ 
 \psi_0 \psi_1 \psi_0 \psi_1 
  	\psi_\SSTU e_{\SSTT_{(j,\beta)}} \psi_\SSTT^*
\end{cases}.
 $$
 depending on whether $\stu(1)$ is on a hyperplane or close to one of the marked point $\pm \alpha_i$. Note that
 $  
  \ell(\SSTU) =\ell-3$ and $s _1(\SSTU)$ is  non-standard, so $\psi_1\psi_\stu e_{\stt_{(k,\beta)}}\psi_\stt^*$ and $\psi_0\psi_1\psi_0\psi_1\psi_\stu e_{\stt_{(k,\beta)}}\psi_\stt^*$ belong to $\Delta	_{<\ell} 	(k,\beta) $ by induction. 
Moreover $\psi_0\psi_\stu e_{\stt_{(k,\beta)}}\psi_\stt^*\in \Delta_{\leq \ell - 2}(k, \beta) \subseteq \Delta_{<\ell}(k,\beta)$. 
    \end{proof}

\begin{cor}\label{qhordering}
The  
  algebras   ${\rm TL}_n(\alpha_1,\alpha_2, \extra )$
and 
 $B_n(\kappa)$ are isomorphic. 
The algebra ${\rm TL}_n(\alpha_1,\alpha_2, \extra )$   is a quasi-hereditary 
graded cellular algebra with respect to the basis 
$$
\{\psi_{\sts\stt}:=\psi_\SSTS e_{\SSTT_{(k,\beta)}} \psi_\SSTT^*
\mid 
\SSTS,\SSTT \in \Std_n(k,\beta), 
 (k,\beta) \in \Lambda_n
\}$$the anti-involution $\ast$ and with 
respect to any total refinement of the  partial order $(\Lambda_n,\leq)$.  
\end{cor}
\begin{proof}
The symplectic blob algebra is a quotient of ${\rm TL}_n(\alpha_1,\alpha_2, \extra )$ by \cref{quotientthm}.
The spanning set of \cref{almostcellular}  is of rank equal to that of the blob algebra (to see this, note that the standard tableaux indexing set in \cref{almostcellular} is the same as that 
of the  blob algebra in \cref{prop:calibrated B modules,thm:calibratedTL}).  
Therefore the quotient map is  an isomorphism and the spanning set of \cref{almostcellular} is, in fact, a basis.  
The anti-involution map is compatible with the cellular structure, by definition. 
The other  axioms of cellularity were verified in \cref{almostcellular}.
That the algebra is quasi-hereditary follows since each cellular ideal is generated by an idempotent (by definition).
\end{proof}

Given $\la \in \Lambda_n$, we define the following left cell ideals in 
${\rm TL}_n(\alpha_1,\alpha_2,\vartheta)$.
\begin{align*}
{\rm TL}_n(\alpha_1,\alpha_2,\vartheta)^{\leq  \la}   
&=  
{\rm TL}_n(\alpha_1,\alpha_2,\vartheta)
e_{\stt_{\la}}
\\ 
{\rm TL}_n(\alpha_1,\alpha_2,\vartheta)^{<\la}  &=   
{\rm TL}_n(\alpha_1,\alpha_2,\vartheta)^{\leq  \la} \cap
\Bbbk \{ \psi_{\sts \stt} \mid 
\sts,\stt \in \Std_n  ( \mu ), \mu < \la \}.
\end{align*}  
The left cell module $\Delta_\Bbbk (\la)$ is given by
$$\Delta_\Bbbk (\la) = {\rm TL}_n(\alpha_1,\alpha_2,\vartheta)^{\leq  \la}  / {\rm TL}_n(\alpha_1,\alpha_2,\vartheta)^{<\la}.$$
By \cref{qhordering}, the cell module has a basis given by 
$$\{ \psi_\sts e_{\stt_\la} + {\rm TL}_n(\alpha_1,\alpha_2,\vartheta)^{<\la} \, | \, \sts \in \Std_n(\la)\}.$$
We abuse notation and write $\psi_\sts$ for $\psi_\sts e_{\stt_\la} + {\rm TL}_n(\alpha_1,\alpha_2,\vartheta)^{<\la} \in \Delta_{\Bbbk}(\la)$.
  We now recall how this  graded  cellular structure allows us to construct the graded 
   simple ${\rm TL}_n(\alpha_1,\alpha_2,\vartheta)$-modules. 
   For each
 $\la \in \Lambda_n$,  we define     a bilinear form  
  $ \langle - ,- \rangle ^{   \la}$ 
  on $ \Delta_\Bbbk(\la)  $ as follows
\begin{equation}\label{geoide}
   \psi  _{\stt_\la \sts}\psi _{\stt \stt_\la } \equiv
  \langle     \psi  _{\sts}, \psi _{\stt} 
  \rangle ^{   \la} \;  
e_{\stt_\la}\pmod{{\rm TL}_n(\alpha_1,\alpha_2,\vartheta)^{ < \lambda}}
  \end{equation}
for any $\sts,\stt \in \Std_n(\lambda  )$.  
Let   $\Bbbk$  be an arbitrary  field of   characteristic not equal to 2.    
Factoring out by the radicals of these forms,  we obtain a complete set of non-isomorphic simple $ {\rm TL}_n(\alpha_1,\alpha_2,\vartheta)$-modules 
   $$ 
  L_\Bbbk(\lambda) =
 \Delta_\Bbbk(\lambda) /
  \rad( \Delta_\Bbbk(\lambda) ),  
 $$for $\la \in \Lambda_n$.

\section{A conjectural  LLT-style algorithm for  graded decomposition matrices}

The original LLT conjecture was phrased in the language of Fock spaces of quantum groups
 \cite{LLT};  the orientifold quiver Temperley--Lieb algebras of this paper have no clear 
 connection to quantum groups (although there are connections to $\imath$quantum groups \cite{MR4666131}) and so the reader might wonder how one can possibly generalise the LLT conjecture to our setting.  
In \cite{KN10}, Kleshchev--Nash observed that every aspect of the original LLT theory
  can be reinterpreted 
 in an entirely elementary fashion within the language of graded tableaux combinatorics, thus recasting the LLT algorithm
 as a natural calculation within Hu--Mathas's graded cellular basis of $\Bbbk\mathfrak{S}_n$ \cite{hm10}. 
 In this paper we have provided precise orientifold analogues of the graded tableaux/paths and graded cellular bases of \cite{hm10,KN10}; in this section  we reap the rewards of our  graded tableaux/paths/cellular basis construction by providing a direct analogue of LLT/Kleshchev--Nash's algorithm and by proving (characteristic-free!) bounds on the dimensions of simple modules.
 
 We verify our conjecture for two important cases: $(i)$ the case where 
 $q$ is not a root of unity (here we see that any block has at most 5 simple modules)
 and $(ii)$ the case where $q$ is a root of unity, but  
 $\alpha_1,\alpha_2\not \in q^{\ZZ}$ (here we see that a block can have arbitrarily many simple modules as $n$ gets large).  These two cases are in some sense orthogonal to one another.
In case $(i)$ we will also see that the naive generalisation of the Nakayama conjecture fails for the orientifold Temperley--Lieb algebras. 
 
We note that there has been some work done already on the (ungraded) decomposition numbers for the symplectic blob algebra over $\CC$. Specifically, the non-zero homomorphisms between certain standard modules constructed in \cite{statmexblob2} imply that some decomposition numbers are non-zero.

\subsection{The conjecture} 
Rather than computing in a Fock space, we will work with the following tableaux-theoretic 
objects.

\begin{defn}For each pair  $\la=(\beta_1,k_1), \mu=(\beta_2,k_2),  \in \Lambda_n$ we  
 define
${\sf CStd}_n(\la, \mu)$ to be the set of standard tableaux $\sts\in \Std_n(\la)$  such that
$ \res(\sts)=\res(\stt_\mu)$ and  we refer to these as the 
 {\sf $\mu$-coloured}  standard tableaux of {\sf shape} $\la$.
 \end{defn}

Motivated by the analogous situation for LLT/Kazhdan--Lusztig theory, we consider the  tableaux-counting polynomials
\begin{align}\label{sghjdhjhgskjdhjhgdhjksghd}\textstyle 
\dim_v( e_{\stt_\mu}\Delta (\la)  )= 
  \sum_{\sts \in {\sf CStd}_n(\la,\mu)}v^{\deg(\sts)} .  
 \end{align} 
 We observe that the matrix 
$$
(\Delta_{\la,\mu})_{\la,\mu \in \Lambda_n}\qquad 
\Delta_{\la,\mu}:=\dim_v( e_{\stt_\mu}\Delta (\la)  )
$$is lower uni-triangular with respect to the natural ordering on $\Lambda_n$ in decreasing order.  
 Now, since  $e_{\stt_{\mu}}$ generates the simple module $L(\mu)$ we have that $[\Delta(\la):L(\mu)]\neq 0$ implies that $\Delta_{\la, \mu}\neq 0$. This implies that the equivalence classes of the equivalence relation on $\Lambda_n$ generated by $\Delta_{\la, \mu}\neq 0$, which we call $\Delta$-equivalence classes, are unions of blocks. Thus to describe the decomposition matrix for ${\rm TL}_n(\alpha_1, \alpha_2, \extra)$, it is enough to describe the submatrices corresponding to each $\Delta$-equivalence class.

Now, the matrix $\Delta$ can be factorised {\em uniquely} as a product
  $\Delta  = NA $
of lower uni-triangular matrices
\[N :=(n_{\la,\nu}(v))_{\la  , \mu \in  \Lambda_n}
\qquad
A :=(a_{\nu,\mu}(v))
_{\nu  ,\mu  \in  \Lambda_n}
\]
such that $n_{\la,\nu}(v)\in v\ZZ[v]$
for $\la \neq \nu$
and
$a_{\nu,\mu}(v)\in  \ZZ[v+v^{-1}]$.
A recursive algorithm for this matrix factorisation is given by  setting
$a_{\la,\la}(v)=1=n_{\la,\la}(v)$ and
  defining the polynomials
\[
a_{\la,\mu}(v)\in \ZZ [v+v^{-1}] \qquad
n_{\la,\mu}(v)\in v\ZZ [v]
\]
  by induction on the  order $\leq$ as follows
\begin{equation*}
a_{\la,\mu}(v) +n_{\lambda,\mu}(v) 
=
{ \sum}_{
\begin{subarray}c
 \sts  \in {\rm  CStd}(\la,  \mu)
\end{subarray}
}
  {  v^{\deg(\sts)}} 
  -
 {\sum}_{
\la <   \nu <    \mu
}     n_{\lambda,\nu}(v)
a_{\nu,\mu}(v).
\end{equation*}
We define  $n_{\la,\mu}(v)$  to be the  {\sf orientifold LLT polynomial} associated to $\la,\mu\in \Lambda_n$.

\begin{conj}\label{conjforus}
Over the complex field, the graded decomposition numbers of the orientifold quiver Temperley--Lieb algebra  are 
given by the orientifold LLT polynomials 
$$
\textstyle \sum_{k\in \ZZ}[\Delta(\la): L(\mu)\langle k \rangle ] v^k = n_{\la,\mu}(v)
$$for $\la,\mu \in \Lambda_n$.

\end{conj}

\begin{rmk} The  complex  graded decomposition matrices for the classical Temperley--Lieb algebras and 
 for  one-boundary Temperley--Lieb algebras (the ``blob algebras'')
 were calculated in \cite{PR13,Pla13}.  
  In both cases, the answer is given in terms of (parabolic) Kazhdan--Lusztig polynomials of type $\widehat{A}_1$, which can be 
  calculated/defined via usual (non-orientifold) tableaux/path combinatorics in the exact same fashion as above, see for example \cite{KN10} and 
  \cite[Chapter 12]{MR4911527}.
  The complex graded decomposition matrices of {\em generalised  blob algebras} 
 were conjectured in \cite{MW00} and proven in   \cite{MR4401509}.   
 In this  case  the answer is given in terms of maximal parabolic  Kazhdan--Lusztig polynomials of type $\widehat{A}_\ell$ which can be 
  calculated/defined   in the exact same fashion as above.  
Our 
 \cref{conjforus} is   heavily  inspired   and motivated by  the theorems of  \cite{PR13,Pla13,MR4401509},
  however unlike in these classical cases our    {\em orientifold} LLT polynomial have no obvious interpretation in terms of Kazhdan--Lusztig theory. 
Given the emerging connections between orientifold quiver Hecke algebras and $\imath$quantum groups
(see for example \cite{MR4666131}) it seems natural to ask whether our orientifold LLT polynomials can be 
interpreted in the language of $\imath$Kazhdan--Lusztig theory of  \cite{MR3864017}, however this is pure speculation at this point in time.  

\end{rmk}

\subsection{Bounds on dimensions of simple modules  }
The following 
observation is quite standard within the theory of graded cellular algebras, but it is also powerful and useful in what follows.

\begin{prop}\label{lad-bound}
Let $\Bbbk$ be a field of   characteristic not equal to 2. 
We have the following bounds on the dimensions of simple modules 
\begin{equation}\label{ladderlem}
\dim(L_\Bbbk(\la))\geq 
\sharp
\{
\sts\in \Std_n(\la) \mid \sts \sim \stt \in \Std_n(\la) \text{ with }{\sf max}(\stt )=\la 
\}
\end{equation}
\end{prop}

\begin{proof}
We have that $e_\sts=e_\stt $ for $\sts,\stt \in \Std_n(\la)$ as in \cref{ladderlem}. 
We note that $  e_\stt\Delta(\mu) =0$ for $\mu >\la$ by \cref{ladderssss}. 
We further note that $[\Delta(\la) :L(\mu)]=0$ for $\mu \not >\la$ by 
\cref{qhordering}.  Putting these two facts together, we immediately deduce that 
\begin{equation}\label{forloic} \dim (e_\stt L(\la) )= \dim (e_\stt \Delta(\la)) = \sharp
\{
\sts\in \Std_n(\la) \mid \sts \sim \stt
\}
\end{equation}and the result follows.
\end{proof}

\subsection{Decomposition matrices 
for $\Bbbk$ arbitrary (of characteristic different from $2$), with  $q$  a root of unity  and 
 $\alpha_1,\alpha_2,\vartheta \not \in q^{\ZZ}$  }
In this case there are no hyperplanes and therefore none of the tiles in our lattice have negative degree.  
Thus the $\Bbbk$-algebra ${\rm TL}_n(\alpha_1,\alpha_2,\vartheta)$ is non-negatively graded for an arbitrary  field $\Bbbk$.  
We will first illustrate what happens in an example and then outline the general case.

\begin{eg}\label{aneasypeasyeg}
We let $\xi$ and $i$ denote primitive    5th and 4th roots of unity. 
We set $q^2=\xi$,  $\alpha_1= i $,  $\alpha_2=i\xi^2$, and $\vartheta= i \xi^3$.  
We colour these roots of unity as   illustrated in \cref{dsakjhlgajlkdfgh} below.

\noindent We shall consider  the $\Delta$-equivalence class given by 
  $$\{(16,\alpha_1 ), (12,\alpha_2 ),  (6,\alpha_1 ),(2,\alpha_2 ),
  (16,\alpha_1^{-1} ) , (10,\alpha_2^{-1}),(6,\alpha_1^{-1}),( 0,\vartheta)\}$$%
  The corresponding  block of the  matrix 
$(\Delta_{\la,\mu})_{\la,\mu \in \Lambda_n}$   is   as follows:
$$\def\arraystretch{1.1}
 \begin{array}{c|cccc|ccc|c}
&(16,\alpha_1)	&(12,\alpha_2)			 &(6,\alpha_1)	&(2,\alpha_2)	
	&(16,\alpha_1^{-1})	&(10,\alpha_2^{-1})		&(6,\alpha_1^{-1})	 		&(0,\vartheta)
\\
  \hline 
(16,\alpha_1)	&			1 & 0  & 0  & 0  & 0  & 0  & 0  & 0   \\	
(12,\alpha_2)	&			 v & 1  & 0  & 0  & 0  & 0  & 0 & 0     \\
(6,\alpha_1 )				&v^2 & v & 1  & 0  & 0  & 0  & 0  & 0    \\
(2,\alpha_2  )				&  v^3 & v^2 & v  & 1  & 0  & 0  & 0  & 0     \\
   \hline
(16,\alpha_1^{-1} )				&0 & 0  & 0  & 0  			& 1  & 0  & 0  & 0    \\
(10,\alpha_2 ^{-1})				&  0 & 0 & 0  & 0  				& v  & 1  & 0  & 0      \\
(6,\alpha_1 ^{-1} )				&0 & 0  &0  & 0  				& v^2  & v  & 1  & 0    \\
   \hline
(0,\vartheta)				& v^4 & v^3 &  v^2  & v  &  v^3  & v^2  & v     & 1   \\
\end{array} 
$$%
This is  not difficult to calculate, for example  the  non-zero entries of the  first column   
correspond to the 1-dimensional spaces spanned by the tableaux in \cref{uisdfuyigfduiogfuiogdfsuoigdfs}.  
These tableaux are simply obtained by ``translating the marked point rightwards''.  All other non-zero entries can be calculated in a similar manner.  
By definition 
$$(\Delta_{\la,\mu})_{\la,\mu \in \Lambda_n}=  (N_{\la,\mu})_{\la,\mu \in \Lambda_n} 
\qquad  (A_{\la,\mu})_{\la,\mu \in \Lambda_n} = {\rm Id} _{  \Lambda_n}.  $$The conjecture   is  true  in this case as the off-diagonal entries 
of $(\Delta_{\la,\mu})_{\la,\mu \in \Lambda_n}$ are all of strictly positive degree  (and the characters of graded simple modules must belong to $\ZZ_{\geq0}[q+q^{-1}]$ by graded cellularity, for a more detailed discussion we refer to \cite[Section 6.7]{MR4911527}).  We note in particular, that this $\Delta$-equivalence class is in fact a block of ${\rm TL}_n(\alpha_1, \alpha_2, \extra)$. 

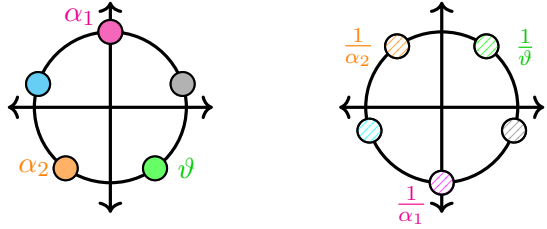
\begin{figure}[ht!]

 $$
 \begin{minipage}{2cm}
 \begin{tikzpicture} 
\clip(-1.75,-1.7) rectangle (1.75,1.7);
 \draw[very thick] (0,0) circle (1cm);
 \draw[very thick, <->](0,-1.4)--(0,1.4);
  \draw[very thick, <->]( -1.35,0)--(1.35,0);
   \draw[thick, fill= white](90:1) circle (4.5pt);
   \draw[thick, fill= magenta,opacity=0.6](90:1) circle (4.5pt);
      \draw[thick,  fill= white](90+72:1) circle (4.5pt);
      \draw[thick, fill= cyan,opacity=0.6](90+72:1) circle (4.5pt);

      \draw[thick, fill= white](90+72*2:1) circle (4.5pt);
      \draw[thick, fill= orange,opacity=0.6](90+72*2:1) circle (4.5pt);
      
            \draw[thick, fill= white](90+72*3:1) circle (4.5pt);
      \draw[thick, fill= green,opacity=0.6](90+72*3:1) circle (4.5pt);

      \draw[thick, fill= white](90+72*4:1) circle (4.5pt);
      \draw[thick, fill= gray,opacity=0.6](90+72*4:1) circle (4.5pt);

\draw(-0.4,1.2) node {$\color{magenta}\alpha_1$};
\draw(-1,-0.8) node {$\color{orange}\alpha_2$};
\draw(1,-0.8) node {$\color{green!80!black}\vartheta$};


 \end{tikzpicture}\end{minipage}
 \qquad \qquad \qquad
 \begin{minipage}{2cm}
 \begin{tikzpicture} [yscale=1,scale=0.5]
\clip(-1.75*2,-1.7*2) rectangle (1.75*2,1.7*2);
 \draw[very thick] (0,0) circle (1*2cm);
 \draw[very thick, <->](0,-1.4*2)--(0,1.4*2);
  \draw[very thick, <->]( -1.35*2,0)--(1.35*2,0);

   \draw[thick, fill= white](-90:1*2) circle (9pt);
     
    \draw[pattern=north east lines, pattern color=magenta, opacity=0.8,   thick] (-90:1*2) circle (9pt);

      \draw[thick, fill= white](-90+72:1*2) circle (9pt);
    \draw[pattern=north east lines, pattern color=gray, opacity=0.8,   thick] (-90+72:1*2) circle (9pt);

      \draw[thick, fill= white](-90+72*2:1*2) circle (9pt);
    \draw[pattern=north east lines, pattern color=green, opacity=0.8,   thick] (-90+72*2:1*2) circle (9pt);
      
            \draw[thick, fill= white](-90+72*3:1*2) circle (9pt);
    \draw[pattern=north east lines, pattern color=orange, opacity=0.8,   thick] (-90+72*3:1*2) circle (9pt);

      \draw[thick, fill= white](-90+72*4:1*2) circle (9pt);
    \draw[pattern=north east lines, pattern color=cyan, opacity=0.8,   thick] (-90+72*4:1*2) circle (9pt);

\draw(-0.4*2,-1.3*2) node {$\color{magenta}\tfrac{1}{\alpha_1}$};
\draw(-1.1*2,0.8*2) node {$\color{orange}\tfrac{1}{\alpha_2}$};
\draw(1.1*2,0.8*2) node {$\color{green!80!black}\tfrac{1}{\vartheta}$};

 \end{tikzpicture}\end{minipage}\quad$$

 \vspace{-0.25cm}
  \caption{The residues and  special points $\alpha_1,\alpha_2,\vartheta$
   and their inverses in \cref{aneasypeasyeg}.  On the left (respectively right) 
   we depict $i\xi^k$ (and $-i\xi^k$) for $0\leq k \leq 4$. 
}
  \label{dsakjhlgajlkdfgh}
 \end{figure}

\end{eg}

 \vspace{-0.4cm}
\begin{figure}[ht!]
 \[\begin{array}{cc}
  \begin{minipage}{12.5cm}\scalefont{0.9}
 \begin{tikzpicture}[scale=0.75]
 \path(0,0) coordinate (X);
 \fill[magenta,opacity=0.4] (X)--++(0:1)--++(-90:1)--++(180:1);
 \path(5,0) coordinate (X);
 \fill[magenta,opacity=0.4] (X)--++(0:1)--++(-90:1)--++(180:1);
 \path(10,0) coordinate (X);
 \fill[magenta,opacity=0.4] (X)--++(0:1)--++(-90:1)--++(180:1);
 \path(15,0) coordinate (X);
 \fill[magenta,opacity=0.4] (X)--++(0:1)--++(-90:1)--++(180:1);
  \path(1,0) coordinate (X);
 \fill[cyan,opacity=0.4] (X)--++(0:1)--++(-90:1)--++(180:1);
 \path(6,0) coordinate (X);
 \fill[cyan,opacity=0.4] (X)--++(0:1)--++(-90:1)--++(180:1);
 \path(11,0) coordinate (X);
 \fill[cyan,opacity=0.4] (X)--++(0:1)--++(-90:1)--++(180:1);
  \path(1+1,0) coordinate (X);
 \fill[orange,opacity=0.4] (X)--++(0:1)--++(-90:1)--++(180:1);
 \path(1+6,0) coordinate (X);
 \fill[orange,opacity=0.4] (X)--++(0:1)--++(-90:1)--++(180:1);
 \path(1+11,0) coordinate (X);
 \fill[orange,opacity=0.4] (X)--++(0:1)--++(-90:1)--++(180:1);
   \path(2+1,0) coordinate (X);
 \fill[green,opacity=0.4] (X)--++(0:1)--++(-90:1)--++(180:1);
 \path(2+6,0) coordinate (X);
 \fill[green,opacity=0.4] (X)--++(0:1)--++(-90:1)--++(180:1);
 \path(2+11,0) coordinate (X);
 \fill[green,opacity=0.4] (X)--++(0:1)--++(-90:1)--++(180:1);
 \path(3+1,0) coordinate (X);
 \fill[gray,opacity=0.4] (X)--++(0:1)--++(-90:1)--++(180:1);
 \path(3+6,0) coordinate (X);
 \fill[gray,opacity=0.4] (X)--++(0:1)--++(-90:1)--++(180:1);
 \path(3+11,0) coordinate (X);
 \fill[gray,opacity=0.4] (X)--++(0:1)--++(-90:1)--++(180:1);
 \draw[thick,path1color](0,0) --++(0:1) coordinate (x1)
  --++(0:1) coordinate (x2)
   --++(0:1) coordinate (x3)
    --++(0:1) coordinate (x4)
     --++(0:1) coordinate (x5)
      --++(0:1) coordinate (x6)
       --++(0:1) coordinate (x7)
        --++(0:1) coordinate (x8)
               --++(0:1) coordinate (x9)
        --++(0:1) coordinate (x10)
               --++(0:1) coordinate (x11)
                      --++(0:1) coordinate (x12)
                             --++(0:1) coordinate (x13)
                                    --++(0:1) coordinate (x14)
                                           --++(0:1) coordinate (x15)
                                                  --++(0:1) coordinate (x16)            
                                                                            --++(-90:1)--++(180:16)--++(90:1);
 \foreach \i in {1,2,3,4,5,6,7,8,...,16}
 {
 \draw[thick,path1color](x\i)--++(-90:1);
 \path(x\i)--++(-90:0.5)--++(180:0.5) coordinate (y\i);
 }
 \path(y8) node {$8$};
 \path(y7) node {$7$};
 \path(y6) node {$6$};
 \path(y5) node {$5$};
 \path(y4) node {$4$};
 \path(y3) node {$3$};
 \path(y2) node {$2$};
 \path(y1) node {$1$};

 \path(y9) node {$9$};
 \path(y10) node {$10$};
 \path(y11) node {$11$};
 \path(y12) node {$12$};
 \path(y13) node {$13$};
 \path(y14) node {$14$};
 \path(y15) node {$15$};
 \path(y16) node {$16$};
 \path(0,0)--++(90:0.4)--++(180:0.25) node {$\alpha_1$};
  \node[M] at (0,0){};
\draw[   thick,,densely dotted] (x8)--++(-90:1.35) ;
\draw[   thick,,densely dotted] (x8)--++(90:0.35) ;

\end{tikzpicture}\end{minipage}
\\ \vspace{-0.2cm}

  \begin{minipage}{12.5cm}\scalefont{0.9}
 \begin{tikzpicture}[scale=0.75]
 \path(0,0) coordinate (X);
 \fill[magenta,opacity=0.4] (X)--++(0:1)--++(-90:1)--++(180:1);
 \path(5,0) coordinate (X);
 \fill[magenta,opacity=0.4] (X)--++(0:1)--++(-90:1)--++(180:1);
 \path(10,0) coordinate (X);
 \fill[magenta,opacity=0.4] (X)--++(0:1)--++(-90:1)--++(180:1);
 \path(15,0) coordinate (X);
 \fill[magenta,opacity=0.4] (X)--++(0:1)--++(-90:1)--++(180:1);
  \path(1,0) coordinate (X);
 \fill[cyan,opacity=0.4] (X)--++(0:1)--++(-90:1)--++(180:1);
 \path(6,0) coordinate (X);
 \fill[cyan,opacity=0.4] (X)--++(0:1)--++(-90:1)--++(180:1);
 \path(11,0) coordinate (X);
 \fill[cyan,opacity=0.4] (X)--++(0:1)--++(-90:1)--++(180:1);
  \path(1+1,0) coordinate (X);
 \fill[orange,opacity=0.4] (X)--++(0:1)--++(-90:1)--++(180:1);
 \path(1+6,0) coordinate (X);
 \fill[orange,opacity=0.4] (X)--++(0:1)--++(-90:1)--++(180:1);
 \path(1+11,0) coordinate (X);
 \fill[orange,opacity=0.4] (X)--++(0:1)--++(-90:1)--++(180:1);
   \path(2+1,0) coordinate (X);
 \fill[green,opacity=0.4] (X)--++(0:1)--++(-90:1)--++(180:1);
 \path(2+6,0) coordinate (X);
 \fill[green,opacity=0.4] (X)--++(0:1)--++(-90:1)--++(180:1);
 \path(2+11,0) coordinate (X);
 \fill[green,opacity=0.4] (X)--++(0:1)--++(-90:1)--++(180:1);
 \path(3+1,0) coordinate (X);
 \fill[gray,opacity=0.4] (X)--++(0:1)--++(-90:1)--++(180:1);
 \path(3+6,0) coordinate (X);
 \fill[gray,opacity=0.4] (X)--++(0:1)--++(-90:1)--++(180:1);
 \path(3+11,0) coordinate (X);
 \fill[gray,opacity=0.4] (X)--++(0:1)--++(-90:1)--++(180:1);
 \draw[thick,path1color](0,0) --++(0:1) coordinate (x1)
  --++(0:1) coordinate (x2)
   --++(0:1) coordinate (x3)
    --++(0:1) coordinate (x4)
     --++(0:1) coordinate (x5)
      --++(0:1) coordinate (x6)
       --++(0:1) coordinate (x7)
        --++(0:1) coordinate (x8)
               --++(0:1) coordinate (x9)
        --++(0:1) coordinate (x10)
               --++(0:1) coordinate (x11)
                      --++(0:1) coordinate (x12)
                             --++(0:1) coordinate (x13)
                                    --++(0:1) coordinate (x14)
                                           --++(0:1) coordinate (x15)
                                                  --++(0:1) coordinate (x16)            
                                                                            --++(-90:1)--++(180:16)--++(90:1);
 \foreach \i in {1,2,3,4,5,6,7,8,...,16}
 {
 \draw[thick,path1color](x\i)--++(-90:1);
 \path(x\i)--++(-90:0.5)--++(180:0.5) coordinate (y\i);
 }
 \path(y8) node {$8$};
 \path(y7) node {$7$};
 \path(y6) node {$6$};
 \path(y5) node {$5$};
 \path(y4) node {$4$};
 \path(y3) node {$3$};
 \path(y2) node {$2$};
 \path(y1) node {$1$};

 \path(y9) node {$9$};
 \path(y10) node {$10$};
 \path(y11) node {$11$};
 \path(y12) node {$12$};
 \path(y13) node {$13$};
 \path(y14) node {$14$};
 \path(y15) node {$15$};
 \path(y16) node {$16$};
 \path(0,0)--++(90:0.4)--++(180:0.25) node {$\phantom{\alpha_2}$};

 \path(x2)--++(90:0.4)--++(180:0.25) node {$\alpha_2$};
  \node[M] at (x2){};
\draw[   thick,,densely dotted] (x8)--++(-90:1.35) ;
\draw[   thick,,densely dotted] (x8)--++(90:0.35) ;

\end{tikzpicture}\end{minipage}
\\ \vspace{-0.1cm}

   \begin{minipage}{12.5cm}\scalefont{0.9}
 \begin{tikzpicture}[scale=0.75]
 \path(0,0) coordinate (X);
 \fill[magenta,opacity=0.4] (X)--++(0:1)--++(-90:1)--++(180:1);
 \path(5,0) coordinate (X);
 \fill[magenta,opacity=0.4] (X)--++(0:1)--++(-90:1)--++(180:1);
 \path(10,0) coordinate (X);
 \fill[magenta,opacity=0.4] (X)--++(0:1)--++(-90:1)--++(180:1);
 \path(15,0) coordinate (X);
 \fill[magenta,opacity=0.4] (X)--++(0:1)--++(-90:1)--++(180:1);
  \path(1,0) coordinate (X);
 \fill[cyan,opacity=0.4] (X)--++(0:1)--++(-90:1)--++(180:1);
 \path(6,0) coordinate (X);
 \fill[cyan,opacity=0.4] (X)--++(0:1)--++(-90:1)--++(180:1);
 \path(11,0) coordinate (X);
 \fill[cyan,opacity=0.4] (X)--++(0:1)--++(-90:1)--++(180:1);
  \path(1+1,0) coordinate (X);
 \fill[orange,opacity=0.4] (X)--++(0:1)--++(-90:1)--++(180:1);
 \path(1+6,0) coordinate (X);
 \fill[orange,opacity=0.4] (X)--++(0:1)--++(-90:1)--++(180:1);
 \path(1+11,0) coordinate (X);
 \fill[orange,opacity=0.4] (X)--++(0:1)--++(-90:1)--++(180:1);
   \path(2+1,0) coordinate (X);
 \fill[green,opacity=0.4] (X)--++(0:1)--++(-90:1)--++(180:1);
 \path(2+6,0) coordinate (X);
 \fill[green,opacity=0.4] (X)--++(0:1)--++(-90:1)--++(180:1);
 \path(2+11,0) coordinate (X);
 \fill[green,opacity=0.4] (X)--++(0:1)--++(-90:1)--++(180:1);
 \path(3+1,0) coordinate (X);
 \fill[gray,opacity=0.4] (X)--++(0:1)--++(-90:1)--++(180:1);
 \path(3+6,0) coordinate (X);
 \fill[gray,opacity=0.4] (X)--++(0:1)--++(-90:1)--++(180:1);
 \path(3+11,0) coordinate (X);
 \fill[gray,opacity=0.4] (X)--++(0:1)--++(-90:1)--++(180:1);
 \draw[thick,path1color](0,0) --++(0:1) coordinate (x1)
  --++(0:1) coordinate (x2)
   --++(0:1) coordinate (x3)
    --++(0:1) coordinate (x4)
     --++(0:1) coordinate (x5)
      --++(0:1) coordinate (x6)
       --++(0:1) coordinate (x7)
        --++(0:1) coordinate (x8)
               --++(0:1) coordinate (x9)
        --++(0:1) coordinate (x10)
               --++(0:1) coordinate (x11)
                      --++(0:1) coordinate (x12)
                             --++(0:1) coordinate (x13)
                                    --++(0:1) coordinate (x14)
                                           --++(0:1) coordinate (x15)
                                                  --++(0:1) coordinate (x16)            
                                                                            --++(-90:1)--++(180:16)--++(90:1);
 \foreach \i in {1,2,3,4,5,6,7,8,...,16}
 {
 \draw[thick,path1color](x\i)--++(-90:1);
 \path(x\i)--++(-90:0.5)--++(180:0.5) coordinate (y\i);
 }
 \path(y8) node {$8$};
 \path(y7) node {$7$};
 \path(y6) node {$6$};
 \path(y5) node {$5$};
 \path(y4) node {$4$};
 \path(y3) node {$3$};
 \path(y2) node {$2$};
 \path(y1) node {$1$};

 \path(y9) node {$9$};
 \path(y10) node {$10$};
 \path(y11) node {$11$};
 \path(y12) node {$12$};
 \path(y13) node {$13$};
 \path(y14) node {$14$};
 \path(y15) node {$15$};
 \path(y16) node {$16$};
 \path(0,0)--++(90:0.4)--++(180:0.25) node {$\phantom{\alpha_2}$};

 \path(x5)--++(90:0.4)--++(180:0.25) node {$\alpha_1$};
  \node[M] at (x5){};
\draw[   thick,,densely dotted] (x8)--++(-90:1.35) ;
\draw[   thick,,densely dotted] (x8)--++(90:0.35) ;

\end{tikzpicture}\end{minipage}
\\ \vspace{-0.125cm}

   \begin{minipage}{12.5cm}\scalefont{0.9}
 \begin{tikzpicture}[scale=0.75]
 \path(0,0) coordinate (X);
 \fill[magenta,opacity=0.4] (X)--++(0:1)--++(-90:1)--++(180:1);
 \path(5,0) coordinate (X);
 \fill[magenta,opacity=0.4] (X)--++(0:1)--++(-90:1)--++(180:1);
 \path(10,0) coordinate (X);
 \fill[magenta,opacity=0.4] (X)--++(0:1)--++(-90:1)--++(180:1);
 \path(15,0) coordinate (X);
 \fill[magenta,opacity=0.4] (X)--++(0:1)--++(-90:1)--++(180:1);
  \path(1,0) coordinate (X);
 \fill[cyan,opacity=0.4] (X)--++(0:1)--++(-90:1)--++(180:1);
 \path(6,0) coordinate (X);
 \fill[cyan,opacity=0.4] (X)--++(0:1)--++(-90:1)--++(180:1);
 \path(11,0) coordinate (X);
 \fill[cyan,opacity=0.4] (X)--++(0:1)--++(-90:1)--++(180:1);
  \path(1+1,0) coordinate (X);
 \fill[orange,opacity=0.4] (X)--++(0:1)--++(-90:1)--++(180:1);
 \path(1+6,0) coordinate (X);
 \fill[orange,opacity=0.4] (X)--++(0:1)--++(-90:1)--++(180:1);
 \path(1+11,0) coordinate (X);
 \fill[orange,opacity=0.4] (X)--++(0:1)--++(-90:1)--++(180:1);
   \path(2+1,0) coordinate (X);
 \fill[green,opacity=0.4] (X)--++(0:1)--++(-90:1)--++(180:1);
 \path(2+6,0) coordinate (X);
 \fill[green,opacity=0.4] (X)--++(0:1)--++(-90:1)--++(180:1);
 \path(2+11,0) coordinate (X);
 \fill[green,opacity=0.4] (X)--++(0:1)--++(-90:1)--++(180:1);
 \path(3+1,0) coordinate (X);
 \fill[gray,opacity=0.4] (X)--++(0:1)--++(-90:1)--++(180:1);
 \path(3+6,0) coordinate (X);
 \fill[gray,opacity=0.4] (X)--++(0:1)--++(-90:1)--++(180:1);
 \path(3+11,0) coordinate (X);
 \fill[gray,opacity=0.4] (X)--++(0:1)--++(-90:1)--++(180:1);
 \draw[thick,path1color](0,0) --++(0:1) coordinate (x1)
  --++(0:1) coordinate (x2)
   --++(0:1) coordinate (x3)
    --++(0:1) coordinate (x4)
     --++(0:1) coordinate (x5)
      --++(0:1) coordinate (x6)
       --++(0:1) coordinate (x7)
        --++(0:1) coordinate (x8)
               --++(0:1) coordinate (x9)
        --++(0:1) coordinate (x10)
               --++(0:1) coordinate (x11)
                      --++(0:1) coordinate (x12)
                             --++(0:1) coordinate (x13)
                                    --++(0:1) coordinate (x14)
                                           --++(0:1) coordinate (x15)
                                                  --++(0:1) coordinate (x16)            
                                                                            --++(-90:1)--++(180:16)--++(90:1);
 \foreach \i in {1,2,3,4,5,6,7,8,...,16}
 {
 \draw[thick,path1color](x\i)--++(-90:1);
 \path(x\i)--++(-90:0.5)--++(180:0.5) coordinate (y\i);
 }
 \path(y8) node {$8$};
 \path(y7) node {$7$};
 \path(y6) node {$6$};
 \path(y5) node {$5$};
 \path(y4) node {$4$};
 \path(y3) node {$3$};
 \path(y2) node {$2$};
 \path(y1) node {$1$};

 \path(y9) node {$9$};
 \path(y10) node {$10$};
 \path(y11) node {$11$};
 \path(y12) node {$12$};
 \path(y13) node {$13$};
 \path(y14) node {$14$};
 \path(y15) node {$15$};
 \path(y16) node {$16$};
 \path(0,0)--++(90:0.4)--++(180:0.25) node {$\phantom{\alpha_2}$};

 \path(x7)--++(90:0.4)--++(180:0.25) node {$\alpha_2$};
  \node[M] at (x7){};
\draw[   thick,,densely dotted] (x8)--++(-90:1.35) ;
\draw[   thick,,densely dotted] (x8)--++(90:0.35) ;

\end{tikzpicture}\end{minipage}
\\ \vspace{-0.125cm}

   \begin{minipage}{12.5cm}\scalefont{0.9}
 \begin{tikzpicture}[scale=0.75]
 \path(0,0) coordinate (X);
 \fill[magenta,opacity=0.4] (X)--++(0:1)--++(-90:1)--++(180:1);
 \path(5,0) coordinate (X);
 \fill[magenta,opacity=0.4] (X)--++(0:1)--++(-90:1)--++(180:1);
 \path(10,0) coordinate (X);
 \fill[magenta,opacity=0.4] (X)--++(0:1)--++(-90:1)--++(180:1);
 \path(15,0) coordinate (X);
 \fill[magenta,opacity=0.4] (X)--++(0:1)--++(-90:1)--++(180:1);
  \path(1,0) coordinate (X);
 \fill[cyan,opacity=0.4] (X)--++(0:1)--++(-90:1)--++(180:1);
 \path(6,0) coordinate (X);
 \fill[cyan,opacity=0.4] (X)--++(0:1)--++(-90:1)--++(180:1);
 \path(11,0) coordinate (X);
 \fill[cyan,opacity=0.4] (X)--++(0:1)--++(-90:1)--++(180:1);
  \path(1+1,0) coordinate (X);
 \fill[orange,opacity=0.4] (X)--++(0:1)--++(-90:1)--++(180:1);
 \path(1+6,0) coordinate (X);
 \fill[orange,opacity=0.4] (X)--++(0:1)--++(-90:1)--++(180:1);
 \path(1+11,0) coordinate (X);
 \fill[orange,opacity=0.4] (X)--++(0:1)--++(-90:1)--++(180:1);
   \path(2+1,0) coordinate (X);
 \fill[green,opacity=0.4] (X)--++(0:1)--++(-90:1)--++(180:1);
 \path(2+6,0) coordinate (X);
 \fill[green,opacity=0.4] (X)--++(0:1)--++(-90:1)--++(180:1);
 \path(2+11,0) coordinate (X);
 \fill[green,opacity=0.4] (X)--++(0:1)--++(-90:1)--++(180:1);
 \path(3+1,0) coordinate (X);
 \fill[gray,opacity=0.4] (X)--++(0:1)--++(-90:1)--++(180:1);
 \path(3+6,0) coordinate (X);
 \fill[gray,opacity=0.4] (X)--++(0:1)--++(-90:1)--++(180:1);
 \path(3+11,0) coordinate (X);
 \fill[gray,opacity=0.4] (X)--++(0:1)--++(-90:1)--++(180:1);
 \draw[thick,path1color](0,0) --++(0:1) coordinate (x1)
  --++(0:1) coordinate (x2)
   --++(0:1) coordinate (x3)
    --++(0:1) coordinate (x4)
     --++(0:1) coordinate (x5)
      --++(0:1) coordinate (x6)
       --++(0:1) coordinate (x7)
        --++(0:1) coordinate (x8)
               --++(0:1) coordinate (x9)
        --++(0:1) coordinate (x10)
               --++(0:1) coordinate (x11)
                      --++(0:1) coordinate (x12)
                             --++(0:1) coordinate (x13)
                                    --++(0:1) coordinate (x14)
                                           --++(0:1) coordinate (x15)
                                                  --++(0:1) coordinate (x16)            
                                                                            --++(-90:1)--++(180:16)--++(90:1);
 \foreach \i in {1,2,3,4,5,6,7,8,...,16}
 {
 \draw[thick,path1color](x\i)--++(-90:1);
 \path(x\i)--++(-90:0.5)--++(180:0.5) coordinate (y\i);
 }
 \path(y8) node {$8$};
 \path(y7) node {$7$};
 \path(y6) node {$6$};
 \path(y5) node {$5$};
 \path(y4) node {$4$};
 \path(y3) node {$3$};
 \path(y2) node {$2$};
 \path(y1) node {$1$};

 \path(y9) node {$9$};
 \path(y10) node {$10$};
 \path(y11) node {$11$};
 \path(y12) node {$12$};
 \path(y13) node {$13$};
 \path(y14) node {$14$};
 \path(y15) node {$15$};
 \path(y16) node {$16$};
 \path(0,0)--++(90:0.4)--++(180:0.25) node {$\phantom{\alpha_2}$};

%
%
%
%
%

%
 %
%
\draw[   thick,,densely dotted] (x8)--++(-90:1.35) ;
\draw[   thick,,densely dotted] (x8)--++(90:0.35) ;
 \path(x8)--++(90:0.4)--++(180:0.25) node {$\vartheta$};
\draw[very thick, fill=white] (x8) circle (3pt);%

\end{tikzpicture}\end{minipage}

\end{array}
\] 
 
 \caption{All 
elements of ${\sf CStd}(\la,\mu)$ for $\la\in   \Lambda_{16} $
 and $\mu=(16,\alpha_1)$ as in \cref{aneasypeasyeg}. 
 We have pictured these as tableaux 
 as it highlights the idea of  ``translating the marked point rightwards''.  
 }
\label{uisdfuyigfduiogfuiogdfsuoigdfs} 
 \end{figure}
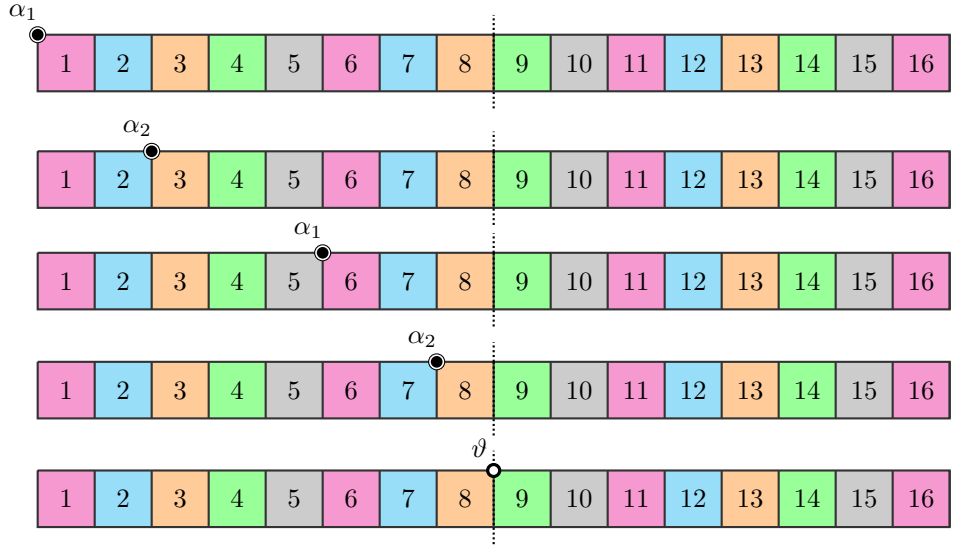

 \vspace{-0.6cm}
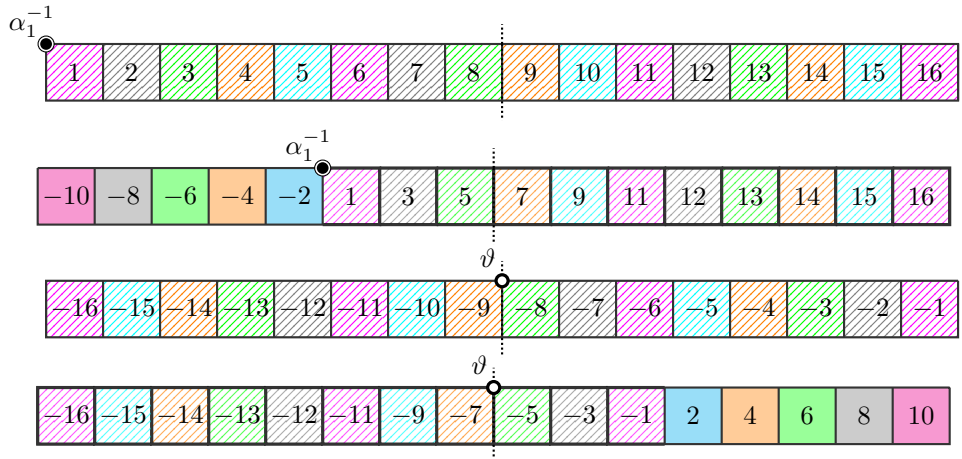
\begin{figure}[ht!]

\[\begin{array}{cc}
  \begin{minipage}{12.5cm}\scalefont{0.9}
 \begin{tikzpicture}[scale=0.75]
 \path(0,0) coordinate (X);
   \draw[pattern=north east lines, pattern  color = magenta, ] (X)--++(0:1)--++(-90:1)--++(180:1);
 \path(5,0) coordinate (X);
   \draw[pattern=north east lines, pattern  color = magenta, ] (X)--++(0:1)--++(-90:1)--++(180:1);
 \path(10,0) coordinate (X);
   \draw[pattern=north east lines, pattern  color = magenta, ] (X)--++(0:1)--++(-90:1)--++(180:1);
 \path(15,0) coordinate (X);
   \draw[pattern=north east lines, pattern  color = magenta, ] (X)--++(0:1)--++(-90:1)--++(180:1);
  \path(1,0) coordinate (X);
   \draw[pattern=north east lines, pattern  color = gray, ] (X)--++(0:1)--++(-90:1)--++(180:1);
 \path(6,0) coordinate (X);
   \draw[pattern=north east lines, pattern  color = gray, ] (X)--++(0:1)--++(-90:1)--++(180:1);
 \path(11,0) coordinate (X);
   \draw[pattern=north east lines, pattern  color = gray, ] (X)--++(0:1)--++(-90:1)--++(180:1);
  \path(1+1,0) coordinate (X);
   \draw[pattern=north east lines, pattern  color = green, ] (X)--++(0:1)--++(-90:1)--++(180:1);
 \path(1+6,0) coordinate (X);
   \draw[pattern=north east lines, pattern  color = green, ] (X)--++(0:1)--++(-90:1)--++(180:1);
 \path(1+11,0) coordinate (X);
   \draw[pattern=north east lines, pattern  color = green, ] (X)--++(0:1)--++(-90:1)--++(180:1);
   \path(2+1,0) coordinate (X);
   \draw[pattern=north east lines, pattern  color = orange, ] (X)--++(0:1)--++(-90:1)--++(180:1);
 \path(2+6,0) coordinate (X);
   \draw[pattern=north east lines, pattern  color = orange, ] (X)--++(0:1)--++(-90:1)--++(180:1);
 \path(2+11,0) coordinate (X);
   \draw[pattern=north east lines, pattern  color = orange, ] (X)--++(0:1)--++(-90:1)--++(180:1);
 \path(3+1,0) coordinate (X);
   \draw[pattern=north east lines, pattern  color = cyan, ] (X)--++(0:1)--++(-90:1)--++(180:1);
 \path(3+6,0) coordinate (X);
   \draw[pattern=north east lines, pattern  color = cyan, ] (X)--++(0:1)--++(-90:1)--++(180:1);
 \path(3+11,0) coordinate (X);
   \draw[pattern=north east lines, pattern  color = cyan, ] (X)--++(0:1)--++(-90:1)--++(180:1);
 \draw[thick,path1color](0,0) --++(0:1) coordinate (x1)
  --++(0:1) coordinate (x2)
   --++(0:1) coordinate (x3)
    --++(0:1) coordinate (x4)
     --++(0:1) coordinate (x5)
      --++(0:1) coordinate (x6)
       --++(0:1) coordinate (x7)
        --++(0:1) coordinate (x8)
               --++(0:1) coordinate (x9)
        --++(0:1) coordinate (x10)
               --++(0:1) coordinate (x11)
                      --++(0:1) coordinate (x12)
                             --++(0:1) coordinate (x13)
                                    --++(0:1) coordinate (x14)
                                           --++(0:1) coordinate (x15)
                                                  --++(0:1) coordinate (x16)            
                                                                            --++(-90:1)--++(180:16)--++(90:1);
 \foreach \i in {1,2,3,4,5,6,7,8,...,16}
 {
 \draw[thick,path1color](x\i)--++(-90:1);
 \path(x\i)--++(-90:0.5)--++(180:0.5) coordinate (y\i);
 }
 \path(y8) node {$8$};
 \path(y7) node {$7$};
 \path(y6) node {$6$};
 \path(y5) node {$5$};
 \path(y4) node {$4$};
 \path(y3) node {$3$};
 \path(y2) node {$2$};
 \path(y1) node {$1$};

 \path(y9) node {$9$};
 \path(y10) node {$10$};
 \path(y11) node {$11$};
 \path(y12) node {$12$};
 \path(y13) node {$13$};
 \path(y14) node {$14$};
 \path(y15) node {$15$};
 \path(y16) node {$16$};
 \path(0,0)--++(90:0.4)--++(180:0.25) node {$\alpha_1^{-1}$};
  \node[M] at (0,0){};
 \path(16,0)--++(90:0.4)--++(0:0.25) node {$\phantom{\alpha_1^{-1}}$};

\draw[   thick,,densely dotted] (x8)--++(-90:1.35) ;
\draw[   thick,,densely dotted] (x8)--++(90:0.35) ;

\end{tikzpicture}\end{minipage}
\\
\vspace{-0.15cm}

 \begin{minipage}{12.5cm}\scalefont{0.9}
 \begin{tikzpicture}[scale=0.75]
 \path(0,0) coordinate (X);
   \fill[magenta,opacity=0.4] (X)--++(0:1)--++(-90:1)--++(180:1);
 \path(5,0) coordinate (X);
  \draw[pattern=north east lines, pattern color=magenta, opacity=0.8, very thick] (X)--++(0:1)--++(-90:1)--++(180:1);
 \path(10,0) coordinate (X);
  \draw[pattern=north east lines, pattern color=magenta, opacity=0.8, very thick] (X)--++(0:1)--++(-90:1)--++(180:1);
 \path(15,0) coordinate (X);
  \draw[pattern=north east lines, pattern color=magenta, opacity=0.8, very thick](X)--++(0:1)--++(-90:1)--++(180:1);

  \path(1,0) coordinate (X);
  \fill[gray,opacity=0.4]  (X)--++(0:1)--++(-90:1)--++(180:1);
 \path(6,0) coordinate (X);
\draw[pattern=north east lines, pattern color=gray, opacity=0.8, very thick]  (X)--++(0:1)--++(-90:1)--++(180:1);
 \path(11,0) coordinate (X);
\draw[pattern=north east lines, pattern color=gray, opacity=0.8, very thick] (X)--++(0:1)--++(-90:1)--++(180:1);

  \path(1+1,0) coordinate (X);
  \fill[green,opacity=0.4]  (X)--++(0:1)--++(-90:1)--++(180:1);
 \path(1+6,0) coordinate (X);
\draw[pattern=north east lines, pattern color=green, opacity=0.8, very thick]  (X)--++(0:1)--++(-90:1)--++(180:1);
 \path(1+11,0) coordinate (X);
\draw[pattern=north east lines, pattern color=green, opacity=0.8, very thick]  (X)--++(0:1)--++(-90:1)--++(180:1);

   \path(2+1,0) coordinate (X);
\fill[orange,opacity=0.4]   (X)--++(0:1)--++(-90:1)--++(180:1);
 \path(2+6,0) coordinate (X);
  \draw[pattern=north east lines, pattern color=orange, opacity=0.8, very thick] (X)--++(0:1)--++(-90:1)--++(180:1);
 \path(2+11,0) coordinate (X);
  \draw[pattern=north east lines, pattern color=orange, opacity=0.8, very thick]  (X)--++(0:1)--++(-90:1)--++(180:1);
 \path(3+1,0) coordinate (X);
  \fill[cyan,opacity=0.4]   (X)--++(0:1)--++(-90:1)--++(180:1);
 \path(3+6,0) coordinate (X);
\draw[pattern=north east lines, pattern color=cyan, opacity=0.8, very thick]  (X)--++(0:1)--++(-90:1)--++(180:1);
 \path(3+11,0) coordinate (X);
\draw[pattern=north east lines, pattern color=cyan, opacity=0.8, very thick] (X)--++(0:1)--++(-90:1)--++(180:1);
 \draw[thick,path1color](0,0) --++(0:1) coordinate (x1)
  --++(0:1) coordinate (x2)
   --++(0:1) coordinate (x3)
    --++(0:1) coordinate (x4)
     --++(0:1) coordinate (x5)
      --++(0:1) coordinate (x6)
       --++(0:1) coordinate (x7)
        --++(0:1) coordinate (x8)
               --++(0:1) coordinate (x9)
        --++(0:1) coordinate (x10)
               --++(0:1) coordinate (x11)
                      --++(0:1) coordinate (x12)
                             --++(0:1) coordinate (x13)
                                    --++(0:1) coordinate (x14)
                                           --++(0:1) coordinate (x15)
                                                  --++(0:1) coordinate (x16)            
                                                                            --++(-90:1)--++(180:16)--++(90:1);
 \foreach \i in {1,2,3,4,5,6,7,8,...,16}
 {
 \draw[thick,path1color](x\i)--++(-90:1);
 \path(x\i)--++(-90:0.5)--++(180:0.5) coordinate (y\i);
 }
 \path(y8) node {$5$};
 \path(y7) node {$3$};
 \path(y6) node {$1$};
 \path(y5) node {$-2$};
 \path(y4) node {$-4$};
 \path(y3) node {$-6$};
 \path(y2) node {$-8$};
 \path(y1) node {$-10$};

 \path(y9) node {$7$};
 \path(y10) node {$9$};
 \path(y11) node {$11$};
 \path(y12) node {$12$};
 \path(y13) node {$13$};
 \path(y14) node {$14$};
 \path(y15) node {$15$};
 \path(y16) node {$16$};%
  \path(0,0)--++(90:0.4)--++(180:0.25) node {$\phantom{\alpha_1}$};
 \path(16,0)--++(90:0.4)--++(0:0.25) node {$\phantom{\alpha_1}$}; 
 \path(x5)--++(90:0.4)--++(180:0.25) node {$\alpha_1^{-1}$};
  \node[M] at (x5){};
\draw[   thick,,densely dotted] (x8)--++(-90:1.35) ;
\draw[   thick,,densely dotted] (x8)--++(90:0.35) ;
\end{tikzpicture}\end{minipage}
 \\

\vspace{-0.15cm}
   \begin{minipage}{12.5cm}\scalefont{0.9}
 \begin{tikzpicture}[scale=0.75,xscale=-1]
 \path(0,0) coordinate (X);
   \draw[pattern=north east lines, pattern  color = magenta, ] (X)--++(0:1)--++(-90:1)--++(180:1);
 \path(5,0) coordinate (X);
   \draw[pattern=north east lines, pattern  color = magenta, ] (X)--++(0:1)--++(-90:1)--++(180:1);
 \path(10,0) coordinate (X);
   \draw[pattern=north east lines, pattern  color = magenta, ] (X)--++(0:1)--++(-90:1)--++(180:1);
 \path(15,0) coordinate (X);
   \draw[pattern=north east lines, pattern  color = magenta, ] (X)--++(0:1)--++(-90:1)--++(180:1);
  \path(1,0) coordinate (X);
   \draw[pattern=north east lines, pattern  color = gray, ] (X)--++(0:1)--++(-90:1)--++(180:1);
 \path(6,0) coordinate (X);
   \draw[pattern=north east lines, pattern  color = gray, ] (X)--++(0:1)--++(-90:1)--++(180:1);
 \path(11,0) coordinate (X);
   \draw[pattern=north east lines, pattern  color = gray, ] (X)--++(0:1)--++(-90:1)--++(180:1);
  \path(1+1,0) coordinate (X);
   \draw[pattern=north east lines, pattern  color = green, ] (X)--++(0:1)--++(-90:1)--++(180:1);
 \path(1+6,0) coordinate (X);
   \draw[pattern=north east lines, pattern  color = green, ] (X)--++(0:1)--++(-90:1)--++(180:1);
 \path(1+11,0) coordinate (X);
   \draw[pattern=north east lines, pattern  color = green, ] (X)--++(0:1)--++(-90:1)--++(180:1);
   \path(2+1,0) coordinate (X);
   \draw[pattern=north east lines, pattern  color = orange, ] (X)--++(0:1)--++(-90:1)--++(180:1);
 \path(2+6,0) coordinate (X);
   \draw[pattern=north east lines, pattern  color = orange, ] (X)--++(0:1)--++(-90:1)--++(180:1);
 \path(2+11,0) coordinate (X);
   \draw[pattern=north east lines, pattern  color = orange, ] (X)--++(0:1)--++(-90:1)--++(180:1);
 \path(3+1,0) coordinate (X);
   \draw[pattern=north east lines, pattern  color = cyan, ] (X)--++(0:1)--++(-90:1)--++(180:1);
 \path(3+6,0) coordinate (X);
   \draw[pattern=north east lines, pattern  color = cyan, ] (X)--++(0:1)--++(-90:1)--++(180:1);
 \path(3+11,0) coordinate (X);
   \draw[pattern=north east lines, pattern  color = cyan, ] (X)--++(0:1)--++(-90:1)--++(180:1);
 \draw[thick,path1color](0,0) --++(0:1) coordinate (x1)
  --++(0:1) coordinate (x2)
   --++(0:1) coordinate (x3)
    --++(0:1) coordinate (x4)
     --++(0:1) coordinate (x5)
      --++(0:1) coordinate (x6)
       --++(0:1) coordinate (x7)
        --++(0:1) coordinate (x8)
               --++(0:1) coordinate (x9)
        --++(0:1) coordinate (x10)
               --++(0:1) coordinate (x11)
                      --++(0:1) coordinate (x12)
                             --++(0:1) coordinate (x13)
                                    --++(0:1) coordinate (x14)
                                           --++(0:1) coordinate (x15)
                                                  --++(0:1) coordinate (x16)            
                                                                            --++(-90:1)--++(180:16)--++(90:1);
 \foreach \i in {1,2,3,4,5,6,7,8,...,16}
 {
 \draw[thick,path1color](x\i)--++(-90:1);
 \path(x\i)--++(-90:0.5)--++(180:0.5) coordinate (y\i);
 }
 \path(y8) node  {$-8$};
 \path(y7) node  {$-7$};
 \path(y6) node  {$-6$};
 \path(y5) node  {$-5$};
 \path(y4) node  {$-4$};
 \path(y3) node  {$-3$};
 \path(y2) node  {$-2$};
 \path(y1) node  {$-1$};

 \path(y9) node  {$-9$};
 \path(y10) node  {$-10$};
 \path(y11) node  {$-11$};
 \path(y12) node  {$-12$};
 \path(y13) node  {$-13$};
 \path(y14) node  {$-14$};
 \path(y15) node  {$-15$};
 \path(y16) node  {$-16$};
 \path(0,0)--++(90:0.4)--++(180:0.25) node {\phantom{$\alpha_1^{-1}$}};
 \path(16,0)--++(90:0.4)--++(0:0.25) node {$\phantom{\alpha_1^{-1}}$};

\draw[   thick,,densely dotted] (x8)--++(-90:1.35) ;
\draw[   thick,,densely dotted] (x8)--++(90:0.35) ;

 \path(x8)--++(90:0.4)--++(0:0.25) node {$\vartheta$};
\draw[very thick, fill=white] (x8) circle (3pt);

\end{tikzpicture}\end{minipage}

\\
\vspace{-0.15cm}

 \begin{minipage}{12.5cm}\scalefont{0.9}
 \begin{tikzpicture}[scale=0.75,xscale=-1]
\path(0,0) coordinate (X);
   \fill[magenta,opacity=0.4] (X)--++(0:1)--++(-90:1)--++(180:1);
 \path(5,0) coordinate (X);
  \draw[pattern=north east lines, pattern color=magenta, opacity=0.8, very thick] (X)--++(0:1)--++(-90:1)--++(180:1);
 \path(10,0) coordinate (X);
  \draw[pattern=north east lines, pattern color=magenta, opacity=0.8, very thick] (X)--++(0:1)--++(-90:1)--++(180:1);
 \path(15,0) coordinate (X);
  \draw[pattern=north east lines, pattern color=magenta, opacity=0.8, very thick](X)--++(0:1)--++(-90:1)--++(180:1);

  \path(1,0) coordinate (X);
  \fill[gray,opacity=0.4]  (X)--++(0:1)--++(-90:1)--++(180:1);
 \path(6,0) coordinate (X);
\draw[pattern=north east lines, pattern color=gray, opacity=0.8, very thick]  (X)--++(0:1)--++(-90:1)--++(180:1);
 \path(11,0) coordinate (X);
\draw[pattern=north east lines, pattern color=gray, opacity=0.8, very thick] (X)--++(0:1)--++(-90:1)--++(180:1);

  \path(1+1,0) coordinate (X);
  \fill[green,opacity=0.4]  (X)--++(0:1)--++(-90:1)--++(180:1);
 \path(1+6,0) coordinate (X);
\draw[pattern=north east lines, pattern color=green, opacity=0.8, very thick]  (X)--++(0:1)--++(-90:1)--++(180:1);
 \path(1+11,0) coordinate (X);
\draw[pattern=north east lines, pattern color=green, opacity=0.8, very thick]  (X)--++(0:1)--++(-90:1)--++(180:1);

   \path(2+1,0) coordinate (X);
\fill[orange,opacity=0.4]   (X)--++(0:1)--++(-90:1)--++(180:1);
 \path(2+6,0) coordinate (X);
  \draw[pattern=north east lines, pattern color=orange, opacity=0.8, very thick] (X)--++(0:1)--++(-90:1)--++(180:1);
 \path(2+11,0) coordinate (X);
  \draw[pattern=north east lines, pattern color=orange, opacity=0.8, very thick]  (X)--++(0:1)--++(-90:1)--++(180:1);
 \path(3+1,0) coordinate (X);
  \fill[cyan,opacity=0.4]   (X)--++(0:1)--++(-90:1)--++(180:1);
 \path(3+6,0) coordinate (X);
\draw[pattern=north east lines, pattern color=cyan, opacity=0.8, very thick]  (X)--++(0:1)--++(-90:1)--++(180:1);
 \path(3+11,0) coordinate (X);
\draw[pattern=north east lines, pattern color=cyan, opacity=0.8, very thick] (X)--++(0:1)--++(-90:1)--++(180:1);
 \draw[thick,path1color](0,0) --++(0:1) coordinate (x1)
  --++(0:1) coordinate (x2)
   --++(0:1) coordinate (x3)
    --++(0:1) coordinate (x4)
     --++(0:1) coordinate (x5)
      --++(0:1) coordinate (x6)
       --++(0:1) coordinate (x7)
        --++(0:1) coordinate (x8)
               --++(0:1) coordinate (x9)
        --++(0:1) coordinate (x10)
               --++(0:1) coordinate (x11)
                      --++(0:1) coordinate (x12)
                             --++(0:1) coordinate (x13)
                                    --++(0:1) coordinate (x14)
                                           --++(0:1) coordinate (x15)
                                                  --++(0:1) coordinate (x16)            
                                                                            --++(-90:1)--++(180:16)--++(90:1);
 \foreach \i in {1,2,3,4,5,6,7,8,...,16}
 {
 \draw[thick,path1color](x\i)--++(-90:1);
 \path(x\i)--++(-90:0.5)--++(180:0.5) coordinate (y\i);
 }
 \path(y8) node {$-5$};
 \path(y7) node {$-3$};
 \path(y6) node {$-1$};
 \path(y5) node {$2$};
 \path(y4) node {$4$};
 \path(y3) node {$6$};
 \path(y2) node {$8$};
 \path(y1) node {$10$};

 \path(y9) node {$-7$};
 \path(y10) node {$-9$};
 \path(y11) node {$-11$};
 \path(y12) node {$-12$};
 \path(y13) node {$-13$};
 \path(y14) node {$-14$};
 \path(y15) node {$-15$};
 \path(y16) node {$-16$};%
\draw[   thick,,densely dotted] (x8)--++(-90:1.35) ;
\draw[   thick,,densely dotted] (x8)--++(90:0.35) ;
 \path(16,0)--++(90:0.4)--++(0:0.25) node {$\phantom{\alpha_1}$};
%
%
%
 \path(x8)--++(90:0.4)--++(0:0.25) node {$\vartheta$};
\draw[very thick, fill=white] (x8) circle (3pt);
\end{tikzpicture}\end{minipage}

\end{array}
\]

 \caption{Let
 $\nu =(16,\alpha_1^{-1})$
  $\mu=(6,\alpha_1^{-1})$ and $\la= (0,\vartheta)$ as in \cref{aneasypeasyeg}. We depict the tableaux
    $\stt_{\nu}$, 
    $\stt_{\mu}$  and the unique elements of ${\sf CStd}_n(\la,\nu)$ and  ${\sf CStd}_n(\la,\mu)$.	
 Notice that we have translated the marked point {\em and also flipped the signs and ordering} (because $\vartheta \not \in q^\ZZ\alpha_1^{-1}$, but $\vartheta^{-1}  \in q^\ZZ\alpha_1^{-1}$). }
 \label{kjhghgjkldfshjkgfsdhkjgdgdhjfghdjkfhjkdgfhjlkgdfhjlkadsfjhkdfjkhldfg}
 
 \end{figure}

The general case (with  $q$   a root of unity, but  
 $\alpha_1,\alpha_2,\vartheta \not \in q^{\ZZ}$) is no more difficult than \cref{aneasypeasyeg}.
For $\la=(k_1,\beta_1) , \mu=(k_2,\beta_2) \in \Lambda_n $ 
we have that
 ${\sf CStd}_n(\la,\mu) \neq \emptyset $ if and only if 
   $\beta_1  = \beta_2   q^{2a} $ for some $0\leq  a <e$ and $k_2 - k_1  \in  2a+ 2e\ZZ_{\geq0}$. 
Each  non-zero weight space is again spanned  by a single  tableau which is   obtained by ``translating the marked point rightwards'' in the exact same fashion;   the conjecture holds (without restriction on $\Bbbk$) because characters of graded simple modules must belong to $\ZZ_{\geq0}[q+q^{-1}]$.
  Note that if $\vartheta\in \alpha_jq^\ZZ$, for  $\mu= (k, \alpha_j^{-1})$ and $\la= (0,\vartheta)$ we  must flip  the order and the signs  of tableau  entries
   as $\alpha_j^{-1} \not \in \vartheta q^\ZZ$, see \cref{kjhghgjkldfshjkgfsdhkjgdgdhjfghdjkfhjkdgfhjlkgdfhjlkadsfjhkdfjkhldfg}.

\subsection{Decomposition matrices for $q$ not a root of unity and  Nakayama's  conjecture }
%
%
%
 Let $\alpha_1,\alpha_2 \in q^\ZZ$ with  $q$ not a root of unity 
and suppose that $\alpha_1=q^{a_1}$ and $\alpha_2=q^{a_2}$ with $0<a_1 < a_2$. 
By definition, for $\la,\mu \in \{(k ,\alpha_1^{\pm1})\mid 0< k \leq n\}$ 
we have that 
 ${\sf CStd}_n(	\la , \mu ) =\emptyset $ unless
 $\la=(k-2a_1,\alpha_1) $ and  $\mu=(k ,\alpha_1^{-1}) $ 
(or the trivial case, where $\la=\mu$). In fact, there is a 
 unique element of 
${\sf CStd}_n( (k-2a_1,\alpha_1), (k ,\alpha_1^{-1}) )$ and it is of degree 0.  
We  claim  that 
\begin{equation}\label{tadddaaaaaaa}
[\Delta(k-2a_1,\alpha_1)  : L(k ,\alpha_1^{-1})]=0
\end{equation}over any field $\Bbbk$ (of characteristic  not equal to 2)  (from which \cref{conjforus} follows in this case).  
 We 
 set $\la=(k-2a_1,\alpha_1)$ and $\mu = (k ,\alpha_1^{-1})$,  and we 
define elements of $\Std_n(\la)$ as follows:
\begin{itemize}
\item we let ${\color{violet}\sts }\in \Std_n(\la)$ denote the tableau obtained from $\stt_\mu$ by 
 translating the marked point $a_1$ boxes rightwards;
 \item we let ${\color{orange}\stt}\in \Std_n(\la)$ denote the unique minimal length   tableau which is not in the same residue class as  a ladder tableau;
 \item we let $\stu<{\color{orange}\stt}$ denote the unique tableau obtained by removing a tile from $\mathbb T_\stt$;  
 \item we let $\stv\in \Std_n(\la) $ denote the unique element  such that $\stu\sim \stv$  (and $\stu \neq \stv$);
\end{itemize}
and we refer to  \cref{asfoiuglisdguiodugiofdug98t5euriougiosdjgkdf} for  examples.
One can prove this by verifying \begin{equation}\label{inthesimplehead}
 \langle \psi_{{\color{violet}\sts}} ,  \psi_{\color{violet}\sts}\rangle e_{\stt_\la}
=
 \psi_{{\color{violet}\sts}}^*\psi_{\color{violet}\sts} e_{\stt_\la}
=\pm \psi_{{\color{orange}\stt}}^* \psi_{\color{orange}\stt} e_{\stt_\la}
=\pm \psi_{{\color{black}\stu}}^* y_1 \psi_{\color{black}\stu} e_{\stt_\la}
=\pm 2 \psi_{\stu}^*   \psi_{\color{black}\stv} e_{\stt_\la}
= \pm 
\langle 2 \psi_{\stu} ,   \psi_{\color{black}\stv}\rangle  e_{\stt_\la}
\end{equation}
for $\langle  \psi_{\stu} ,    \psi_{\color{black}\stv}\rangle \neq 0$,
The first equality of \eqref{inthesimplehead} is by definition. 
The second equality of \eqref{inthesimplehead}
follows by the commuting relations and \eqref{skjghskljdhgkjldhkjgdshgkjdshfgjkdfhgkjdsfghkdshgfjkdshfgjksfhdgkjhgdksjlghsd}. 
The third follows by relation (4.7). 
The fifth equality holds by definition, by grading considerations
 $\langle \psi_\stu ,  \psi_\stu\rangle = 0=
\langle \psi_\stv ,  \psi_\stv\rangle $; and by  
 \cref{forloic} the Gram matrix of this weight space has full rank (as $\stu$ is a ladder tableau) and hence 
 $\langle  \psi_{\stu} ,   \psi_{\color{black}\stv}\rangle \neq 0$. 
 We will prove the fourth equality in a specific example (the general case is similar).
We then note that \eqref{inthesimplehead} immediately implies that $\psi_{\color{violet}\sts }\not \in \rad \langle - , - \rangle_\la$ and so 
$[\Delta(k-2a_1,\alpha_1)  : L(k ,\alpha_1^{-1})]=0$.

   \vspace{-0.15cm}
 \begin{figure}[ht!] 
  \[
\def\pE{28}
\def\pN{18}
\def\pA{8}
\def\pB{3}
\TIKZ[yscale=-1, scale=.3, font=\scriptsize]{
\clip(-10,-3) rectangle (15,18);

\pgfmathsetmacro{\pEParity}{Mod(\pE,2)}
\pgfmathsetmacro{\pNParity}{Mod(\pN,2)}
\coordinate (max-start) at (\pA-1, 0);
\path (max-start) to ++(\pN-2*\pB, \pN) coordinate (end);
\path (end) to ++(-\pN, -\pN) coordinate (min-start);
\path (end) to ++(\pB, -\pB) coordinate (p-corner);

 \path(max-start)--++(-4,0) coordinate (akunam);

 \draw[pathsregion] (akunam)--++(12,12) --++(-6,6)--++(-18,-18)--++(4,0) coordinate (INVER);

\coordinate (beta2) at (4, -1); \coordinate (beta2inv) at (-4, -1);
\coordinate (beta1) at (8, -1); \coordinate (beta1inv) at (-8, -1);
\filldraw[plusone] (beta2) \TILE;
\draw[tauplane] (0,-1) to +(0,\pN+1);
\draw[sigmaplane] (-\pE,-1) to +(0,\pN+1);
\draw[sigmaplane] (\pE,-1) to +(0,\pN+1);
\begin{scope}
\clip (-17, -1) rectangle (17, \pN);
	\foreach \x in {-20, ..., 20}{
		\draw[gridstyle] (2*\x,-1) to +(\pN+1,\pN+1);
		\draw[gridstyle] (2*\x,-1) to +(-\pN-1,\pN+1);
	}
\end{scope}
\foreach \x in {beta1, beta1inv}
	{\node[M1] at (\x){};}
\foreach \x in {beta2, beta2inv}
	{\node[M2] at (\x){};}
\node[above] at (beta1) {\small \strut $\alpha_2$};
	\node[above] at (beta2) {\small \strut $\alpha_1$};
\node[above] at (beta1inv) {\small \strut $\alpha_2^{-1}$};
	\node[above] at (beta2inv) {\small \strut $\alpha_1^{-1}$};

  \draw[line width=3, cyan] (akunam)--++(1,1)	--++(-1,1)
 --++(1,1)	--++(-1,1)
  --++(1,1)	--++(-1,1)
   --++(1,1)	--++(-1,1)
    --++(1,1)	--++(-1,1)
     --++(1,1)	--++(-1,1)
 --++(2,2) --++(4,4)	 ;

 \draw[line width=1.4] (INVER)--++(1,1)	--++(-1,1)
 --++(1,1)	--++(-1,1)
 --++(10,10)--++(4,4)	 ;

	 }
	 \qquad 
	 \TIKZ[yscale=-1, scale=.3, font=\scriptsize]{
\clip(-10,-3) rectangle (15,18);

\pgfmathsetmacro{\pEParity}{Mod(\pE,2)}
\pgfmathsetmacro{\pNParity}{Mod(\pN,2)}
\coordinate (max-start) at (\pA-1, 0);
\path (max-start) to ++(\pN-2*\pB, \pN) coordinate (end);
\path (end) to ++(-\pN, -\pN) coordinate (min-start);
\path (end) to ++(\pB, -\pB) coordinate (p-corner);

 \path(max-start)--++(-4,0) coordinate (akunam);

 \draw[pathsregion] (akunam)--++(12,12) --++(-6,6)--++(-18,-18)--++(4,0) coordinate (INVER);

\coordinate (beta2) at (4, -1); \coordinate (beta2inv) at (-4, -1);
\coordinate (beta1) at (8, -1); \coordinate (beta1inv) at (-8, -1);
\filldraw[plusone] (beta2) \TILE;
\draw[tauplane] (0,-1) to +(0,\pN+1);
\draw[sigmaplane] (-\pE,-1) to +(0,\pN+1);
\draw[sigmaplane] (\pE,-1) to +(0,\pN+1);
\begin{scope}
\clip (-17, -1) rectangle (17, \pN);
	\foreach \x in {-20, ..., 20}{
		\draw[gridstyle] (2*\x,-1) to +(\pN+1,\pN+1);
		\draw[gridstyle] (2*\x,-1) to +(-\pN-1,\pN+1);
	}
\end{scope}
\foreach \x in {beta1, beta1inv}
	{\node[M1] at (\x){};}
\foreach \x in {beta2, beta2inv}
	{\node[M2] at (\x){};}
\node[above] at (beta1) {\small \strut $\alpha_2$};
	\node[above] at (beta2) {\small \strut $\alpha_1$};
\node[above] at (beta1inv) {\small \strut $\alpha_2^{-1}$};
	\node[above] at (beta2inv) {\small \strut $\alpha_1^{-1}$};

%
%

 \draw[line width=3.2,orange] (INVER) 
 --++(9,9) 	--++(-1,1)
 --++(1,1)	--++(-1,1)	 --++(2,2)--++(4,4)	 ;

\path(INVER) --++(2,0) coordinate (INVER);
 
 \draw[line width=1.4,black]  (INVER) 	--++(-1,1)
 --++(4,4) coordinate (X)  --++(4,4) 	--++(-1,1)
 --++(1,1)	--++(-1,1)	 --++(2,2)--++(4,4)	 ;

  \draw[line width=1.4, black] (X)--++(4,-4)--++(-1,-1);

  \fill[magenta,opacity=0.4] (X)--++(1,-1)--++(-1,-1)--++(1,-1)--++(-1,-1)--++(1,-1)--++(-1,-1)
  --++(-1,1)--++(1,1)  --++(-1,1)--++(1,1)  --++(-1,1)--++(1,1);

\path(X)--++(1,-1) coordinate (Y);
  \fill[cyan,opacity=0.4] (Y)--++(1,-1)--++(-1,-1)--++(1,-1)--++(-1,-1)
  --++(-1,1)--++(1,1)  --++(-1,1)--++(1,1);

\path(Y)--++(1,-1) coordinate (Y);
  \fill[green,opacity=0.4] (Y)--++(1,-1)--++(-1,-1)--++(1,-1)--++(-1,-1)
  --++(-1,1)--++(1,1)  --++(-1,1)--++(1,1);

\path(Y)--++(1,-1) coordinate (Y);
  \fill[orange,opacity=0.4] (Y)--++(1,-1)--++(-1,-1) --++(-1,1)--++(1,1);

\path(X)--++(-1,-1) coordinate (Y);
  \fill[cyan,opacity=0.4] (Y)--++(1,-1)--++(-1,-1)--++(1,-1)--++(-1,-1)
  --++(-1,1)--++(1,1)  --++(-1,1)--++(1,1);

\path(Y)--++(-1,-1) coordinate (Y);
  \fill[green,opacity=0.4] (Y)--++(1,-1)--++(-1,-1)--++(1,-1)--++(-1,-1)
  --++(-1,1)--++(1,1)  --++(-1,1)--++(1,1);

\path(Y)--++(-1,-1) coordinate (Y);
  \fill[orange,opacity=0.4] (Y)--++(1,-1)--++(-1,-1) --++(-1,1)--++(1,1);

	 }
\]

\caption{Here $\alpha_1=q^4$ and $\alpha_2=q^8$.  
On the left we depict the non-ladder tableau $\sts$ and the tableau $\color{cyan}\stt_{\la}$. 
On the right we depict the non-ladder tableau $\color{orange}\stt$ and the pair of tableaux 
$\stu \sim \stv$ such that $\stu$ is a ladder tableau.  These are the tableaux from the proof  in \cref{surprisesemisimple}. 
}
\label{asfoiuglisdguiodugiofdug98t5euriougiosdjgkdf}
\end{figure}

\begin{eg} \label{surprisesemisimple}
 Let $\alpha_1,\alpha_2 \in q^\ZZ$ with  $q$ not a root of unity 
and suppose that $\alpha_1=q^{4}$ and $\alpha_2=q^{8}$. 
We will show that 
$y_1\psi_\stu=2\psi_\stv$ for $\stu,\stv$ as in \cref{asfoiuglisdguiodugiofdug98t5euriougiosdjgkdf}.  We have that
\begin{align*}
y_1\psi_\stu &=	 y_1 \color{orange} \psi_1
\color{green!70!black}  (\psi_0\psi_2)
\color{cyan}(\psi_1\psi_3)  
\color{magenta} (\psi_0\psi_2 \psi_{4}) 
\color{cyan}  (\psi_1\psi_3) 
\color{green!70!black} (\psi_0\psi_2)
\color{orange}\psi_1
\color{black} \psi_\stv
\\
&=  
  \psi_1
    \psi_0\psi_2 
  \psi_1\psi_3   
   \psi_0\psi_2  (\psi_{4}y_5  +1)
    \psi_1\psi_3  
   \psi_0\psi_2 
 \psi_1
   \psi_\stv
   \\
   &=  
     \psi_1
    \psi_0\psi_2 
  \psi_1	 
   \psi_0	(\psi_3   \psi_2   \psi_3  )
    \psi_1	 
   \psi_0\psi_2 
 \psi_1
   \psi_\stv
   \\
   &=  
     \psi_1
    \psi_0\psi_2 
  \psi_1	 
   \psi_0	(\psi_2   \psi_3   \psi_2 + 1  )
    \psi_1	 
   \psi_0\psi_2 
 \psi_1
   \psi_\stv
   \\
   &=  
     \psi_1
    \psi_0\psi_2 
(  \psi_1	 
   \psi_0 
    \psi_1	 
   \psi_0)\psi_2 
 \psi_1
   \psi_\stv
\\
   &=  
     \psi_1
    \psi_0\psi_2 
(  	 
   \psi_0 
    \psi_1	 
   \psi_0 \psi_1+2\psi_0)\psi_2 
 \psi_1
   \psi_\stv
\\     &=  
  2    \psi_1
    \psi_0\psi_0\psi_2 
  \psi_2 
 \psi_1
   \psi_\stv
=  2   \psi_\stv
\end{align*} 
where the first equality follows by definition (compare the colouring with \cref{asfoiuglisdguiodugiofdug98t5euriougiosdjgkdf}); the  second   follows by repeated applications of (4.2); 
the third, fifth, and seventh equalities follow by   residue considerations and  (4.9) (and the commutation relations);
 the fourth equality follows by (4.5);
 the sixth   follows by (4.8); the eighth  follows by  (4.4).  
 Thus we have verified \eqref{inthesimplehead} and 
 hence verified \eqref{tadddaaaaaaa} in this case.
   For all other pairs $\la,\mu \in \Lambda_n\setminus \{(0,\vartheta)\}$ in this block, we have that 
 $$
 \textstyle \sum_{\sts \in {\sf CStd}_n(\la,\mu)}v^{\deg(\sts)}\in \ZZ[q]$$and so 
the remaining entries of the decomposition matrix can be deduced immediately. For   $\Bbbk$ of characteristic not 2,
 the corresponding block of the graded decomposition matrix is
independent of the field $\Bbbk$ and is   
 as follows
$$\def\arraystretch{1.1}
  \begin{array}{c|cccc} 
   & (18,\alpha_2^{-1}) & (14,\alpha_1^{-1}) & (6,\alpha_1) & (2,\alpha_2) 
   \\
   \hline 
   (18,\alpha_2^{-1}) & 1		&0&0&0
   \\
   (14,\alpha_1^{-1}) &v		&1&0&0
   \\ 
   (6,\alpha_1) 		& v		&0&1&0	
   \\
   (2,\alpha_2) 		&v^2		&v	&v	&1
   \end{array}
$$ 
\end{eg}

\begin{rmk}
The above example  illustrates that Nakayama's conjecture fails for  ${\rm TL}_n(\alpha_1,\alpha_2,\vartheta)$.  
By which we mean: the blocks of ${\rm TL}_n(\alpha_1,\alpha_2,\vartheta)$ are not given by the $W(C_n)$-orbits of  residue classes 
of 
 shapes/tableaux
 (in contrast to the case of cyclotomic Hecke and Temperley--Lieb algebras of type $A$). 
To see this, simply note that if $\alpha_1 \in q^{a_1}$ but $\alpha_2  \not \in q^\ZZ$ and $\vartheta \not \in q^\ZZ$ then  the two simples in \cref{surprisesemisimple} have the same $W(C_n)$-orbit of residues, but do not belong to the same block.
\end{rmk}

\begin{Acknowledgements*}
The first  author is grateful for funding from EPSRC grant EP/V00090X/1.  This paper was written while all five authors were at the ICERM programme  {`Categorification and Computation in Algebraic Combinatorics'} in Autumn 2025.
\end{Acknowledgements*}

 \scalefont{0.9}


\begin{thebibliography}{dGNPR05}
\scalefont{0.9}
\bibitem[AP23]{MR4666131}
A.~Appel and T.~Prze\'zdziecki, \emph{Generalized {S}chur-{W}eyl dualities for
  quantum affine symmetric pairs and orientifold {KLR} algebras}, Adv. Math.
  \textbf{435} (2023), Paper No. 109383, 52. 

 
\bibitem[BN05]{MR2174270}
D.~Bar-Natan, \emph{Khovanov's homology for tangles and cobordisms}, Geom.
  Topol. \textbf{9} (2005), 1443--1499. 

\bibitem[Bax82]{statmex1}
R.~J. Baxter, \emph{Exactly solved models in statistical mechanics}, Academic
  Press, London, 1982.


\bibitem[Bow22]{MR4401509}
C.~Bowman, \emph{The many integral graded cellular bases of {H}ecke algebras of
  complex reflection groups}, Amer. J. Math. \textbf{144} (2022), no.~2,
  437--504. 

\bibitem[Bow25]{MR4911527}
\bysame, \emph{Diagrammatic algebra}, Universitext, Springer,  
  \copyright 2025, With a foreword by Geordie Williamson. 


\bibitem[BC18]{MR3820251}
C.~Bowman and A.  Cox, \emph{Modular decomposition numbers of cyclotomic
  {H}ecke and diagrammatic {C}herednik algebras: a path theoretic approach},
  Forum Math. Sigma \textbf{6} (2018), Paper No. e11, 66. 


\bibitem[BCH23]{MR4611117}
C.~Bowman, A.~Cox, and A.~Hazi, \emph{Path isomorphisms between quiver {H}ecke
  and diagrammatic {B}ott-{S}amelson endomorphism algebras}, Adv. Math.
  \textbf{429} (2023), Paper No. 109185, 106.


 





\bibitem[BDD+]{BDDHMS2}
C.~Bowman, A.~Dell'Arciprete, M.~{De Visscher}, A.~Hazi, R.~Muth, and
  C.~Stroppel, \emph{Quiver presentations and {S}chur--{W}eyl duality for
  {K}hovanov arc algebras}, (to appear in Math. Z.) \href{https://arxiv.org/abs/2411.15520}{\color{black}arXiv:2411.15520}.

 \bibitem[BHDS]{stropandus}
C.~Bowman, A.~Hazi, M.~{De~{V}isscher}, and C.~Stroppel, \emph{Quiver
  presentations and isomorphisms of Hecke categories and Khovanov arc
  algebras}, \href{https://arxiv.org/abs/2309.13695}{\color{black}arXiv:2309.13695}.



\bibitem[BK09]{bk09}
J.~Brundan and A.~Kleshchev,
  \emph{\href{http://dx.doi.org/10.1016/j.aim.2009.06.018}{\color{black}Graded decomposition
  numbers for cyclotomic {{H}ecke} algebras}}, Adv.\ Math. \textbf{222} (2009),
  no.~6, 1883--1942.

\bibitem[BKW11]{bkw11}
J.~Brundan, A.~Kleshchev, and W.~Wang,
  \emph{\href{http://dx.doi.org/10.1515/CRELLE.2011.033}{\color{black}Graded {Specht}
  modules}}, J.\ Reine Angew.\ Math. \textbf{655} (2011), 61--87.

 
\bibitem[BS11]{MR2918294}
J.~Brundan and C.~Stroppel, \emph{Highest weight categories arising from
  {K}hovanov's diagram algebra {I}: cellularity}, Mosc. Math. J. \textbf{11}
  (2011), no.~4, 685--722, 821--822.

\bibitem[BW18]{MR3864017}
H.~Bao and W.~Wang, \emph{A new approach to {K}azhdan-{L}usztig theory of type
  {$B$} via quantum symmetric pairs}, Ast\'erisque (2018), no.~402, vii+134.

\bibitem[dGN09]{GN}
J.~de~Gier and A. Nichols, \emph{The two-boundary {T}emperley--{L}ieb
  algebra}, J. Algebra \textbf{321} (2009), no.~4, 1132--1167. 

\bibitem[dGNPR05]{statmex6}
J.~de~Gier, A.~Nichols, P.~Pyatov, and V. Rittenberg, \emph{Magic in the
  spectra of the {XXZ} quantum chain with boundaries at {$\Delta=0$} and
  {$\Delta=-1/2$}}, Nuclear Physics B \textbf{729} (2005), 387--418.

\bibitem[dGP04]{statmex5}
J.~de~Gier and P. Pyatov, \emph{Bethe ansatz for the {T}emperley--{L}ieb
  loop model with open boundaries}, Journal of Statistical Mechanics: Theory
  and Experiment \textbf{2004} (2004), no.~03, P03002.

\bibitem[DR25a]{DR25b}
Z.~Daugherty and A.~Ram, \emph{Calibrated representations of two boundary
  {T}emperley--{L}ieb algebras}, Ann. Represent. Theory \textbf{2} (2025),
  405--438. 

\bibitem[DR25b]{DR25a}
\bysame, \emph{Two boundary {H}ecke algebras and combinatorics of type {$C$}},
  Ann. Represent. Theory \textbf{2} (2025),   355--404. 

\bibitem[EK06]{MR2279279}
N.~Enomoto and M.~Kashiwara, \emph{Symmetric crystals and affine {H}ecke
  algebras of type {B}}, Proc. Japan Acad. Ser. A Math. Sci. \textbf{82}
  (2006), no.~8, 131--136. 

\bibitem[EL]{ELpaper}
B.~Elias and I.~Losev, \emph{Modular representation theory in type $A$ via
  Soergel bimodules},
  \href{https://arxiv.org/pdf/1701.00560.pdf}{\color{black}arXiv:1701.00560}.

\bibitem[Eli10]{MR2726291}
B.~Elias, \emph{A diagrammatic {T}emperley--{L}ieb categorification}, Int. J.
  Math. Math. Sci. (2010), Art. ID 530808, 47. 

\bibitem[Ern12]{MR2927180}
D.~Ernst, \emph{Diagram calculus for a type affine {$C$} {T}emperley--{L}ieb
  algebra, {I}}, J. Pure Appl. Algebra \textbf{216} (2012),   67--88.

\bibitem[Ern18]{MR3818281}
\bysame, \emph{Diagram calculus for a type affine {$C$} {T}emperley--{L}ieb
  algebra, {II}}, J. Pure Appl. Algebra \textbf{222} (2018),  
  795--830. 

\bibitem[GMP08]{statmexblob3}
R.~M. Green, P.~P. Martin, and A.~E. Parker, \emph{On the non-generic
  representation theory of the symplectic blob algebra}, arXiv preprint (2008).

\bibitem[GMP12]{MR2928127}
\bysame, \emph{A presentation for the symplectic blob algebra}, J. Algebra
  Appl. \textbf{11} (2012), no.~3, 1250060, 22. 

\bibitem[GMP17]{statmexblob2}
\bysame, \emph{On quasi-heredity and cell module homomorphisms in the
  symplectic blob algebra}, arXiv:0807.4101.



\bibitem[HGP19]{statmexblob66}
A.~Harbat, C.~Gonz\'alez, and D.~Plaza, \emph{Type $\widehat{C}$ Temperley--Lieb
  algebra quotients and Catalan combinatorics}, arXiv:1904.08351.


\bibitem[HM10]{hm10}
J.~Hu and A.~Mathas,
  \emph{\href{http://dx.doi.org/10.1016/j.aim.2010.03.002}{\color{black}Graded cellular
  bases for the cyclotomic {Khovanov}--{Lauda}--{Rouquier} algebras of type\
  {$A$}}}, Adv.\ Math. \textbf{225} (2010), no.~2, 598--642.

\bibitem[Kho00]{MR1740682}
M.~Khovanov, \emph{A categorification of the {J}ones polynomial}, Duke Math. J.
  \textbf{101} (2000), no.~3, 359--426.

\bibitem[KK12]{kk12}
S.-J. Kang and M.~Kashiwara,
  \emph{\href{http://dx.doi.org/10.1007/s00222-012-0388-1}{\color{black}Categorification of
  highest weight modules via {K}hovanov--{L}auda--{R}ouquier algebras}},
  Invent.\ Math. \textbf{190} (2012), no.~3, 699--742.

\bibitem[KMP16]{statmexblob1}
O.~H. King, P.~P. Martin, and A.~E. Parker, \emph{Decomposition matrices and
  blocks for the symplectic blob algebra over the complex field}, arXiv
  preprint (2016), 41 pages; arXiv:1611.06968 [math.RT].

\bibitem[KN10]{KN10}
A.~Kleshchev and D.~Nash,
  \emph{\href{http://dx.doi.org/10.1080/00927870903386536}{\color{black}An interpretation of
  the {Lascoux--Leclerc--Thibon} algorithm and graded representation theory}},
  Comm.\ Algebra \textbf{38} (2010), no.~12, 4489--4500.

\bibitem[LLT96]{LLT}
A.~Lascoux, B.~Leclerc, and J.-Y. Thibon,
  \emph{\href{http://projecteuclid.org/euclid.cmp/1104287629}{\color{black}{H}ecke algebras
  at roots of unity and crystal bases of quantum affine algebras}}, Comm.\
  Math.\ Phys. \textbf{181} (1996), no.~1, 205--263.

\bibitem[LP20]{MR4100120}
N.~Libedinsky and D.~Plaza, \emph{Blob algebra approach to modular
  representation theory}, Proc. Lond. Math. Soc. (3) \textbf{121} (2020),
  no.~3, 656--701. 

\bibitem[LS22]{MR4353348}
A.~D. Lauda and J.~Sussan, \emph{An invitation to categorification}, Notices
  Amer. Math. Soc. \textbf{69} (2022), no.~1, 11--21. 

\bibitem[LV11]{MR2822211}
A.~D. Lauda and M.~Vazirani, \emph{Crystals from categorified quantum groups},
  Adv. Math. \textbf{228} (2011), no.~2, 803--861. 

\bibitem[Man22]{Manolescu2022FourDimensionalTopology}
C.~Manolescu, \emph{Four-dimensional topology},
  \url{https://web.stanford.edu/~cm5/4D.pdf}, 2022, Accessed: 2025-10-29.

\bibitem[MGP07]{MR2354870}
P.~Martin, R.~M. Green, and A.~Parker, \emph{Towers of recollement and bases
  for diagram algebras: planar diagrams and a little beyond}, J. Algebra
  \textbf{316} (2007), no.~1, 392--452. 

\bibitem[MS94]{MR1267001}
P.~Martin and H.~Saleur, \emph{The blob algebra and the periodic
  {T}emperley--{L}ieb algebra}, Lett. Math. Phys. \textbf{30} (1994), no.~3,
  189--206. 

\bibitem[MW00]{MW00}
P.~Martin and D.~Woodcock,
  \emph{\href{http://dx.doi.org/10.1006/jabr.1999.7948}{\color{black}On the structure of the
  blob algebra}}, J.\ Algebra \textbf{225} (2000), no.~2, 957--988.

\bibitem[Nic06a]{statmex32}
A.~Nichols, \emph{Structure of the two-boundary {XXZ} model with non-diagonal
  boundary terms}, Journal of Statistical Mechanics: Theory and Experiment
  (2006), no.~02, L02004.

\bibitem[Nic06b]{statmex31}
\bysame, \emph{The {T}emperley--{L}ieb algebra and its generalizations in the
  {Potts} and {XXZ} models}, Journal of Statistical Mechanics: Theory and
  Experiment (2006), no.~01, P01003.

\bibitem[PdR21]{MR4250039}
L.~Poulain~d'Andecy and S.~Rostam, \emph{Morita equivalences for cyclotomic
  {H}ecke algebras of types {B} and {D}}, Bull. Soc. Math. France \textbf{149}
  (2021), no.~1, 179--233. 

\bibitem[PdW20]{MR4085039}
L.~Poulain~d'Andecy and R.~Walker, \emph{Affine {H}ecke algebras and
  generalizations of quiver {H}ecke algebras of type {$B$}}, Proc. Edinb. Math.
  Soc. (2) \textbf{63} (2020), no.~2, 531--578. 

\bibitem[Pic20]{MR4076631}
L.~Piccirillo, \emph{The {C}onway knot is not slice}, Ann. of Math. (2)
  \textbf{191} (2020), no.~2, 581--591.

\bibitem[Pla13]{Pla13}
D.~Plaza, \emph{\href{http://dx.doi.org/10.1016/j.jalgebra.2013.07.017}{\color{black}Graded
  decomposition numbers for the blob algebra}}, J.\ Algebra \textbf{394}
  (2013), 182--206.

\bibitem[PRH14]{PR13}
D.~Plaza and S.~Ryom-Hansen,
  \emph{\href{http://dx.doi.org/10.1007/s10801-013-0481-6}{\color{black}Graded cellular
  bases for {{T}emperley--{L}ieb} algebras of type {$A$} and {$B$}}}, J.\
  Algebraic Combin. \textbf{40} (2014), no.~1, 137--177.

\bibitem[Ree18]{statmexblob67}
A.~Reeves, \emph{Tilting modules for the symplectic blob algebra}, arXiv
  preprint (2018).

\bibitem[Str]{ICM1}
C.~Stroppel, \emph{Categorification: tangle invariants and {TQFT}s},
  \href{https://arxiv.org/abs/2207.05139}{\color{black}arXiv:2207.05139}.

\bibitem[Str09]{MR2521250}
\bysame, \emph{Parabolic category {$\mathscr O$}, perverse sheaves on
  {G}rassmannians, {S}pringer fibres and {K}hovanov homology}, Compos. Math.
  \textbf{145} (2009), no.~4, 954--992.

\bibitem[TL71]{MR498284}
H.~N.~V. {T}emperley and E.~H. {L}ieb, \emph{Relations between the
  ``percolation'' and ``colouring'' problem and other graph-theoretical
  problems associated with regular planar lattices: some exact results for the
  ``percolation'' problem}, Proc. Roy. Soc. London Ser. A \textbf{322} (1971),
  no.~1549, 251--280. 

\bibitem[VV11]{MR2827096}
M.~Varagnolo and E.~Vasserot, \emph{Canonical bases and affine {H}ecke algebras
  of type {B}}, Invent. Math. \textbf{185} (2011), no.~3, 593--693.

\end{thebibliography}
%

 \end{document}